\documentclass[nosumlimits,twoside]{amsart}
\usepackage{eurosym}
\usepackage{amsfonts, amsmath, amssymb}
\usepackage[bookmarksnumbered,plainpages,driverfallback=dvipdfm]{hyperref} 
\usepackage{graphicx}
\usepackage{float}
\usepackage{srcltx}
\usepackage[all]{xy}
\usepackage{version}
\usepackage[T1]{fontenc}
\usepackage{color}
\usepackage{enumerate}
\usepackage[normalem]{ulem}

\newtheorem{theorem}{\sc Theorem}[section]
\newtheorem{proposition}[theorem]{\sc Proposition}
\newtheorem{notation}[theorem]{\sc Notation}

\newtheorem{lemma}[theorem]{\sc Lemma}
\newtheorem{corollary}[theorem]{\sc Corollary}
\theoremstyle{definition}
\newtheorem{definition}[theorem]{\sc Definition}

\newtheorem{example}[theorem]{\sc Example}

\theoremstyle{remark}
\newtheorem{remark}[theorem]{\sc Remark}

\newcommand\id{\mathrm{Id}}
\newcommand\M{\mathcal{M}}

\newcommand\Alg{\mathrm{Alg}}
\newcommand\Bialg{\mathrm{Bialg}}

\renewcommand\Vec{\mathrm{Vec}}

\newcommand\U{\mathcal{U}}
\newcommand\yd{{_{H}^{H}\mathcal{YD}}}
\newcommand{\duplica}[1]{#1\otimes #1}

\def\pulb{\ar@{}[dr]|(0.2){\mbox{\Large{$\lrcorner$}}}}
\def\push{\ar@{}[dr]|(0.8){\mbox{\Large{$\ulcorner$}}}}

\allowdisplaybreaks
\IfFileExists{tcilatex.tex}{\input{tcilatex}}{}
\newenvironment{invisible}{{\noindent\sc \underline{\color{blue}Invisible (To be hidden)}:\quad}\color{red}}{\medskip}
\excludeversion{invisible}

\begin{document}
\title{Monadic vs Adjoint Decomposition}
\author{Alessandro Ardizzoni}
\address{%
\parbox[b]{\linewidth}{University of Turin, Department of Mathematics ``G. Peano'', via
Carlo Alberto 10, I-10123 Torino, Italy}}
\email{alessandro.ardizzoni@unito.it}
\urladdr{sites.google.com/site/aleardizzonihome}
\author{Claudia Menini}
\address{University of Ferrara, Department of Mathematics and Computer Science, Via Machiavelli
30, Ferrara, I-44121, Italy}
\email{men@unife.it}
\urladdr{sites.google.com/a/unife.it/claudia-menini}
\subjclass[2010]{Primary 18C15; Secondary 18A40}
\thanks{This paper was written while the authors were members of the
"National Group for Algebraic and Geometric Structures and their
Applications" (GNSAGA-INdAM). They were both partially supported by MIUR
within the National Research Project PRIN 2017. The authors would like to express their gratitude to the referee for the accurate checking, the careful reading and the useful comments that improved an earlier version of this paper.}

\begin{abstract}
It is known that the so-called monadic decomposition, applied to the adjunction connecting the category of bialgebras to the category of vector spaces via the tensor and the primitive functors, returns the usual adjunction between bialgebras and (restricted) Lie algebras.
Moreover, in this framework, the notions of augmented monad and combinatorial rank play a central role. In order to set these results into a wider context, we are led to substitute the monadic decomposition by what we call the adjoint decomposition. This construction has the
advantage of reducing the computational complexity when compared to the first
one. We connect the two decompositions by means of an embedding and we
investigate its properties by using a relative version of Grothendieck
fibration. As an application, in this wider setting, by using the notion of augmented monad, we introduce a notion of combinatorial rank  that, among other things, is expected to give some hints
on the length of the monadic decomposition.
\end{abstract}

\keywords{Adjunctions, Monads, Fibrations, Separability, Combinatorial Rank}
\maketitle
\tableofcontents

\section*{Introduction}

Let $\mathcal{A}$ be a category with all coequalizers and let $L\dashv R:\mathcal{A}\to \mathcal{B} $ be an adjunction with
unit $\eta :\mathrm{Id}\rightarrow RL$ and counit $\epsilon :LR\rightarrow
\mathrm{Id}$. Consider the Eilenberg-Moore category $\mathcal{B}_{1}$ of algebras
over the monad $\left( RL,R\epsilon L,\eta \right) .$ Then the comparison
functor $R_{1}:\mathcal{A}\rightarrow \mathcal{B}_{1}$ has a left adjoint $%
L_{1},$ with unit $\eta _{1}:\mathrm{Id}\rightarrow R_{1}L_{1}$ and counit $%
\epsilon _{1}:L_{1}R_{1}\rightarrow \mathrm{Id}$, and we can compute the
Eilenberg-Moore  category $\mathcal{B}_{2}$ of algebras over the monad $\left(
R_{1}L_{1},R_{1}\epsilon _{1}L_{1},\eta _{1}\right) .$ Going on this way we
obtain a tower
\begin{equation*}
\vcenter{\xymatrixcolsep{2cm}\xymatrixrowsep{0.7cm}
\xymatrix{\mathcal{A}\ar@<.5ex>[d]^{R}&\mathcal{A}\ar@<.5ex>[d]^{R_1}%
\ar[l]_{\mathrm{Id}_{\mathcal{A}}}&\mathcal{A}\ar@<.5ex>[d]^{R_2}\ar[l]_{%
\mathrm{Id}_{\mathcal{A}}}&\cdots
\ar[l]_{\mathrm{Id}_{\mathcal{A}}}\quad\cdots&\mathcal{A}\ar@<.5ex>[d]^{R_N}%
\ar[l]_{\mathrm{Id}_{\mathcal{A}}}\\
\mathcal{B}\ar@<.5ex>@{.>}[u]^{L}&\mathcal{B}_1\ar@<.5ex>@{.>}[u]^{L_1}
\ar[l]_{U_{0,1}}&\mathcal{B}_2 \ar@<.5ex>@{.>}[u]^{L_2}
\ar[l]_{U_{1,2}}&\cdots\ar[l]_{U_{2,3}}\quad\cdots&\mathcal{B}_N%
\ar@<.5ex>@{.>}[u]^{L_N}\ar[l]_{U_{N-1,N}} }}
\end{equation*}%
where $U_{n,n+1}$ denotes the forgetful functor and $U_{n,n+1}\circ
R_{n+1}=R_{n}.$ If this process stops exactly after $N$ steps, meaning that $%
N$ is the smallest positive integer such that $U_{N,N+1}$ is a category
isomorphism, then $R$ is said to have a\emph{\ monadic decomposition of
monadic length }$N$. For relevant outcomes of this notion we refer to \cite%
{AGM-MonadicLie1, AGM-Restricted, AM-MM}. We just mention here how our interest in this construction stems from the case when $(L,R)$ is the adjunction $(\widetilde{T},P)$, where $P$ is the functor that associates to any bialgebra, over a base field $\Bbbk$, its space of primitive elements and its left adjoint $\widetilde{T}$ associates to a vector space $V$ its tensor algebra $TV$ endowed with the usual bialgebra structure in which the elements of $V$ are primitive. One of the outcomes of the papers quoted above is the existence of an equivalence $\Lambda$ between the category $\Vec_2$ and the category $\mathrm{Lie}$ of either Lie algebras, if $\mathrm{char}(\Bbbk)=0$, or restricted Lie algebras, if $\mathrm{char}(\Bbbk)>0$. Moreover one has $\Lambda\circ P_2=\mathcal{P}$ and $H\circ\Lambda=U_{0,1}U_{1,2}$. Here $(\mathcal{U},\mathcal{P})$ is the usual adjunction between $\Bialg$ and $\mathrm{Lie}$ given by the (restricted) universal enveloping algebra functor and the primitive functor.
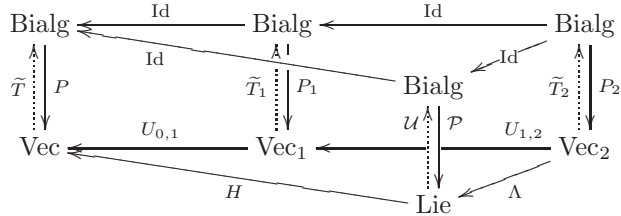
\begin{equation*}
\xymatrixcolsep{1cm}\xymatrixrowsep{0.30cm}\xymatrix{\Bialg%
\ar@<.5ex>[dd]^{P}&&\Bialg\ar@<.5ex>[dd]^{P_1}|(.30)\hole%
\ar[ll]_{\id}&&\Bialg%
\ar@<.5ex>[dd]^{P_2}\ar[ll]_{\id}\ar[dl]|{\id}\\
&&&\Bialg\ar@<.5ex>[dd]^(.30){\mathcal{P}}\ar[ulll]^(.70){%
\id}\\
\Vec\ar@<.5ex>@{.>}[uu]^{\widetilde{T}}&&\Vec_1%
\ar@<.5ex>@{.>}[uu]^{ \widetilde{T}_1}|(.70)\hole
\ar[ll]_{U_{0,1}}&&\Vec_2 \ar@<.5ex>@{.>}[uu]^{\widetilde{T}_2}
\ar[ll]_(.20){U_{1,2}}|\hole \ar[dl]^{\Lambda} \\
&&&\mathrm{Lie}\ar@<.5ex>@{.>}[uu]^(.70){\mathcal{U}}\ar[ulll]^{H}}
\end{equation*}
The starting functor $P$ comes out to have monadic decomposition of monadic length at most $2$ and this reflects the fact that the functor $\mathcal{U}$ is fully faithful, or equivalently the unit $\id\to\mathcal{PU}$ of the adjunction $(\mathcal{U},\mathcal{P})$ is invertible, which is part of the so-called Milnor-Moore theorem. Thus if the input $(L,R)$ of the monadic decomposition procedure is the adjunction $(\widetilde{T},P)$, then the corresponding output, when the iteration stops, is the adjunction $(\mathcal{U},\mathcal{P})$ up to equivalence.

We point out that unit $\widetilde{\eta}$ of the adjunction $(\widetilde{T},P)$ splits via a suitable natural retraction $\gamma:P\widetilde{T}\to\id$, i.e. an augmentation for the associated monad, that allows to define a functor $\Gamma_1:\Vec\to \Vec_1, V\mapsto (V,\gamma V)$. The composite functor $S_1:=\widetilde{T}_1\Gamma_1:\Vec\to \Bialg$ associates to a vector space $V$ the tensor bialgebra $\widetilde{T}V$ factored out by the ideal generated by its homogeneous primitive elements of degree at least two. If $\widetilde{\eta}_1$ denotes the unit of $(\widetilde{T}_1,P_1)$, it comes out that $\widetilde{\eta}_1\Gamma_1V$ is invertible for every $V$ (this is equivalent to ask that $V$ has \emph{combinatorial rank} at most one) and this plays a central role in proving that  the iteration stops after two steps.

It is natural to wonder what happens to monadic decomposition if we substitute the category of vector spaces over $\Bbbk$ and the category of bialgebras over $\Bbbk$ by an arbitrary braided monoidal category $\M$ and the category $\Bialg(\M)$ of bialgebras in $\M$ respectively, once we made the proper assumptions on $\M$ to have an analogue of the adjunction $(\widetilde{T},P)$. Partial results have been obtained in \cite{AM-MM} giving rise to the notion of Milnor-Moore Category. It is worth to notice that, to the best of our knowledge, even in the more restrictive case when $\M$ is a symmetric monoidal category it is an open problem whether the monadic length is still at most $2$.

In order to look at the problem from a more general perspective, unconstrained by the particular features of the examples considered above, we think one has to investigate the stationarity of monadic decomposition at the level of  an arbitrary adjunction $(L,R)$. The notions of augmented monad and of combinatorial rank, mentioned above, are expected to play a relevant role in the picture. Moreover, since the procedure may, in principle, stop at some level higher than $2$, the functor $\Gamma_1$, arising from the augmentation, should be extended to some functor $\Gamma_n:\mathcal{B}\to \mathcal{B}_n$. A first attempt to define such a functor shows how it is inconvenient to prove that the candidate object $\Gamma_n B$ belongs to $\mathcal{B}_n$ for every $B$ in $\mathcal{B}$. This is due to the fact that to test if an object
belongs to this category several equalities have to be checked. The first
aim of this paper is to reduce drastically the number of equalities to
verify by replacing the category $\mathcal{B}_{n}$ by a new category $%
\mathcal{B}_{\left[ n\right] }$. More precisely, we construct a kind of
monadic decomposition that we call an \emph{adjoint decomposition }as
follows,
\begin{equation*}
\xymatrixcolsep{2cm}\xymatrixrowsep{0.7cm}\xymatrix{\mathcal{A}%
\ar@<.5ex>[d]^{R}&\mathcal{A}\ar@<.5ex>[d]^{R_{[1]}}\ar[l]_{\mathrm{Id}}&%
\mathcal{A}\ar@<.5ex>[d]^{R_{[2]}}\ar[l]_{\mathrm{Id}}&\cdots
\ar[l]_{\mathrm{Id}}\quad\cdots&\mathcal{A}\ar@<.5ex>[d]^{R_{[n]}}\ar[l]_{%
\mathrm{Id}}\\
\mathcal{B}\ar@<.5ex>@{.>}[u]^{L}&\mathcal{B}_{[1]}%
\ar@<.5ex>@{.>}[u]^{L_{[1]}} \ar[l]_{U_{[0,1]}}&\mathcal{B}_{[2]}
\ar@<.5ex>@{.>}[u]^{L_{[2]}}
\ar[l]_{U_{[1,2]}}&\cdots\ar[l]_{U_{[2,3]}}\quad\cdots&\mathcal{B}_{[n]}%
\ar@<.5ex>@{.>}[u]^{L_{[n]}}\ar[l]_{U_{[n-1,n]}}}
\end{equation*}%
where $(L_{[n]},R_{[n]},\epsilon _{\lbrack n]},\eta _{\lbrack n]})$ is a
suitable adjunction. Denote by $U_{[a,b]}:\mathcal{B}_{[b]}\rightarrow
\mathcal{B}_{[a]}$ the composite functor $U_{[a,a+1]}\circ
U_{[a+1,a+2]}\circ \cdots \circ U_{[b-2,b-1]}\circ U_{[b-1,b]}$ for all $%
a\leq b.$ An object in $\mathcal{B}_{[n]}$ is a pair $B_{[n]}:=(B_{[n-1]},b_{%
\left[ n-1\right] })$ where $B_{[n-1]}$ is an object in $\mathcal{B}_{[n-1]}$
and $b_{\left[ n-1\right] }:RL_{\left[ n-1\right] }B_{\left[ n-1\right]
}\rightarrow U_{[0,n-1]}B_{[n-1]}$ is a morphism in $\mathcal{B}$. Thus it can
be regarded as a datum $B_{[n]}:=(B_{[0]},b_{\left[ 0\right] },b_{\left[ 1%
\right] },\ldots ,b_{\left[ n-1\right] })$, where $b_{\left[ t\right] }:RL_{%
\left[ t\right] }B_{\left[ t\right] }\rightarrow B_{[0]}$ and $%
B_{[t]}:=U_{[t,n]}B_{[n]}$, for each $t\in \left\{ 0,\ldots ,n-1\right\} $.
A morphism $f_{\left[ n\right] }:B_{\left[ n\right] }\rightarrow B_{\left[ n%
\right] }^{\prime }$ in $\mathcal{B}_{\left[ n\right] }$ is a morphism $f_{%
\left[ n-1\right] }:B_{\left[ n-1\right] }\rightarrow B_{\left[ n-1\right]
}^{\prime }$ in $\mathcal{B}_{\left[ n-1\right] }$ such that $U_{\left[ 0,n-1%
\right] }f_{\left[ n-1\right] }\circ b_{[n-1]}=b_{[n-1]}^{\prime }\circ
RL_{[n-1]}f_{[n-1]}$.

For every $n\geq 1,$ we can construct a fully faithful functor $\Lambda _{n}:%
\mathcal{B}_{n}\rightarrow \mathcal{B}_{[n]}$ which satisfies the equalities
$\Lambda _{n}\circ R_{n}=R_{[n]}$ and $U_{\left[ n-1,n\right] }\circ \Lambda
_{n}=\Lambda _{n-1}\circ U_{n-1,n}$ i.e. that makes commutative the solid
faces of the following diagram.
\begin{equation}
\vcenter{\xymatrixcolsep{1.5cm}\xymatrixrowsep{0.3cm}\xymatrix{&\mathcal{A}%
\ar[dl]|{\mathrm{Id}}\ar@<.5ex>[dd]^(.70){R_{n-1}}|\hole&&\mathcal{A}%
\ar@<.5ex>[dd]^{R_n}\ar[ll]_{\mathrm{Id}}\ar[dl]|{\mathrm{Id}}\\
\mathcal{A}\ar@<.5ex>[dd]^(.30){R_{[n-1]}}&&\mathcal{A}%
\ar@<.5ex>[dd]^(.30){R_{[n]}}\ar[ll]|(.30){\mathrm{Id}}\\
&\mathcal{B}_{n-1}\ar[dl]|{\Lambda_{n-1}}%
\ar@<.5ex>@{.>}[uu]^(.30){L_{n-1}}|(.51)\hole &&\mathcal{B}_n
\ar@<.5ex>@{.>}[uu]^{L_n}
\ar[ll]^(.70){U_{n-1,n}}|(.47)\hole|(.49)\hole\ar[dl]|{\Lambda_{n}} \\
\mathcal{B}_{[n-1]}\ar@<.5ex>@{.>}[uu]^(.70){
L_{[n-1]}}&&\mathcal{B}_{[n]}\ar@<.5ex>@{.>}[uu]^(.70){
L_{[n]}}\ar[ll]^(.30){U_{[n-1,n]}}}}  \label{diag:Lambda}
\end{equation}

Furthermore we have an isomorphism $\lambda _{n}:L_{[n]}\Lambda
_{n}\rightarrow L_{n}$. By means of a relative version of Grothendieck
fibration, we are able to give sufficient conditions for an object in $%
\mathcal{B}_{[n]}$ to be the image through $\Lambda _{n}$ of an object in $%
\mathcal{B}_{n}$. As an instance of how this strategy works we construct,
under appropriate conditions, involving an augmentation for the monad $RL$, a family of functors $\Gamma _{\left[ n\right]
}:\mathcal{B}\rightarrow \mathcal{B}_{[n]}$, $n\in \mathbb{N}$, that factor
through $\Lambda _{n}$ returning the desired functor of $\Gamma_n:\mathcal{B}\to \mathcal{B}_n$.
These constructions apply to the adjunction $\widetilde{T}\dashv P:\Bialg(\M)\to \M $. In the particular case when $\M$ is the category $\yd$ of Yetter-Drinfeld modules over a Hopf algebra $H$, we obtain an explicit description of the functors $S_{[n]}:=\widetilde{T}_{[n]}\Gamma_{[n]}\cong \widetilde{T}_{n}\Gamma_{n}=:S_n$, which extend the functor $S_1$ mentioned above. The combinatorial rank of an object $V$ in $\yd$, regarded as a braided vector space through the braiding of $\yd$, is exactly the smallest $n$ such that the canonical projection $S_{[n]}V\to S_{[n+1]}V$ is invertible and in this case $S_{[n]}V$ is isomorphic to the Nichols algebra of $V$. Since the previous projection makes sense also if we start from a general adjunction $L\dashv R:\mathcal{A}\to \mathcal{B} $ and an object $B$ in $\mathcal{B}$, we are led to a notion of combinatorial rank in this wide
setting that, among other things, is expected to give some hints on the
length of the monadic decomposition. Finally we propose possible lines of future investigation.

\subsection*{Description of main results and applications}

The paper is organized as follows.

In Section \ref{sec:1} we recall the notion of monadic decomposition and the
definition of inserter category together with its properties needed in the
paper.

In Section \ref{sec:2} we revise the notion of Adjoint triangle introduced
by Dubuc. In Proposition \ref{pro:adj}, we give a procedure to associate a
new adjoint triangle to a given one. By means of this result, we construct
iteratively the adjoint decomposition.

In Section \ref{sec:3} we compare monadic and adjoint decompositions. More
explicitly, we construct a fully faithful functor $\Lambda _{n}:\mathcal{B%
}_{n}\rightarrow \mathcal{B}_{[n]},$ which is injective on objects,
connecting the two decompositions. This is obtained in Remark \ref{rem:adj}
by applying iteratively Proposition \ref{pro:5}.

In Section \ref{sec:4}, we investigate a relative version of Grothendieck
fibrations. As a byproduct, we deduce other properties of the functor $%
\Lambda _{n}.$ In particular, in Theorem \ref{teo:LambdaFibr}, we prove it
is an $\mathsf{M}\left( U_{\left[ n\right] }\right) $-fibration, where $%
\mathsf{M}\left( U_{\left[ n\right] }\right) $ stands for the class of
morphisms in $\mathcal{B}_{[n]}$ whose image in $\mathcal{B}$ via the
forgetful functor $U_{\left[ n\right] }:\mathcal{B}_{[n]}\rightarrow
\mathcal{B}$ are monomorphisms. As a consequence, in Theorem \ref%
{teo:main}, we give conditions guaranteeing that an object $B_{\left[ n%
\right] }\in \mathcal{B}_{[n]}$ is the image of an object in $\mathcal{B}_{n}$
through $\Lambda _{n}.$ These conditions enable to reduce the number of
equalities to check in order to establish that an object lives inside $%
\mathcal{B}_{n}.$ In Corollary \ref{coro:ff} we are able to prove that if $%
L_{\left[ N\right] }$ is fully faithful for some $N,$ then $R$ has a monadic
decomposition of monadic length at most $N.$

In Section \ref{sec:5} we connect these results to the notion of augmented monad. Explicitly, given a
suitable diagram involving two adjunctions $\left( L,R\right) $ and $\left(
L^{\prime },R^{\prime }\right) $, in Theorem \ref{thm:main}, we prove that
under certain assumptions, if the monad $R^{\prime }L^{\prime }$ is augmented, then so is $RL$
and we can construct a family of functors $\Gamma _{\left[ n\right] }:%
\mathcal{B}\rightarrow \mathcal{B}_{[n]}$, $n\in \mathbb{N}$. Any object of
the form $\Gamma _{\left[ n\right] }B\in \mathcal{B}_{[n]},$ with $B\in
\mathcal{B}$, fulfills the conditions mentioned above and hence it belongs
to the image of $\Lambda _{n}$. As a consequence $\Gamma _{\left[ n\right] }$
factors through $\Lambda _{n}$, see Proposition \ref{pro:Gamman}.

In Section \ref{sec:6}, we study our prototype example for Theorem \ref%
{thm:main} which also explains the relevant role played by the functors $%
\Gamma _{\left[ n\right] }$. Given a preadditive braided monoidal category $%
\mathcal{M}$ having equalizers, denumerable coproducts and coequalizers of
reflexive pairs of morphisms and such that all of them are preserved by the
tensor products, we construct a diagram, as in Theorem \ref{thm:main},
\begin{equation*}
\xymatrixcolsep{1.5cm}\xymatrixrowsep{0.7cm}\xymatrix{\Bialg(\M)\ar[r]^-{%
\mho ^{+}}\ar@<.5ex>@{.>}[d]^{P}&\Alg^+(\M)\ar@<.5ex>@{.>}[d]^{\Omega
^{+}}\\ \M\ar[r]^{\id}\ar@<.5ex>[u]^{\widetilde{T}}&\M\ar@<.5ex>[u]^{T^+}}
\end{equation*}

\begin{invisible}
\begin{equation*}
\begin{array}{ccc}
\mathrm{Bialg}\left( \mathcal{M}\right) & \overset{\mho ^{+}}{%
\longrightarrow } & \mathrm{Alg}^{+}\left( \mathcal{M}\right) \\
\widetilde{T}\uparrow \downarrow P &  & T^{+}\uparrow \downarrow \Omega ^{+}
\\
\mathcal{M} & \overset{\mathrm{Id}}{\longrightarrow } & \mathcal{M}%
\end{array}%
\end{equation*}
\end{invisible}

where $\mathrm{Bialg}\left( \mathcal{M}\right) $ is the category of
bialgebras in $\mathcal{M}$, $\mathrm{Alg}^{+}\left( \mathcal{M}\right) $ is
the category of augmented algebras in $\mathcal{M},$ $\widetilde{T}$ is the
tensor bialgebra functor, $P$ is the primitive functor, $T^{+}$ is
essentially the tensor algebra functor and $\Omega ^{+}$ associates to an
augmented algebra $\left( A,\varepsilon \right) $ the kernel in $\mathcal{M}$
of its augmentation $\varepsilon $. The functor $\mho ^{+}$ is just the
forgetful functor. By the foregoing we get a family of functors $\Gamma _{%
\left[ n\right] }:\mathcal{M}\rightarrow \mathcal{M}_{[n]}$, $n\in \mathbb{N}
$, that factor through $\Lambda _{n}$, as desired.

In Section \ref{sec:7} we describe explicitly these functors $\Gamma _{\left[
n\right] }$ in the case when $\mathcal{M}$ is the category $_{H}^{H}\mathcal{%
YD}$ of Yetter-Drinfeld modules over a finite-dimensional Hopf algebra $H,$
the particular case of the category $\mathrm{Vec}$ of vector spaces being
obtained by taking $H=\Bbbk .$ Concurrently we are lead to define a possible
analogue of the notion of combinatorial rank $\kappa \left( V,c\right) $ of
a braided vector space $\left( V,c\right) $ as defined in \cite[Section 5]%
{Ar-OntheComb} by mimicking V. K. Kharchenko's definition in \cite[%
Definition 5.4]{Kharchenko}. We refer to \cite{Kharchenko-LNM} for an
overview on the notion of combinatorial rank and its importance. Recall that
a braided vector space $\left( V,c\right) $ is a vector space $V$ endowed
with a braiding $c:V\otimes V\rightarrow V\otimes V.$ The tensor algebra $TV$
can be endowed with a braided bialgebra structure (this means to have a
braided vector space endowed with an algebra and a coalgebra structure
suitably compatible with the braiding), arising from the braiding of $V,$
that we denote by $T\left( V,c\right) .$ If we divide out $T\left(
V,c\right) $ by the ideal generated by its homogeneous primitive elements of
degree at least two we obtain a new braided bialgebra, say $S_{\left[ 1%
\right] }\left( V,c\right) $. We can repeat the same procedure on this
braided bialgebra obtaining a new quotient braided bialgebra $S_{\left[ 2%
\right] }\left( V,c\right) $ and go on this way. At the limit this procedure
yields the so-called Nichols algebra $B\left( V,c\right) $ and the number of
steps occurred is exactly $\kappa \left( V,c\right) .$

\begin{invisible}
In other words $\kappa \left( V,c\right) $ is the combinatorial rank of $%
B\left( V,c\right) $ as defined by Kharchenko in \cite[Definition 5.4]%
{Kharchenko}, once we adapt this notion from the world of character Hopf
algebras to that of braided bialgebras by replacing the free enveloping Hopf
algebra $G\left\langle X\right\rangle $, used therein, by the braided tensor
algebra $T\left( V,c\right) $.
\end{invisible}

Now it is well-known that under some finiteness conditions, a braided vector
space $\left( V,c\right) $ can be realized as an object in the category $%
_{H}^{H}\mathcal{YD}$ for some Hopf algebra $H$ and $c$ becomes the braiding
$c_{V,V}$ of $_{H}^{H}\mathcal{YD}$ applied to $V,$ see \cite[3.2.9]%
{Schauenburg}. On the other hand $_{H}^{H}\mathcal{YD}$ is a braided
monoidal category and any bialgebra in it becomes in a natural way a braided
bialgebra in the above sense if we forget the Yetter-Drinfeld module
structure and we just keep the underlying braiding, algebra and coalgebra
structures. In particular $\widetilde{T}V\in \mathrm{Bialg}\left( _{H}^{H}%
\mathcal{YD}\right) $ becomes the braided tensor algebra $T\left( V,c\right)
$ mentioned above. Define the functors $S_{\left[ n\right] }:=\widetilde{T}_{%
\left[ n\right] }\Gamma _{\left[ n\right] }:{}_{H}^{H}\mathcal{YD}%
\rightarrow \mathrm{Bialg}\left( _{H}^{H}\mathcal{YD}\right) $. In Example %
\ref{ex:YD}, we shows that $S_{\left[ n\right] }V\in \mathrm{Bialg}\left(
_{H}^{H}\mathcal{YD}\right) $ becomes the braided bialgebra $S_{\left[ n%
\right] }\left( V,c\right) $ mentioned above, for each $n\in \mathbb{N}$. As
a consequence the combinatorial rank of $V,$ regarded as braided vector
space through the braiding $c=c_{V,V}$ of $_{H}^{H}\mathcal{YD}$ as above,
is the smallest $n$ such that the canonical projection $S_{\left[ n\right]
}V\rightarrow S_{\left[ n+1\right] }V$ is invertible, if such an $n$ exists,
and in this case we have $S_{\left[ n\right] }V=\mathcal{B}\left( V,c\right)
$.

Since, in the setting of Theorem \ref{thm:main}, we can always define $%
S_{[n]}:=L_{[n]}\Gamma _{\lbrack n]}:\mathcal{B}\rightarrow \mathcal{A}$,
for every $B\in \mathcal{B}$ we are lead to define (see Definition \ref%
{def:combrank}) the \emph{combinatorial rank} of an object $B\in \mathcal{B}$%
, with respect to the adjunction $\left( L,R\right) $, to be the smallest $n$
such that the canonical projection $S_{[n]}B\rightarrow S_{[n+1]}B$ is
invertible (see Lemma \ref{lem:rank}), if such an $n$ exists. Thus a concept
of combinatorial rank can be introduced and investigated in this very
general setting in which there is neither a bialgebra nor a braided vector
spaces but just an adjunction $(L,R)$ as in Theorem \ref{thm:main}. In the
case when $\mathcal{M}$ is the category $\mathrm{Vec}$ of vector spaces and
the adjunction is $(\widetilde{T},P),$ every object in $\mathrm{Vec}$ has
combinatorial rank at most one (Example \ref{ex:vec}), but this is not true
for an arbitrary $\mathcal{M}$, e.g. the category ${_{H}^{H}\mathcal{YD}},$
see Example \ref{ex:YDrank2}. In Theorem \ref{thm:bound}, we prove that, if
the adjunction $\left( L_{N},R_{N}\right) $ is idempotent for some positive
integer $N$ (e.g. $R$ has a monadic decomposition of length $N$), then every
object in the domain of $R$ has combinatorial rank at most $N$ with respect
to the adjunction $\left( L,R\right) .$ As a corollary we obtain that every
symmetric MM-category in the sense of \cite[Definition 7.4]{AM-MM} has all
objects of combinatorial rank at most one, with respect to the adjunction $(%
\widetilde{T},P),$ see Corollary \ref{coro:MM}.

A possible idea for a future investigation is to establish whether the
general framework, in which the notion of combinatorial rank is settled
here, can give new hints on the existence of some bound for the
combinatorial rank of objects in a proper category $\mathcal{B}$ with
respect to an adjunction $\left( L,R\right) $ (or more specifically in a
category $\mathcal{M}$ with respect to the adjunction $(\widetilde{T},P)$)
as it happens in $\mathrm{Vec}$. The fact that all objects in $\mathrm{Vec}$
have combinatorial rank at most one constitutes one of the main ingredients in \cite%
{AGM-MonadicLie1} to prove that the monadic decomposition of $P:\mathrm{Bialg%
}(\mathrm{Vec})\rightarrow \mathrm{Vec}$ has length at most two. A natural
question, that we also leave to future investigations, is to determine
whether a similar result still holds in the setting of $\mathcal{B}$ as
above for the functor $R$. Such a result would be related to an analogue of
the so-called Milnor-Moore theorem, see Remark \ref{rem:Milnor-Moore}. More
generally one can ask whether the length of the monadic decomposition of the
functor $R$ is upper-bounded in case the combinatorial rank of objects in $%
\mathcal{B}$ with respect to $\left( L,R\right) $ is upper-bounded.

\section{Monadic Decomposition and Inserter Category\label{sec:1}}

Throughout this paper $\Bbbk $ will denote a field. All vector spaces and
(co)algebras will be defined over $\Bbbk $. The unadorned tensor product $%
\otimes$ will denote the tensor product over $\Bbbk $ if not stated
otherwise. We denote either by $\mathfrak{M}$ or $\mathrm{Vec}$ the category of vector
spaces.

\begin{definition}
Recall that a \emph{monad} on a category $\mathcal{A}$ is a triple $\mathbb{Q%
}:=\left( Q,m,u\right) ,$ where $Q:\mathcal{A}\rightarrow \mathcal{A}$ is a
functor, $m:QQ\rightarrow Q$ and $u:\mathcal{A}\rightarrow Q$ are functorial
morphisms satisfying the associativity and the unitality conditions $m\circ
mQ=m\circ Qm$ and $m\circ Qu=\mathrm{Id}_{Q}=m\circ uQ.$ An\emph{\ algebra}
over a monad $\mathbb{Q}$ on $\mathcal{A}$ (or simply a $\mathbb{Q}$\emph{%
-algebra}) is a pair $\left( X,{\mu }\right) $ where $X\in \mathcal{A}$ and $%
{\mu }:QX\rightarrow X$ is a morphism in $\mathcal{A}$ such that ${\mu }%
\circ Q{\mu }={\mu }\circ mX$ and ${\mu }\circ uX=\mathrm{Id}_{X}.$ A \emph{%
morphism between two} $\mathbb{Q}$-\emph{algebras} $\left( X,{\mu }\right) $
and $\left( X^{\prime },{\mu }^{\prime }\right) $ is a morphism $%
f:X\rightarrow X^{\prime }$ in $\mathcal{A}$ such that ${\mu }^{\prime
}\circ Qf=f\circ {\mu }$. We will denote by ${_{\mathbb{Q}}\mathcal{A}}$ the
category of $\mathbb{Q}$-algebras and their morphisms. This is the so-called
\emph{Eilenberg-Moore category} of the monad $\mathbb{Q}$ (which is
sometimes also denoted by ${\mathcal{A}}^{\mathbb{Q}}$ in the literature).
When the multiplication and unit of the monad are clear from the context, we
will just write $Q$ instead of $\mathbb{Q}$.
\end{definition}

A monad $\mathbb{Q}$ on $\mathcal{A}$ gives rise to an adjunction $\left(
F,U\right) :=\left( {_{\mathbb{Q}}}F,{_{\mathbb{Q}}}U\right) $ where $U:{_{%
\mathbb{Q}}\mathcal{A\rightarrow A}}$ is the forgetful functor and $F:%
\mathcal{A}\rightarrow {_{\mathbb{Q}}\mathcal{A}}$ is the free functor.
Explicitly:
\begin{equation*}
U\left( X,{\mu }\right) :=X,{\mathcal{\quad }}Uf:=f{\mathcal{\qquad }}\text{%
and}{\mathcal{\qquad }F}X:=\left( QX,mX\right) ,{\mathcal{\quad }}Ff:=Qf.
\end{equation*}%
Note that $UF=Q.$ The unit of the adjunction $\left( F,U\right) $ is given
by the unit $u:\mathcal{A}\rightarrow UF=Q$ of the monad $\mathbb{Q}$. The
counit $\lambda :FU\rightarrow {_{\mathbb{Q}}\mathcal{A}}$ of this
adjunction is uniquely determined by the equality $U\left( \lambda \left( X,{%
\mu }\right) \right) ={\mu }$ for every $\left( X,{\mu }\right) \in {_{%
\mathbb{Q}}\mathcal{A}}.$ It is well-known that the forgetful functor $U:{_{%
\mathbb{Q}}\mathcal{A}}\rightarrow \mathcal{A}$ is faithful and reflects
isomorphisms (see e.g. \cite[Proposition 4.1.4]{Borceux2}).

Let $L\dashv R:\mathcal{A}\to \mathcal{B} $ be an adjunction with unit $\eta :\id_{\mathcal{B}}\to RL$ and counit $\epsilon:LR\to\id_{\mathcal{A}}$. Then $\left( RL,R\epsilon L,\eta \right) $ is a monad on $\mathcal{B}$
and we can consider the so-called \emph{comparison functor} $K:\mathcal{A}%
\rightarrow {_{RL}\mathcal{B}}$ which
is defined by $KX:=\left( RX,R\epsilon X\right) $ and $Kf:=Rf.$ Note that $%
_{RL}U\circ K=R$.

\begin{definition}
An adjunction $L\dashv R:\mathcal{A}\to \mathcal{B} $ is called \emph{monadic} (tripleable in Beck's
terminology \cite[Definition 3, page 8]{Beck}) whenever the comparison
functor $K:\mathcal{A}\rightarrow {_{RL}}\mathcal{B}$ is an equivalence of
categories. A functor $R$ is called \emph{monadic} if it has a left adjoint $%
L$ such that the adjunction $(L,R)$ is monadic, see \cite[Definition 3',
page 8]{Beck}.
\end{definition}

\begin{definition}
\label{def:MonDec} (See \cite[Definition 2.7]{AGM-MonadicLie1}, \cite[%
Definition 2.1]{AHW} and \cite[Definitions 2.10 and 2.14]{MS}) Fix a $N\in
\mathbb{N}$. We say that a functor $R$ has a\emph{\ monadic decomposition of
monadic length }$N$\emph{\ }whenever there exists a sequence $\left(
R_{n}\right) _{n\leq N}$ of functors $R_{n}$ such that

1) $R_{0}=R$;

2) for $0\leq n\leq N$, the functor $R_{n}$ has a left adjoint functor $%
L_{n} $;

3) for $0\leq n\leq N-1$, the functor $R_{n+1}$ is the comparison functor
induced by the adjunction $\left( L_{n},R_{n}\right) $ with respect to its
associated monad;

4) $L_{N}$ is fully faithful while $L_{n}$ is not fully faithful for $%
0\leq n\leq N-1.$

Compare with the construction performed in \cite[1.5.5, page 49]{Manes-PhD}.

For $R:\mathcal{A}\to \mathcal{B}$, as above we have a diagram
\begin{equation}
\vcenter{\xymatrixcolsep{2cm}\xymatrixrowsep{.5cm}
\xymatrix{\mathcal{A}\ar@<.5ex>[d]^{R_0}&\mathcal{A}\ar@<.5ex>[d]^{R_1}%
\ar[l]_{\mathrm{Id}_{\mathcal{A}}}&\mathcal{A}\ar@<.5ex>[d]^{R_2}\ar[l]_{%
\mathrm{Id}_{\mathcal{A}}}&\cdots
\ar[l]_{\mathrm{Id}_{\mathcal{A}}}\quad\cdots&\mathcal{A}\ar@<.5ex>[d]^{R_N}%
\ar[l]_{\mathrm{Id}_{\mathcal{A}}}\\
\mathcal{B}_0\ar@<.5ex>@{.>}[u]^{L_0}&\mathcal{B}_1\ar@<.5ex>@{.>}[u]^{L_1}
\ar[l]_{U_{0,1}}&\mathcal{B}_2 \ar@<.5ex>@{.>}[u]^{L_2}
\ar[l]_{U_{1,2}}&\cdots\ar[l]_{U_{2,3}}\quad\cdots&\mathcal{B}_N%
\ar@<.5ex>@{.>}[u]^{L_N}\ar[l]_{U_{N-1,N}} }}  \label{diag:MonadicDec}
\end{equation}

where $\mathcal{B}_{0}=\mathcal{B}$ and, for $1\leq n\leq N,$

\begin{itemize}
\item $\mathcal{B}_{n}$ is the category of $ Q_{n-1}$-algebras ${_{Q_{n-1}}\mathcal{B}_{n-1}}$, where $Q_{n-1}:=R_{n-1}L_{n-1}$;

\item $U_{n-1,n}:\mathcal{B}_{n}\rightarrow \mathcal{B}_{n-1}$ is the
forgetful functor ${_{Q_{n-1}}}U$.
\end{itemize}

We will denote by $\eta _{n}:\mathrm{Id}_{\mathcal{B}_{n}}\rightarrow
R_{n}L_{n}$ and $\epsilon _{n}:L_{n}R_{n}\rightarrow \mathrm{Id}_{\mathcal{A}%
}$ the unit and counit of the adjunction $\left( L_{n},R_{n}\right) $
respectively for $0\leq n\leq N$. Note that one can introduce the forgetful
functor $U_{m,n}:\mathcal{B}_{n}\rightarrow \mathcal{B}_{m}$ for all $m\leq
n $ with $0\leq m,n\leq N$.

We point out that $L_N$ is full and faithful is equivalent to the fact
that the forgetful functor $U_{N,N+1}$ is a category isomorphism, see e.g.
\cite[Remark 2.4]{AGM-MonadicLie1}.
\end{definition}

We refer to \cite[Remarks 2.8 and 2.10]{AGM-MonadicLie1} for further
comments on monadic decompositions.\medskip

We now recall the notion of inserter category which will be a crucial tool in the construction of the adjoint
decomposition.

\begin{definition}
Let $F,G:\mathcal{A}\rightarrow \mathcal{B}$ be functors. The \emph{inserter
category} $\left\langle F|G\right\rangle $ has objects the pairs $\left(
A,\alpha _{A}\right) $ where $A\in \mathcal{A}$ and $\alpha
_{A}:FA\rightarrow GA$ is a morphism in $\mathcal{B}$. A morphism $f:\left( A,\alpha
_{A}\right) \rightarrow \left( A^{\prime },\alpha _{A^{\prime }}\right) $ is
a morphism $f:A\rightarrow A^{\prime }$ in $\mathcal{A}$ such that the
following diagram commutes%
\begin{equation*}
\xymatrixcolsep{1.5cm}\xymatrixrowsep{0.5cm}\xymatrix{FA\ar[d]_{\alpha_A}%
\ar[r]^{Ff}&FA'\ar[d]^{\alpha_{A'}}\\ GA\ar[r]^{Gf}&GA' }
\end{equation*}%
If we denote by
\begin{equation*}
P=P_{\left\langle F|G\right\rangle }:\left\langle F|G\right\rangle
\rightarrow \mathcal{A},\qquad\left( A,\alpha _{A}\right) \mapsto A,\qquad
f\mapsto f
\end{equation*}%
the forgetful functor, then there is a natural transformation
\begin{equation*}
\psi :=\psi _{\left\langle F|G\right\rangle }:FP\rightarrow GP
\end{equation*}%
which is defined by $\psi \left( A,\alpha \right) =\alpha $ for every $%
\left( A,\alpha \right) \in \left\langle F\downarrow G\right\rangle .$

Given functors $F,G,F^{\prime },G^{\prime }:\mathcal{A}\rightarrow \mathcal{B%
}$ and natural transformations $\phi :F^{\prime }\rightarrow F$ and $\gamma
:G\rightarrow G^{\prime }$ we can define the functor%
\begin{equation*}
\left\langle \phi |\gamma \right\rangle :\left\langle F|G\right\rangle
\rightarrow \left\langle F^{\prime }|G^{\prime }\right\rangle, \quad\left( A,FA%
\overset{\alpha _{A}}{\rightarrow }GA\right) \mapsto \left( A,F^{\prime }A%
\overset{\phi A}{\rightarrow }FA\overset{\alpha _{A}}{\rightarrow }GA\overset%
{\gamma A}{\rightarrow }G^{\prime }A\right) ,\quad f\mapsto f.
\end{equation*}
\end{definition}

\begin{remark}
We point out that $\left\langle F|G\right\rangle $ is exactly the inserter $\mathrm{\mathbf{Insert}}(F,G) $ in the $2$-category $\mathrm{Cat%
}$, see e.g. \cite[page 157]{CS}.
\end{remark}

\begin{lemma}
\label{lem:lift}1) Let $F,G:\mathcal{A}\rightarrow \mathcal{B}$ be functors
and let $Q:\mathcal{Q}\rightarrow \mathcal{A}$ be a functor endowed with a
natural transformation $q:FQ\rightarrow GQ.$ Then there is a unique functor $%
Q\left[ q\right] :\mathcal{Q}\rightarrow \left\langle F|G\right\rangle $
such that $P\circ Q\left[ q\right] =Q$ and $\psi Q\left[ q\right] =q.$
Explicitly $Q\left[ q\right] X:=\left( QX,qX\right) \in \left\langle
F|G\right\rangle $ for every $X\in \mathcal{Q}.$ Clearly any functor $N:%
\mathcal{Q}\rightarrow \left\langle F|G\right\rangle $ is of the form $Q%
\left[ q\right] $ for $Q=PN$ and $q=\psi N.$

2)\ Let $Q\left[ q\right] ,K\left[ k\right] :\mathcal{Q}\rightarrow
\left\langle F|G\right\rangle $ be functors and let $\pi :Q\rightarrow K$ be
such that $G\pi \circ q=k\circ F\pi .$ Then there is a unique natural
transformation $\widetilde{\pi }:Q\left[ q\right] \rightarrow K\left[ k%
\right] $ such that $P\widetilde{\pi }=\pi .$ Clearly any natural
transformation $\nu :Q\left[ q\right] \rightarrow K\left[ k\right] $ is of
the form $\widetilde{\pi }$ for $\pi =P\nu .$

3)\ Let $\phi :F^{\prime }\rightarrow F$ and $\gamma :G\rightarrow G^{\prime
}.$ Then $\left\langle \phi |\gamma \right\rangle \circ Q\left[ q\right] =Q%
\left[ \gamma Q\circ q\circ \phi Q\right] .$
\end{lemma}

\begin{proof}
1)\ For every $X\in \mathcal{Q}$ define $Q\left[ q\right] X:=\left(
QX,qX:FQX\rightarrow GQX\right) \in \left\langle F|G\right\rangle .$ Given $%
f:X\rightarrow Y$ in $\mathcal{Q},$ by naturality of $q$ we have $qY\circ
FQf=GQf\circ qX$ so that $f$ induces a morphism $Q\left[ q\right] f:Q\left[ q%
\right] X\rightarrow Q\left[ q\right] Y$ such that $PQ\left[ q\right] f=Qf.$
Thus the functor $Q\left[ q\right] :\mathcal{Q}\rightarrow \left\langle
F|G\right\rangle $ is defined. Moreover $\psi Q\left[ q\right] X=\psi \left(
QX,qX\right) =qX$ so that $\psi Q\left[ q\right] =q.$ Let us check that $Q%
\left[ q\right] $ is unique. Given a functor $N:\mathcal{Q}\rightarrow
\left\langle F|G\right\rangle $ such that $PN=Q$ and $\psi N=q$ we have that
$PNX=QX$ so that $NX=\left( QX,\alpha \right) $ for some $\alpha .$ Moreover
$qX=\psi NX=\psi \left( QX,\alpha \right) =\alpha $ and hence $NX=\left(
QX,qX\right) =Q\left[ q\right] X.$ Moreover, given $f:X\rightarrow Y$ in $%
\mathcal{Q},$ we have $PNf=Qf=PQ\left[ q\right] f.$ Since $P$ is faithful,
we deduce $Nf=Q\left[ q\right] f$ and hence $W=Q\left[ q\right] .$

2) For $X\in \mathcal{Q}$ we have $G\pi X\circ qX=kX\circ F\pi X.$ Since $Q%
\left[ q\right] X:=\left( QX,qX\right) $ and $K\left[ k\right] X:=\left(
KX,kX\right) ,$ we get that $\pi X$ induces $\widetilde{\pi }X:Q\left[ q%
\right] X\rightarrow K\left[ k\right] X$ such that $P\widetilde{\pi }X=\pi
X. $ The naturality of $\pi X$ induces the one of $\widetilde{\pi }X$ so
that we get $\widetilde{\pi }:Q\left[ q\right] \rightarrow K\left[ k\right] $
such that $P\widetilde{\pi }=\pi .$

3) First we have%
\begin{eqnarray*}
P_{\left\langle F^{\prime }|G^{\prime }\right\rangle }\circ \left\langle
\phi |\gamma \right\rangle \circ Q\left[ q\right] &=&P_{\left\langle
F|G\right\rangle }\circ Q\left[ q\right] =Q \\
\psi _{\left\langle F^{\prime }|G^{\prime }\right\rangle }\left(
\left\langle \phi |\gamma \right\rangle \circ Q\left[ q\right] \right)
&=&\left( \gamma P_{\left\langle F|G\right\rangle }\circ \psi \circ \phi
P_{\left\langle F|G\right\rangle }\right) Q\left[ q\right] \\
&=&\gamma P_{\left\langle F|G\right\rangle }Q\left[ q\right] \circ \psi Q%
\left[ q\right] \circ \phi P_{\left\langle F|G\right\rangle }Q\left[ q\right]
=\gamma Q\circ q\circ \phi Q
\end{eqnarray*}%
so that $\left\langle \phi |\gamma \right\rangle \circ Q\left[ q\right] =Q%
\left[ \gamma Q\circ q\circ \phi Q\right] .$
\end{proof}

\begin{invisible}
$\left\langle F|G\right\rangle $ is a sort of pullback as follows.%
\begin{eqnarray*}
&&%
\begin{array}{ccc}
\left\langle F|G\right\rangle & \overset{P}{\longrightarrow } & \mathcal{A}
\\
P\downarrow & \overset{\psi :FP\rightarrow GP}{\Leftarrow } & F\downarrow \\
\mathcal{A} & \overset{G}{\longrightarrow } & \mathcal{B}%
\end{array}%
\begin{array}{ccc}
\mathcal{Q} & \overset{Q}{\longrightarrow } & \mathcal{A} \\
Q\downarrow & \overset{q:FQ\rightarrow GQ}{\Leftarrow } & F\downarrow \\
\mathcal{A} & \overset{G}{\longrightarrow } & \mathcal{B}%
\end{array}
\\
&\Rightarrow &\exists \mathcal{Q}\rightarrow \left\langle F|G\right\rangle
:X\mapsto \left( QX,q_{X}:FQX\rightarrow GQX\right) .
\end{eqnarray*}%
Let $F,G:\mathcal{A}\rightarrow \mathcal{B}$ be functors. We have an
equalizer of functors
\begin{equation*}
\left\langle F|G\right\rangle \overset{E}{\longrightarrow }\left\langle
F\downarrow G\right\rangle \overset{P_{1}}{\underset{P_{2}}{%
\rightrightarrows }}\mathcal{A}
\end{equation*}%
where $\left\langle F\downarrow G\right\rangle $ is the comma category while
$P_{1}$ and $P_{2}$ are the obvious projection functors. To prove this fact,
recall the comma category has objects $\left( X,Y,\alpha
_{X,Y}:FX\rightarrow GY\right) $ and morphisms $\left( f,g\right) :\left(
X,Y,\alpha _{X,Y}\right) \rightarrow \left( X^{\prime },Y^{\prime },\alpha
_{X^{\prime },Y^{\prime }}\right) $ such that%
\begin{equation*}
\begin{array}{ccc}
FX & \overset{\alpha _{X,Y}}{\longrightarrow } & GY \\
Ff\downarrow &  & Gg\downarrow \\
FX^{\prime } & \overset{\alpha _{X^{\prime },Y^{\prime }}}{\longrightarrow }
& GY^{\prime }%
\end{array}%
.
\end{equation*}%
We have a faithful functor%
\begin{equation*}
\left\langle F|G\right\rangle \overset{E}{\longrightarrow }\left\langle
F\downarrow G\right\rangle :\left( A,\alpha _{A}\right) \mapsto \left(
A,A,\alpha _{A}\right) ;\qquad f\mapsto \left( f,f\right) .
\end{equation*}%
Via this functor we get an equalizer of categories%
\begin{equation*}
\left\langle F|G\right\rangle \overset{E}{\longrightarrow }\left\langle
F\downarrow G\right\rangle \overset{P_{1}}{\underset{P_{2}}{%
\rightrightarrows }}\mathcal{A}
\end{equation*}%
where $P_{1}$ and $P_{2}$ are the obvious projection functors. In fact
\begin{eqnarray*}
P_{1}E\left( A,\alpha _{A}\right) &=&P_{1}\left( A,A,\alpha _{A}\right)
=A=P_{2}\left( A,A,\alpha _{A}\right) =P_{2}E\left( A,\alpha _{A}\right) \\
P_{1}Ef &=&P_{1}\left( f,f\right) =f=P_{2}\left( f,f\right) =P_{2}Ef.
\end{eqnarray*}%
Let $Q:\mathcal{C}\rightarrow \left\langle F\downarrow G\right\rangle $ be a
functor such that $P_{1}Q=P_{2}Q.$ For every $C\in \mathcal{C}$ we write $%
QC=\left( X,Y,\alpha _{X,Y}:FX\rightarrow GY\right) $ for some $X,Y,$ and $%
\alpha _{X,Y}.$ Then $X=P_{1}QC=P_{2}QC=Y$ so that $\overline{Q}C:=\left(
X,\alpha _{X,X}:FX\rightarrow GX\right) \in \left\langle F|G\right\rangle $
and $QC=E\overline{Q}C.$ Given $h\in \mathcal{C}$ we write $Qh=\left(
f,g\right) $ for some $f$ and $g.$ Then $f=P_{1}Qh=P_{2}Qh=g$ so that $%
\overline{Q}h:=f\in \left\langle F|G\right\rangle .$ Then $Qh=E\overline{Q}%
h. $ Clearly this define a functor $\overline{Q}:\mathcal{C}\rightarrow
\left\langle F|G\right\rangle $ such that $E\overline{Q}=Q.$ Since $E$ is
injective on morphisms and objects, $\overline{Q}$ is unique.
\end{invisible}

\begin{proposition}
\label{pro:aureo}Consider the forgetful functor $P=P_{\left\langle
F|G\right\rangle }:\left\langle F|G\right\rangle \rightarrow \mathcal{A}.$
Let $f:\left( A,a\right) \rightarrow \left( C,c\right) $ and $g:\left(
B,b\right) \rightarrow \left( C,c\right) $ morphisms in $\left\langle
F|G\right\rangle $ and let $h:A\rightarrow B$ be a morphism in $\mathcal{A}$
such that $Pf=Pg\circ h.$ If $GPg$ is a monomorphism, then there is a
(unique) morphism $h^{\prime }:\left( A,a\right) \rightarrow \left(
B,b\right) $ such that $Ph^{\prime }=h$ and $f=g\circ h^{\prime }.$

\begin{invisible}
\begin{equation*}
\begin{array}{ccc}
A &  &  \\
h\downarrow & \searrow Pf &  \\
B & \overset{Pg}{\longrightarrow } & C%
\end{array}%
\begin{array}{ccc}
\left( A,a\right) &  &  \\
\exists h^{\prime }\downarrow & \searrow f &  \\
\left( B,b\right) & \overset{g}{\longrightarrow } & \left( C,c\right)%
\end{array}%
\end{equation*}
\end{invisible}
\end{proposition}

\begin{proof}
Consider the following diagram%
\begin{equation*}
\xymatrixcolsep{1.5cm}\xymatrixrowsep{0.5cm}\xymatrix{FA\ar[d]^{a}%
\ar[r]^{Fh}&FB\ar[d]^{b}\ar[r]^{FPg}&FC\ar[d]^{c}\\GA\ar[r]^{Gh}&GB%
\ar[r]^{GPg}&GC}
\end{equation*}%
Since $f$ is a morphism in $\left\langle F|G\right\rangle $, then the
external diagram commutes, and since $g$ is a morphism in $\left\langle
F|G\right\rangle $, so does the right square. Using that $GPg$ is
a monomorphism, we deduce that the left square commutes as well i.e. $h$ induces a morphism $h^{\prime }:\left( A,a\right) \rightarrow
\left( B,b\right) $ such that $Ph^{\prime }=h.$ Now $Pf=Pg\circ h=P\left(
g\circ h^{\prime }\right) $ and $P$ is faithful imply $f=g\circ h^{\prime }.$
Since $P$ is faithful, $h^{\prime }$ is unique.
\end{proof}

\begin{example}
Recall from \cite[Definition 2.2.2]{Pi}, that given an endofunctor $F:\mathcal{A}\rightarrow \mathcal{A},$ the category of $%
F $\textbf{-algebras} (not to be confused with an Eilenberg-Moore algebra)
is $F$-$\mathrm{Alg}=\left\langle F|\mathrm{Id}_{\mathcal{A}}\right\rangle $.
\end{example}

Let $F,G:\mathcal{A}\rightarrow \mathcal{B}$ be functors and let $\epsilon
:F\rightarrow G$ be a natural transformation. If $\mathcal{B}$ has
coequalizers we can define the functor
\begin{equation}\label{def:Ueps}
\mathcal{U}\left( \epsilon \right) :\left\langle F|G\right\rangle
\rightarrow \mathcal{B}
\end{equation}%
by the following coequalizer of natural transformations%
\begin{equation}  \label{coeq:pieps}
\xymatrixcolsep{1.5cm} \xymatrix{FP\ar@<.5ex>[r]^{\psi}\ar@<-.5ex>[r]_{%
\epsilon P}&GP\ar[r]^{\pi :=\pi \left( \epsilon \right) }&
\mathcal{U}(\epsilon)}
\end{equation}

\begin{invisible}
\begin{equation}
FP\overset{\psi }{\underset{\epsilon P}{\rightrightarrows }}GP\overset{\pi
:=\pi \left( \epsilon \right) }{\rightarrow }\mathcal{U}\left( \epsilon
\right) .
\end{equation}
Given $\left( A,\alpha \right) \in \left\langle F|G\right\rangle $, the
object $\mathcal{U}\left( A,\alpha \right) $ is defined by the following
coequalizer
\begin{equation*}
FA\overset{\alpha }{\underset{\epsilon _{A}}{\rightrightarrows }}GA\overset{%
\pi _{\left( A,\alpha \right) }}{\rightarrow }\mathcal{U}\left( A,\alpha
\right) .
\end{equation*}%
Given $f:\left( A,\alpha \right) \rightarrow \left( A^{\prime },\alpha
^{\prime }\right) $ in $\left\langle F|G\right\rangle ,$ we have $\alpha
^{\prime }\circ FPf=GPf\circ \alpha .$ Thus
\begin{eqnarray*}
\pi _{\left( A^{\prime },\alpha ^{\prime }\right) }\circ GPf\circ \alpha
&=&\pi _{\left( A^{\prime },\alpha ^{\prime }\right) }\circ \alpha ^{\prime
}\circ FPf \\
&=&\pi _{\left( A^{\prime },\alpha ^{\prime }\right) }\circ \epsilon
_{A^{\prime }}\circ FPf\overset{\text{nat. }\epsilon }{=}\pi _{\left(
A^{\prime },\alpha ^{\prime }\right) }\circ GPf\circ \epsilon _{A}
\end{eqnarray*}%
so that $\pi _{\left( A^{\prime },\alpha ^{\prime }\right) }\circ GPf$
coequalizes the pair $\left( \alpha ,\epsilon _{A}\right) $ and hence there
is a unique morphism $\mathcal{U}f:\mathcal{U}\left( A,\alpha \right)
\rightarrow \mathcal{U}\left( A^{\prime },\alpha ^{\prime }\right) $ such
that%
\begin{equation*}
\mathcal{U}f\circ \pi _{\left( A,\alpha \right) }=\pi _{\left( A^{\prime
},\alpha ^{\prime }\right) }\circ GPf.
\end{equation*}%
This defines a functor $\mathcal{U}:=\mathcal{U}\left( \epsilon \right)
:\left\langle F|G\right\rangle \rightarrow \mathcal{B}$ and it means that $%
\pi =\left( \pi _{\left( A,\alpha \right) }\right) _{\left( A,\alpha \right)
\in \left\langle F|G\right\rangle }$ is a natural transformation. Given a
natural transformation $\xi :GP\rightarrow H$ such that $\xi \circ \psi =\xi
\circ \epsilon P$ we get that $\xi _{\left( A,\alpha \right) }\circ \alpha
=\xi _{\left( A,\alpha \right) }\circ \epsilon _{A}$ so that there is a
unique morphism $\overline{\xi }_{\left( A,\alpha \right) }:\mathcal{U}%
\left( A,\alpha \right) \rightarrow H\left( A,\alpha \right) $ such that $%
\overline{\xi }_{\left( A,\alpha \right) }\circ \pi _{\left( A,\alpha
\right) }=\xi _{\left( A,\alpha \right) }.$ Given $f:\left( A,\alpha \right)
\rightarrow \left( A^{\prime },\alpha ^{\prime }\right) $ in $\left\langle
F|G\right\rangle $, we get%
\begin{equation*}
\overline{\xi }_{\left( A^{\prime },\alpha ^{\prime }\right) }\circ \mathcal{%
U}f\circ \pi _{\left( A,\alpha \right) }=\overline{\xi }_{\left( A^{\prime
},\alpha ^{\prime }\right) }\circ \pi _{\left( A^{\prime },\alpha ^{\prime
}\right) }\circ GPf=\xi _{\left( A,\alpha \right) }\circ GPf=Hf\circ \xi
_{\left( A,\alpha \right) }=Hf\circ \overline{\xi }_{\left( A,\alpha \right)
}\circ \pi _{\left( A,\alpha \right) }
\end{equation*}%
and hence $\overline{\xi }_{\left( A^{\prime },\alpha ^{\prime }\right)
}\circ \mathcal{U}f=Hf\circ \overline{\xi }_{\left( A,\alpha \right) }$
which means that $\overline{\xi }:=\left( \overline{\xi }_{\left( A,\alpha
\right) }\right) _{\left( A,\alpha \right) \in \left\langle F|G\right\rangle
}:\mathcal{U}\rightarrow H$ is a natural transformation such that $\overline{%
\xi }\circ \pi =\xi .$ Moreover $\overline{\xi }$ is unique as $\pi $ is
surjective on components.

See Lemma 2.1 in the paper [Pre-torsors and Galois Comodules Over Mixed
Distributive Laws] of Claudia for a general proof.
\end{invisible}

\section{Adjoint Triangles and Adjoint Decomposition\label{sec:2}}

In this section we construct iteratively the category $\mathcal{B}_{\left[ n%
\right] }\ $and an analogue of the monadic decomposition that will be called
the adjoint decomposition. Our first aim is to obtain an analogue of the
Eilenberg-Moore category. For this purpose we will use the notion of adjoint
triangle.

\begin{definition}
\label{def:augtriang} \cite[Definition 1]{Dubuc} By an \emph{adjoint triangle%
}, we mean a diagram of functors

\begin{invisible}
\begin{equation}
\mathbb{T}:=%
\begin{array}{ccc}
\mathcal{A} & \overset{\mathrm{Id}_{\mathcal{A}}}{\longrightarrow } &
\mathcal{A} \\
L\uparrow \downarrow R & \zeta & L^{\prime }\uparrow \downarrow R^{\prime }
\\
\mathcal{B} & \overset{G}{\longrightarrow } & \mathcal{B}^{\prime }%
\end{array}%
\quad \vcenter{\xymatrixcolsep{1.5cm}\xymatrixrowsep{1.5cm}\xymatrix{&%
\mathcal{A}\ar@<.3ex>[dl]^*-<0.2cm>{^{R}}\ar@<.3ex>[dr]^*-<0.2cm>{^{R'}}%
\ar@{}[d]|-{\zeta}\\
\mathcal{B}\ar@<.3ex>@{.>}[ur]^*-<0.2cm>{^{L}}\ar[rr]^G&&\mathcal{B}'%
\ar@<.3ex>@{.>}[ul]^*-<0.2cm>{^{L'}}}}
\end{equation}
\end{invisible}

\begin{equation}
\vcenter{\xymatrixcolsep{1.5cm}\xymatrixrowsep{0.7cm}\xymatrix{\mathcal{A}\ar[r]^\id%
\ar@{}[dr]|-{\zeta}\ar@<.4ex>[d]^*-<0.2cm>{^{R}}&
\mathcal{A}\ar@<.4ex>[d]^*-<0.2cm>{^{R'}}\\
\mathcal{B}\ar@<.4ex>@{.>}[u]^*-<0.2cm>{^{L}}\ar[r]^G&\mathcal{B}'%
\ar@<.4ex>@{.>}[u]^*-<0.2cm>{^{L'}}}}  \label{diag:AugTr}
\end{equation}%
where $\left( L,R,\eta ,\epsilon \right) $ and $\left( L^{\prime },R^{\prime
},\eta ^{\prime },\epsilon ^{\prime }\right) $ are adjunctions and $%
GR=R^{\prime }$.

The letter $\zeta $ inserted in (\ref{diag:AugTr}) is the unique natural
transformation $\zeta :L^{\prime }G\rightarrow L$ such that $\epsilon \circ
\zeta R=\epsilon ^{\prime }$ namely $\zeta :=\epsilon ^{\prime }L\circ
L^{\prime }G\eta $. The convention to write this natural transformation inside the respective adjoint triangle will be useful to state our results along the paper. It is easy to
check that%
\begin{equation}
R^{\prime }\zeta \circ \eta ^{\prime }G=G\eta .  \label{form:zeta1}
\end{equation}
Note that diagram \eqref{diag:AugTr} has been drawn as a square to make
it more readable, although the two copies of $\mathcal{A}$ on the top can be
glued together to give rise, in fact, to a triangle.
\end{definition}

\begin{invisible}
Assume $\epsilon \circ \zeta R=\epsilon ^{\prime }.$ Then $\zeta =\epsilon
L\circ L\eta \circ \zeta =\epsilon L\circ \zeta RL\circ L^{\prime }G\eta
=\left( \epsilon \circ \zeta R\right) L\circ L^{\prime }G\eta =\epsilon
^{\prime }L\circ L^{\prime }G\eta .$ Assume $\zeta =\epsilon ^{\prime
}L\circ L^{\prime }G\eta .$ We compute $\epsilon \circ \zeta R=\epsilon
\circ \epsilon ^{\prime }LR\circ L^{\prime }G\eta R=\epsilon ^{\prime }\circ
L^{\prime }R^{\prime \prime }G\eta R=\epsilon ^{\prime }\circ L^{\prime
\prime }G\eta R=\epsilon ^{\prime }.$

We also have%
\begin{equation*}
R^{\prime }\zeta \circ \eta ^{\prime }G=R^{\prime }\epsilon ^{\prime }L\circ
R^{\prime }L^{\prime }G\eta \circ \eta ^{\prime }G=R^{\prime }\epsilon
^{\prime }L\circ \eta ^{\prime }GRL\circ G\eta =R^{\prime }\epsilon ^{\prime
}L\circ \eta ^{\prime }R^{\prime }L\circ G\eta =G\eta .
\end{equation*}
\end{invisible}

\begin{remark}
\label{rem:comptr}As a particular case of horizontal composition of adjoint
squares (see \cite[I,6.8 ]{Gray-FormalCat}), we can define the (horizontal)
composition $\mathbb{T}^{\prime \prime }:=\mathbb{T}^{\prime }\ast \mathbb{T}
$ of two adjoint triangles $\mathbb{T}^{\prime }$ and $\mathbb{T}$ by

\begin{invisible}
\begin{equation*}
\mathbb{T}^{\prime }:=\vcenter{\xymatrix{&\mathcal{A}%
\ar@<.3ex>[dl]^*-<0.2cm>{^{R''}}\ar@<.3ex>[dr]^*-<0.2cm>{^{R}}%
\ar@{}[d]|-(0.7){\theta}\\
\mathcal{B}''\ar@<.3ex>@{.>}[ur]^*-<0.2cm>{^{L''}}\ar[rr]_\Theta&&%
\mathcal{B}\ar@<.3ex>@{.>}[ul]^*-<0.2cm>{^{L}}}} \qquad \mathbb{T}:=%
\vcenter{\xymatrix{&\mathcal{A}\ar@<.3ex>[dl]^*-<0.2cm>{^{R}}%
\ar@<.3ex>[dr]^*-<0.2cm>{^{R'}}\ar@{}[d]|-(0.7){\zeta}\\
\mathcal{B}\ar@<.3ex>@{.>}[ur]^*-<0.2cm>{^{L}}\ar[rr]_G&&\mathcal{B}'%
\ar@<.3ex>@{.>}[ul]^*-<0.2cm>{^{L'}}}} \qquad \mathbb{T}^{\prime \prime }:=%
\vcenter{\xymatrix{&\mathcal{A}\ar@<.3ex>[dl]^(.3)*-<0.2cm>{^{R''}}%
\ar@<.3ex>[dr]^(.3)*-<0.2cm>{^{R'}}\ar@{}[d]|-(0.7){\zeta=\theta*\zeta}\\
\mathcal{B}''\ar@<.3ex>@{.>}[ur]^(.7)*-<0.2cm>{^{L''}}\ar[rr]_{G''=G%
\Theta}&&\mathcal{B}'\ar@<.3ex>@{.>}[ul]^(.7)*-<0.2cm>{^{L'}}}}
\end{equation*}
\end{invisible}

\begin{equation*}
\mathbb{T}^{\prime }:=\vcenter{\xymatrixcolsep{1.5cm}\xymatrixrowsep{.7cm}%
\xymatrix{\mathcal{A}\ar[r]^\id\ar@{}[dr]|-{\theta}%
\ar@<.4ex>[d]^*-<0.2cm>{^{R''}}& \mathcal{A}\ar@<.4ex>[d]^*-<0.2cm>{^{R}}\\
\mathcal{B}''\ar@<.4ex>@{.>}[u]^*-<0.2cm>{^{L''}}\ar[r]^\Theta&\mathcal{B}%
\ar@<.4ex>@{.>}[u]^*-<0.2cm>{^{L}}}}\qquad\mathbb{T}:=\vcenter{%
\xymatrixcolsep{1.5cm}\xymatrixrowsep{.71cm}\xymatrix{\mathcal{A}\ar[r]^\id%
\ar@{}[dr]|-{\zeta}\ar@<.4ex>[d]^*-<0.2cm>{^{R}}&
\mathcal{A}\ar@<.4ex>[d]^*-<0.2cm>{^{R'}}\\
\mathcal{B}\ar@<.4ex>@{.>}[u]^*-<0.2cm>{^{L}}\ar[r]^G&\mathcal{B}'%
\ar@<.4ex>@{.>}[u]^*-<0.2cm>{^{L'}}} }\qquad\mathbb{T}^{\prime \prime }:=%
\vcenter{\xymatrixcolsep{1.5cm}\xymatrixrowsep{.7cm}\xymatrix{\mathcal{A}%
\ar[r]^\id\ar@{}[dr]|-{\zeta''=\theta*\zeta}\ar@<.4ex>[d]^*-<0.2cm>{^{R''}}&
\mathcal{A}\ar@<.4ex>[d]^*-<0.2cm>{^{R'}}\\
\mathcal{B}''\ar@<.4ex>@{.>}[u]^*-<0.2cm>{^{L''}}\ar[r]^{G''=G\Theta}&%
\mathcal{B}'\ar@<.4ex>@{.>}[u]^*-<0.2cm>{^{L'}}}}
\end{equation*}

\begin{invisible}
\begin{equation*}
\mathbb{T}^{\prime }:=%
\begin{array}{ccc}
\mathcal{A} & \overset{\mathrm{Id}_{\mathcal{A}}}{\longrightarrow } &
\mathcal{A} \\
L^{\prime \prime }\uparrow \downarrow R^{\prime \prime } & \theta &
L\uparrow \downarrow R \\
\mathcal{B}^{\prime \prime } & \overset{\Theta }{\longrightarrow } &
\mathcal{B}%
\end{array}%
\text{ and }\mathbb{T}:=%
\begin{array}{ccc}
\mathcal{A} & \overset{\mathrm{Id}_{\mathcal{A}}}{\longrightarrow } &
\mathcal{A} \\
L\uparrow \downarrow R & \zeta & L^{\prime }\uparrow \downarrow R^{\prime }
\\
\mathcal{B} & \overset{G}{\longrightarrow } & \mathcal{B}^{\prime }%
\end{array}%
\end{equation*}%
namely the adjoint triangle $\mathbb{T}^{\prime \prime }:=\mathbb{T}^{\prime
}\ast \mathbb{T}$ defined as follows
\begin{equation*}
\mathbb{T}^{\prime \prime }:=%
\begin{array}{ccc}
\mathcal{A} & \overset{\mathrm{Id}_{\mathcal{A}}}{\longrightarrow } &
\mathcal{A} \\
L^{\prime \prime }\uparrow \downarrow R^{\prime \prime } & \zeta ^{\prime
\prime }:=\theta \ast \zeta & L^{\prime }\uparrow \downarrow R^{\prime } \\
\mathcal{B}^{\prime \prime } & \overset{G^{\prime \prime }:=G\Theta }{%
\longrightarrow } & \mathcal{B}^{\prime }%
\end{array}%
\end{equation*}
\end{invisible}

where $\theta \ast \zeta :=\left( L^{\prime }G\Theta \overset{\zeta \Theta }{%
\rightarrow }L\Theta \overset{\theta }{\rightarrow }L^{\prime \prime
}\right) .$ In fact $\left( G\Theta \right) R^{\prime \prime }=GR=R^{\prime
} $ and
\begin{equation*}
\epsilon ^{\prime \prime }\circ \left( \theta \ast \zeta \right) R^{\prime
\prime }=\epsilon ^{\prime \prime }\circ \theta R^{\prime \prime }\circ
\zeta \Theta R^{\prime \prime }=\epsilon \circ \zeta R=\epsilon ^{\prime }.
\end{equation*}
\end{remark}

To any adjoint triangle $\mathbb{T}$ as in (\ref{diag:AugTr}) we would like to associate
a new adjoint triangle $\mathbb{T}^2$ as follows%
\begin{gather}\label{diag:firstT2}
\mathbb{T}:=\vcenter{\xymatrixcolsep{1.5cm}\xymatrixrowsep{.71cm}\xymatrix{%
\mathcal{A}\ar[r]^\id\ar@{}[dr]|-{\zeta}\ar@<.4ex>[d]^*-<0.2cm>{^{R}}&
\mathcal{A}\ar@<.4ex>[d]^*-<0.2cm>{^{R'}}\\
\mathcal{B}\ar@<.4ex>@{.>}[u]^*-<0.2cm>{^{L}}\ar[r]^G&\mathcal{B}'%
\ar@<.4ex>@{.>}[u]^*-<0.2cm>{^{L'}}} } \qquad \rightsquigarrow \qquad
\mathbb{T}^2:=\vcenter{\xymatrixcolsep{1.5cm}\xymatrixrowsep{.71cm}%
\xymatrix{\mathcal{A}\ar[r]^\id\ar@<.4ex>[d]^*-<0.2cm>{^{R\left(
\mathbb{T}\right)}}& \mathcal{A}\ar@<.4ex>[d]^*-<0.2cm>{^{R}}\\
\mathrm{I}(\mathbb{T})\ar@<.4ex>@{.>}[u]^*-<0.2cm>{^{L\left(
\mathbb{T}\right)}}\ar[r]^{P(\mathbb{T})}&\mathcal{B}%
\ar@<.4ex>@{.>}[u]^*-<0.2cm>{^{L}}} }
\end{gather}

\begin{invisible}
\begin{gather*}
\mathbb{T}:=%
\begin{array}{ccc}
\mathcal{A} & \overset{\mathrm{Id}_{\mathcal{A}}}{\longrightarrow } &
\mathcal{A} \\
L\uparrow \downarrow R & \zeta & L^{\prime }\uparrow \downarrow R^{\prime }
\\
\mathcal{B} & \overset{G}{\longrightarrow } & \mathcal{B}^{\prime }%
\end{array}%
\qquad \rightsquigarrow \qquad
\begin{array}{ccc}
\mathcal{A} & \overset{\mathrm{Id}_{\mathcal{A}}}{\longrightarrow } &
\mathcal{A} \\
L\left( \mathbb{T}\right) \downarrow R\left( \mathbb{T}\right) & \zeta
\left( \mathbb{T}\right) & L^{\prime }\uparrow \downarrow R^{\prime } \\
\mathrm{I}\left( \mathbb{T}\right) & \overset{G\left( \mathbb{T}\right) }{%
\longrightarrow } & \mathcal{B}^{\prime }%
\end{array}
\\
\xymatrixcolsep{1.5cm}\xymatrixrowsep{1.5cm}\xymatrix{&\mathcal{A}%
\ar@<.3ex>[dl]^*-<0.2cm>{^{R}}\ar@<.3ex>[dr]^*-<0.2cm>{^{R'}}\ar@{}[d]|-{%
\zeta}\\
\mathcal{B}\ar@<.3ex>@{.>}[ur]^*-<0.2cm>{^{L}}\ar[rr]^G&&\mathcal{B}'%
\ar@<.3ex>@{.>}[ul]^*-<0.2cm>{^{L'}}}\qquad \rightsquigarrow \qquad %
\xymatrixcolsep{1.5cm}\xymatrixrowsep{1.5cm}\xymatrix{&\mathcal{A}%
\ar@<.3ex>[dl]^*-<0.2cm>{^{R(\mathbb{T})}}\ar@<.3ex>[dr]^*-<0.2cm>{^{R'}}%
\ar@{}[d]|-{\zeta(\mathbb{T})}\\
\textrm{I}(\mathbb{T})\ar@<.3ex>@{.>}[ur]^*-<0.2cm>{^{L(\mathbb{T})}}%
\ar[rr]^{G(\mathbb{T})}&&\mathcal{B}'\ar@<.3ex>@{.>}[ul]^*-<0.2cm>{^{L'}}}
\end{gather*}
which are faces of a tetrahedron
\begin{equation*}
\xy 0;/r.25pc/:
(-20,-10 )*{\mathcal{B}}="1";
(16,-20)*{\mathcal{B}_{[1]}}="2";
(30,0)*{\mathcal{B}'}="3";
(2,24)*{\mathcal{A}}="4";
{\ar@{->}@<.3ex>^*-<0.2cm>{^R} "4";"1" };
{\ar@{.>}@<.3ex>^*-<0.2cm>{^L} "1";"4" };
{\ar@{->}@<.3ex>^*-<0.2cm>{^{R'}} "4";"3" };
{\ar@{.>}@<.3ex>^*-<0.2cm>{^{L'}} "3";"4" };
{\ar@{->}@<.3ex>^*-<0.2cm>{^{R_{[1]}}} "4";"2" };
{\ar@{.>}@<.3ex>^*-<0.2cm>{^{L_{[1]}}} "2";"4" };
{\ar@{.>}|(.35)G "1";"3" };
{\ar_{G_{[1]}} "2";"3" };
{\ar^{U_{[0,1]}} "2";"1" };
\endxy%
\end{equation*}
\end{invisible}

First we have to introduce the category
\begin{equation*}
\mathrm{I}\left( \mathbb{T}\right) :=\left\langle R^{\prime
}L|G\right\rangle .
\end{equation*}

For any category $\mathcal{A}$ we can consider the functor%
\begin{equation*}
D:\mathcal{A}\rightarrow \left\langle \mathrm{Id}_{\mathcal{A}}|\mathrm{Id}_{%
\mathcal{A}}\right\rangle, \qquad A\mapsto \left( A,\mathrm{Id}_{A}\right)
,\qquad h\mapsto h.
\end{equation*}

If $F,G:\mathcal{A}\rightarrow \mathcal{B}$, $H:\mathcal{B}\rightarrow
\mathcal{B}^{\prime }$ and $K:\mathcal{A}^{\prime }\rightarrow \mathcal{A}$
are functors, we can define%
\begin{eqnarray*}
S^{H} :\left\langle F|G\right\rangle \rightarrow \left\langle
HF|HG\right\rangle,&\left( A,\alpha _{A}:FA\rightarrow GA\right) \mapsto
\left( A,H\alpha _{A}:HFA\rightarrow HGA\right) ,\qquad f\mapsto f \\
D^{K} :\left\langle FK|GK\right\rangle \rightarrow \left\langle
F|G\right\rangle,&\left( A^{\prime },\alpha _{A^{\prime }}:FKA^{\prime
}\rightarrow GKA^{\prime }\right) \mapsto \left( KA^{\prime },\alpha
_{A^{\prime }}:FKA^{\prime }\rightarrow GKA^{\prime }\right) ,\quad
f\mapsto Kf.
\end{eqnarray*}%
Given $\epsilon :F\rightarrow G$ we define the functor
\begin{equation*}
\mathcal{S}\left( \epsilon \right) :=\left(\mathcal{A}\overset{D}{\rightarrow }%
\left\langle \mathrm{Id}_{\mathcal{A}}|\mathrm{Id}_{\mathcal{A}%
}\right\rangle \overset{S^{G}}{\rightarrow }\left\langle G|G\right\rangle
\overset{\left\langle \epsilon |\mathrm{Id}_{\mathcal{A}}\right\rangle }{%
\rightarrow }\left\langle F|G\right\rangle \right).
\end{equation*}
Explicitly, by the notation of Lemma \ref{lem:lift}, we have
\begin{equation*}
\mathcal{S}\left( \epsilon \right) =\mathrm{Id}_{\mathcal{A}}\left[ \epsilon %
\right] :\mathcal{A}\rightarrow \left\langle F|G\right\rangle ,\qquad A\mapsto
\left( A,\epsilon A\right) ,\qquad f\mapsto f.
\end{equation*}%
Let us show how $\mathcal{S}\left( \epsilon \right)$ relates to the functor $\mathcal{U}\left( \epsilon \right)$ of \eqref{def:Ueps} in the particular case when $G=\id_{\mathcal{A}}.$

\begin{invisible}
CONTROLLARE\ CHE LA\ COSTRUZIONE\ DI\ $\mathcal{U}\left( \epsilon \right) $
NON\ SIA\ IL\ MODO\ STANDARD\ DI\ COSTRUIRE\ L'AGGIUNTO\ SINISTRO.
\end{invisible}

\begin{lemma}
\label{lem:Ueps}Let $F:\mathcal{A}\rightarrow \mathcal{A}$ be a functor and
let $\epsilon :F\rightarrow \mathrm{Id}_{\mathcal{A}}$ be a natural
transformation. Assume $\mathcal{A}$ has coequalizers. Then $\pi \left( \epsilon \right) \mathcal{S}\left( \epsilon
\right) $ is invertible and $\left( \mathcal{U}\left( \epsilon \right) ,%
\mathcal{S}\left( \epsilon \right) ,\eta \left( \epsilon \right) ,\left( \pi
\left( \epsilon \right) \mathcal{S}\left( \epsilon \right) \right)
^{-1}\right) $ is an adjunction where $\eta \left( \epsilon \right) :\mathrm{%
Id}_{\left\langle F|\mathrm{Id}_{\mathcal{A}}\right\rangle }\rightarrow
\mathcal{S}\left( \epsilon \right) \mathcal{U}\left( \epsilon \right) $ is
uniquely determined by $P\eta \left( \epsilon \right) =\pi \left( \epsilon
\right) .$
\end{lemma}

\begin{proof}
Set $\pi :=\pi \left( \epsilon \right) .$ Note that $\psi \mathcal{S}\left(
\epsilon \right) =\epsilon $ and $P\circ \mathcal{S}\left( \epsilon \right) =%
\mathrm{Id}_{\mathcal{A}}.$ Thus, if we evaluate the left-hand side coequalizer
below on $\mathcal{S}(\epsilon)$, we obtain the right-hand side one.%
\begin{equation*}
\xymatrixcolsep{1cm} \xymatrix{FP\ar@<.5ex>[r]^{\psi}\ar@<-.5ex>[r]_{%
\epsilon P}&P\ar[r]^-{\pi }& \mathcal{U}(\epsilon)} \qquad %
\xymatrix{F\ar@<.5ex>[r]^{\epsilon}\ar@<-.5ex>[r]_{\epsilon}&\id_{%
\mathcal{A}}\ar[r]^-{\pi \mathcal{S}(\epsilon)}&
\mathcal{U}(\epsilon)\mathcal{S}(\epsilon)}
\end{equation*}

\begin{invisible}
\begin{equation*}
\begin{array}{ccccc}
FP & \overset{\psi }{\underset{\epsilon P}{\rightrightarrows }} & P &
\overset{\pi }{\rightarrow } & \mathcal{U}\left( \epsilon \right)%
\end{array}%
\end{equation*}%
with $\mathcal{S}\left( \epsilon \right) $, we obtain
\begin{equation*}
\begin{array}{ccccc}
F & \overset{\epsilon }{\underset{\epsilon }{\rightrightarrows }} & \mathrm{%
Id}_{\mathcal{A}} & \overset{\pi \mathcal{S}\left( \epsilon \right) }{%
\rightarrow } & \mathcal{U}\left( \epsilon \right) \mathcal{S}\left(
\epsilon \right)%
\end{array}%
\end{equation*}
\end{invisible}

This means $\pi \mathcal{S}\left( \epsilon \right) $ is invertible. Let us
check that there is $\eta \left( \epsilon \right) :\mathrm{Id}_{\left\langle
F|\mathrm{Id}_{\mathcal{A}}\right\rangle }\rightarrow \mathcal{S}\left(
\epsilon \right) \mathcal{U}\left( \epsilon \right) $ such that $P\eta
\left( \epsilon \right) =\pi .$ We have%
\begin{equation*}
\pi \circ \psi =\pi \circ \epsilon P=\epsilon \mathcal{U}\left( \epsilon
\right) \circ F\pi .
\end{equation*}%
Since $\mathrm{Id}_{\left\langle F|\mathrm{Id}_{\mathcal{A}}\right\rangle }=P%
\left[ \psi \right] $ and $\mathcal{S}\left( \epsilon \right) \mathcal{U}%
\left( \epsilon \right) =\mathcal{U}\left( \epsilon \right) \left[ \psi
\mathcal{S}\left( \epsilon \right) \mathcal{U}\left( \epsilon \right) \right]
=\mathcal{U}\left( \epsilon \right) \left[ \epsilon \mathcal{U}\left(
\epsilon \right) \right] ,$ by Lemma \ref{lem:lift} there is a unique
natural transformation $\eta \left( \epsilon \right) :\mathrm{Id}%
_{\left\langle F|\mathrm{Id}_{\mathcal{A}}\right\rangle }\rightarrow
\mathcal{S}\left( \epsilon \right) \mathcal{U}\left( \epsilon \right) $ such
that $P\eta \left( \epsilon \right) =\pi .$ We have%
\begin{equation*}
P\eta \left( \epsilon \right) \mathcal{S}\left( \epsilon \right) =\pi
\mathcal{S}\left( \epsilon \right) =P\mathcal{S}\left( \epsilon \right) \pi
\mathcal{S}\left( \epsilon \right)
\end{equation*}%
so that $\eta \left( \epsilon \right) \mathcal{S}\left( \epsilon \right) =%
\mathcal{S}\left( \epsilon \right) \pi \mathcal{S}\left( \epsilon \right) .$
Moreover%
\begin{equation*}
\mathcal{U}\left( \epsilon \right) \eta \left( \epsilon \right) \circ \pi
=\pi \mathcal{S}\left( \epsilon \right) \mathcal{U}\left( \epsilon \right)
\circ P\eta \left( \epsilon \right) =\pi \mathcal{S}\left( \epsilon \right)
\mathcal{U}\left( \epsilon \right) \circ \pi
\end{equation*}%
and hence $\mathcal{U}\left( \epsilon \right) \eta \left( \epsilon \right)
=\pi \mathcal{S}\left( \epsilon \right) \mathcal{U}\left( \epsilon \right) .$
Therefore $\left( \mathcal{U}\left( \epsilon \right) ,\mathcal{S}\left(
\epsilon \right) ,\eta \left( \epsilon \right) ,\left( \pi \left( \epsilon
\right) \mathcal{S}\left( \epsilon \right) \right) ^{-1}\right) $ is an
adjunction.
\end{proof}

\begin{lemma}
\label{lem:US}Let $F,F^{\prime }:\mathcal{A}\rightarrow \mathcal{A}$ be
functors. Let $\epsilon :F\rightarrow \mathrm{Id}_{\mathcal{A}}$ and $\phi
:F^{\prime }\rightarrow F$ be natural transformations. Set $\epsilon
^{\prime }:=\epsilon \circ \phi :F^{\prime }\rightarrow \mathrm{Id}_{%
\mathcal{A}}$. If $\mathcal{A}$ has coequalizers, we have the
adjoint triangle%
\begin{equation*}
\xymatrixcolsep{1.5cm}\xymatrixrowsep{0.7cm}\xymatrix{\mathcal{A}\ar[r]^\id%
\ar@{}[dr]|{\phi^*}\ar@<.4ex>[d]^*-<0.2cm>{^{\mathcal{S}\left( \epsilon
\right)}}& \mathcal{A}\ar@<.4ex>[d]^*-<0.2cm>{^{\mathcal{S}\left( \epsilon'
\right)}}\\ \left\langle
F|\id\right\rangle\ar@<.4ex>@{.>}[u]^*-<0.2cm>{^{\mathcal{U}\left( \epsilon
\right)}}\ar[r]_{\left\langle \phi|\id\right\rangle}&\left\langle
F'|\id\right\rangle\ar@<.4ex>@{.>}[u]^*-<0.2cm>{^{\mathcal{U}\left(
\epsilon' \right)}}}
\end{equation*}
\begin{invisible}
\begin{equation*}
\begin{array}{ccc}
\mathcal{A} & \overset{\mathrm{Id}_{\mathcal{A}}}{\longrightarrow } &
\mathcal{A} \\
\mathcal{U}\left( \epsilon \right) \uparrow \downarrow \mathcal{S}\left(
\epsilon \right) & \phi ^{\ast } & \mathcal{U}\left( \epsilon ^{\prime
}\right) \uparrow \downarrow \mathcal{S}\left( \epsilon ^{\prime }\right) \\
\left\langle F|\mathrm{Id}_{\mathcal{A}}\right\rangle & \overset{%
\left\langle \phi |\mathrm{Id}\right\rangle }{\rightarrow } & \left\langle
F^{\prime }|\mathrm{Id}_{\mathcal{A}}\right\rangle%
\end{array}%
\end{equation*}
\end{invisible}
Moreover the natural transformation $\phi ^{\ast }:\mathcal{U}\left(
\epsilon ^{\prime }\right) \circ \left\langle \phi |\mathrm{Id}\right\rangle
\rightarrow \mathcal{U}\left( \epsilon \right) $ satisfies
\begin{equation*}
\phi ^{\ast }\circ \pi \left( \epsilon ^{\prime }\right) \left\langle \phi |%
\mathrm{Id}\right\rangle =\pi \left( \epsilon \right)
\end{equation*}%
and $\phi ^{\ast }\mathcal{S}\left( \epsilon \right) \circ \pi ^{\prime }%
\mathcal{S}\left( \epsilon ^{\prime }\right) =\pi \mathcal{S}\left( \epsilon
\right) .$ In particular $\phi ^{\ast }\mathcal{S}\left( \epsilon \right) $
is invertible.
\end{lemma}

\begin{proof}
By Lemma \ref{lem:Ueps}, we have that $\mathcal{U}\left( \epsilon \right)
\dashv \mathcal{S}\left( \epsilon \right) $ and $\mathcal{U}\left( \epsilon
^{\prime }\right) \dashv \mathcal{S}\left( \epsilon ^{\prime }\right) .$

Set $P:=P_{\left\langle F|\mathrm{Id}_{\mathcal{A}}\right\rangle }$ and $%
P^{\prime }:=P_{\left\langle F^{\prime }|\mathrm{Id}_{\mathcal{A}%
}\right\rangle }.$ Set also $\pi =\pi \left( \epsilon \right) $ and $\pi
^{\prime }=\pi \left( \epsilon ^{\prime }\right) $. By Lemma \ref{lem:lift},
we have%
\begin{equation*}
\left\langle \phi |\mathrm{Id}\right\rangle \circ \mathcal{S}\left( \epsilon
\right) =\left\langle \phi |\mathrm{Id}\right\rangle \circ \mathrm{Id}_{%
\mathcal{A}}\left[ \epsilon \right] =\mathrm{Id}_{\mathcal{A}}\left[
\epsilon \circ \phi \right] =\mathrm{Id}_{\mathcal{A}}\left[ \epsilon
^{\prime }\right] =\mathcal{S}\left( \epsilon ^{\prime }\right) .
\end{equation*}%
so that $\left\langle \phi |\mathrm{Id}\right\rangle \circ \mathcal{S}\left(
\epsilon \right) =\mathcal{S}\left( \epsilon ^{\prime }\right) $ and hence
the diagram in the statement is an adjoint triangle.

By definition $\phi ^{\ast }=\left( \pi \left( \epsilon ^{\prime }\right)
\mathcal{S}\left( \epsilon ^{\prime }\right) \mathcal{U}\left( \epsilon
\right) \right) ^{-1}\circ \mathcal{U}\left( \epsilon ^{\prime }\right)
\left\langle \phi |\mathrm{Id}\right\rangle \eta \left( \epsilon \right) .$
Then%
\begin{eqnarray*}
\phi ^{\ast }\circ \pi ^{\prime }\left\langle \phi |\mathrm{Id}\right\rangle
&=&\left( \pi ^{\prime }\mathcal{S}\left( \epsilon ^{\prime }\right)
\mathcal{U}\left( \epsilon \right) \right) ^{-1}\circ \mathcal{U}\left(
\epsilon ^{\prime }\right) \left\langle \phi |\mathrm{Id}\right\rangle \eta
\left( \epsilon \right) \circ \pi ^{\prime }\left\langle \phi |\mathrm{Id}%
\right\rangle \\
&=&\left( \pi ^{\prime }\mathcal{S}\left( \epsilon ^{\prime }\right)
\mathcal{U}\left( \epsilon \right) \right) ^{-1}\circ \pi ^{\prime
}\left\langle \phi |\mathrm{Id}\right\rangle \mathcal{S}\left( \epsilon
\right) \mathcal{U}\left( \epsilon \right) \circ P^{\prime }\left\langle
\phi |\mathrm{Id}\right\rangle \eta \left( \epsilon \right) \\
&=&\left( \pi ^{\prime }\mathcal{S}\left( \epsilon ^{\prime }\right)
\mathcal{U}\left( \epsilon \right) \right) ^{-1}\circ \pi ^{\prime }\mathcal{%
S}\left( \epsilon ^{\prime }\right) \mathcal{U}\left( \epsilon \right) \circ
P\eta \left( \epsilon \right) =\pi .
\end{eqnarray*}%
Moreover%
\begin{equation*}
\phi ^{\ast }\mathcal{S}\left( \epsilon \right) \circ \pi ^{\prime }\mathcal{%
S}\left( \epsilon ^{\prime }\right) =\phi ^{\ast }\mathcal{S}\left( \epsilon
\right) \circ \pi ^{\prime }\left\langle \phi |\mathrm{Id}\right\rangle
\mathcal{S}\left( \epsilon \right) =\left( \phi ^{\ast }\circ \pi ^{\prime
}\left\langle \phi |\mathrm{Id}\right\rangle \right) \mathcal{S}\left(
\epsilon \right) =\pi \mathcal{S}\left( \epsilon \right) .
\end{equation*}%
Since $\pi ^{\prime }\mathcal{S}\left( \epsilon ^{\prime }\right) $ and $\pi
\mathcal{S}\left( \epsilon \right) $ are invertible, then so is $\phi ^{\ast }%
\mathcal{S}\left( \epsilon \right) .$
\end{proof}

Given an adjoint triangle as in \eqref{diag:AugTr}, assume $\mathcal{A}$ has coequalizers and consider the following diagram where we apply Lemma \ref{lem:US} to the counit $%
\epsilon :LR\rightarrow \mathrm{Id}_{\mathcal{A}}$ and $\phi :=\zeta
R:L^{\prime }R^{\prime }\rightarrow LR$ to get the adjoint triangle with $%
\left( \zeta R\right) ^{\ast }$ in the middle.%
\begin{equation*}
\xymatrixcolsep{1.5cm}\xymatrixrowsep{0.7cm}\xymatrix{\mathcal{A}\ar[rr]^\id%
\ar@{}[drr]|{(\zeta R)^*}\ar@<.4ex>[d]^*-<0.2cm>{^{\mathcal{S}\left(
\epsilon \right)}}&& \mathcal{A}\ar@<.4ex>[d]^*-<0.2cm>{^{\mathcal{S}\left(
\epsilon' \right)}}\\ \left\langle
LR|\id\right\rangle\ar[rd]_{\mathcal{R}\left(
\mathbb{T}\right)}\ar@<.4ex>@{.>}[u]^*-<0.2cm>{^{\mathcal{U}\left( \epsilon
\right)}}\ar[rr]_{\left\langle \zeta R|\id\right\rangle}&&\left\langle
L'R'|\id\right\rangle\ar@<.4ex>@{.>}[u]^*-<0.2cm>{^{\mathcal{U}\left(
\epsilon' \right)}}\\ &\mathrm{I}\left( \mathbb{T}\right)=\left\langle
R'L|G\right\rangle\ar[ur]_{\mathcal{L}\left( \mathbb{T}\right)}}
\end{equation*}

\begin{invisible}
\begin{equation*}
\mathbb{T}:=%
\begin{array}{ccc}
\mathcal{A} & \overset{\mathrm{Id}_{\mathcal{A}}}{\longrightarrow } &
\mathcal{A} \\
L\uparrow \downarrow R &
\begin{array}{c}
\zeta :L^{\prime }G\rightarrow L \\
R^{\prime }=GR%
\end{array}
& L^{\prime }\uparrow \downarrow R^{\prime } \\
\mathcal{B} & \overset{G}{\longrightarrow } & \mathcal{B}^{\prime }%
\end{array}%
\quad
\begin{array}{ccc}
\mathcal{A} & \overset{\mathrm{Id}_{\mathcal{A}}}{\longrightarrow } &
\mathcal{A} \\
\mathcal{U}\left( \epsilon \right) \uparrow \downarrow \mathcal{S}\left(
\epsilon \right) & \left( \zeta R\right) ^{\ast } & \mathcal{U}\left(
\epsilon ^{\prime }\right) \uparrow \downarrow \mathcal{S}\left( \epsilon
^{\prime }\right) \\
\left\langle LR|\mathrm{Id}_{\mathcal{A}}\right\rangle & \overset{%
\left\langle \zeta R|\mathrm{Id}_{\mathcal{A}}\right\rangle }{%
\longrightarrow } & \left\langle L^{\prime }R^{\prime }|\mathrm{Id}_{%
\mathcal{A}}\right\rangle \\
\mathcal{R}\left( \mathbb{T}\right) \searrow &  & \nearrow \mathcal{L}\left(
\mathbb{T}\right) \\
& \left\langle R^{\prime }L|G\right\rangle &
\end{array}%
\end{equation*}
\end{invisible}

The functors $\mathcal{L}\left( \mathbb{T}\right) $ and $\mathcal{R}\left(
\mathbb{T}\right) $ appearing in the diagram above are defined as follows
\begin{eqnarray*}
\mathcal{L}\left( \mathbb{T}\right) =LP_{\mathrm{I}\left( \mathbb{T}\right)}\left[ \zeta P_{\mathrm{I}\left( \mathbb{T}\right)}\circ L^{\prime }\psi_{\mathrm{I}\left( \mathbb{T}\right)} %
\right] &:&\mathrm{I}\left( \mathbb{T}\right) :=\left\langle R^{\prime
}L|G\right\rangle \overset{S^{L^{\prime }}}{\rightarrow }\left\langle
L^{\prime }R^{\prime }L|L^{\prime }G\right\rangle \overset{\left\langle
\mathrm{Id}_{L^{\prime }R^{\prime }L}|\zeta \right\rangle }{\rightarrow }%
\left\langle L^{\prime }R^{\prime }L|L\right\rangle \overset{D^{L}}{%
\rightarrow }\left\langle L^{\prime }R^{\prime }|\mathrm{Id}_{\mathcal{A}%
}\right\rangle, \\
\left( B,b\right) &\mapsto &\left( LB,\alpha \left( B,b\right) :=\zeta
B\circ L^{\prime }b\right) ,\qquad h\mapsto Lh
\end{eqnarray*}%
\begin{gather*}
\mathcal{R}\left( \mathbb{T}\right) =RP_{\left\langle LR|\mathrm{Id}_{\mathcal{A}}\right\rangle}\left[ R^{\prime }\psi_{\left\langle LR|\mathrm{Id}_{\mathcal{A}}\right\rangle} \right]
\quad:\quad\left\langle LR|\mathrm{Id}_{\mathcal{A}}\right\rangle \overset{%
S^{R^{\prime }}}{\rightarrow }\left\langle R^{\prime }LR|R^{\prime
}\right\rangle =\left\langle R^{\prime }LR|GR\right\rangle \overset{D^{R}}{%
\rightarrow }\left\langle R^{\prime }L|G\right\rangle =\mathrm{I}\left(
\mathbb{T}\right), \\
\left( A,\alpha _{A}:LRA\rightarrow
A\right) \mapsto\left( RA,R^{\prime }\alpha _{A}:R^{\prime }LRA\rightarrow
R^{\prime }A=GRA\right) ,\qquad f\mapsto Rf.
\end{gather*}%
Set%
\begin{eqnarray*}
\mathrm{R}\left( \mathbb{T}\right) =R\left[ R^{\prime }\epsilon \right] =
\mathcal{R}\left( \mathbb{T}\right) \circ \mathcal{S}\left( \epsilon \right)
:&&\mathcal{A}\rightarrow \mathrm{I}\left( \mathbb{T}\right),\quad A\mapsto \left(
RA,R^{\prime }\epsilon A\right) ,\quad f\mapsto Rf, \\
\mathrm{L}\left( \mathbb{T}\right) =\mathcal{U}\left( \epsilon ^{\prime
}\right) \circ \mathcal{L}\left( \mathbb{T}\right) :&&\mathrm{I}\left( \mathbb{%
T}\right) \rightarrow \mathcal{A},\quad \left( B,b\right) \mapsto \mathcal{U}%
\left( \epsilon ^{\prime }\right) \left( LB,\zeta B\circ L^{\prime }b\right)
,\quad h\mapsto \mathcal{U}\left( \epsilon ^{\prime }\right) \left(
Lh\right) .
\end{eqnarray*}

Consider also the forgetful functor%
\begin{equation*}
P\left( \mathbb{T}\right) =P_{\mathrm{I}\left( \mathbb{T}\right) }:\mathrm{I}%
\left( \mathbb{T}\right) \rightarrow \mathcal{B},\quad\left( B,b\right) \mapsto
B,\quad f\mapsto f
\end{equation*}%
and the functor%
\begin{equation*}
G\left( \mathbb{T}\right) :=G\circ P\left( \mathbb{T}\right),\quad\mathrm{I}%
\left( \mathbb{T}\right) \rightarrow \mathcal{B}^{\prime }:\left( B,b\right)
\mapsto GB,\quad h\mapsto Gh.
\end{equation*}%
Note that,if we set $P^{\prime }:=P_{\left\langle L^{\prime }R^{\prime }|\mathrm{Id}%
_{\mathcal{A}}\right\rangle }$, we get%
\begin{eqnarray}
P\left( \mathbb{T}\right) \mathrm{R}\left( \mathbb{T}\right) &=&R,
\label{form:ptrt} \\
P^{\prime }\mathcal{L}\left( \mathbb{T}\right) \mathrm{R}\left( \mathbb{T}%
\right) &=&LR  \label{form:P'ltrt}
\end{eqnarray}%

We are now ready to construct the adjoint triangle $\mathbb{T}^2$ announced in \eqref{diag:firstT2}.

\begin{proposition}
\label{pro:adj}Assume $\mathcal{A}$ has coequalizers. Given an adjoint triangle $\mathbb{T}$ as in (\ref{diag:AugTr}%
), then%
\begin{equation*}
\mathbb{T}^2:=\ \vcenter{\xymatrixcolsep{2.5cm}\xymatrixrowsep{.8cm}%
\xymatrix{\mathcal{A}\ar[r]^\id\ar@{}[dr]|-{\pi \left( \epsilon ^{\prime
}\right) \mathcal{L}\left(
\mathbb{T}\right)}\ar@<.4ex>[d]^*-<0.2cm>{^{\mathrm{R}\left( \mathbb{T}\right)}}&
\mathcal{A}\ar@<.4ex>[d]^*-<0.2cm>{^{R}}\\
\mathrm{I}(\mathbb{T})\ar@<.4ex>@{.>}[u]^*-<0.2cm>{^{\mathrm{L}\left(
\mathbb{T}\right)}}\ar[r]_{P(\mathbb{T})}&\mathcal{B}%
\ar@<.4ex>@{.>}[u]^*-<0.2cm>{^{L}}} }
\end{equation*}

\begin{invisible}
\begin{equation*}
\mathbb{T}^{2}:=%
\begin{array}{ccc}
\mathcal{A} & \overset{\mathrm{Id}_{\mathcal{A}}}{\longrightarrow } &
\mathcal{A} \\
\mathrm{L}\left( \mathbb{T}\right) \uparrow \downarrow \mathrm{R}\left(
\mathbb{T}\right) & \pi \left( \epsilon ^{\prime }\right) \mathcal{L}\left(
\mathbb{T}\right) & L\uparrow \downarrow R \\
\mathrm{I}\left( \mathbb{T}\right) & \overset{P\left( \mathbb{T}\right) }{%
\longrightarrow } & \mathcal{B}%
\end{array}%
\end{equation*}
\end{invisible}
is an adjoint triangle where the adjunction $\left( \mathrm{L}\left( \mathbb{T}\right) ,\mathrm{R}%
\left( \mathbb{T}\right) ,\eta \left( \mathbb{T}\right) ,\epsilon \left(
\mathbb{T}\right) \right) $ is uniquely determined by the following
equalities%
\begin{gather}
P\left( \mathbb{T}\right) \eta \left( \mathbb{T}\right) =R\pi \left(
\epsilon ^{\prime }\right) \mathcal{L}\left( \mathbb{T}\right) \circ \eta
P\left( \mathbb{T}\right) ,  \label{form:defetaT} \\
\epsilon \left( \mathbb{T}\right) \circ \pi \left( \epsilon ^{\prime
}\right) \mathcal{L}\left( \mathbb{T}\right) \mathrm{R}\left( \mathbb{T}%
\right) =\epsilon .  \label{form:defepsT}
\end{gather}
\end{proposition}

\begin{proof}
Set $\psi ^{\prime }:=\psi _{\left\langle L^{\prime }R^{\prime }|\mathrm{Id}%
_{\mathcal{A}}\right\rangle }$ and $\pi ^{\prime }:=\pi \left( \epsilon
^{\prime }\right) $ and let us construct $\epsilon \left( \mathbb{T}\right) $.

One easily checks that $\mathcal{L}\left( \mathbb{T}\right) \mathrm{R}\left(
\mathbb{T}\right) =\left( LR\right) \left[ \zeta R\circ L^{\prime }R^{\prime
}\epsilon \right] $. Moreover $\mathcal{S}\left( \epsilon ^{\prime }\right) =%
\mathrm{Id}_{\mathcal{A}}\left[ \epsilon ^{\prime }\right] .$

\begin{invisible}
$\mathcal{L}\left( \mathbb{T}\right) \mathrm{R}\left( \mathbb{T}\right) A=%
\mathcal{L}\left( \mathbb{T}\right) \left( RA,R^{\prime }\epsilon A\right)
=\left( LRA,\zeta RA\circ L^{\prime }R^{\prime }\epsilon A\right)
\Rightarrow \mathcal{L}\left( \mathbb{T}\right) \mathrm{R}\left( \mathbb{T}%
\right) =\left( LR\right) \left[ \zeta R\circ L^{\prime }R^{\prime }\epsilon %
\right] $
\end{invisible}

Since $\epsilon \circ \left( \zeta R\circ L^{\prime }R^{\prime }\epsilon
\right) =\epsilon ^{\prime }\circ L^{\prime }R^{\prime }\epsilon ,$ by Lemma %
\ref{lem:lift}, the counit $\epsilon :LR\to\id_\mathcal{A}$ induces $\widetilde{\epsilon }:\mathcal{L%
}\left( \mathbb{T}\right) \mathrm{R}\left( \mathbb{T}\right) \rightarrow
\mathcal{S}\left( \epsilon ^{\prime }\right) $ such that $P^{\prime }%
\widetilde{\epsilon }=\epsilon .$ Define
\begin{equation*}
\epsilon \left( \mathbb{T}\right) :=\left( \pi \left( \epsilon ^{\prime
}\right) \mathcal{S}\left( \epsilon ^{\prime }\right) \right) ^{-1}\circ
\mathcal{U}\left( \epsilon ^{\prime }\right) \widetilde{\epsilon }:\mathrm{L}%
\left( \mathbb{T}\right) \mathrm{R}\left( \mathbb{T}\right) \rightarrow
\mathrm{Id}_{\mathcal{A}}.
\end{equation*}%
Now we define $\eta \left( \mathbb{T}\right) :\mathrm{Id}_{\mathrm{I}\left(
\mathbb{T}\right) }\rightarrow \mathrm{R}\left( \mathbb{T}\right) \mathrm{L}%
\left( \mathbb{T}\right) .$ One easily checks that $\mathrm{R}\left( \mathbb{T}\right) \mathrm{L}\left(
\mathbb{T}\right) =\left( R\mathrm{L}\left( \mathbb{T}\right) \right) \left[
R^{\prime }\epsilon \mathrm{L}\left( \mathbb{T}\right) \right] $ and $%
\mathrm{Id}_{\mathrm{I}\left( \mathbb{T}\right) }=P\left( \mathbb{T}\right) %
\left[ \psi _{\mathrm{I}\left( \mathbb{T}\right) }\right] .$
\begin{invisible}
$\mathrm{R}\left( \mathbb{T}\right) \mathrm{L}\left( \mathbb{T}\right)
\left( B,b\right) =\left( R\mathrm{L}\left( \mathbb{T}\right) \left(
B,b\right) ,R^{\prime }\epsilon \mathrm{L}\left( \mathbb{T}\right) \left(
B,b\right) \right) \Rightarrow \mathrm{R}\left( \mathbb{T}\right) \mathrm{L}%
\left( \mathbb{T}\right) =\left( R\mathrm{L}\left( \mathbb{T}\right) \right) %
\left[ R^{\prime }\epsilon \mathrm{L}\left( \mathbb{T}\right) \right] $
\end{invisible}
Set
\begin{equation*}
\alpha :=\left( P\left( \mathbb{T}\right) \overset{\eta P\left( \mathbb{T}%
\right) }{\longrightarrow }RLP\left( \mathbb{T}\right)=RP'\mathcal{L}\left( \mathbb{T}\right) \overset{R\pi \left(
\epsilon ^{\prime }\right) \mathcal{L}\left( \mathbb{T}\right) }{%
\longrightarrow }R\mathcal{U}(\epsilon')\mathcal{L}\left( \mathbb{T}\right)=R\mathrm{L}\left( \mathbb{T}\right) \right) .
\end{equation*}

We compute%
\begin{eqnarray*}
\epsilon ^{\prime }\mathrm{L}\left( \mathbb{T}\right) \circ L^{\prime
}\left( G\alpha \circ \psi _{\mathrm{I}\left( \mathbb{T}\right) }\right)
&=&\epsilon ^{\prime }\mathrm{L}\left( \mathbb{T}\right) \circ L^{\prime
}G\alpha \circ L^{\prime }\psi _{\mathrm{I}\left( \mathbb{T}\right) } \\
&=&\epsilon ^{\prime }\mathrm{L}\left( \mathbb{T}\right) \circ L^{\prime
}GR\pi \left( \epsilon ^{\prime }\right) \mathcal{L}\left( \mathbb{T}\right)
\circ L^{\prime }G\eta P\left( \mathbb{T}\right) \circ L^{\prime }\psi _{%
\mathrm{I}\left( \mathbb{T}\right) } \\
&=&\epsilon ^{\prime }\mathrm{L}\left( \mathbb{T}\right) \circ L^{\prime
}R^{\prime }\pi \left( \epsilon ^{\prime }\right) \mathcal{L}\left( \mathbb{T%
}\right) \circ L^{\prime }G\eta P\left( \mathbb{T}\right) \circ L^{\prime
}\psi _{\mathrm{I}\left( \mathbb{T}\right) } \\
&=&\pi \left( \epsilon ^{\prime }\right) \mathcal{L}\left( \mathbb{T}\right)
\circ \epsilon ^{\prime }LP\left( \mathbb{T}\right) \circ L^{\prime }G\eta
P\left( \mathbb{T}\right) \circ L^{\prime }\psi _{\mathrm{I}\left( \mathbb{T}%
\right) } \\
&=&\pi \left( \epsilon ^{\prime }\right) \mathcal{L}\left( \mathbb{T}\right)
\circ \zeta P\left( \mathbb{T}\right) \circ L^{\prime }\psi _{\mathrm{I}%
\left( \mathbb{T}\right) } \\
&\overset{(\ast )}{=}&\pi \left( \epsilon ^{\prime }\right) \mathcal{L}%
\left( \mathbb{T}\right) \circ \psi ^{\prime }\mathcal{L}\left( \mathbb{T}%
\right) \\
&\overset{\text{def.}\pi \left( \epsilon ^{\prime }\right) }{=}&\pi \left(
\epsilon ^{\prime }\right) \mathcal{L}\left( \mathbb{T}\right) \circ
\epsilon ^{\prime }P^{\prime }\mathcal{L}\left( \mathbb{T}\right) \\
&=&\pi \left( \epsilon ^{\prime }\right) \mathcal{L}\left( \mathbb{T}\right)
\circ \epsilon ^{\prime }LP\left( \mathbb{T}\right) \\
&=&\pi \left( \epsilon ^{\prime }\right) \mathcal{L}\left( \mathbb{T}\right)
\circ \epsilon LP\left( \mathbb{T}\right) \circ L\eta P\left( \mathbb{T}%
\right) \circ \epsilon ^{\prime }LP\left( \mathbb{T}\right) \\
&=&\epsilon \mathrm{L}\left( \mathbb{T}\right) \circ LR\pi \left( \epsilon
^{\prime }\right) \mathcal{L}\left( \mathbb{T}\right) \circ L\eta P\left(
\mathbb{T}\right) \circ \epsilon ^{\prime }LP\left( \mathbb{T}\right) \\
&=&\epsilon \mathrm{L}\left( \mathbb{T}\right) \circ L\alpha \circ \epsilon
^{\prime }LP\left( \mathbb{T}\right) \\
&=&\epsilon ^{\prime }\mathrm{L}\left( \mathbb{T}\right) \circ L^{\prime
}R^{\prime }\left( \epsilon \mathrm{L}\left( \mathbb{T}\right) \circ L\alpha
\right) =\epsilon ^{\prime }\mathrm{L}\left( \mathbb{T}\right) \circ
L^{\prime }\left( R^{\prime }\epsilon \mathrm{L}\left( \mathbb{T}\right)
\circ R^{\prime }L\alpha \right)
\end{eqnarray*}%
where (*) follows by the equality $\zeta P\left( \mathbb{T}\right) \circ
L^{\prime }\psi _{\mathrm{I}\left( \mathbb{T}\right) }=\psi ^{\prime }%
\mathcal{L}\left( \mathbb{T}\right) $ that can be easily checked.

\begin{invisible}
Here is the computation%
\begin{equation*}
\zeta P\left( \mathbb{T}\right) \left( B,b\right) \circ L^{\prime }\psi _{%
\mathrm{I}\left( \mathbb{T}\right) }\left( B,b\right) =\zeta B\circ
L^{\prime }b=\psi ^{\prime }\mathcal{L}\left( \mathbb{T}\right) \left(
B,b\right) .
\end{equation*}
\end{invisible}

We have so proved that $\epsilon ^{\prime }\mathrm{L}\left( \mathbb{T}%
\right) \circ L^{\prime }\left( G\alpha \circ \psi _{\mathrm{I}\left(
\mathbb{T}\right) }\right) =\epsilon ^{\prime }\mathrm{L}\left( \mathbb{T}%
\right) \circ L^{\prime }\left( R^{\prime }\epsilon \mathrm{L}\left( \mathbb{%
T}\right) \circ R^{\prime }L\alpha \right) .$ By the adjunction this is
equivalent to $G\alpha \circ \psi _{\mathrm{I}\left( \mathbb{T}\right)
}=R^{\prime }\epsilon \mathrm{L}\left( \mathbb{T}\right) \circ R^{\prime
}L\alpha .$ By Lemma \ref{lem:lift}, the map $\alpha $ induces $\eta \left(
\mathbb{T}\right) :\mathrm{Id}_{\mathrm{I}\left( \mathbb{T}\right)
}\rightarrow \mathrm{R}\left( \mathbb{T}\right) \mathrm{L}\left( \mathbb{T}%
\right) $ such that $P\left( \mathbb{T}\right) \eta \left( \mathbb{T}\right)
=\alpha .$ We compute
\begin{eqnarray*}
&&\epsilon \left( \mathbb{T}\right) \mathrm{L}\left( \mathbb{T}\right) \circ
\mathrm{L}\left( \mathbb{T}\right) \eta \left( \mathbb{T}\right) \circ \pi
\left( \epsilon ^{\prime }\right) \mathcal{L}\left( \mathbb{T}\right) \\
&=&\left( \pi \left( \epsilon ^{\prime }\right) \mathcal{S}\left( \epsilon
^{\prime }\right) \mathrm{L}\left( \mathbb{T}\right) \right) ^{-1}\circ
\mathcal{U}\left( \epsilon ^{\prime }\right) \widetilde{\epsilon }\mathrm{L}%
\left( \mathbb{T}\right) \circ \mathcal{U}\left( \epsilon ^{\prime }\right)
\mathcal{L}\left( \mathbb{T}\right) \eta \left( \mathbb{T}\right) \circ \pi
\left( \epsilon ^{\prime }\right) \mathcal{L}\left( \mathbb{T}\right) \\
&\overset{\text{nat.}\pi \left( \epsilon ^{\prime }\right) \text{ }}{=}%
&\left( \pi \left( \epsilon ^{\prime }\right) \mathcal{S}\left( \epsilon
^{\prime }\right) \mathrm{L}\left( \mathbb{T}\right) \right) ^{-1}\circ \pi
\left( \epsilon ^{\prime }\right) \mathcal{S}\left( \epsilon ^{\prime
}\right) \mathrm{L}\left( \mathbb{T}\right) \circ P^{\prime }\widetilde{%
\epsilon }\mathrm{L}\left( \mathbb{T}\right) \circ P^{\prime }\mathcal{L}%
\left( \mathbb{T}\right) \eta \left( \mathbb{T}\right) \\
&=&P^{\prime }\widetilde{\epsilon }\mathrm{L}\left( \mathbb{T}\right) \circ
LP\left( \mathbb{T}\right) \eta \left( \mathbb{T}\right) =\epsilon \mathrm{L}%
\left( \mathbb{T}\right) \circ L\alpha \\
&=&\epsilon \mathrm{L}\left( \mathbb{T}\right) \circ LR\pi \left( \epsilon
^{\prime }\right) \mathcal{L}\left( \mathbb{T}\right) \circ L\eta P\left(
\mathbb{T}\right) \\
&=&\pi \left( \epsilon ^{\prime }\right) \mathcal{L}\left( \mathbb{T}\right)
\circ \epsilon LP\left( \mathbb{T}\right) \circ L\eta P\left( \mathbb{T}%
\right) =\pi \left( \epsilon ^{\prime }\right) \mathcal{L}\left( \mathbb{T}%
\right)
\end{eqnarray*}%
so that $\epsilon \left( \mathbb{T}\right) \mathrm{L}\left( \mathbb{T}%
\right) \circ \mathrm{L}\left( \mathbb{T}\right) \eta \left( \mathbb{T}%
\right) =\mathrm{Id}_{\mathrm{L}\left( \mathbb{T}\right) }.$

We compute%
\begin{eqnarray*}
P\left( \mathbb{T}\right) \left( \mathrm{R}\left( \mathbb{T}\right) \epsilon
\left( \mathbb{T}\right) \circ \eta \left( \mathbb{T}\right) \mathrm{R}%
\left( \mathbb{T}\right) \right) &=&P\left( \mathbb{T}\right) \mathrm{R}%
\left( \mathbb{T}\right) \epsilon \left( \mathbb{T}\right) \circ P\left(
\mathbb{T}\right) \eta \left( \mathbb{T}\right) \mathrm{R}\left( \mathbb{T}%
\right) \\
&=&R\epsilon \left( \mathbb{T}\right) \circ \alpha \mathrm{R}\left( \mathbb{T%
}\right) \\
&=&\left( R\pi \left( \epsilon ^{\prime }\right) \mathcal{S}\left( \epsilon
^{\prime }\right) \right) ^{-1}\circ R\mathcal{U}\left( \epsilon ^{\prime
}\right) \widetilde{\epsilon }\circ R\pi \left( \epsilon ^{\prime }\right)
\mathcal{L}\left( \mathbb{T}\right) \mathrm{R}\left( \mathbb{T}\right) \circ
\eta P\left( \mathbb{T}\right) \mathrm{R}\left( \mathbb{T}\right) \\
&=&\left( R\pi \left( \epsilon ^{\prime }\right) \mathcal{S}\left( \epsilon
^{\prime }\right) \right) ^{-1}\circ R\pi \left( \epsilon ^{\prime }\right)
\mathcal{S}\left( \epsilon ^{\prime }\right) \circ RP^{\prime }\widetilde{%
\epsilon }\circ \eta R \\
&=&R\epsilon \circ \eta R=\mathrm{Id}_{R}.
\end{eqnarray*}

We have so proved that $\left( \mathrm{L}\left( \mathbb{T}\right) ,\mathrm{R}%
\left( \mathbb{T}\right) ,\eta \left( \mathbb{T}\right) ,\epsilon \left(
\mathbb{T}\right) \right) $ is an adjunction. We compute%
\begin{eqnarray*}
\epsilon \mathrm{L}\left( \mathbb{T}\right) \circ LP\left( \mathbb{T}\right)
\eta \left( \mathbb{T}\right) &=&\epsilon \mathrm{L}\left( \mathbb{T}\right)
\circ L\alpha \\
&=&\epsilon \mathrm{L}\left( \mathbb{T}\right) \circ LR\pi \left( \epsilon
^{\prime }\right) \mathcal{L}\left( \mathbb{T}\right) \circ L\eta P\left(
\mathbb{T}\right) \\
&=&\pi \left( \epsilon ^{\prime }\right) \mathcal{L}\left( \mathbb{T}\right)
\circ \epsilon LP\left( \mathbb{T}\right) \circ L\eta P\left( \mathbb{T}%
\right) =\pi \left( \epsilon ^{\prime }\right) \mathcal{L}\left( \mathbb{T}%
\right) .
\end{eqnarray*}%
Thus, since $P\left( \mathbb{T}\right) \mathrm{R}\left( \mathbb{T}\right)
=R, $ the diagram in the statement is an adjoint triangle and the natural
transformation inside it is the correct one. Note that (\ref{form:defepsT})
follows by definition of adjoint triangle and, since $\pi \left( \epsilon
^{\prime }\right) \mathcal{L}\left( \mathbb{T}\right) \mathrm{R}\left(
\mathbb{T}\right) $ is an epimorphism, it uniquely determines $\epsilon
\left( \mathbb{T}\right) .$
\end{proof}

Starting from an adjunction $\left( L,R,\eta ,\epsilon \right) ,$ with $R:%
\mathcal{A}\rightarrow \mathcal{B}$ and where $\mathcal{A}$ has all
coequalizers, we are now able to construct a kind of monadic decomposition that
will be called an \textbf{adjoint decomposition }as follows, where we set $%
L_{\left[ 0\right] }:=L,R_{\left[ 0\right] }:=R$, $\eta_{[0]}:=\eta$, $\epsilon_{[0]}:=\epsilon$  and $\mathcal{B}_{\left[ 0%
\right] }:=\mathcal{B}$.
\begin{equation}\label{diag:adjdecomp}
\vcenter{\xymatrixcolsep{2cm}\xymatrix{\mathcal{A}\ar@<.5ex>[d]^{R_{[0]}}&\mathcal{A}%
\ar@<.5ex>[d]^{R_{[1]}}\ar[l]_{\mathrm{Id}}&\mathcal{A}%
\ar@<.5ex>[d]^{R_{[2]}}\ar[l]_{\mathrm{Id}}&\cdots
\ar[l]_{\mathrm{Id}}\quad\cdots&\mathcal{A}\ar@<.5ex>[d]^{R_{[n]}}\ar[l]_{%
\mathrm{Id}}\\
\mathcal{B}_{[0]}\ar@{}[ur]|{\pi_{[0,1]}}\ar@<.5ex>@{.>}[u]^{L_{[0]}}&%
\mathcal{B}_{[1]}\ar@{}[ur]|{\pi_{[1,2]}}\ar@<.5ex>@{.>}[u]^{L_{[1]}}
\ar[l]_{U_{[0,1]}}\ar@<0.3cm>@{}[l]^*+[F-,]{\mathbb{T}_{[0,1]}}&%
\mathcal{B}_{[2]} \ar@<.5ex>@{.>}[u]^{L_{[2]}}
\ar[l]_{U_{[1,2]}}\ar@<0.3cm>@{}[l]^*+[F-,]{\mathbb{T}_{[1,2]}}&\cdots%
\ar[l]_{U_{[2,3]}}\ar@<0.3cm>@{}[l]^*+[F-,]{\mathbb{T}_{[2,3]}}\quad\cdots&%
\mathcal{B}_{[n]}\ar@<.5ex>@{.>}[u]^{L_{[n]}}\ar[l]_{U_{[n-1,n]}}%
\ar@<0.3cm>@{}[l]^*+[F-,]{\mathbb{T}_{[n-1,n]}}}}
\end{equation}%
In the diagram above we label by $\mathbb{T}_{\left[ 0,1\right] }$ the first
adjoint triangle from left, by $\mathbb{T}_{\left[ 1,2\right] }$ the second
one and in general by $\mathbb{T}_{\left[ n-1,n\right] }$ the $n$-th one.
Denote by $\mathbb{T}_{\left[ n\right] }$ the composition of the first $n$
adjoint triangles. They are constructed iteratively as follows. The adjoint
triangle $\mathbb{T}_{\left[ 0\right] }$ is defined as in the following
diagram while, for $n>0,$ we set $\mathbb{T}_{\left[ n-1,n\right] }:=\left(
\mathbb{T}_{\left[ n-1\right] }\right) ^{2}$ (see Proposition \ref{pro:adj})
and $\mathbb{T}_{\left[ n\right] }:=\mathbb{T}_{\left[ n-1,n\right] }\ast
\mathbb{T}_{\left[ n-1\right] }.$
\begin{gather*}
\mathbb{T}_{[0]}:=\vcenter{\xymatrixcolsep{1.5cm}\xymatrixrowsep{.71cm}%
\xymatrix{\mathcal{A}\ar[r]^\id\ar@{}[dr]|-{\pi _{\left[ 0\right]
}=\mathrm{Id}_{L}}\ar@<.4ex>[d]^*-<0.2cm>{^{R}}&
\mathcal{A}\ar@<.4ex>[d]^*-<0.2cm>{^{R}}\\
\mathcal{B}\ar@<.4ex>@{.>}[u]^*-<0.2cm>{^{L}}\ar[r]^\id&\mathcal{B}%
\ar@<.4ex>@{.>}[u]^*-<0.2cm>{^{L}}} } \qquad\qquad \mathbb{T}_{[n]}:=%
\vcenter{\xymatrixcolsep{1.5cm}\xymatrixrowsep{.71cm}\xymatrix{\mathcal{A}%
\ar[r]^\id\ar@{}[dr]|-{\pi_{[n]}}\ar@<.4ex>[d]^*-<0.2cm>{^{R_{[n]}}}&
\mathcal{A}\ar@<.4ex>[d]^*-<0.2cm>{^{R}}\\
\mathcal{B}_{[n]}\ar@<.4ex>@{.>}[u]^*-<0.2cm>{^{L_{[n]}}}\ar[r]^{U_{[n]}}&%
\mathcal{B}\ar@<.4ex>@{.>}[u]^*-<0.2cm>{^{L}}} }
\end{gather*}
\begin{invisible}
\begin{equation*}
\mathbb{T}_{\left[ 0\right] }:=%
\begin{array}{ccc}
\mathcal{A} & \overset{\mathrm{Id}_{\mathcal{A}}}{\longrightarrow } &
\mathcal{A} \\
L\uparrow \downarrow R & \pi _{\left[ 0\right] }=\mathrm{Id}_{L} & L\uparrow
\downarrow R \\
\mathcal{B} & \overset{\mathrm{Id}_{\mathcal{B}}}{\longrightarrow } &
\mathcal{B}%
\end{array}%
\end{equation*}%
Thus%
\begin{equation*}
\mathbb{T}_{\left[ n\right] }=%
\begin{array}{ccc}
\mathcal{A} & \overset{\mathrm{Id}_{\mathcal{A}}}{\longrightarrow } &
\mathcal{A} \\
L_{\left[ n\right] }\uparrow \downarrow R_{\left[ n\right] } & \pi _{\left[ n%
\right] } & L\uparrow \downarrow R \\
\mathcal{B}_{\left[ n\right] } & \overset{U_{\left[ n\right] }}{%
\longrightarrow } & \mathcal{B}%
\end{array}%
\end{equation*}
\end{invisible}
Explicitly $\mathcal{B}_{\left[ 1\right] }=\mathrm{I}\left( \mathbb{T}_{%
\left[ 0\right] }\right) =\left\langle R^{\prime }L|G\right\rangle
=\left\langle RL|\mathrm{Id}_{\mathcal{B}}\right\rangle ,U_{\left[ 1\right]
}=U_{\left[ 0,1\right] }=P\left( \mathbb{T}_{\left[ 0\right] }\right) $ the
forgetful functor.

\begin{eqnarray*}
\mathcal{L}\left( \mathbb{T}_{\left[ 0\right] }\right) &:&\mathcal{B}_{\left[
1\right] }\rightarrow \left\langle LR|\mathrm{Id}_{\mathcal{A}}\right\rangle
,\quad\left( B,b\right) \mapsto \left( LB,\zeta B\circ Lb\right) =\left(
LB,Lb\right) ,\qquad h\mapsto Lh, \\
R_{\left[ 1\right] } &=&\mathrm{R}\left( \mathbb{T}_{\left[ 0\right]
}\right) =\mathcal{R}\left( \mathbb{T}_{\left[ 0\right] }\right) \circ
\mathcal{S}\left( \epsilon \right) :\mathcal{A}\rightarrow \mathrm{I}\left(
\mathbb{T}_{\left[ 0\right] }\right) ,\quad A\mapsto \left( RA,R\epsilon A\right)
,\quad f\mapsto Rf, \\
L_{\left[ 1\right] } &=&\mathrm{L}\left( \mathbb{T}_{\left[ 0\right]
}\right) =\mathcal{U}\left( \epsilon \right) \circ \mathcal{L}\left( \mathbb{%
T}_{\left[ 0\right] }\right) :\mathrm{I}\left( \mathbb{T}_{\left[ 0\right]
}\right) \rightarrow \mathcal{A},\quad\left( B,b\right) \mapsto \mathcal{U}\left(
\epsilon \right) \left( LB,Lb\right) ,\quad h\mapsto \mathcal{U}\left(
\epsilon \right) \left( Lh\right) , \\
\pi _{\left[ 1\right] } &=&\pi _{\left[ 0,1\right] } = \pi \left( \epsilon \right) \mathcal{L}%
\left( \mathbb{T}_{\left[ 0\right] }\right) .
\end{eqnarray*}%
The unit $\eta _{\left[ 1\right] }=\eta \left( \mathbb{T}_{\left[ 0\right]
}\right) $ and the counit $\epsilon _{\left[ 1\right] }=\epsilon \left(
\mathbb{T}_{\left[ 0\right] }\right) $ of the adjunction $\left( L_{\left[ 1%
\right] },R_{\left[ 1\right] }\right) $ are uniquely defined by
\begin{equation}
U_{\left[ 1\right] }\eta _{\left[ 1\right] }=R\pi _{\left[ 1\right] }\circ
\eta U_{\left[ 1\right] }\qquad \text{and}\qquad \epsilon _{\left[ 1\right]
}\circ \pi _{\left[ 1\right] }R_{\left[ 1\right] }=\epsilon =\epsilon _{%
\left[ 0\right] }.  \label{form:epseta[1]}
\end{equation}%
Note that for every $B_{\left[ 1\right] }:=\left( B,b\right) \in \mathcal{B}%
_{\left[ 1\right] }$ we have the following coequalizer

\begin{equation}  \label{coeq:L1}
\xymatrixcolsep{1.5cm} \xymatrix{LRLB\ar@<.5ex>[r]^{Lb}\ar@<-.5ex>[r]_{%
\epsilon LB}&LB\ar[r]^{\pi _{[1]}B_{[1]}}& L_{[1]}B_{[1]}}
\end{equation}

\begin{invisible}
\begin{equation}
LRLB\overset{Lb}{\underset{\epsilon LB}{\rightrightarrows }}LB\overset{\pi _{%
\left[ 1\right] }B_{\left[ 1\right] }}{\longrightarrow }L_{\left[ 1\right]
}B_{\left[ 1\right] }.
\end{equation}
\end{invisible}

Next $\mathcal{B}_{\left[ 2\right] }=\mathrm{I}\left( \mathbb{T}_{\left[ 1%
\right] }\right) =\left\langle RL_{\left[ 1\right] }|U_{\left[ 1\right]
}\right\rangle ,U_{\left[ 1,2\right] }:=P\left( \mathbb{T}_{\left[ 1\right]
}\right) $ and $U_{\left[ 2\right] }=U_{\left[ 0,1\right] }\circ U_{\left[
1,2\right] }.$ Moreover%
\begin{eqnarray*}
\mathcal{L}\left( \mathbb{T}_{\left[ 1\right] }\right) &:&\mathcal{B}_{\left[
2\right] }\rightarrow \left\langle LR|\mathrm{Id}_{\mathcal{A}}\right\rangle
:\quad\left( B_{\left[ 1\right] },b_{\left[ 1\right] }\right) \mapsto \left( L_{%
\left[ 1\right] }B_{\left[ 1\right] },\pi _{\left[ 1\right] }B_{\left[ 1%
\right] }\circ Lb_{\left[ 1\right] }\right) ,\quad h\mapsto L_{\left[ 1%
\right] }h, \\
R_{\left[ 2\right] } &=&\mathrm{R}\left( \mathbb{T}_{\left[ 1\right]
}\right) =\mathcal{R}\left( \mathbb{T}_{\left[ 1\right] }\right) \circ
\mathcal{S}\left( \epsilon _{\left[ 1\right] }\right) :\mathcal{A}%
\rightarrow \mathcal{B}_{\left[ 2\right] },\quad A\mapsto \left( R_{\left[ 1\right]
}A,R\epsilon _{\left[ 1\right] }A\right) ,\quad f\mapsto R_{\left[ 1\right]
}f, \\
L_{\left[ 2\right] } &=&\mathrm{L}\left( \mathbb{T}_{\left[ 1\right]
}\right) =\mathcal{U}\left( \epsilon \right) \circ \mathcal{L}\left( \mathbb{%
T}_{\left[ 1\right] }\right) :\mathcal{B}_{\left[ 2\right] }\rightarrow
\mathcal{A},\\&&\quad \left( B_{\left[ 1\right] },b_{\left[ 1\right] }\right) \mapsto
\mathcal{U}\left( \epsilon \right) \left( L_{\left[ 1\right] }B_{\left[ 1%
\right] },\pi _{\left[ 1\right] }B_{\left[ 1\right] }\circ Lb_{\left[ 1%
\right] }\right) ,\quad h\mapsto \mathcal{U}\left( \epsilon \right) \left(
L_{\left[ 1\right] }h\right) , \\
\pi _{\left[ 1,2\right] } &=&\pi \left( \epsilon \right) \mathcal{L}\left(
\mathbb{T}_{\left[ 1\right] }\right) \\
\pi _{\left[ 2\right] } &=&\pi _{\left[ 1,2\right] }\ast \pi _{\left[ 1,0%
\right] }=\pi _{\left[ 1,2\right] }\circ \pi _{\left[ 1\right] }U_{\left[ 1,2%
\right] }
\end{eqnarray*}%
The unit $\eta _{\left[ 2\right] }=\eta \left( \mathbb{T}_{\left[ 1\right]
}\right) $ and the counit $\epsilon _{\left[ 2\right] }=\epsilon \left(
\mathbb{T}_{\left[ 1\right] }\right) $ of the adjunction $\left( L_{\left[ 2%
\right] },R_{\left[ 2\right] }\right) $ are uniquely determined by
\begin{equation}
U_{\left[ 1,2\right] }\eta _{\left[ 2\right] }=R_{\left[ 1\right] }\pi _{%
\left[ 1,2\right] }\circ \eta _{\left[ 1\right] }U_{\left[ 1,2\right]
}\qquad \text{and}\qquad \epsilon _{\left[ 2\right] }\circ \pi _{\left[ 1,2%
\right] }R_{\left[ 2\right] }=\epsilon _{\left[ 1\right] }
\label{form:epseta[2]}
\end{equation}%
Note that for every $B_{\left[ 2\right] }:=\left( B_{\left[ 1\right] },b_{%
\left[ 1\right] }\right) \in \mathcal{B}_{\left[ 2\right] }$ we have the
following coequalizer%
\begin{equation*}
\xymatrixcolsep{2cm} \xymatrix{LRL_{[1]}B_{[1]}\ar@<.5ex>[r]^{%
\pi_{[1]}B_{[1]}\circ Lb_{[1]}}\ar@<-.5ex>[r]_{\epsilon
L_{[1]}B_{[1]}}&L_{[1]}B_{[1]}\ar[r]^{\pi _{[1,2]}B_{[2]}}& L_{[2]}B_{[2]}}
\end{equation*}

\begin{invisible}
\begin{equation*}
LRL_{\left[ 1\right] }B_{\left[ 1\right] }\overset{\pi _{\left[ 1\right] }B_{%
\left[ 1\right] }\circ Lb_{\left[ 1\right] }}{\underset{\epsilon L_{\left[ 1%
\right] }B_{\left[ 1\right] }}{\rightrightarrows }}L_{\left[ 1\right] }B_{%
\left[ 1\right] }\overset{\pi _{\left[ 1,2\right] }B_{\left[ 2\right] }}{%
\longrightarrow }L_{\left[ 2\right] }B_{\left[ 2\right] }
\end{equation*}
\end{invisible}

Finally $\mathcal{B}_{\left[ n+1\right] }=\mathrm{I}\left( \mathbb{T}_{\left[
n\right] }\right) =\left\langle RL_{\left[ n\right] }|U_{\left[ n\right]
}\right\rangle ,U_{\left[ n,n+1\right] }=P\left( \mathbb{T}_{\left[ n\right]
}\right) $ and $U_{\left[ n+1\right] }=U_{\left[ n\right] }\circ U_{\left[
n,n+1\right] }$. Moreover%
\begin{eqnarray*}
\mathcal{L}\left( \mathbb{T}_{\left[ n\right] }\right) &:&\mathcal{B}_{\left[
n+1\right] }\rightarrow \left\langle LR|\mathrm{Id}_{\mathcal{A}%
}\right\rangle ,\quad \left( B_{\left[ n\right] },b_{\left[ n\right] }\right)
\mapsto \left( L_{\left[ n\right] }B_{\left[ n\right] },\pi _{\left[ n\right]
}B_{\left[ n\right] }\circ Lb_{\left[ n\right] }\right) ,\qquad h\mapsto L_{%
\left[ n\right] }h, \\
R_{\left[ n+1\right] } &=&\mathrm{R}\left( \mathbb{T}_{\left[ n\right]
}\right) =\mathcal{R}\left( \mathbb{T}_{\left[ n\right] }\right) \circ
\mathcal{S}\left( \epsilon _{\left[ n\right] }\right) :\mathcal{A}%
\rightarrow \mathcal{B}_{\left[ n+1\right] },\quad A\mapsto \left( R_{\left[ n%
\right] }A,R\epsilon _{\left[ n\right] }A\right) ,\qquad f\mapsto R_{\left[ n%
\right] }f, \\
L_{\left[ n+1\right] } &=&\mathrm{L}\left( \mathbb{T}_{\left[ n\right]
}\right) =\mathcal{U}\left( \epsilon \right) \circ \mathcal{L}\left( \mathbb{%
T}_{\left[ n\right] }\right) ,\\&&\quad \mathcal{B}_{\left[ n+1\right] }\rightarrow
\mathcal{A}:\left( B_{\left[ n\right] },b_{\left[ n\right] }\right) \mapsto
\mathcal{U}\left( \epsilon \right) \left( L_{\left[ n\right] }B_{\left[ n%
\right] },\pi _{\left[ n\right] }B_{\left[ n\right] }\circ Lb_{\left[ n%
\right] }\right) ,\quad h\mapsto \mathcal{U}\left( \epsilon \right) \left(
L_{\left[ n\right] }h\right) , \\
\pi _{\left[ n,n+1\right] } &=&\pi \left( \epsilon \right) \mathcal{L}\left(
\mathbb{T}_{\left[ n\right] }\right) \\
\pi _{\left[ n+1\right] } &=&\pi _{\left[ n,n+1\right] }\ast \pi _{\left[ n%
\right] }=\pi _{\left[ n,n+1\right] }\circ \pi _{\left[ n\right] }U_{\left[
n,n+1\right] }
\end{eqnarray*}%
The unit $\eta _{\left[ n+1\right] }=\eta \left( \mathbb{T}_{\left[ n\right]
}\right) $ and the counit $\epsilon _{\left[ n+1\right] }=\epsilon \left(
\mathbb{T}_{\left[ n\right] }\right) $ of the adjunction $\left( L_{\left[
n+1\right] },R_{\left[ n+1\right] }\right) $ are uniquely determined by
\begin{equation}
U_{\left[ n,n+1\right] }\eta _{\left[ n+1\right] }=R_{\left[ n\right] }\pi _{%
\left[ n,n+1\right] }\circ \eta _{\left[ n\right] }U_{\left[ n,n+1\right]
}\qquad \text{and}\qquad \epsilon _{\left[ n+1\right] }\circ \pi _{\left[
n,n+1\right] }R_{\left[ n+1\right] }=\epsilon _{\left[ n\right] }
\label{form:epseta[n+1]}
\end{equation}%
Note that for every $B_{\left[ n+1\right] }:=\left( B_{\left[ n\right] },b_{%
\left[ n\right] }\right) \in \mathcal{B}_{\left[ n+1\right] }$ we have the
following coequalizer%
\begin{equation}  \label{coeq:Ln+1}
\xymatrixcolsep{2cm} \xymatrix{LRL_{[n]}B_{[n]}\ar@<.5ex>[r]^{%
\pi_{[n]}B_{[n]}\circ Lb_{[n]}}\ar@<-.5ex>[r]_{\epsilon
L_{[n]}B_{[n]}}&L_{[n]}B_{[n]}\ar[r]^-{\pi _{[n,n+1]}B_{[n+1]}}&
L_{[n+1]}B_{[n+1]}}
\end{equation}

\begin{invisible}
\begin{equation}
LRL_{\left[ n\right] }B_{\left[ n\right] }\overset{\pi _{\left[ n\right] }B_{%
\left[ n\right] }\circ Lb_{\left[ n\right] }}{\underset{\epsilon L_{\left[ n%
\right] }B_{\left[ n\right] }}{\rightrightarrows }}L_{\left[ n\right] }B_{%
\left[ n\right] }\overset{\pi _{\left[ n,n+1\right] }B_{\left[ n+1\right] }}{%
\longrightarrow }L_{\left[ n+1\right] }B_{\left[ n+1\right] }
\end{equation}
We have $\pi _{\left[ n\right] }B_{\left[ n\right] }\circ Lb_{\left[ n\right]
}=\epsilon L_{\left[ n\right] }B_{\left[ n\right] }\circ Le_{\left[ n\right]
}$ where
\begin{equation*}
e_{\left[ n\right] }:=U_{\left[ n\right] }\eta _{\left[ n\right] }B_{\left[ n%
\right] }\circ b_{\left[ n\right] }:RL_{\left[ n\right] }B_{\left[ n\right]
}\rightarrow RL_{\left[ n\right] }B_{\left[ n\right] }.
\end{equation*}%
Set $A:=L_{\left[ n\right] }B_{\left[ n\right] };e_{\left[ n\right]
}:RA\rightarrow RA,$
\begin{equation*}
LRA\overset{\epsilon A\circ Le_{\left[ n\right] }}{\underset{\epsilon A}{%
\rightrightarrows }}A\overset{\pi _{\left[ n,n+1\right] }B_{\left[ n+1\right]
}}{\longrightarrow }L_{\left[ n+1\right] }B_{\left[ n+1\right] }.
\end{equation*}%
$\mathcal{E}=\left\{ f:RA\rightarrow RA\mid \exists \mathrm{Coeq}\left(
\epsilon A\circ Lf,\epsilon A\right) \right\} \subseteq \mathcal{A}.$

\begin{equation*}
B_{\left[ n+1\right] }:=\left\{ \left( B_{\left[ n\right] },b_{\left[ n%
\right] }\right) \in \left\langle RL_{\left[ n\right] }|U_{\left[ n\right]
}\right\rangle \mid U_{\left[ n\right] }\eta _{\left[ n\right] }B_{\left[ n%
\right] }\circ b_{\left[ n\right] }\in \mathcal{E}\right\} .
\end{equation*}
\end{invisible}

so that%
\begin{equation}
\pi _{\left[ n+1\right] }B_{\left[ n+1\right] }\circ Lb_{\left[ n\right]
}=\pi _{\left[ n,n+1\right] }B_{\left[ n+1\right] }\circ \epsilon L_{\left[ n%
\right] }B_{\left[ n\right] }  \label{form:pieps}
\end{equation}

By composing the functors on the bottom of \eqref{diag:adjdecomp} and the corresponding natural
transformations one defines, for $0\leq t\leq n,$
\begin{eqnarray*}
U_{\left[ t,n\right] } &=&U_{\left[ t,t+1\right] }\circ U_{\left[ t+1,t+2%
\right] }\circ \cdots \circ U_{\left[ n-2,n-1\right] }\circ U_{\left[ n-1,n%
\right] }, \\
\pi _{\left[ t,n\right] } &=&\pi _{\left[ n-1,n\right] }\ast \pi _{\left[
n-2,n-1\right] }\ast \cdots \ast \pi _{\left[ t+1,t+2\right] }\ast \pi _{%
\left[ t,t+1\right] } \\
&=&\pi _{\left[ n-1,n\right] }\circ \pi _{\left[ n-2,n-1\right] }U_{\left[
n-1,n\right] }\circ \cdots \circ \pi _{\left[ t,t+1\right] }U_{\left[ t+1,n%
\right] }.
\end{eqnarray*}

\begin{invisible}
\begin{equation*}
\begin{array}{ccc}
\mathcal{A} & \overset{\mathrm{Id}_{\mathcal{A}}}{\longrightarrow } &
\mathcal{A} \\
L_{\left[ n\right] }\uparrow \downarrow R_{\left[ n\right] } & \pi _{\left[
n-1,n\right] } & L_{\left[ n-1\right] }\uparrow \downarrow R_{\left[ n-1%
\right] } \\
\mathcal{B}_{\left[ n\right] } & \overset{U_{\left[ n-1,n\right] }}{%
\longrightarrow } & \mathcal{B}_{\left[ n-1\right] }%
\end{array}%
\begin{array}{ccc}
\mathcal{A} & \overset{\mathrm{Id}_{\mathcal{A}}}{\longrightarrow } &
\mathcal{A} \\
L_{\left[ n-1\right] }\uparrow \downarrow R_{\left[ n-1\right] } & \left[
n-2,n-1\right] & L_{\left[ n-2\right] }\uparrow \downarrow R_{\left[ n-2%
\right] } \\
\mathcal{B}_{\left[ n-1\right] } & \overset{U_{\left[ n-2,n-1\right] }}{%
\longrightarrow } & \mathcal{B}_{\left[ n-2\right] }%
\end{array}%
\cdots
\begin{array}{ccc}
\mathcal{A} & \overset{\mathrm{Id}_{\mathcal{A}}}{\longrightarrow } &
\mathcal{A} \\
L_{\left[ t+1\right] }\uparrow \downarrow R_{\left[ t+1\right] } & \pi _{
\left[ t,t+1\right] } & L_{\left[ t\right] }\uparrow \downarrow R_{\left[ t%
\right] } \\
\mathcal{B}_{\left[ t+1\right] } & \overset{U_{\left[ t,t+1\right] }}{%
\longrightarrow } & \mathcal{B}_{\left[ t\right] }%
\end{array}%
\end{equation*}%
\begin{equation*}
L_{t}U_{\left[ t,n\right] }\cdots \overset{\pi _{\left[ n-3,n-2\right] }U_{%
\left[ n-2,n\right] }}{\longrightarrow }L_{\left[ n-2\right] }U_{\left[ n-2,n%
\right] }\overset{\pi _{\left[ n-2,n-1\right] }U_{\left[ n-1,n\right] }}{%
\longrightarrow }L_{\left[ n-1\right] }U_{\left[ n-1,n\right] }\overset{\pi
_{\left[ n-1,n\right] }}{\longrightarrow }L_{\left[ n\right] }
\end{equation*}
\end{invisible}

Let us give a more explicit description of objects and morphisms in the category $\mathcal{B}_{\left[ n\right] }$ for $n\in\mathbb{N}$. First $\mathcal{B}_{\left[ 0\right] }=\mathcal{B}.$ An object in $%
\mathcal{B}_{\left[ 1\right] }$ is a pair $B_{\left[ 1\right] }=\left( B,b_{%
\left[ 0\right] }:RL_{\left[ 0\right] }B\rightarrow B\right) $ where $B\in
\mathcal{B},b_{\left[ 0\right] }\in \mathcal{B}.$ An object in $\mathcal{B}_{%
\left[ 2\right] }$ is a pair $B_{\left[ 2\right] }=\left( B_{\left[ 1\right]
},b_{\left[ 1\right] }:RL_{\left[ 1\right] }B_{\left[ 1\right] }\rightarrow
B\right) ,$ where $B_{\left[ 1\right] }=\left( B,b_{\left[ 0\right] }\right)
\in \mathcal{B}_{\left[ 1\right] },b_{\left[ 1\right] }\in \mathcal{B}.$ Thus we can regard $B_{\left[ 2\right] }$ as the triple $%
\left( B,b_{\left[ 0\right] }:RL_{\left[ 0\right] }B\rightarrow B,b_{\left[ 1%
\right] }:RL_{\left[ 1\right] }B_{\left[ 1\right] }\rightarrow B\right) .$
Going on this way, an object in $\mathcal{B}_{\left[ n\right] }$ has the
form $B_{\left[ n\right] }=\left( B,b_{\left[ 0\right] },b_{\left[ 1\right]
},\ldots ,b_{\left[ n-1\right] }\right) $ where $b_{\left[ t\right] }:RL_{%
\left[ t\right] }B_{\left[ t\right] }\rightarrow B$ and $B_{\left[ t\right]
}=U_{\left[ t,n\right] }B_{\left[ n\right] }=\left( B,b_{\left[ 0\right] },b_{\left[ 1\right]
},\ldots ,b_{\left[ t-1\right] }\right)$ for each $t\in \left\{
0,\ldots ,n-1\right\} .$

The lower case $n=0$ can also be included in the notation $B_{\left[ n\right]
}=\left( B,b_{\left[ 0\right] },b_{\left[ 1\right] },\ldots ,b_{\left[ n-1%
\right] }\right) $ by thinking that the $b_{\left[ i\right] }$'s disappear.
A datum such as $\left( b_{\left[ 0\right] },b_{\left[ 1\right]
},\ldots ,b_{\left[ n-1\right] }\right) $ is called a $R$-structured sink in
the literature.

A morphism $f_{\left[ 1\right] }:B_{\left[ 1\right] }\rightarrow B_{\left[ 1%
\right] }^{\prime }$ in $\mathcal{B}_{\left[ 1\right] }$ means a morphism $%
f=U_{\left[ 1\right] }f_{\left[ 1\right] }:B\rightarrow B^{\prime }$ such
that
\begin{equation*}
\xymatrixcolsep{1.5cm}\xymatrixrowsep{0.5cm}\xymatrix{RL_{[0]
}B\ar[d]_{b_{[0] }}\ar[r]^{RL_{[0] }f}&RL_{[0] }B'\ar[d]^{b_{[0] }'}\\
B\ar[r]^{f}&B' }
\end{equation*}

\begin{invisible}
\begin{equation*}
\begin{array}{ccc}
RL_{[0] }B & \overset{b_{\left[ 0\right] }}{\longrightarrow } & B \\
RL_{\left[ 0\right] }f\downarrow &  & \downarrow f \\
RL_{\left[ 0\right] }B^{\prime } & \overset{b_{\left[ 0\right] }^{\prime }}{%
\longrightarrow } & B^{\prime }%
\end{array}%
\end{equation*}
\end{invisible}

For $n>1,$ a morphism $f_{\left[ n\right] }:B_{\left[ n\right] }\rightarrow
B_{\left[ n\right] }^{\prime }$ in $\mathcal{B}_{\left[ n\right] }$ is a
morphism $f_{\left[ n-1\right] }=U_{\left[ n-1,n\right] }f_{\left[ n\right]
}:B_{\left[ n-1\right] }\rightarrow B_{\left[ n-1\right] }^{\prime }$ such
that%
\begin{equation*}
\xymatrixcolsep{2.3cm}\xymatrixrowsep{0.5cm}\xymatrix{RL_{[n-1]
}B_{[n-1]}\ar[d]_{b_{[n-1] }}\ar[r]^{RL_{[n-1] }f_{[n-1]}}&RL_{[n-1]
}B_{[n-1]}'\ar[d]^{b_{[n-1] }'}\\ B\ar[r]^{f}&B' }
\end{equation*}

\begin{invisible}
\begin{equation*}
\begin{array}{ccc}
RL_{\left[ n-1\right] }B_{\left[ n-1\right] } & \overset{b_{\left[ n-1\right]
}}{\longrightarrow } & B \\
RL_{\left[ n-1\right] }f_{\left[ n-1\right] }\downarrow &  & \downarrow f \\
RL_{\left[ n-1\right] }B_{\left[ n-1\right] }^{\prime } & \overset{b_{\left[
n-1\right] }^{\prime }}{\longrightarrow } & B^{\prime }%
\end{array}%
\end{equation*}
\end{invisible}
where $f=U_{[n]}f_{[n]}:B\to B'$.

\section{Comparing monadic and adjoint decompositions\label{sec:3}}

Next aim is to connect the monadic and adjoint decompositions by
constructing functors $(\Lambda _{n})_{n\in \mathbb{N}}$ making commutative
the solid faces of diagram \eqref{diag:Lambda} for every $n\geq 1$.

To this aim we first prove some technical results needed to obtain
Proposition \ref{pro:5} which is the main tool to iteratively construct $%
(\Lambda _{n})_{n\in \mathbb{N}}$ in Remark \ref{rem:adj}.

\begin{proposition}
\label{pro:1}Assume $\mathcal{A}$ has coequalizers and consider the two adjoint triangles $\mathbb{T}$, $\mathbb{T}%
^{\prime }$ and their composition $\mathbb{T}^{\prime \prime }$ of Remark %
\ref{rem:comptr}. Then there is an adjoint triangle
\begin{equation}  \label{diag:Ilamb}
\xymatrixcolsep{2.3cm}\xymatrixrowsep{0.7cm}\xymatrix{\mathcal{A}\ar[r]^\id%
\ar@{}[dr]|-{\mathcal{U}\left( \epsilon ^{\prime }\right) \boldsymbol{\theta
}}\ar@<.4ex>[d]^*-<0.2cm>{^{\mathrm{R}(\mathbb{T}^{\prime \prime })}}&
\mathcal{A}\ar@<.4ex>[d]^*-<0.2cm>{^{\mathrm{R}(\mathbb{T})}}\\
\mathrm{I}(\mathbb{T}^{\prime \prime
})\ar@<.4ex>@{.>}[u]^*-<0.2cm>{^{\mathrm{L}(\mathbb{T}^{\prime \prime
})}}\ar[r]^{\mathrm{I}\left( \theta
\right)}&{\mathrm{I}(\mathbb{T})}\ar@<.4ex>@{.>}[u]^*-<0.2cm>{^{\mathrm{L}(%
\mathbb{T})}}}
\end{equation}

\begin{invisible}
\begin{equation}
\begin{array}{ccc}
\mathcal{A} & \overset{\mathrm{Id}_{\mathcal{A}}}{\longrightarrow } &
\mathcal{A} \\
\mathrm{L}(\mathbb{T}^{\prime \prime }) \uparrow \downarrow \mathrm{R}\left(
\mathbb{T}^{\prime \prime }\right) & \mathcal{U}\left( \epsilon ^{\prime
}\right) \boldsymbol{\theta } & \mathrm{L}\left( \mathbb{T}\right) \uparrow
\downarrow \mathrm{R}\left( \mathbb{T}\right) \\
\mathrm{I}\left( \mathbb{T}^{\prime \prime }\right) & \overset{\mathrm{I}%
\left( \theta \right) }{\longrightarrow } & \mathrm{I}\left( \mathbb{T}%
\right)%
\end{array}%
\end{equation}
\end{invisible}

The functor $\mathrm{I}\left( \theta \right) $ is defined by%
\begin{equation*}
\mathrm{I}\left( \theta \right) :\mathrm{I}\left( \mathbb{T}^{\prime \prime
}\right) \rightarrow \mathrm{I}\left( \mathbb{T}\right) ,\qquad \left(
B^{\prime \prime },\mu ^{\prime \prime }\right) \mapsto \left( \Theta
B^{\prime \prime },\mu ^{\prime \prime }\circ R^{\prime }\theta B^{\prime
\prime }\right) ,\qquad f\mapsto \Theta P\left( \mathbb{T}^{\prime \prime
}\right) f
\end{equation*}%
and it satisfies%
\begin{equation*}
P\left( \mathbb{T}\right) \circ \mathrm{I}\left( \theta \right) =\Theta
\circ P\left( \mathbb{T}^{\prime \prime }\right) \qquad \text{and}\qquad
G\left( \mathbb{T}\right) \circ \mathrm{I}\left( \theta \right) =G^{\prime
\prime }\circ P\left( \mathbb{T}^{\prime \prime }\right) .
\end{equation*}%
The natural transformation $\boldsymbol{\theta }:\mathcal{L}\left( \mathbb{T}%
\right) \mathrm{I}\left( \theta \right) \rightarrow \mathcal{L}\left(
\mathbb{T}^{\prime\prime }\right) $ appearing inside the adjoint triangle is
defined uniquely by $U_{\left[ 1\right] }^{\prime }\boldsymbol{\theta }=\theta
P\left( \mathbb{T}^{\prime \prime }\right) $, where $U_{[1]}^{\prime }:\langle L'R'|\id \rangle\to \id_{\mathcal{A}}$ is the forgetful functor.

If $\Theta $ is faithful, then so is $\mathrm{I}\left( \theta \right) $.

If $\theta $ is invertible (resp. the identity), then so is $\boldsymbol{\theta } $%
.
\end{proposition}

\begin{proof}
The functor $\mathrm{I}\left( \theta \right) $ can be more properly defined
as follows
\begin{eqnarray*}
\mathrm{I}\left( \theta \right) &:=&D^{\Theta }\circ \left\langle R^{\prime
}\theta \mid \mathrm{Id}_{G^{\prime \prime }}\right\rangle :\mathrm{I}\left(
\mathbb{T}^{\prime \prime }\right) =\left\langle R^{\prime }L^{\prime \prime
}|G^{\prime \prime }\right\rangle \rightarrow \mathrm{I}\left( \mathbb{T}%
\right) =\left\langle R^{\prime }L|G\right\rangle, \\
&&\left( B^{\prime \prime },R^{\prime }L^{\prime \prime }B^{\prime \prime }%
\overset{\mu ^{\prime \prime }}{\rightarrow }G^{\prime \prime }B^{\prime
\prime }\right) \mapsto \left( \Theta B^{\prime \prime },R^{\prime }L\Theta
B^{\prime \prime }\overset{R^{\prime }\theta B^{\prime \prime }}{\rightarrow
}R^{\prime }L^{\prime \prime }B^{\prime \prime }\overset{\mu ^{\prime
\prime }}{\rightarrow }G^{\prime \prime }B^{\prime \prime }=G\Theta
B^{\prime \prime }\right)
\end{eqnarray*}

\begin{invisible}
Thus $\mathrm{I}\left( \theta \right) =\left( \Theta P\left( \mathbb{T}%
^{\prime \prime }\right) \right) \left[ \psi _{\mathrm{I}\left( \mathbb{T}%
^{\prime \prime }\right) }\circ R^{\prime }\theta P\left( \mathbb{T}^{\prime
\prime }\right) \right] .$
\end{invisible}

We compute%
\begin{eqnarray*}
P\left( \mathbb{T}\right) \circ \mathrm{I}\left( \theta \right) &=&P\left(
\mathbb{T}\right) \circ D^{\Theta }\circ \left\langle R^{\prime }\theta \mid
\mathrm{Id}_{G^{\prime \prime }}\right\rangle =\Theta \circ P_{\left\langle
R^{\prime }L\Theta \mid G^{\prime \prime }\right\rangle }\circ \left\langle
R^{\prime }\theta \mid \mathrm{Id}_{G^{\prime \prime }}\right\rangle =\Theta
\circ P\left( \mathbb{T}^{\prime \prime }\right) , \\
G\left( \mathbb{T}\right) \circ \mathrm{I}\left( \theta \right) &=&G\circ
P\left( \mathbb{T}\right) \circ \mathrm{I}\left( \theta \right) =G\circ
\Theta \circ P\left( \mathbb{T}^{\prime \prime }\right) =G^{\prime \prime
}\circ P\left( \mathbb{T}^{\prime \prime }\right) .
\end{eqnarray*}%
Let us construct $\boldsymbol{\theta }:\mathcal{L}\left( \mathbb{T}\right)
\mathrm{I}\left( \theta \right) \rightarrow \mathcal{L}\left( \mathbb{T}%
^{\prime \prime}\right) .$ It is easy to check that
\begin{eqnarray*}
\mathcal{L}\left( \mathbb{T}\right) \mathrm{I}\left( \theta \right)
&=&\left( L\Theta P\left( \mathbb{T}^{\prime \prime }\right) \right) \left[
\zeta \Theta P\left( \mathbb{T}^{\prime \prime }\right) \circ L^{\prime
}\psi _{\mathrm{I}\left( \mathbb{T}^{\prime \prime }\right) }\circ L^{\prime
}R^{\prime }\theta P\left( \mathbb{T}^{\prime \prime }\right) \right] , \\
\mathcal{L}\left( \mathbb{T}^{\prime\prime }\right) &=&\left( L^{\prime \prime
}P\left( \mathbb{T}^{\prime \prime }\right) \right) \left[ \theta P\left(
\mathbb{T}^{\prime \prime }\right) \circ \zeta \Theta P\left( \mathbb{T}%
^{\prime \prime }\right) \circ L^{\prime }\psi _{\mathrm{I}\left( \mathbb{T}%
^{\prime \prime }\right) }\right]
\end{eqnarray*}

\begin{invisible}
In fact
\begin{eqnarray*}
\mathcal{L}\left( \mathbb{T}\right) \mathrm{I}\left( \theta \right) \left(
B^{\prime \prime },\beta ^{\prime \prime }\right) &=&\mathcal{L}\left(
\mathbb{T}\right) \left( \Theta B^{\prime \prime },\beta ^{\prime \prime
}\circ R^{\prime }\theta B^{\prime \prime }\right) =\left( L\Theta B^{\prime
\prime },\zeta \Theta B^{\prime \prime }\circ L^{\prime }\beta ^{\prime
\prime }\circ L^{\prime }R^{\prime }\theta _{B^{\prime \prime }}\right) , \\
\mathcal{L}\left( \mathbb{T}^{\prime }\right) \left( B^{\prime \prime
},\beta ^{\prime \prime }\right) &=&\left( L^{\prime \prime }B^{\prime
\prime },\zeta ^{\prime \prime }B^{\prime \prime }\circ L^{\prime }\beta
^{\prime \prime }\right) =\left( L^{\prime \prime }B^{\prime \prime },\theta
B^{\prime \prime }\circ \zeta \Theta B^{\prime \prime }\circ L^{\prime
}\beta ^{\prime \prime }\right) .
\end{eqnarray*}
\end{invisible}

Since
\begin{eqnarray*}
\theta P\left( \mathbb{T}^{\prime \prime }\right) \circ \left( \zeta
\Theta P\left( \mathbb{T}^{\prime \prime }\right) \circ L^{\prime }\psi _{%
\mathrm{I}\left( \mathbb{T}^{\prime \prime }\right) }\circ L^{\prime
}R^{\prime }\theta P\left( \mathbb{T}^{\prime \prime }\right) \right) =\left( \theta P\left( \mathbb{T}^{\prime \prime }\right) \circ \zeta
\Theta P\left( \mathbb{T}^{\prime \prime }\right) \circ L^{\prime }\psi _{%
\mathrm{I}\left( \mathbb{T}^{\prime \prime }\right) }\right) \circ L^{\prime
}R^{\prime }\theta P\left( \mathbb{T}^{\prime \prime }\right) ,
\end{eqnarray*}%
by Lemma \ref{lem:lift}, there is a unique $\boldsymbol{\theta }:\mathcal{L}%
\left( \mathbb{T}\right) \mathrm{I}\left( \theta \right) \rightarrow
\mathcal{L}\left( \mathbb{T}^{\prime \prime}\right) $ such that $U_{\left[ 1\right]
}^{\prime }\boldsymbol{\theta }=\theta P\left( \mathbb{T}^{\prime \prime
}\right) .$

Consider $\mathcal{U}\left( \epsilon ^{\prime }\right) \boldsymbol{\theta }:%
\mathcal{U}\left( \epsilon ^{\prime }\right) \mathcal{L}\left( \mathbb{T}%
\right) \mathrm{I}\left( \theta \right) \rightarrow \mathcal{U}\left(
\epsilon ^{\prime }\right) \mathcal{L}\left( \mathbb{T}^{\prime \prime}\right) $
i.e. $\mathcal{U}\left( \epsilon ^{\prime }\right) \boldsymbol{\theta }:%
\mathrm{L}\left( \mathbb{T}\right) \mathrm{I}\left( \theta \right)
\rightarrow \mathrm{L}\left( \mathbb{T}^{\prime \prime }\right) .$ This
gives rise the adjoint triangle (\ref{diag:Ilamb}). In fact we have
\begin{eqnarray*}
P\left( \mathbb{T}\right) \circ \mathrm{I}\left( \theta \right) \circ
\mathrm{R}\left( \mathbb{T}^{\prime \prime }\right) &=&\Theta \circ P\left(
\mathbb{T}^{\prime \prime }\right) \circ \mathrm{R}\left( \mathbb{T}^{\prime
\prime }\right) =\Theta \circ R^{\prime \prime }=R=P\left( \mathbb{T}\right)
\circ \mathrm{R}\left( \mathbb{T}\right) , \\
\psi _{\mathrm{I}\left( \mathbb{T}\right) }\mathrm{I}\left( \theta \right)
\mathrm{R}\left( \mathbb{T}^{\prime \prime }\right) &\overset{(*)}{=}&\left( \psi _{\mathrm{%
I}\left( \mathbb{T}^{\prime \prime }\right) }\circ R^{\prime }\theta P\left(
\mathbb{T}^{\prime \prime }\right) \right) \mathrm{R}\left( \mathbb{T}%
^{\prime \prime }\right) =\psi _{\mathrm{I}\left( \mathbb{T}^{\prime \prime
}\right) }\mathrm{R}\left( \mathbb{T}^{\prime \prime }\right) \circ
R^{\prime }\theta P\left( \mathbb{T}^{\prime \prime }\right) \mathrm{R}%
\left( \mathbb{T}^{\prime \prime }\right) \\
&=&R^{\prime }\epsilon ^{\prime \prime }\circ R^{\prime }\theta R^{\prime
\prime }=R^{\prime }\left( \epsilon ^{\prime \prime }\circ \theta R^{\prime
\prime }\right) =R^{\prime }\epsilon =\psi _{\mathrm{I}\left( \mathbb{T}%
\right) }\mathrm{R}\left( \mathbb{T}\right)
\end{eqnarray*}%
where in $(*)$ we used the equality $\psi _{\mathrm{I}\left( \mathbb{T}\right) }\mathrm{I}\left( \theta \right)=\psi _{\mathrm{%
I}\left( \mathbb{T}^{\prime \prime }\right) }\circ R^{\prime }\theta P\left(
\mathbb{T}^{\prime \prime }\right)$ which follows from the computation $\psi _{\mathrm{I}\left( \mathbb{T}\right) }\mathrm{I}\left( \theta \right)(B'',\mu'')=\psi _{\mathrm{I}\left( \mathbb{T}\right) }(\Theta B'',\mu''\circ R'\theta B'')=\mu''\circ R'\theta B''=\psi _{\mathrm{%
I}\left( \mathbb{T}^{\prime \prime }\right) }(B'',\mu'')\circ R^{\prime }\theta P\left(
\mathbb{T}^{\prime \prime }\right)(B'',\mu'')$.

By Lemma \ref{lem:lift}, we get $\mathrm{I}\left( \theta \right) \circ
\mathrm{R}\left( \mathbb{T}^{\prime \prime }\right) =\mathrm{R}\left(
\mathbb{T}\right) .$ Consider the morphisms in the following diagram
\begin{equation*}
\xymatrixcolsep{0.7cm}\xymatrixrowsep{0.7cm}
\xymatrix{&\mathcal{U}(\epsilon')\mathcal{L}(\mathbb{T})\mathrm{I}(\theta) \mathrm{R}(\mathbb{T}'')\ar[rr]^-{\mathcal{U}%
\left( \epsilon ^{\prime }\right) \boldsymbol{\theta }\mathrm{R}\left(
\mathbb{T}^{\prime \prime }\right)}\ar[dr]_{\epsilon \left( \mathbb{T}\right)}&&\mathcal{U}(\epsilon')\mathcal{L}(\mathbb{T}^{\prime \prime })\mathrm{R}(\mathbb{T}^{\prime \prime })=\mathrm{L}(\mathbb{T}^{\prime \prime })\mathrm{R}(\mathbb{T}^{\prime \prime })\ar[dl]^{\epsilon \left( \mathbb{T}^{\prime \prime }\right)}\\U'_{[1]}\mathcal{L}(\mathbb{T})\mathrm{I}(\theta) \mathrm{R}(\mathbb{T}'')\ar@{->>}[ur]^(.35){\pi (\epsilon^{\prime })\mathcal{L}\left( \mathbb{T}%
\right) \mathrm{I}\left( \theta \right) \mathrm{R}\left( \mathbb{T}^{\prime
\prime }\right)}&&\id_{\mathcal{A}}}
\end{equation*}
Then 
\begin{eqnarray*}
\epsilon \left( \mathbb{T}^{\prime \prime }\right) \circ \mathcal{U}\left(
\epsilon ^{\prime }\right) \boldsymbol{\theta }\mathrm{R}\left( \mathbb{T}%
^{\prime \prime }\right) \circ \pi (\epsilon^{\prime })\mathcal{L}\left( \mathbb{T}%
\right) \mathrm{I}\left( \theta \right) \mathrm{R}\left( \mathbb{T}^{\prime
\prime }\right) &=&\epsilon \left( \mathbb{T}^{\prime \prime }\right) \circ
\pi (\epsilon^{\prime })\mathcal{L}\left( \mathbb{T}^{\prime \prime }\right) \mathrm{R}%
\left( \mathbb{T}^{\prime \prime }\right) \circ U_{\left[ 1\right] }^{\prime
}\boldsymbol{\theta }\mathrm{R}\left( \mathbb{T}^{\prime \prime }\right) \\
\overset{(\ref{form:defepsT})}{=}\epsilon ^{\prime \prime }\circ \theta
P\left( \mathbb{T}^{\prime \prime }\right) \mathrm{R}\left( \mathbb{T}%
^{\prime \prime }\right) &=&\epsilon ^{\prime \prime }\circ \theta R^{\prime
\prime }=\epsilon \overset{(\ref{form:defepsT})}{=}\epsilon \left( \mathbb{T}%
\right) \circ \pi (\epsilon^{\prime })\mathcal{L}\left( \mathbb{T}\right) \mathrm{R}%
\left( \mathbb{T}\right) \\
&=&\epsilon \left( \mathbb{T}\right) \circ \pi (\epsilon^{\prime })\mathcal{L}\left(
\mathbb{T}\right) \mathrm{I}\left( \theta \right) \mathrm{R}\left( \mathbb{T}%
^{\prime \prime }\right)
\end{eqnarray*}%
so that $\epsilon \left( \mathbb{T}^{\prime \prime }\right) \circ \mathcal{U}%
\left( \epsilon ^{\prime }\right) \boldsymbol{\theta }\mathrm{R}\left(
\mathbb{T}^{\prime \prime }\right) =\epsilon \left( \mathbb{T}\right) $
which means that $\mathcal{U}\left( \epsilon ^{\prime }\right) \boldsymbol{%
\theta }$ is the correct natural transformation to put inside the adjoint
triangle.


If $\Theta $ is faithful, from $P\left( \mathbb{T}\right) \circ \mathrm{I}%
\left( \theta \right) =\Theta \circ P\left( \mathbb{T}^{\prime \prime
}\right) $ and the fact that $P\left( \mathbb{T}^{\prime \prime }\right) $
is faithful we get that $\mathrm{I}\left( \theta \right) $ is faithful too.

If $\theta $ is invertible, from $U_{\left[ 1\right] }\boldsymbol{\theta }%
=\theta P\left( \mathbb{T}^{\prime \prime }\right) $ and the fact that $U_{%
\left[ 1\right] }$ reflects isomorphisms, we deduce that $\boldsymbol{\theta
}$ is invertible as well.

If $\theta $ is the identity, then $L\Theta =L^{\prime \prime }$ so that, by
the foregoing, we get
\begin{equation*}
\mathcal{L}\left( \mathbb{T}\right) \mathrm{I}\left( \theta \right) \left(
B^{\prime \prime },\mu ^{\prime \prime }\right) =\left( L\Theta B^{\prime
\prime },\zeta \Theta B^{\prime \prime }\circ L^{\prime }\mu ^{\prime
\prime }\circ L^{\prime }R^{\prime }\theta _{B^{\prime \prime }}\right)
=\left( L^{\prime \prime }B^{\prime \prime },\theta B^{\prime \prime }\circ
\zeta \Theta B^{\prime \prime }\circ L^{\prime }\mu ^{\prime \prime
}\right) =\mathcal{L}\left( \mathbb{T}^{\prime\prime }\right) \left( B^{\prime
\prime },\mu ^{\prime \prime }\right) .
\end{equation*}%
Hence the domain and codomain of $\boldsymbol{\theta }\left( B^{\prime
\prime },\mu ^{\prime \prime }\right) $ are the same. Thus, since $U_{%
\left[ 1\right] }^{\prime }\boldsymbol{\theta }\left( B^{\prime \prime
},\mu ^{\prime \prime }\right) =\theta B^{\prime \prime }=\mathrm{Id}%
_{L^{\prime \prime }B^{\prime \prime }}=U_{\left[ 1\right] }^{\prime }%
\mathrm{Id}_{\mathcal{L}\left( \mathbb{T}^{\prime\prime }\right) \left( B^{\prime
\prime },\mu ^{\prime \prime }\right) }$ and $U_{\left[ 1\right] }^{\prime
}$ is faithful, we obtain $\boldsymbol{\theta }\left( B^{\prime \prime
},\mu ^{\prime \prime }\right) =\mathrm{Id}_{\mathcal{L}\left( \mathbb{T}%
^{\prime \prime}\right) \left( B^{\prime \prime },\mu ^{\prime \prime }\right)
}. $
\end{proof}

\begin{proposition}
\label{pro:2}Assume $\mathcal{A}$ has coequalizers. Given an adjoint triangle $\mathbb{T}$ as in (\ref{diag:AugTr}%
), then
\begin{equation*}
\xymatrixcolsep{1.5cm}\xymatrixrowsep{0.7cm}\xymatrix{\mathcal{A}\ar[r]^\id%
\ar@{}[dr]|-{\sigma _{[1]}}\ar@<.4ex>[d]^*-<0.2cm>{^{R_{[1]}}}&
\mathcal{A}\ar@<.4ex>[d]^*-<0.2cm>{^{\mathrm{R}(\mathbb{T})}}\\
\mathcal{B}_{[1]}\ar@<.4ex>@{.>}[u]^*-<0.2cm>{^{L_{[1]}}}\ar[r]^{S^G}&%
\mathrm{I}(\mathbb{T})\ar@<.4ex>@{.>}[u]^*-<0.2cm>{^{\mathrm{L}(%
\mathbb{T})}}}
\end{equation*}

\begin{invisible}
\begin{equation*}
\begin{array}{ccc}
\mathcal{A} & \overset{\mathrm{Id}_{\mathcal{A}}}{\longrightarrow } &
\mathcal{A} \\
L_{\left[ 1\right] }\uparrow \downarrow R_{\left[ 1\right] } & \sigma _{
\left[ 1\right] } & \mathrm{L}\left( \mathbb{T}\right) \uparrow \downarrow
\mathrm{R}\left( \mathbb{T}\right) \\
\mathcal{B}_{\left[ 1\right] } & \overset{S^{G}}{\longrightarrow } & \mathrm{%
I}\left( \mathbb{T}\right)%
\end{array}%
\end{equation*}
\end{invisible}
is an adjoint triangle too. If $\zeta R$ is epimorphism on each component then $\sigma _{\left[ 1\right]
}$ is invertible.
\end{proposition}

\begin{proof}
Recall that $\mathcal{B}_{\left[ 1\right] }=\left\langle RL\mid \mathrm{Id}%
\right\rangle .$ Note that%
\begin{equation*}
S^{G}R_{\left[ 1\right] }A=S^{G}\left( RA,R\epsilon A\right) =\left(
RA,GR\epsilon A\right) =\left( RA,R^{\prime }\epsilon A\right) =\mathrm{R}%
\left( \mathbb{T}\right) A
\end{equation*}%
and since $S^{G}R_{\left[ 1\right] }$ and $\mathrm{R}\left( \mathbb{T}%
\right) $ coincide also on morphisms we get they are equal. Hence we have an
adjoint triangle as in the statement where $\sigma _{\left[ 1\right]
}:=\epsilon \left( \mathbb{T}\right) L_{\left[ 1\right] }\circ \mathrm{L}%
\left( \mathbb{T}\right) S^{G}\eta _{\left[ 1\right] }.$ Call $\mathbb{T}%
^{\prime }$ the diagram in the statement and let $\mathbb{T}^{2}$ be the
diagram of Proposition \ref{pro:adj}. Since $P\left( \mathbb{T}\right) \circ
S^{G}=U_{\left[ 1\right] }$ we get that $\mathbb{T}^{\prime }\ast \mathbb{T}%
^{2}=\mathbb{T}_{\left[ 1\right] }.$ Thus
\begin{equation*}
\pi _{\left[ 1\right] }=\sigma _{\left[ 1\right] }\ast \pi \left( \epsilon
^{\prime }\right) \mathcal{L}\left( \mathbb{T}\right) =\sigma _{\left[ 1%
\right] }\circ \pi \left( \epsilon ^{\prime }\right) \mathcal{L}\left(
\mathbb{T}\right) S^{G}.
\end{equation*}

\begin{invisible}
Here it is
\begin{equation*}
\begin{array}{ccc}
\mathcal{A} & \overset{\mathrm{Id}_{\mathcal{A}}}{\longrightarrow } &
\mathcal{A} \\
L_{\left[ 1\right] }\uparrow \downarrow R_{\left[ 1\right] } & \sigma _{
\left[ 1\right] } & \mathrm{L}\left( \mathbb{T}\right) \uparrow \downarrow
\mathrm{R}\left( \mathbb{T}\right) \\
\mathcal{B}_{\left[ 1\right] } & \overset{S^{G}}{\longrightarrow } & \mathrm{%
I}\left( \mathbb{T}\right)%
\end{array}%
\ast
\begin{array}{ccc}
\mathcal{A} & \overset{\mathrm{Id}_{\mathcal{A}}}{\longrightarrow } &
\mathcal{A} \\
\mathrm{L}\left( \mathbb{T}\right) \uparrow \downarrow \mathrm{R}\left(
\mathbb{T}\right) & \pi \left( \epsilon ^{\prime }\right) \mathcal{L}\left(
\mathbb{T}\right) & L\uparrow \downarrow R \\
\mathrm{I}\left( \mathbb{T}\right) & \overset{P\left( \mathbb{T}\right) }{%
\longrightarrow } & \mathcal{B}%
\end{array}%
=%
\begin{array}{ccc}
\mathcal{A} & \overset{\mathrm{Id}_{\mathcal{A}}}{\longrightarrow } &
\mathcal{A} \\
L_{\left[ 1\right] }\uparrow \downarrow R_{\left[ 1\right] } & \pi _{\left[ 1%
\right] } & L\uparrow \downarrow R \\
\mathcal{B}_{\left[ 1\right] } & \overset{P\left( \mathbb{T}\right) S^{G}}{%
\longrightarrow } & \mathcal{B}%
\end{array}%
\end{equation*}
\end{invisible}

It is easy to check that $U_{\left[ 1\right] }\circ \mathcal{L}\left(
\mathbb{T}\right) \circ S^{G}=L\circ U_{\left[ 1\right] }.$ Moreover, by
definition of $\mathcal{L}\left( \mathbb{T}\right) $ one gets
\begin{equation*}
\psi _{\left\langle L^{\prime }R^{\prime }\mid \mathrm{Id}\right\rangle }%
\mathcal{L}\left( \mathbb{T}\right) =\zeta P\left( \mathbb{T}\right) \circ
L^{\prime }\psi _{\mathrm{I}\left( \mathbb{T}\right) }.
\end{equation*}

\begin{invisible}
For every $\left( B,b\right) \in \mathrm{I}\left( \mathbb{T}\right)
=\left\langle R^{\prime }L|G\right\rangle ,$ we have
\begin{equation*}
\psi _{\left\langle L^{\prime }R^{\prime }\mid \mathrm{Id}\right\rangle }%
\mathcal{L}\left( \mathbb{T}\right) \left( B,b\right) =\psi \left( LB,\zeta
_{B}\circ L^{\prime }b\right) =\zeta _{B}\circ L^{\prime }b=\zeta _{B}\circ
L^{\prime }\psi _{\mathcal{B}\left( \mathbb{T}\right) }\left( B,b\right) .
\end{equation*}
\end{invisible}

In particular
\begin{equation*}
\psi _{\left\langle L^{\prime }R^{\prime }\mid \mathrm{Id}\right\rangle }%
\mathcal{L}\left( \mathbb{T}\right) S^{G}=\zeta P\left( \mathbb{T}\right)
S^{G}\circ L^{\prime }\psi _{\mathrm{I}\left( \mathbb{T}\right) }S^{G}=\zeta
U_{\left[ 1\right] }\circ L^{\prime }G\psi_{\left\langle RL\mid \mathrm{Id}%
\right\rangle } \overset{\text{nat.}\zeta \text{ }%
}{=}L\psi _{\left\langle RL\mid \mathrm{Id}\right\rangle }\circ \zeta RLU_{%
\left[ 1\right] }
\end{equation*}%
so that the following diagram of coequalizers serially commutes
\begin{equation*}
\xymatrixcolsep{2.3cm} \xymatrix{L^{\prime }R^{\prime }P_{\langle L'R'\mid\id\rangle}\mathcal{L}\left( \mathbb{T}\right) S^{G}\ar[d]_{\zeta
RLU_{[1]}}\ar@<.5ex>[r]^-{\psi _{\left\langle L^{\prime }R^{\prime }\mid
\mathrm{Id}\right\rangle }\mathcal{L}\left( \mathbb{T}\right)
S^{G}}\ar@<-.5ex>[r]_-{\epsilon ^{\prime }P_{\langle L'R'\mid\id\rangle}\mathcal{L}\left( \mathbb{T}\right) S^{G}}&P_{\langle L'R'\mid\id\rangle}\mathcal{L}\left( \mathbb{T}\right) S^{G}\ar[d]^{\id}\ar[r]^-{\pi \left(
\epsilon ^{\prime }\right) \mathcal{L}\left( \mathbb{T}\right) S^{G}}&
\mathcal{U}\left( \epsilon ^{\prime }\right) \mathcal{L}\left(
\mathbb{T}\right) S^{G}=\mathrm{L}\left( \mathbb{T}\right)
S^{G}\ar[d]^{\sigma_{[1]}}\\ LRLU_{[1]}\ar@<.5ex>[r]^{L\psi _{\left\langle
RL\mid \mathrm{Id}\right\rangle }}\ar@<-.5ex>[r]_{\epsilon LU_{\left[
1\right] }}&LU_{[1]}\ar[r]^{\pi _{[1]}}& L_{[1]}}
\end{equation*}

\begin{invisible}
\begin{equation*}
\begin{array}{ccccc}
L^{\prime }R^{\prime }U_{\left[ 1\right] }\mathcal{L}\left( \mathbb{T}%
\right) S^{G} & \overset{\psi _{\left\langle L^{\prime }R^{\prime }\mid
\mathrm{Id}\right\rangle }\mathcal{L}\left( \mathbb{T}\right) S^{G}}{%
\underset{\epsilon ^{\prime }U_{\left[ 1\right] }\mathcal{L}\left( \mathbb{T}%
\right) S^{G}}{\rightrightarrows }} & U_{\left[ 1\right] }\mathcal{L}\left(
\mathbb{T}\right) S^{G} & \overset{\pi \left( \epsilon ^{\prime }\right)
\mathcal{L}\left( \mathbb{T}\right) S^{G}}{\rightarrow } & \mathcal{U}\left(
\epsilon ^{\prime }\right) \mathcal{L}\left( \mathbb{T}\right) S^{G}=\mathrm{%
L}\left( \mathbb{T}\right) S^{G} \\
\zeta RLU_{\left[ 1\right] }\downarrow &  & \downarrow \mathrm{Id} &  &
\downarrow \sigma _{\left[ 1\right] } \\
LRLU_{\left[ 1\right] } & \overset{L\psi _{\left\langle RL\mid \mathrm{Id}%
\right\rangle }}{\underset{\epsilon LU_{\left[ 1\right] }}{\rightrightarrows
}} & LU_{\left[ 1\right] } & \overset{\pi _{\left[ 1\right] }}{\rightarrow }
& L_{\left[ 1\right] }%
\end{array}%
\end{equation*}
\end{invisible}

If $\zeta R$ is epimorphism on each component it is then easy to check that $%
\pi \left( \epsilon ^{\prime }\right) \mathcal{L}\left( \mathbb{T}\right)
S^{G}$ is a coequalizer for the pair $\left( L\psi _{\left\langle RL\mid
\mathrm{Id}\right\rangle },\epsilon LU_{\left[ 1\right] }\right) $ and hence
$\sigma _{\left[ 1\right] }$ is invertible.

\begin{invisible}
Consider the particular case when $\mathbb{T=T}_{\left[ n\right] }$
\begin{equation*}
\mathbb{T}_{\left[ n\right] }:=%
\begin{array}{ccc}
\mathcal{A} & \overset{\mathrm{Id}_{\mathcal{A}}}{\longrightarrow } &
\mathcal{A} \\
L_{\left[ n\right] }\uparrow \downarrow R_{\left[ n\right] } & \pi _{\left[ n%
\right] }:LU_{\left[ n\right] }\rightarrow L_{\left[ n\right] } & L\uparrow
\downarrow R \\
\mathcal{B}_{\left[ n\right] } & \overset{U_{\left[ n\right] }}{%
\longrightarrow } & \mathcal{B}%
\end{array}%
\qquad
\begin{array}{ccc}
\mathcal{A} & \overset{\mathrm{Id}_{\mathcal{A}}}{\longrightarrow } &
\mathcal{A} \\
L_{\left[ n\right] \left[ 1\right] }\uparrow \downarrow R_{\left[ n\right] %
\left[ 1\right] } & \sigma _{\left[ 1\right] } & L_{\left[ n+1\right]
}\uparrow \downarrow R_{\left[ n+1\right] } \\
\left\langle R_{\left[ n\right] }L_{\left[ n\right] }\mid \mathrm{Id}%
\right\rangle & \overset{S^{U_{\left[ n\right] }}}{\longrightarrow } &
\mathcal{B}_{\left[ n+1\right] }%
\end{array}%
\end{equation*}%
\begin{eqnarray*}
&&L_{\left[ n+1\right] }S^{U_{\left[ n\right] }}\left( B_{\left[ n\right]
},b:R_{\left[ n\right] }L_{\left[ n\right] }B_{\left[ n\right] }\rightarrow
B_{\left[ n\right] }\right) \\
&=&L_{\left[ n+1\right] }\left( B_{\left[ n\right] },U_{\left[ n\right]
}b:RL_{\left[ n\right] }B_{\left[ n\right] }\rightarrow B\right) =Coeq\left(
\pi _{\left[ n\right] }B_{\left[ n\right] }\circ LU_{\left[ n\right]
}b,\epsilon L_{\left[ n\right] }B_{\left[ n\right] }\right) \\
&=&Coeq\left( L_{\left[ n\right] }b\circ \pi _{\left[ n\right] }R_{\left[ n%
\right] }L_{\left[ n\right] }B_{\left[ n\right] },\epsilon L_{\left[ n\right]
}B_{\left[ n\right] }\right) \\
&=&Coeq\left( L_{\left[ n\right] }b\circ \pi _{\left[ n\right] }R_{\left[ n%
\right] }L_{\left[ n\right] }B_{\left[ n\right] },\epsilon _{\left[ n\right]
}L_{\left[ n\right] }B_{\left[ n\right] }\circ \pi _{\left[ n\right] }R_{%
\left[ n\right] }L_{\left[ n\right] }B_{\left[ n\right] }\right) \\
&=&Coeq\left( L_{\left[ n\right] }b,\epsilon _{\left[ n\right] }L_{\left[ n%
\right] }B_{\left[ n\right] }\right) =L_{\left[ n\right] \left[ 1\right]
}\left( B_{\left[ n\right] },b\right)
\end{eqnarray*}%
\begin{eqnarray*}
\mathrm{L}\left( \mathbb{T}\right) S^{G}\left( B,b:RLB\rightarrow B\right)
&=&\mathrm{L}\left( \mathbb{T}\right) \left( B,Gb:GRLB\rightarrow GB\right)
\\
&=&Coeq\left( \zeta B\circ L^{\prime }Gb,\epsilon ^{\prime }LB\right) \\
&=&Coeq\left( Lb\circ \zeta RLB,\epsilon LB\circ \zeta RLB\right) \\
&=&Coeq\left( Lb,\epsilon LB\right) =L_{\left[ 1\right] }\left( B,b\right)
\end{eqnarray*}
\end{invisible}
\end{proof}

\begin{proposition}
\label{pro:3} Assume $\mathcal{A}$ has coequalizers.  Consider the two adjoint triangles $\mathbb{T}$, $\mathbb{T}%
^{\prime }$ and their composition $\mathbb{T}^{\prime \prime }$ of Remark %
\ref{rem:comptr}. Then we can define a new adjoint triangle
\begin{equation}  \label{diag:pro3}
\xymatrixcolsep{1.5cm}\xymatrixrowsep{0.7cm}\xymatrix{\mathcal{A}\ar[r]^\id%
\ar@{}[dr]|-{\theta _{[1]}}\ar@<.4ex>[d]^*-<0.2cm>{^{R_{[1]}''}}&
\mathcal{A}\ar@<.4ex>[d]^*-<0.2cm>{^{\mathrm{R}(\mathbb{T})}}\\
\mathcal{B}_{[1]}''\ar@<.4ex>@{.>}[u]^*-<0.2cm>{^{L_{[1]}''}}\ar[r]^{\Theta
_{[1]}}&\mathrm{I}(\mathbb{T})\ar@<.4ex>@{.>}[u]^*-<0.2cm>{^{\mathrm{L}(%
\mathbb{T})}}}
\end{equation}

\begin{invisible}
\begin{equation*}
\begin{array}{ccc}
\mathcal{A} & \overset{\mathrm{Id}_{\mathcal{A}}}{\longrightarrow } &
\mathcal{A} \\
L_{\left[ 1\right] }^{\prime \prime }\uparrow \downarrow R_{\left[ 1\right]
}^{\prime \prime } & \theta _{\left[ 1\right] } & \mathrm{L}\left( \mathbb{T}%
\right) \uparrow \downarrow \mathrm{R}\left( \mathbb{T}\right) \\
\mathcal{B}_{\left[ 1\right] }^{\prime \prime } & \overset{\Theta _{\left[ 1%
\right] }}{\longrightarrow } & \mathrm{I}\left( \mathbb{T}\right)%
\end{array}%
\end{equation*}
\end{invisible}

where
\begin{equation*}
\Theta _{\left[ 1\right] }:\mathcal{B}_{\left[ 1\right] }^{\prime \prime
}\rightarrow \mathrm{I}\left( \mathbb{T}\right) ,\qquad \left( V^{\prime
\prime },\mu ^{\prime \prime }\right) \mapsto \left( \Theta V^{\prime \prime
},G^{\prime \prime }\mu ^{\prime \prime }\circ R^{\prime }\theta V^{\prime
\prime }\right) ,\qquad f\mapsto \Theta U_{\left[ 1\right] }^{\prime \prime
}f,
\end{equation*}%
and such that%
\begin{equation*}
P\left( \mathbb{T}\right) \Theta _{\left[ 1\right] }=\Theta U_{\left[ 1%
\right] }^{\prime \prime }\qquad \text{and}\qquad G\left( \mathbb{T}\right)
\Theta _{\left[ 1\right] }=G^{\prime \prime }U_{\left[ 1\right] }^{\prime
\prime }.
\end{equation*}

\begin{itemize}
\item[1)] If $\Theta $ is faithful, then so is $\Theta _{\left[ 1\right] }$.

\item[2)] If $\theta $ is invertible and any component of $\zeta R$ is an
epimorphism, then $\theta _{\left[ 1\right] }$ is invertible.
\end{itemize}
\end{proposition}

\begin{proof}
By composing the two adjoint triangles obtained in Proposition \ref{pro:1}
and Proposition \ref{pro:2}, the latter applied to $\mathbb{T}^{\prime
\prime }$, i.e.
\begin{equation*}
\xymatrixcolsep{2.5cm}\xymatrixrowsep{0.7cm}\xymatrix{\mathcal{A}\ar[r]^\id%
\ar@{}[dr]|{\sigma'' _{[1]}}\ar@<.4ex>[d]^*-<0.2cm>{^{R_{[1]}''}}&
\mathcal{A}\ar@{}[dr]|-{\mathcal{U}\left( \epsilon ^{\prime }\right)
\boldsymbol{\theta
}}\ar@<.4ex>[d]^*-<0.2cm>{^{\mathrm{R}(\mathbb{T''})}}\ar[r]^\id&\mathcal{A}%
\ar@<.4ex>[d]^*-<0.2cm>{^{\mathrm{R}(\mathbb{T})}}\\
\mathcal{B}_{[1]}''=\left\langle
R''L''\mid\id\right\rangle\ar@<.4ex>@{.>}[u]^*-<0.2cm>{^{L_{[1]}''}}%
\ar[r]^{S^{G''}}&\left\langle G''R''L''\mid
G''\right\rangle=\mathrm{I}(\mathbb{T''})\ar@<.4ex>@{.>}[u]^*-<0.2cm>{^{%
\mathrm{L}(\mathbb{T''})}}\ar[r]^-{\mathrm{I}(\theta)}&\mathrm{I}(%
\mathbb{T})\ar@<.4ex>@{.>}[u]^*-<0.2cm>{^{\mathrm{L}(\mathbb{T})}} }
\end{equation*}

\begin{invisible}

\begin{equation*}
\begin{array}{ccccc}
\mathcal{A} & \overset{\mathrm{Id}_{\mathcal{A}}}{\longrightarrow } &
\mathcal{A} & \overset{\mathrm{Id}_{\mathcal{A}}}{\longrightarrow } &
\mathcal{A} \\
L_{\left[ 1\right] }^{\prime \prime }\uparrow \downarrow R_{\left[ 1\right]
}^{\prime \prime } & \sigma _{\left[ 1\right] }^{\prime \prime } & \mathrm{L}%
\left( \mathbb{T}^{\prime \prime }\right) \uparrow \downarrow \mathrm{R}%
\left( \mathbb{T}^{\prime \prime }\right) & \mathcal{U}\left( \epsilon
^{\prime }\right) \boldsymbol{\theta } & \mathrm{L}\left( \mathbb{T}\right)
\uparrow \downarrow \mathrm{R}\left( \mathbb{T}\right) \\
\mathcal{B}_{\left[ 1\right] }^{\prime \prime }=\left\langle R^{\prime
\prime }L^{\prime \prime }\mid \mathrm{Id}\right\rangle & \overset{%
S^{G^{\prime \prime }}}{\longrightarrow } & \left\langle G^{\prime \prime
}R^{\prime \prime }L^{\prime \prime }\mid G^{\prime \prime }\right\rangle =%
\mathrm{I}\left( \mathbb{T}^{\prime \prime }\right) & \overset{\mathrm{I}%
\left( \theta \right) }{\longrightarrow } & \mathrm{I}\left( \mathbb{T}%
\right)%
\end{array}%
\end{equation*}
\end{invisible}

we obtain the triangle \eqref{diag:pro3} with $\theta _{\left[ 1\right]
}:=\sigma _{\left[ 1\right] }^{\prime \prime }\ast \mathcal{U}\left(
\epsilon ^{\prime }\right) \boldsymbol{\theta }=\sigma _{\left[ 1\right]
}^{\prime \prime }\circ \mathcal{U}\left( \epsilon ^{\prime }\right)
\boldsymbol{\theta }S^{G^{\prime \prime }}$ and $\Theta _{\left[ 1\right] }=%
\mathrm{I}\left( \theta \right) \circ S^{G^{\prime \prime }}.$ Explicitly
for every $\left( V^{\prime \prime },\mu ^{\prime \prime }\right) \in
\mathcal{B}_{\left[ 1\right] }^{\prime \prime }$ we have
\begin{equation*}
\Theta _{\left[ 1\right] }\left( V^{\prime \prime },\mu ^{\prime \prime
}\right) =\mathrm{I}\left( \theta \right) S^{G^{\prime \prime }}\left(
V^{\prime \prime },\mu ^{\prime \prime }\right) =\mathrm{I}\left( \theta
\right) \left( V^{\prime \prime },G^{\prime \prime }\mu ^{\prime \prime
}\right) =\left( \Theta V^{\prime \prime },G^{\prime \prime }\mu ^{\prime
\prime }\circ R^{\prime }\theta V^{\prime \prime }\right)
\end{equation*}%
and for every morphism $f\in \mathcal{B}_{\left[ 1\right]
}^{\prime \prime }$ we have $\Theta _{\left[ 1\right] }f=\mathrm{I}\left(
\theta \right) S^{G^{\prime \prime }}f=\Theta U_{\left[ 1\right] }^{\prime
\prime }f$.

If $\Theta $ is faithful, then, by Proposition \ref{pro:1}, so is $\mathrm{I}%
\left( \theta \right) $. Since $S^{G^{\prime \prime }}$ acts as the identity
on morphisms it is faithful too and we get that $\Theta _{\left[ 1\right] }$
is faithful as a composition of faithful functors.

Assume $\theta $ is invertible and that any component of $\zeta R$ is an
epimorphism. By Proposition \ref{pro:1}, $\boldsymbol{\theta }$ is
invertible. Now%
\begin{equation*}
\zeta ^{\prime \prime }R^{\prime \prime }=\left( \theta \ast \zeta \right)
R^{\prime \prime }=\theta R^{\prime \prime }\circ \zeta \Theta R^{\prime
\prime }=\theta R^{\prime \prime }\circ \zeta R
\end{equation*}%
which is an epimorphism on each component. Thus, by Proposition \ref{pro:2},
we get that $\sigma _{\left[ 1\right] }^{\prime \prime }$ is invertible.
Hence $\theta _{\left[ 1\right] }$ is invertible as a composition of
invertible natural transformations.
\begin{invisible}
Here we investigate a different way to split the adjoint triangle
involving $\Theta _{\left[ 1\right] }.$

Consider $U_{\left[ 1\right] }:\mathcal{B}_{\left[ 1\right] }=\left\langle
L^{\prime }R^{\prime },\mathrm{Id}_{\mathcal{A}}\right\rangle \rightarrow
\mathcal{A},$ $\psi ^{\prime }:=\psi _{\mathcal{B}_{\left[ 1\right]
}}:L^{\prime }R^{\prime }U_{\left[ 1\right] }\rightarrow U_{\left[ 1\right]
}.$ Note that $U_{\left[ 1\right] }\mathcal{L}\left( \mathbb{T}\right)
=LP\left( \mathbb{T}\right) $ and $\mathrm{L}\left( \mathbb{T}\right) =%
\mathcal{U}\left( \epsilon ^{\prime }\right) \circ \mathcal{L}\left( \mathbb{%
T}\right) .$

We can consider $\mathcal{L}\left( \mathbb{T}^{\prime \prime }\right)
:\left\langle R^{\prime }L^{\prime \prime }|G^{\prime \prime }\right\rangle
\rightarrow \left\langle L^{\prime }R^{\prime }|\mathrm{Id}_{\mathcal{A}%
}\right\rangle $ and we have $U_{\left[ 1\right] }\circ \mathcal{L}\left(
\mathbb{T}^{\prime \prime }\right) =L^{\prime \prime }\circ P\left( \mathbb{T%
}^{\prime \prime }\right) .$

One can also define the functor%
\begin{equation*}
\Gamma _{\left[ 1\right] }:\mathcal{B}_{\left[ 1\right] }^{\prime \prime
}=\left\langle R^{\prime \prime }L^{\prime \prime }\mid \mathrm{Id}%
\right\rangle \overset{S^{\Theta }}{\longrightarrow }\left\langle \Theta
R^{\prime \prime }L^{\prime \prime }\mid \Theta \right\rangle =\left\langle
RL^{\prime \prime }\mid \Theta \right\rangle \overset{\left\langle R^{\prime
}\theta \mid \Theta \right\rangle }{\longrightarrow }\left\langle RL\Theta
\mid \Theta \right\rangle \overset{D^{\Theta }}{\longrightarrow }%
\left\langle RL\mid \mathrm{Id}\right\rangle =\mathcal{B}_{\left[ 1\right] }.
\end{equation*}%
Explicitly
\begin{equation*}
\Gamma _{\left[ 1\right] }:\left( B^{\prime \prime },\beta ^{\prime \prime
}\right) \mapsto \left( \Theta B^{\prime \prime },\Theta \beta ^{\prime
\prime }\circ R\theta B^{\prime \prime }\right) ;\qquad f\mapsto \Theta f.
\end{equation*}%
We have $\Gamma _{\left[ 1\right] }R_{\left[ 1\right] }^{\prime \prime }$
and $R_{\left[ 1\right] }$ coincides on morphisms and
\begin{eqnarray*}
\Gamma _{\left[ 1\right] }R_{\left[ 1\right] }^{\prime \prime }A &=&\Gamma _{%
\left[ 1\right] }\left( R^{\prime \prime }A,R^{\prime \prime }\epsilon
^{\prime \prime }A\right) =\left( \Theta R^{\prime \prime }A,\Theta
R^{\prime \prime }\epsilon ^{\prime \prime }A\circ R\theta R^{\prime \prime
}A\right) \\
&=&\left( RA,R\epsilon ^{\prime \prime }A\circ R\theta R^{\prime \prime
}A\right) =\left( RA,R\left( \epsilon ^{\prime \prime }\circ \theta
R^{\prime \prime }\right) A\right) =\left( RA,R\epsilon A\right) =R_{\left[ 1%
\right] }A
\end{eqnarray*}%
so that $\Gamma _{\left[ 1\right] }R_{\left[ 1\right] }^{\prime \prime }=R_{%
\left[ 1\right] }.$ Consider the following diagram%
\begin{equation*}
\begin{array}{ccccc}
LRLU_{\left[ 1\right] }\Gamma _{\left[ 1\right] } & \overset{L\psi \Gamma _{%
\left[ 1\right] }}{\underset{\epsilon LU_{\left[ 1\right] }\Gamma _{\left[ 1%
\right] }}{\rightrightarrows }} & LU_{\left[ 1\right] }\Gamma _{\left[ 1%
\right] } & \overset{\pi _{\left[ 1\right] }\Gamma _{\left[ 1\right] }}{%
\rightarrow } & L_{\left[ 1\right] }\Gamma _{\left[ 1\right] } \\
\theta R^{\prime \prime }\theta U_{\left[ 1\right] }^{\prime \prime
}\downarrow &  & \downarrow \theta U_{\left[ 1\right] }^{\prime \prime } &
& \downarrow \exists \gamma _{\left[ 1\right] } \\
L^{\prime \prime }R^{\prime \prime }L^{\prime \prime }U_{\left[ 1\right]
}^{\prime \prime } & \overset{L^{\prime \prime }\psi ^{\prime \prime }}{%
\underset{\epsilon ^{\prime \prime }L^{\prime \prime }U_{\left[ 1\right]
}^{\prime \prime }}{\rightrightarrows }} & L^{\prime \prime }U_{\left[ 1%
\right] }^{\prime \prime } & \overset{\pi _{\left[ 1\right] }^{\prime \prime
}}{\rightarrow } & L_{\left[ 1\right] }^{\prime \prime }%
\end{array}%
\end{equation*}%
Note that $U_{\left[ 1\right] }\Gamma _{\left[ 1\right] }=\Theta U_{\left[ 1%
\right] }^{\prime \prime }$ so that $LU_{\left[ 1\right] }\Gamma _{\left[ 1%
\right] }=L\Theta U_{\left[ 1\right] }^{\prime \prime }.$ Moreover%
\begin{equation*}
\psi \Gamma _{\left[ 1\right] }\left( B^{\prime \prime },\beta ^{\prime
\prime }\right) =\psi \left( \Theta B^{\prime \prime },\Theta \beta ^{\prime
\prime }\circ R\theta B^{\prime \prime }\right) =\Theta \beta ^{\prime
\prime }\circ R\theta B^{\prime \prime }
\end{equation*}%
and hence%
\begin{equation*}
\psi \Gamma _{\left[ 1\right] }=\Theta \psi ^{\prime \prime }\circ R\theta
U_{\left[ 1\right] }^{\prime \prime }.
\end{equation*}%
Using this equality we get%
\begin{eqnarray*}
\theta U_{\left[ 1\right] }^{\prime \prime }\circ L\psi \Gamma _{\left[ 1%
\right] } &=&\theta U_{\left[ 1\right] }^{\prime \prime }\circ L\Theta \psi
^{\prime \prime }\circ LR\theta U_{\left[ 1\right] }^{\prime \prime } \\
&=&L^{\prime \prime }\psi ^{\prime \prime }\circ \theta R^{\prime \prime
}L^{\prime \prime }U_{\left[ 1\right] }^{\prime \prime }\circ L\Theta
R^{\prime \prime }\theta U_{\left[ 1\right] }^{\prime \prime }=L^{\prime
\prime }\psi ^{\prime \prime }\circ \theta R^{\prime \prime }\theta U_{\left[
1\right] }^{\prime \prime }
\end{eqnarray*}%
and
\begin{eqnarray*}
\theta U_{\left[ 1\right] }^{\prime \prime }\circ \epsilon LU_{\left[ 1%
\right] }\Gamma _{\left[ 1\right] } &=&\theta U_{\left[ 1\right] }^{\prime
\prime }\circ \epsilon L\Theta U_{\left[ 1\right] }^{\prime \prime }\overset{%
\text{nat. }\epsilon }{=}\epsilon L^{\prime \prime }U_{\left[ 1\right]
}^{\prime \prime }\circ LR\theta U_{\left[ 1\right] }^{\prime \prime
}=\left( \epsilon ^{\prime \prime }\circ \theta R^{\prime \prime }\right)
L^{\prime \prime }U_{\left[ 1\right] }^{\prime \prime }\circ LR\theta U_{%
\left[ 1\right] }^{\prime \prime } \\
&=&\epsilon ^{\prime \prime }L^{\prime \prime }U_{\left[ 1\right] }^{\prime
\prime }\circ \theta R^{\prime \prime }L^{\prime \prime }U_{\left[ 1\right]
}^{\prime \prime }\circ LR\theta U_{\left[ 1\right] }^{\prime \prime
}=\epsilon ^{\prime \prime }L^{\prime \prime }U_{\left[ 1\right] }^{\prime
\prime }\circ \theta R^{\prime \prime }\theta U_{\left[ 1\right] }^{\prime
\prime }.
\end{eqnarray*}%
Thus the diagram above serially commutes and hence there exists $\gamma _{%
\left[ 1\right] }:L_{\left[ 1\right] }G_{\left[ 1\right] }^{\prime \prime
}\rightarrow L_{\left[ 1\right] }^{\prime \prime }$ such that $\gamma _{%
\left[ 1\right] }\circ \pi _{\left[ 1\right] }\Gamma _{\left[ 1\right] }=\pi
_{\left[ 1\right] }^{\prime \prime }\circ \theta U_{\left[ 1\right]
}^{\prime \prime }.$ Hence we have the adjoint triangle in the left square
below.
\begin{equation*}
\begin{array}{ccccc}
\mathcal{A} & \overset{\mathrm{Id}_{\mathcal{A}}}{\longrightarrow } &
\mathcal{A} & \overset{\mathrm{Id}_{\mathcal{A}}}{\longrightarrow } &
\mathcal{A} \\
L_{\left[ 1\right] }^{\prime \prime }\uparrow \downarrow R_{\left[ 1\right]
}^{\prime \prime } & \gamma _{\left[ 1\right] }:L_{\left[ 1\right] }\Gamma _{%
\left[ 1\right] }\rightarrow L_{\left[ 1\right] }^{\prime \prime } & L_{
\left[ 1\right] }\uparrow \downarrow R_{\left[ 1\right] } & \sigma _{\left[ 1%
\right] }:\mathrm{L}\left( \mathbb{T}\right) S^{G}\rightarrow L_{\left[ 1%
\right] } & \mathrm{L}\left( \mathbb{T}\right) \uparrow \downarrow \mathrm{R}%
\left( \mathbb{T}\right) \\
\mathcal{B}_{\left[ 1\right] }^{\prime \prime }=\left\langle R^{\prime
\prime }L^{\prime \prime }\mid \mathrm{Id}\right\rangle & \overset{\Gamma _{%
\left[ 1\right] }}{\longrightarrow } & \mathcal{B}_{\left[ 1\right]
}=\left\langle RL\mid \mathrm{Id}\right\rangle & \overset{S^{G}}{%
\longrightarrow } & \mathrm{I}\left( \mathbb{T}\right)%
\end{array}%
\end{equation*}%
If we compose it with the right one we expect to get the adjoint triangle
involving $\Theta _{\left[ 1\right] }.$ Indeed we have%
\begin{eqnarray*}
S^{G}\Gamma _{\left[ 1\right] }\left( B^{\prime \prime },\beta ^{\prime
\prime }\right) &=&S^{G}\left( \Theta B^{\prime \prime },\Theta \beta
^{\prime \prime }\circ R\theta B^{\prime \prime }\right) =\left( \Theta
B^{\prime \prime },G\Theta \beta ^{\prime \prime }\circ GR\theta B^{\prime
\prime }\right) \\
&=&\left( \Theta B^{\prime \prime },G^{\prime \prime }\beta ^{\prime \prime
}\circ R^{\prime }\theta B^{\prime \prime }\right) =\Theta _{\left[ 1\right]
}\left( B^{\prime \prime },\beta ^{\prime \prime }\right) .
\end{eqnarray*}%
On morphisms they clearly coincide so that $S^{G}\Gamma _{\left[ 1\right]
}=\Theta _{\left[ 1\right] }.$ Moreover%
\begin{equation*}
\gamma _{\left[ 1\right] }\ast \sigma _{\left[ 1\right] }=\gamma _{\left[ 1%
\right] }\circ \sigma _{\left[ 1\right] }\Gamma _{\left[ 1\right] }
\end{equation*}%
so that%
\begin{eqnarray*}
\left( \gamma _{\left[ 1\right] }\ast \sigma _{\left[ 1\right] }\right)
\circ \pi \mathcal{L}\left( \mathbb{T}\right) S^{G}\Gamma _{\left[ 1\right]
} &=&\gamma _{\left[ 1\right] }\circ \left( \sigma _{\left[ 1\right] }\circ
\pi \mathcal{L}\left( \mathbb{T}\right) S^{G}\right) \Gamma _{\left[ 1\right]
} \\
&=&\gamma _{\left[ 1\right] }\circ \pi _{\left[ 1\right] }\Gamma _{\left[ 1%
\right] }=\pi _{\left[ 1\right] }^{\prime \prime }\circ \theta U_{\left[ 1%
\right] }^{\prime \prime }.
\end{eqnarray*}%
Let us check that this coincide with
\begin{eqnarray*}
&&\theta _{\left[ 1\right] }\circ \pi \mathcal{L}\left( \mathbb{T}\right)
S^{G}\Gamma _{\left[ 1\right] } \\
&=&\sigma _{\left[ 1\right] }^{\prime \prime }\circ \mathcal{U}\left(
\epsilon ^{\prime }\right) \boldsymbol{\theta }S^{G^{\prime \prime }}\circ
\pi \mathcal{L}\left( \mathbb{T}\right) \Theta _{\left[ 1\right] } \\
&=&\sigma _{\left[ 1\right] }^{\prime \prime }\circ \mathcal{U}\left(
\epsilon ^{\prime }\right) \boldsymbol{\theta }S^{G^{\prime \prime }}\circ
\pi \mathcal{L}\left( \mathbb{T}\right) \mathcal{B}\left( \theta \right)
S^{G^{\prime \prime }} \\
&=&\sigma _{\left[ 1\right] }^{\prime \prime }\circ \pi ^{\prime }\mathcal{L}%
\left( \mathbb{T}^{\prime \prime }\right) S^{G^{\prime \prime }}\circ U_{%
\left[ 1\right] }\boldsymbol{\theta }S^{G^{\prime \prime }} \\
&=&\pi _{\left[ 1\right] }^{\prime \prime }\circ U_{\left[ 1\right] }%
\boldsymbol{\theta }S^{G^{\prime \prime }}=\pi _{\left[ 1\right] }^{\prime
\prime }\circ \theta P\left( \mathbb{T}^{\prime \prime }\right) S^{G^{\prime
\prime }}=\pi _{\left[ 1\right] }^{\prime \prime }\circ \theta U_{\left[ 1%
\right] }^{\prime \prime }
\end{eqnarray*}%
As a consequence $\theta _{\left[ 1\right] }=\gamma _{\left[ 1\right] }\ast
\sigma _{\left[ 1\right] }.$ Note that we have the following two pyramids%
\begin{equation*}
\begin{array}{ccccc}
\mathrm{I}\left( \mathbb{T}^{\prime \prime }\right) & \longleftarrow &
S^{G^{\prime \prime }} & - & \mathcal{B}_{\left[ 1\right] }^{\prime \prime }
\\
| & \nwarrow \mathrm{R}\left( \mathbb{T}^{\prime \prime }\right) & \sigma _{
\left[ 1\right] }^{\prime \prime } & R_{\left[ 1\right] }^{\prime \prime
}\nearrow & | \\
\mathrm{I}\left( \theta \right) & \mathcal{U}\left( \epsilon ^{\prime
}\right) \boldsymbol{\theta } & \mathcal{A} & \theta & \Gamma _{\left[ 1%
\right] } \\
\downarrow & \swarrow \mathrm{R}\left( \mathbb{T}\right) & \sigma _{\left[ 1%
\right] } & R_{\left[ 1\right] }\searrow & \downarrow \\
\mathrm{I}\left( \mathbb{T}\right) & \longleftarrow & S^{G} & - & \mathcal{B}%
_{\left[ 1\right] }%
\end{array}%
\qquad
\begin{array}{ccccc}
\mathcal{B}_{\left[ 1\right] }^{\prime \prime } & - & U_{\left[ 1\right]
}^{\prime \prime } & \longrightarrow & \mathcal{B}^{\prime \prime } \\
| & \nwarrow R_{\left[ 1\right] }^{\prime \prime } & \pi _{\left[ 1\right]
}^{\prime \prime } & R^{\prime \prime }\nearrow & | \\
\Gamma _{\left[ 1\right] } & \gamma _{\left[ 1\right] } & \mathcal{A} &
\theta & \Theta \\
\downarrow & \swarrow R_{\left[ 1\right] } & \pi _{\left[ 1\right] } &
R\searrow & \downarrow \\
\mathcal{B}_{\left[ 1\right] } & - & U_{\left[ 1\right] } & \longrightarrow
& \mathcal{B}%
\end{array}%
\end{equation*}
\end{invisible}
\end{proof}

\begin{notation}\label{not:Thetaprime}
By applying Proposition \ref{pro:3} in the particular case when $G=\mathrm{Id},$ $L^{\prime }=L,R^{\prime }=R$, we get the functor
\begin{equation*}
\Theta' _{\left[ 1\right] }:\mathcal{B}_{\left[ 1\right] }^{\prime \prime }\rightarrow \mathcal{B}_{%
\left[ 1\right] } ,\qquad \left( V^{\prime
\prime },\mu ^{\prime \prime }\right) \mapsto \left( \Theta V^{\prime \prime
},\Theta\mu ^{\prime \prime }\circ R\theta V^{\prime
\prime }\right) ,\qquad f\mapsto \Theta U_{\left[ 1\right] }^{\prime \prime
}f.
\end{equation*}%
This functor will be used in the following section.
\end{notation}

Recall from Definition \ref{def:MonDec} that given an adjunction $L\dashv R:\mathcal{A}\rightarrow \mathcal{B}$ we can consider the category $\mathcal{B}_{1}$ of $RL$-algebras over $\mathcal{B}$ and the corresponding forgetful functor $U_{0,1 }:\mathcal{B}_{1}\to \mathcal{B}$. The functor $R=R_0$ induces the comparison functor $R_1:\mathcal{A}\to \mathcal{B}_{1}$.

\begin{proposition}
\label{pro:4}Assume $\mathcal{A}$ has coequalizers. Consider an adjunction $L\dashv R:\mathcal{A}\rightarrow \mathcal{B}$. Then the inclusion functor $%
\Lambda _{1}:\mathcal{B}_{1}\rightarrow \mathcal{B}_{\left[ 1\right] }$
gives rise to an adjoint triangle
\begin{equation*}
\xymatrixcolsep{1.5cm}\xymatrixrowsep{0.7cm}\xymatrix{\mathcal{A}\ar[r]^\id%
\ar@{}[dr]|-{\lambda_1=\id_{L_1}}\ar@<.4ex>[d]^*-<0.2cm>{^{R_1}}&
\mathcal{A}\ar@<.4ex>[d]^*-<0.2cm>{^{R_{[1]}}}\\
\mathcal{B}_1\ar@<.4ex>@{.>}[u]^*-<0.2cm>{^{L_1}}\ar[r]^-{\Lambda_1}&%
\mathcal{B}_{[1]}=\left\langle
RL\mid\id\right\rangle\ar@<.4ex>@{.>}[u]^*-<0.2cm>{^{L_{[1]}}}}
\end{equation*}

\begin{invisible}
\begin{equation*}
\begin{array}{ccc}
\mathcal{A} & \overset{\mathrm{Id}_{\mathcal{A}}}{\longrightarrow } &
\mathcal{A} \\
L_{1}\uparrow \downarrow R_{1} & \lambda _{1}=\mathrm{Id}_{L_{1}} & L_{\left[
1\right] }\uparrow \downarrow R_{\left[ 1\right] } \\
\mathcal{B}_{1} & \overset{\Lambda _{1}}{\longrightarrow } & \mathcal{B}_{%
\left[ 1\right] }=\left\langle RL|\mathrm{Id}_{\mathcal{B}}\right\rangle%
\end{array}%
\end{equation*}
\end{invisible}

such that $U_{\left[ 1\right] }\circ \Lambda _{1}=U_{0,1}$ and $L_{\left[ 1%
\right] }\circ \Lambda _{1}=L_{1}$
\end{proposition}

\begin{proof}
Clearly $\mathcal{B}_{1}$ is a full subcategory of $\mathcal{B}_{\left[ 1%
\right] }$ through $\Lambda _{1}$ and one has $\Lambda _{1}\circ R_{1}=R_{%
\left[ 1\right] }.$ By construction of $L_1$ and $L_{[1]}$, for every $(B,b)$ in $\mathcal{B}_1$, we have that
$L_{\left[ 1\right]}\Lambda _{1}(B,b)=L_{\left[ 1\right]}(B,b)=\mathcal{U}(\epsilon)(LB,Lb)$ which is the coequalizer of the pair $(Lb,\epsilon LB)$ namely $L_1(B,b)$. Similarly $L_{\left[ 1\right]}\circ \Lambda _{1}$ and $L_{1}$ agree on morphisms and hence $L_{1}=L_{\left[ 1\right]
}\circ \Lambda _{1}$. Moreover the components of $\pi _{\left[ 1\right] }\Lambda _{1}=\pi
_{1}:LU_{1}\rightarrow L_{1}$ are the universal morphisms defining the coequalizer $L_1(B,b)$. Note that $L_{%
\left[ 1\right] }R_{\left[ 1\right] }=L_{\left[ 1\right] }\Lambda
_{1}R_{1}=L_{1}R_{1}$ and%
\begin{equation*}
\epsilon _{\left[ 1\right] }\circ \pi _{1}R_{1}=\epsilon _{\left[ 1\right]
}\circ \pi _{\left[ 1\right] }\Lambda _{1}R_{1}=\epsilon _{\left[ 1\right]
}\circ \pi _{\left[ 1\right] }R_{\left[ 1\right] }\overset{(\ref%
{form:epseta[1]})}{=}\epsilon =\epsilon _{1}\circ \pi _{1}R_{1}
\end{equation*}%
so that $\epsilon _{\left[ 1\right] }=\epsilon _{1}.$ Note that the last
equality, in the above displayed formula, is just the definition of the
counit $\epsilon _{1}$ of $\left( L_{1},R_{1}\right) .$ As a consequence $%
\lambda _{1}=\epsilon _{\left[ 1\right] }L_{1}\circ L_{\left[ 1\right]
}\Lambda _{1}\eta _{1}=\epsilon _{1}L_{1}\circ L_{1}\eta _{1}=\mathrm{Id}%
_{L_{1}}.$
\end{proof}

\begin{proposition}
\label{pro:5}Assume $\mathcal{A}$ has coequalizers. Consider the two adjoint triangles $\mathbb{T}$, $\mathbb{T}%
^{\prime }$ and their composition $\mathbb{T}^{\prime \prime }$ of Remark %
\ref{rem:comptr}. Then we can define a new adjoint triangle
\begin{equation*}
\xymatrixcolsep{1.5cm}\xymatrixrowsep{0.7cm}\xymatrix{\mathcal{A}\ar[r]^\id%
\ar@{}[dr]|-{\theta_1}\ar@<.4ex>[d]^*-<0.2cm>{^{R''_1}}&
\mathcal{A}\ar@<.4ex>[d]^*-<0.2cm>{^{\mathrm{R}(\mathbb{T})}}\\
\mathcal{B}''_1\ar@<.4ex>@{.>}[u]^*-<0.2cm>{^{L''_1}}\ar[r]^-{\Theta_1}&%
\mathrm{I}(\mathbb{T})\ar@<.4ex>@{.>}[u]^*-<0.2cm>{^{\mathrm{L}(%
\mathbb{T})}}}
\end{equation*}

\begin{invisible}
\begin{equation*}
\begin{array}{ccc}
\mathcal{A} & \overset{\mathrm{Id}_{\mathcal{A}}}{\longrightarrow } &
\mathcal{A} \\
L_{1}^{\prime \prime }\uparrow \downarrow R_{1}^{\prime \prime } & \theta
_{1} & \mathrm{L}\left( \mathbb{T}\right) \uparrow \downarrow \mathrm{R}%
\left( \mathbb{T}\right) \\
\mathcal{B}_{1}^{\prime \prime } & \overset{\Theta _{1}}{\longrightarrow } &
\mathrm{I}\left( \mathbb{T}\right)%
\end{array}%
\end{equation*}
\end{invisible}

where
\begin{equation*}
\Theta _{1}:\mathcal{B}_{1}^{\prime \prime }\rightarrow \mathrm{I}\left(
\mathbb{T}\right) ,\qquad \left( V^{\prime \prime },\mu ^{\prime \prime
}\right) \mapsto \left( \Theta V^{\prime \prime },G^{\prime \prime }\mu
^{\prime \prime }\circ R^{\prime }\theta V^{\prime \prime }\right) ,\qquad
f\mapsto \Theta U_{0,1}^{\prime \prime }f,
\end{equation*}%
is such that%
\begin{equation*}
P\left( \mathbb{T}\right) \Theta _{1}=\Theta U_{0,1}^{\prime \prime }\qquad
\text{and}\qquad G\left( \mathbb{T}\right) \Theta _{1}=G^{\prime \prime
}U_{0,1}^{\prime \prime }.
\end{equation*}

\begin{itemize}
\item[1)] If $\Theta $ is faithful, then so is $\Theta _{1}$.

\item[2)] Assume that $\theta $ is invertible and that any component of $%
\zeta R$ is an epimorphism.

Then $\theta _{1}$ is invertible. Moreover if $\Theta $ is full, then so is $%
\Theta _{1}$ and, if $\Theta $ is injective on objects, then so is $\Theta _{1}$.
\end{itemize}
\end{proposition}

\begin{proof}
Compose the adjoint triangles of Proposition \ref{pro:4} (applied to the
adjunction $\left( L^{\prime \prime },R^{\prime \prime }\right) $) and
Proposition \ref{pro:3}%
\begin{equation*}
\xymatrixcolsep{2cm}\xymatrixrowsep{0.7cm}\xymatrix{\mathcal{A}\ar[r]^\id%
\ar@{}[dr]|-{\lambda_1=\id_{L_1''}}\ar@<.4ex>[d]^*-<0.2cm>{^{R''_1}}&%
\mathcal{A}\ar@{}[dr]|-{\theta_{[1]}}\ar@<.4ex>[d]^*-<0.2cm>{^{R''_{[1]}}}%
\ar[r]^\id & \mathcal{A}\ar@<.4ex>[d]^*-<0.2cm>{^{\mathrm{R}(\mathbb{T})}}\\
\mathcal{B}''_1\ar@<.4ex>@{.>}[u]^*-<0.2cm>{^{L''_1}}\ar[r]^-{\Lambda_1}&%
\mathcal{B}''_{[1]}=\left\langle
R''L''\mid\id\right\rangle\ar@<.4ex>@{.>}[u]^*-<0.2cm>{^{L''_{[1]}}}%
\ar[r]^-{\Theta_{[1]}}&\mathrm{I}(\mathbb{T})\ar@<.4ex>@{.>}[u]^*-<0.2cm>{^{%
\mathrm{L}(\mathbb{T})}}}
\end{equation*}

\begin{invisible}

\begin{equation*}
\begin{array}{ccc}
\mathcal{A} & \overset{\mathrm{Id}_{\mathcal{A}}}{\longrightarrow } &
\mathcal{A} \\
L_{1}^{\prime \prime }\uparrow \downarrow R_{1}^{\prime \prime } & \lambda
_{1}=\mathrm{Id}_{L_{1}^{\prime \prime }} & L_{\left[ 1\right] }^{\prime
\prime }\uparrow \downarrow R_{\left[ 1\right] }^{\prime \prime } \\
\mathcal{B}_{1}^{\prime \prime } & \overset{\Lambda _{1}}{\longrightarrow }
& \mathcal{B}_{\left[ 1\right] }^{\prime \prime }=\left\langle R^{\prime
\prime }L^{\prime \prime }|\mathrm{Id}_{\mathcal{B}^{\prime \prime
}}\right\rangle%
\end{array}%
\ast
\begin{array}{ccc}
\mathcal{A} & \overset{\mathrm{Id}_{\mathcal{A}}}{\longrightarrow } &
\mathcal{A} \\
L_{\left[ 1\right] }^{\prime \prime }\uparrow \downarrow R_{\left[ 1\right]
}^{\prime \prime } & \theta _{\left[ 1\right] } & \mathrm{L}\left( \mathbb{T}%
\right) \uparrow \downarrow \mathrm{R}\left( \mathbb{T}\right) \\
\mathcal{B}_{\left[ 1\right] }^{\prime \prime } & \overset{\Theta _{\left[ 1%
\right] }}{\longrightarrow } & \mathrm{I}\left( \mathbb{T}\right)%
\end{array}%
\end{equation*}
\end{invisible}

to get the adjoint triangle in the present statement. Thus $\Theta
_{1}=\Theta _{\left[ 1\right] }\circ \Lambda _{1}$ and $\theta _{1}=\lambda
_{1}\ast \theta _{\left[ 1\right] }=\lambda _{1}\circ \theta _{\left[ 1%
\right] }\Lambda _{1}=\theta _{\left[ 1\right] }\Lambda _{1}.$

If $\Theta $ is faithful, then,  by Proposition \ref{pro:3}, so is $\Theta _{\left[
1\right] }$. Since the inclusion functor $\Lambda _{1}$ is faithful it is
then clear that $\Theta _{1}$ is faithful too as a composition of faithful
functors.

Assume that $\theta $ is invertible and that any component of $\zeta R$ is
an epimorphism. Still by Proposition \ref{pro:3}, we deduce that $\theta _{%
\left[ 1\right] }$ is invertible. Thus $\theta _{1}=\theta _{\left[ 1\right]
}\Lambda _{1}$ is invertible too.

It remains to prove that if $\Theta $ is full, then so is $\Theta _{1}$. Let $\xi
:\Theta _{1}\left( V^{\prime \prime },\mu _{V^{\prime \prime }}\right)
\rightarrow \Theta _{1}\left( W^{\prime \prime },\mu _{W^{\prime \prime
}}\right) $ be a morphism in $\mathrm{I}\left( \mathbb{T}\right)
=\left\langle R^{\prime }L|G\right\rangle .$ This means to have a morphism $%
h:=P\left( \mathbb{T}\right) \xi :\Theta V^{\prime \prime }\rightarrow
\Theta W^{\prime \prime }$ such that
\begin{equation*}
Gh\circ \left( G^{\prime \prime }\mu _{V^{\prime \prime }}\circ R^{\prime
}\theta V^{\prime \prime }\right) =\left( G^{\prime \prime }\mu _{W^{\prime
\prime }}\circ R^{\prime }\theta W^{\prime \prime }\right) \circ R^{\prime
}Lh.
\end{equation*}%
Since $G^{\prime \prime }=G\Theta ,$ $R=\Theta R^{\prime \prime }$ and $%
R^{\prime }=GR$ we can rewrite this equality as%
\begin{equation*}
G\left( h\circ \Theta \mu _{V^{\prime \prime }}\circ \Theta R^{\prime \prime
}\theta V^{\prime \prime }\right) =G\left( \Theta \mu _{W^{\prime \prime
}}\circ \Theta R^{\prime \prime }\theta W^{\prime \prime }\circ \Theta
R^{\prime \prime }Lh\right) .
\end{equation*}%
Since $\Theta $ is full, there is a morphism $g:V^{\prime \prime
}\rightarrow W^{\prime \prime }$ such that $h=\Theta g$ so that, using $%
G^{\prime \prime }=G\Theta ,$ we can further rewrite%
\begin{equation*}
G^{\prime \prime }\left( g\circ \mu _{V^{\prime \prime }}\circ R^{\prime
\prime }\theta V^{\prime \prime }\right) =G^{\prime \prime }\left( \mu
_{W^{\prime \prime }}\circ R^{\prime \prime }\theta W^{\prime \prime }\circ
R^{\prime \prime }L\Theta g\right) .
\end{equation*}%
By naturality of $\theta $ we have $\mu _{W^{\prime \prime }}\circ R^{\prime
\prime }\theta W^{\prime \prime }\circ R^{\prime \prime }L\Theta g=\mu
_{W^{\prime \prime }}\circ R^{\prime \prime }L^{\prime \prime }g\circ
R^{\prime \prime }\theta V^{\prime \prime }$ so that, since $\theta $ is
invertible, we obtain%
\begin{equation*}
G^{\prime \prime }\left( g\circ \mu _{V^{\prime \prime }}\right) =G^{\prime
\prime }\left( \mu _{W^{\prime \prime }}\circ R^{\prime \prime }L^{\prime
\prime }g\right)
\end{equation*}%
and hence%
\begin{eqnarray*}
&&L^{\prime \prime }\left( g\circ \mu _{V^{\prime \prime }}\right) \circ
\zeta ^{\prime \prime }R^{\prime \prime }L^{\prime \prime }V^{\prime \prime }%
\overset{\text{nat. }\zeta ^{\prime \prime }}{=}\zeta ^{\prime \prime
}W^{\prime \prime }\circ L^{\prime }G^{\prime \prime }\left( g\circ \mu
_{V^{\prime \prime }}\right) \\
&=&\zeta ^{\prime \prime }W^{\prime \prime }\circ L^{\prime }G^{\prime
\prime }\left( \mu _{W^{\prime \prime }}\circ R^{\prime \prime }L^{\prime
\prime }g\right) \overset{\text{nat. }\zeta ^{\prime \prime }}{=}L^{\prime
\prime }\left( \mu _{W^{\prime \prime }}\circ R^{\prime \prime }L^{\prime
\prime }g\right) \circ \zeta ^{\prime \prime }R^{\prime \prime }L^{\prime
\prime }V^{\prime \prime }.
\end{eqnarray*}%
Now $\zeta ^{\prime \prime }=\theta \ast \zeta =\theta \circ \zeta \Theta $
so that $\zeta ^{\prime \prime }R^{\prime \prime }=\theta R^{\prime \prime
}\circ \zeta \Theta R^{\prime \prime }=\theta R^{\prime \prime }\circ \zeta
R $ which is an epimorphism on each component. Thus we arrive at%
\begin{equation*}
L^{\prime \prime }\left( g\circ \mu _{V^{\prime \prime }}\right) =L^{\prime
\prime }\left( \mu _{W^{\prime \prime }}\circ R^{\prime \prime }L^{\prime
\prime }g\right) .
\end{equation*}

Using this equality we compute
\begin{eqnarray*}
g\circ \mu _{V^{\prime \prime }} &=&\mu _{W^{\prime \prime }}\circ \eta
^{\prime \prime }W^{\prime \prime }\circ g\circ \mu _{V^{\prime \prime
}}=\mu _{W^{\prime \prime }}\circ R^{\prime \prime }L^{\prime \prime }\left(
g\circ \mu _{V^{\prime \prime }}\right) \circ \eta ^{\prime \prime
}R^{\prime \prime }L^{\prime \prime }V^{\prime \prime } \\
&=&\mu _{W^{\prime \prime }}\circ R^{\prime \prime }L^{\prime \prime }\left(
\mu _{W^{\prime \prime }}\circ R^{\prime \prime }L^{\prime \prime }g\right)
\circ \eta ^{\prime \prime }R^{\prime \prime }L^{\prime \prime }V^{\prime
\prime } \\
&=&\mu _{W^{\prime \prime }}\circ \eta ^{\prime \prime }W^{\prime \prime
}\circ \mu _{W^{\prime \prime }}\circ R^{\prime \prime }L^{\prime \prime
}g=\mu _{W^{\prime \prime }}\circ R^{\prime \prime }L^{\prime \prime }g
\end{eqnarray*}%
This means there is a morphism $g_{1}:\left( V^{\prime \prime },\mu
_{V^{\prime \prime }}\right) \rightarrow \left( W^{\prime \prime },\mu
_{W^{\prime \prime }}\right) $ such that $U_{0,1}^{\prime \prime }g_{1}=g.$
We have%
\begin{equation*}
P\left( \mathbb{T}\right) \Theta _{1}g_{1}=\Theta U_{0,1}^{\prime \prime
}g_{1}=\Theta g=h=P\left( \mathbb{T}\right) \xi .
\end{equation*}%
Since $P\left( \mathbb{T}\right) $ is faithful, we deduce that $\Theta
_{1}g_{1}=\xi .$ Thus $\Theta _{1}$ is full.

Assume that $\Theta $ is injective on objects and $\Theta _{1}\left(
V^{\prime \prime },\mu _{V^{\prime \prime }}\right) =\Theta _{1}\left(
W^{\prime \prime },\mu _{W^{\prime \prime }}\right) .$ Then we can apply the
above argument for $\xi :=\mathrm{Id.}$ In this case $h:=P\left( \mathbb{T}%
\right) \xi =\mathrm{Id}:\Theta V^{\prime \prime }\rightarrow \Theta
W^{\prime \prime }.$ The fact that $\Theta $ is injective on objects tells
that $V^{\prime \prime }=W^{\prime \prime }$ so that we write $h=\Theta g$
for $g=\mathrm{Id}\ ($and the above assumption that $\Theta $ is full can be
dropped out). As above we arrive at $g\circ \mu _{V^{\prime \prime }}=\mu
_{W^{\prime \prime }}\circ R^{\prime \prime }L^{\prime \prime }g$ i.e. $\mu
_{V^{\prime \prime }}=\mu _{W^{\prime \prime }}.$ We have so proved that $%
\left( V^{\prime \prime },\mu _{V^{\prime \prime }}\right) =\left( W^{\prime
\prime },\mu _{W^{\prime \prime }}\right) $ and hence $\Theta _{1}\ $is
injective on objects.
\end{proof}

\begin{remark}
\label{rem:adj}Consider an adjunction $\left( L,R,\eta ,\epsilon \right) $ and assume $\mathcal{A}$ has coequalizers.

Apply Proposition \ref{pro:4} to  obtain the adjoint triangle $\boldsymbol{\Lambda }%
_{[1] }$
\begin{equation*}
\boldsymbol{\Lambda }_{[1] }:=\ \vcenter{\xymatrixcolsep{1.7cm}%
\xymatrixrowsep{0.7cm}\xymatrix{\mathcal{A}\ar[r]^\id\ar@{}[dr]|-{\lambda_1=\id_{L_1}}%
\ar@<.4ex>[d]^*-<0.2cm>{^{R_1}}&
\mathcal{A}\ar@<.4ex>[d]^*-<0.2cm>{^{R_{[1]}}}\\
\mathcal{B}_1\ar@<.4ex>@{.>}[u]^*-<0.2cm>{^{L_1}}\ar[r]^{\Lambda_1}&%
\mathcal{B}_{[1]}\ar@<.4ex>@{.>}[u]^*-<0.2cm>{^{L_{[1]}}}} } \qquad
\boldsymbol{\Lambda }_{[2] }:=\ \vcenter{\xymatrixcolsep{1.7cm}%
\xymatrixrowsep{0.7cm}\xymatrix{\mathcal{A}\ar[r]^\id\ar@{}[dr]|-{%
\lambda_2:=(\lambda_1)_1}\ar@<.4ex>[d]^*-<0.2cm>{^{R_2}}&
\mathcal{A}\ar@<.4ex>[d]^*-<0.2cm>{^{R_{[2]}}}\\
\mathcal{B}_2\ar@<.4ex>@{.>}[u]^*-<0.2cm>{^{L_2}}\ar[r]^{\Lambda_2:=(%
\Lambda_1)_1}&\mathcal{B}_{[2]}\ar@<.4ex>@{.>}[u]^*-<0.2cm>{^{L_{[2]}}}} }
\end{equation*}

\begin{invisible}
\begin{equation*}
\boldsymbol{\Lambda }_{\left[ 1\right] }:=%
\begin{array}{ccc}
\mathcal{A} & \overset{\mathrm{Id}_{\mathcal{A}}}{\longrightarrow } &
\mathcal{A} \\
L_{1}\uparrow \downarrow R_{1} & \lambda _{1} & L_{\left[ 1\right] }\uparrow
\downarrow R_{\left[ 1\right] } \\
\mathcal{B}_{1} & \overset{\Lambda _{1}}{\longrightarrow } & \mathcal{B}_{%
\left[ 1\right] }%
\end{array}%
\end{equation*}
\end{invisible}

where $U_{\left[ 0,1\right] }\Lambda
_{1}=U_{0,1}$ and $\Lambda _{1}$ is fully faithful and injective on objects.
Recall that any component of $\pi _{\left[ 1\right] }R_{\left[ 1\right] }$
is an epimorphism.

Hence all the conditions in Proposition \ref{pro:5} are verified for $%
\mathbb{T}^{\prime }=\boldsymbol{\Lambda }_{\left[ 1\right] }$ and $\mathbb{T%
}=\mathbb{T}_{\left[ 1\right] }$ and we obtain the adjoint triangle $%
\boldsymbol{\Lambda }_{[2] }$
\begin{invisible}
\begin{equation*}
\boldsymbol{\Lambda }_{\left[ 2\right] }:=%
\begin{array}{ccc}
\mathcal{A} & \overset{\mathrm{Id}_{\mathcal{A}}}{\longrightarrow } &
\mathcal{A} \\
L_{2}\uparrow \downarrow R_{2} & \lambda _{2}:=\left( \lambda _{1}\right)
_{1} & L_{\left[ 2\right] }\uparrow \downarrow R_{\left[ 2\right] } \\
\mathcal{B}_{2} & \overset{\Lambda _{2}:=\left( \Lambda _{1}\right) _{1}}{%
\longrightarrow } & \mathcal{B}_{\left[ 2\right] }%
\end{array}%
\end{equation*}
\end{invisible}
where $\lambda _{2}:L_{[2]}\Lambda_2\to L_2$ is invertible. Moreover $U_{\left[ 1,2\right] }\Lambda
_{2}=\Lambda _{1}U_{1,2}$ and $U_{\left[ 2\right] }\Lambda _{2}=U_{\left[ 0,1%
\right] }\Lambda _{1}U_{1,2}$ i.e. $U_{\left[ 0,2\right] }\Lambda
_{2}=U_{0,1}U_{1,2}=U_{0,2}.$ Furthermore $\Lambda _{2}$ is fully faithful
and injective on objects. Recall that any component of $\pi _{\left[ 2\right]
}R_{\left[ 2\right] }$ is an epimorphism.

Going on this way we construct iteratively $\left( \Lambda _{n}\right)
_{n\in \mathbb{N}}$ such that $\Lambda _{0}:=\mathrm{Id}$ and $U_{\left[
n-1,n\right] }\Lambda _{n}=\Lambda _{n-1}U_{n-1,n}$, for every $n\geq 1.$
Moreover $\lambda _{n}:L_{[n]}\Lambda_n\to L_n$ is invertible, $U_{\left[ 0,n\right] }\Lambda
_{n}=U_{0,n}$ for every $n\in \mathbb{N}$ and $\Lambda _{n}$ is fully
faithful and injective on objects.
\end{remark}

\begin{remark}
Note that, by construction $\Lambda _{n}$ is defined as follows
\begin{equation*}
\Lambda _{n}:\mathcal{B}_{n}\rightarrow \mathcal{B}_{\left[ n\right] },\
\left( V_{n-1},\mu _{n-1}\right) \mapsto \left( \Lambda _{n-1}V_{n-1},U_{%
\left[ n-1\right] }\Lambda _{n-1}\mu _{n-1}\circ R\lambda
_{n-1}V_{n-1}\right) ,\  f\mapsto \Lambda _{n-1}U_{n-1,n}f
\end{equation*}%
i.e.
\begin{equation*}
\Lambda _{n}:\mathcal{B}_{n}\rightarrow \mathcal{B}_{\left[ n\right]
},\qquad \left( V_{n-1},\mu _{n-1}\right) \mapsto \left( \Lambda
_{n-1}V_{n-1},U_{n-1}\mu _{n-1}\circ R\lambda _{n-1}V_{n-1}\right) ,\qquad
f\mapsto \Lambda _{n-1}U_{n-1,n}f.
\end{equation*}

\begin{invisible}
At the beginning we thought that $\lambda _{n}=\mathrm{Id}$ and
\begin{eqnarray*}
\Lambda _{n} &:&\left( V,\mu _{0}:RLV\rightarrow V,\mu
_{1}:R_{1}L_{1}V_{1}\rightarrow V_{1},\ldots ,\mu
_{n-1}:R_{n-1}L_{n-1}V_{n-1}\rightarrow V_{n-1}\right) \\
&\mapsto &\left( V,\mu _{0}:RLV\rightarrow V,U_{1}\mu _{1}:RL_{\left[ 1%
\right] }V_{\left[ 1\right] }\rightarrow V,\ldots ,U_{n-1}\mu _{n-1}:RL_{%
\left[ n-1\right] }V_{\left[ n-1\right] }\rightarrow V\right) .
\end{eqnarray*}%
because we know that $\lambda _{1}=\mathrm{Id.}$ Let us consider $\lambda
_{2}.$ Take
\begin{equation*}
B_{\left[ 2\right] }:=\Lambda _{2}\left( V_{1},\mu _{1}\right) =\left(
\Lambda _{1}V_{1},U_{1}\mu _{1}\circ R\lambda _{1}V_{1}\right) =\left(
\Lambda _{1}V_{1},U_{1}\mu _{1}\right) =:\left( B_{\left[ 1\right] },b_{%
\left[ 1\right] }\right) .
\end{equation*}
Then $L_{\left[ 2\right] }B_{\left[ 2\right] }$ is defined by the
coequalizer
\begin{equation*}
LRL_{\left[ 1\right] }B_{\left[ 1\right] }\overset{\pi _{\left[ 1\right] }B_{%
\left[ 1\right] }\circ Lb_{\left[ 1\right] }}{\underset{\epsilon L_{\left[ 1%
\right] }B_{\left[ 1\right] }}{\rightrightarrows }}L_{\left[ 1\right] }B_{%
\left[ 1\right] }\overset{\pi _{\left[ 1,2\right] }B_{\left[ 2\right] }}{%
\longrightarrow }L_{\left[ 2\right] }B_{\left[ 2\right] }
\end{equation*}%
i.e.%
\begin{equation*}
LRL_{\left[ 1\right] }\Lambda _{1}V_{1}\overset{\pi _{\left[ 1\right]
}\Lambda _{1}V_{1}\circ LU_{1}\mu _{1}}{\underset{\epsilon L_{\left[ 1\right]
}\Lambda _{1}V_{1}}{\rightrightarrows }}L_{\left[ 1\right] }\Lambda _{1}V_{1}%
\overset{\pi _{\left[ 1,2\right] }\Lambda _{2}\left( V_{1},\mu _{1}\right) }{%
\longrightarrow }L_{\left[ 2\right] }\Lambda _{2}\left( V_{1},\mu
_{1}\right) .
\end{equation*}%
By the proof of Proposition \ref{pro:4}, we know that $L_{\left[ 1\right]
}\Lambda _{1}=L_{1}$ and $\pi _{\left[ 1\right] }\Lambda _{1}=\pi _{1}$ so
that we obtain%
\begin{equation*}
LRL_{1}V_{1}\overset{\pi _{1}V_{1}\circ LU_{1}\mu _{1}}{\underset{\epsilon
L_{1}V_{1}}{\rightrightarrows }}L_{1}V_{1}\overset{\pi _{\left[ 1,2\right]
}\Lambda _{2}\left( V_{1},\mu _{1}\right) }{\longrightarrow }L_{\left[ 2%
\right] }\Lambda _{2}\left( V_{1},\mu _{1}\right) .
\end{equation*}%
Note that we necessarily have%
\begin{eqnarray*}
&&%
\begin{array}{ccc}
\mathcal{A} & \overset{\mathrm{Id}_{\mathcal{A}}}{\longrightarrow } &
\mathcal{A} \\
L_{2}\uparrow \downarrow R_{2} & \lambda _{2} & L_{\left[ 2\right] }\uparrow
\downarrow R_{\left[ 2\right] } \\
\mathcal{B}_{2} & \overset{\Lambda _{2}}{\longrightarrow } & \mathcal{B}_{%
\left[ 2\right] }%
\end{array}%
\ast
\begin{array}{ccc}
\mathcal{A} & \overset{\mathrm{Id}_{\mathcal{A}}}{\longrightarrow } &
\mathcal{A} \\
L_{\left[ 2\right] }\uparrow \downarrow R_{\left[ 2\right] } & \pi _{\left[
1,2\right] } & L_{\left[ 1\right] }\uparrow \downarrow R_{\left[ 1\right] }
\\
\mathcal{B}_{\left[ 2\right] } & \overset{U_{\left[ 1,2\right] }}{%
\longrightarrow } & \mathcal{B}_{\left[ 1\right] }%
\end{array}
\\
&=&%
\begin{array}{ccc}
\mathcal{A} & \overset{\mathrm{Id}_{\mathcal{A}}}{\longrightarrow } &
\mathcal{A} \\
L_{2}\uparrow \downarrow R_{2} & \pi _{1,2} & L_{1}\uparrow \downarrow R_{1}
\\
\mathcal{B}_{2} & \overset{U_{1,2}}{\longrightarrow } & \mathcal{B}_{1}%
\end{array}%
\ast
\begin{array}{ccc}
\mathcal{A} & \overset{\mathrm{Id}_{\mathcal{A}}}{\longrightarrow } &
\mathcal{A} \\
L_{1}\uparrow \downarrow R_{1} & \lambda _{1} & L_{\left[ 1\right] }\uparrow
\downarrow R_{\left[ 1\right] } \\
\mathcal{B}_{1} & \overset{\Lambda _{1}}{\longrightarrow } & \mathcal{B}_{%
\left[ 1\right] }%
\end{array}%
\end{eqnarray*}%
so that $\lambda _{2}\ast \pi _{\left[ 1,2\right] }=\pi _{1,2}\ast \lambda
_{1}$ i.e. $\lambda _{2}\circ \pi _{\left[ 1,2\right] }\Lambda _{2}=\pi
_{1,2}\circ \lambda _{1}U_{1,2}.$ Since $\lambda _{1}=\mathrm{I,}$ we arrive
at $\lambda _{2}\circ \pi _{\left[ 1,2\right] }\Lambda _{2}=\pi _{1,2}.$
Hence the following diagram serially commutes
\begin{equation*}
\begin{array}{ccccc}
LRL_{1}V_{1} & \overset{\pi _{1}V_{1}\circ LU_{1}\mu _{1}}{\underset{%
\epsilon L_{1}V_{1}}{\rightrightarrows }} & L_{1}V_{1} & \overset{\pi _{%
\left[ 1,2\right] }\Lambda _{2}\left( V_{1},\mu _{1}\right) }{%
\longrightarrow } & L_{\left[ 2\right] }\Lambda _{2}\left( V_{1},\mu
_{1}\right) \\
\pi _{1}R_{1}L_{1}V_{1}\downarrow &  & \downarrow \mathrm{Id} &  &
\downarrow \lambda _{2}\left( V_{1},\mu _{1}\right) \\
L_{1}R_{1}L_{1}V_{1} & \overset{L_{1}\mu _{1}}{\underset{\epsilon
_{1}L_{1}V_{1}}{\rightrightarrows }} & L_{1}V_{1} & \overset{\pi
_{1,2}\left( V_{1},\mu _{1}\right) }{\longrightarrow } & L_{2}\left(
V_{1},\mu _{1}\right)%
\end{array}%
\end{equation*}%
Since $\pi _{1}R_{1}L_{1}V_{1}$ is an epimorphism, it is once more clear
that $\lambda _{2}\left( V_{1},\mu _{1}\right) $ must be invertible. But it
is now an identification and no longer an identity as in case of $\lambda
_{1}.$
\end{invisible}
\end{remark}

Since $\Lambda _{n}:\mathcal{B}_{n}\rightarrow \mathcal{B}_{\left[ n\right]
} $ is fully faithful, we get that $\mathcal{B}_{n}$ is equivalent to the
essential image of $\Lambda _{n}.$ Later on we will look for handy criteria
for an object in $\mathcal{B}_{\left[ n\right] }$ to belong to the image of $%
\Lambda _{n}.$

\section{Relative Grothendieck fibrations\label{sec:4}}

In order to deduce properties of the functors $\Lambda_n$, a relative
version of the notion of Grothendieck fibration is needed. We collect here
its definition and properties.

\begin{definition}
Let $F:\mathcal{A}\rightarrow \mathcal{B}$ be a functor and let $\mathsf{M}$
be a class of morphisms in $\mathcal{B}$.

We say that a morphism $f\in \mathcal{A}$ is \textbf{cartesian} (with
respect to $F$) over a morphism $f^{\prime }\in \mathcal{B}$ whenever $%
Ff=f^{\prime }$ and given $g\in \mathcal{A}$ and $h\in \mathcal{B}$ such
that $Ff\circ h=Fg,$ then there exists a unique morphism $k\in \mathcal{A}$
such that $Fk=h$ and $f\circ k=g,$ \cite[Definition 8.1.2]{Borceux2}.
\begin{equation*}
\xymatrixcolsep{1.5cm}\xymatrixrowsep{0.5cm}\xymatrix{&FZ\ar[dl]_h\ar[d]^{Fg}\\ FX\ar[r]|{Ff}&FY }\qquad\qquad
\xymatrixcolsep{1.5cm}\xymatrixrowsep{0.5cm}\xymatrix{&Z\ar@{.>}[dl]_k\ar[d]^{g}\\ X\ar[r]|{f}&Y }
\end{equation*}
We say that $F$ is an $\mathsf{M}$-\textbf{fibration} if every morphism $%
f^{\prime }:B\rightarrow FA$ in $\mathsf{M}$ there is $f:A^{\prime
}\rightarrow A\ $which is cartesian over $f^{\prime }$. When $\mathsf{M}$ is
the class of all morphisms in $\mathcal{B}$ we recover the notion of
fibration, see \cite[Definition 8.1.3]{Borceux2}.
\end{definition}

\begin{remark}\label{rem:cartpullb}
A morphism $f:X\to Y$ is cartesian over $Ff$ if and only if following
diagram is a pullback for every object $Z$, where the vertical maps are obtained by evaluating $F$.
\begin{equation*}
\xymatrixcolsep{1.5cm}\xymatrixrowsep{0.5cm}\xymatrix{\mathrm{Hom}\left(
Z,X\right)\ar[d]_{F_{Z,X}}\ar[r]^{\mathrm{Hom}\left(
Z,f\right)}&\mathrm{Hom}\left( Z,Y\right)\ar[d]^{F_{Z,Y}}\\
\mathrm{Hom}\left( FZ,FX\right)\ar[r]^{\mathrm{Hom}\left(
FZ,Ff\right)}&\mathrm{Hom}\left( FZ,FY\right) }
\end{equation*}
In fact the map $\mathrm{Hom}(Z,X)\to \mathrm{Hom}(Z,Y)\times_{\mathrm{Hom}(FZ,FY)}\mathrm{Hom}(FZ,FX):k\mapsto(f\circ k,Fk)$ into the pullback becomes bijective. This fact is well-known, see e.g. \cite[Definition 4.32.1]{stacks}.
\end{remark}

\begin{lemma}[Cf. {\cite[Proposition 3.4(ii)]{Vistoli}}]
\label{lem:transitive}Being cartesian is transitive.
\end{lemma}

\begin{proof}
Since the vertical composition of pullbacks is a pullback (\cite[Proposition 2.5.9]{Borceux1}), it follows from Remark \ref{rem:cartpullb}. 
\end{proof}

Recall that an \textbf{isofibration} (called \emph{transportable functor} in
\cite[Corollaire 4.4]{Gro-SGA1}) is a functor $F:\mathcal{A}\rightarrow
\mathcal{B}$ such that for any object $A\in \mathcal{A}$ and any isomorphism $%
f^{\prime }:B\rightarrow FA$, there exists an isomorphism $f:A^{\prime
}\rightarrow A$ such that $Ff=f^{\prime }$. A \textbf{discrete isofibration}
is an isofibration such that $f$ is unique (see \cite[page 13]{LP}).

It is known that every fibration is an isofibration. Let us prove a relative
version of this result.

\begin{proposition}
\label{pro:isofibr}Let $\mathsf{Iso}$ be the class of all isomorphisms in $%
\mathcal{B}$. Then $F:\mathcal{A}\rightarrow \mathcal{B}$ is an $\mathsf{Iso}
$-fibration if and only if it is an isofibration.
\end{proposition}

\begin{proof}
$\left( \Rightarrow \right) $ Let $f^{\prime }:B\rightarrow FA$ be any
isomorphism. Then $f^{\prime }\in \mathsf{Iso.}$ Since $F$ is an $\mathsf{Iso%
}$-fibration we have that there is $f:A^{\prime }\rightarrow A\ $which is
cartesian over $f^{\prime }$. In particular $Ff=f^{\prime }$ and $FA^{\prime
}=B.$ From $Ff\circ \left( f^{\prime }\right) ^{-1}=F\mathrm{Id}_{A}$ we
deduce that there exists a unique morphism $k:A\rightarrow A^{\prime }$ such
that $Fk=\left( f^{\prime }\right) ^{-1}$ and $f\circ k=\mathrm{Id}_{A}.$
Similarly, from $Ff\circ \mathrm{Id}_{B}=Ff,$ we get a unique morphism $%
\lambda :A^{\prime }\rightarrow A^{\prime }$ such that $F\lambda =\mathrm{Id}%
_{B}$ and $f\circ \lambda =f.$ Since $F\left( k\circ f\right) =Fk\circ
Ff=\left( f^{\prime }\right) ^{-1}\circ f^{\prime }=\mathrm{Id}$ and $f\circ
\left( k\circ f\right) =f$, we get $\lambda =k\circ f.$ On the other hand
since $F\mathrm{Id}_{A^{\prime }}=\mathrm{Id}_{B}$ and $f\circ \mathrm{Id}%
_{A^{\prime }}=f$ we also have $\lambda =\mathrm{Id}_{A^{\prime }}.$ Hence $%
k\circ f=\mathrm{Id.}$ We have so proved that $f$ is an isomorphism. Thus $F$
is an isofibration.

$\left( \Leftarrow \right) $ Let $f^{\prime }:B\rightarrow FA$ be in $%
\mathsf{Iso}$. Then it is an isomorphism. Since $F$ is an isofibration there
is there exists an isomorphism $f:A^{\prime }\rightarrow A$ such that $%
Ff=f^{\prime }$. Let $g\in \mathcal{A}$ and $h\in \mathcal{B}$ such that $%
Ff\circ h=Fg.$ Then we can take $k=f^{-1}\circ g$ to get $Fk=F\left(
f^{-1}\right) \circ F\left( g\right) =F\left( f^{-1}\right) \circ Ff\circ
h=h $ and $f\circ k=g.$ On the other hand any morphism $k$ such that $Fk=h$
and $f\circ k=g,$ from the latter equality is $f^{-1}\circ g.$ We have so
proved that $f$ is cartesian over $f^{\prime }.$ Hence $F$ is an $\mathsf{M}$%
-fibration.
\end{proof}

\begin{corollary}
\label{coro:isofibr}If $\mathsf{M}\supseteq \mathsf{Iso}$ and $F:\mathcal{A}%
\rightarrow \mathcal{B}$ is an $\mathsf{M}$-fibration then $F$ is an
isofibration.
\end{corollary}

\begin{proof}
Clearly from $\mathsf{M}\supseteq \mathsf{Iso}$ we deduce that $F$ is $%
\mathsf{M}$-fibration implies $F$ is an $\mathsf{Iso}$-fibration. By
Proposition \ref{pro:isofibr}, $F$ is an isofibration.
\end{proof}

\begin{remark}
\label{rem:isofibr}If $F$ is an isofibration which is faithful and injective
on objects then $F$ is a discrete isofibration. In fact, if there is another
$t:A^{\prime \prime }\rightarrow A$ such that $Ft=f^{\prime },$ then $%
FA^{\prime \prime }=B=FA^{\prime }$ so that $A=A^{\prime }.$ Moreover $%
Ft=f^{\prime }=Ff$ so that $t=f.$ Hence $f$ is unique.
\end{remark}

\begin{definition}
\label{def:images}Given a functor $F:\mathcal{A}\rightarrow \mathcal{B}$, if we define the \textbf{image} of $F$, denoted by  $\mathrm{Im}\left( F\right) $, as the class of objects in $\mathcal{B}$ of the form $FA$ for some $A\in \mathcal{A}$ together with the class of morphisms in $\mathcal{B}$ of the form $Ff$ for some $f\in \mathcal{A}$, it is not necessarily true that $\mathrm{Im}\left( F\right) $ is a subcategory of $\mathcal{B}$ in general. However this holds in some particular cases e.g. when $F$ is either injective on objects, see e.g.  \cite[page 62]{Mitchell}, or full, see Lemma \ref{lem:isofibfull} below.

In general we can consider the following categories.
\end{definition}

\begin{itemize}

\item $\mathrm{\mathrm{Eim}}\left( F\right) ,$ the \textbf{essential image} of $F,$
i.e. the full subcategory of $\mathcal{B}$ whose objects are those isomorphic to $FA$ for
some $A\in \mathcal{A}$.

\item $\mathrm{Im}^{\prime }\left( F\right) ,$ i.e. the full subcategory of $\mathcal{B}$ whose objects are of the form $FA$ for some $A\in \mathcal{A}$.
\end{itemize}

Clearly $\mathrm{Im}\left( F\right) \subseteq \mathrm{Im}^{\prime
}\left( F\right) \subseteq \mathrm{\mathrm{Eim}}\left(
F\right) $ hold always.

\begin{lemma}
\label{lem:isofibfull}Let $F:\mathcal{A}\rightarrow \mathcal{B}$ be a
functor.

\begin{enumerate}
\item[1)] If $F$ is an isofibration then $\mathrm{\mathrm{Eim}}\left(
F\right) \subseteq \mathrm{Im}^{\prime }\left( F\right) .$

\item[2)] If $F$ is full then $\mathrm{Im}\left( F\right)$ is a category and $\mathrm{Im}^{\prime }\left( F\right) =%
\mathrm{Im}\left( F\right) .$
\end{enumerate}
\end{lemma}

\begin{proof}
$1).$ Given an object $B$ in $\mathrm{\mathrm{Eim}}\left( F\right) $ then $%
B\in \mathcal{B}$ is endowed with an isomorphism $f^{\prime }:B\rightarrow
FA $ for some $A\in \mathcal{A}$. Since $F$ is an isofibration we get an
isomorphism $f:A^{\prime }\rightarrow A$ such that $Ff=f^{\prime }.$ In
particular $FA^{\prime }=B$ whence $B\in \mathrm{Im}\left( F\right) .$ Since
$\mathrm{\mathrm{Eim}}\left( F\right) $ and $\mathrm{Im}^{\prime }\left(
F\right) $ ar both full subcategories of $\mathcal{B}$ we get $\mathrm{%
\mathrm{Eim}}\left( F\right) \subseteq \mathrm{Im}^{\prime }\left( F\right)
. $

$2).$ Let $g:FA^{\prime
}\rightarrow FA$ be a morphism in $\mathrm{Im}^{\prime }\left( F\right) .$
Since $F$ is full, there is $f:A^{\prime }\rightarrow A$ such that $g=Ff.$
Then $g$ is a morphism in $\mathrm{Im}\left( F\right) $. As a consequence, the composition in $\mathrm{Im}^{\prime }\left( F\right) $ of two morphisms in $\mathrm{Im}\left( F\right) $ lies in $\mathrm{Im}\left( F\right) $ too. Since
$\mathrm{Im}^{\prime }\left( F\right) $ and $\mathrm{Im}\left( F\right) $ have the same objects we get $\mathrm{Im}^{\prime }\left(
F\right) \subseteq \mathrm{Im}\left( F\right) $ and hence the equality.
\end{proof}

\begin{lemma}
\label{lem:fibrfulfaith}Let $F:\mathcal{A}\rightarrow \mathcal{B}$ be a fully faithful functor. Then $F$ is an $\mathsf{M}$-fibration if and only if
for every morphism $f^{\prime }:B\rightarrow FA$ in $\mathsf{M}$ there is $%
A^{\prime }\in \mathcal{A}\ $such that $FA^{\prime }=B.$
\end{lemma}

\begin{proof}
Assume that for every morphism $f^{\prime }:B\rightarrow FA$ in $\mathsf{M}$
there is $A^{\prime }\in \mathcal{A}\ $such that $FA^{\prime }=B.$

Since $F$ is fully faithful there is $f:A^{\prime }\rightarrow A$ such
that $Ff=f^{\prime }.$ In order to prove that $f:A^{\prime }\rightarrow A\ $%
is cartesian over $f^{\prime },$ let $g:C\rightarrow A$ in $\mathcal{A}$ and
$h:FC\rightarrow FA^{\prime }$ in $\mathcal{B}$ such that $Ff\circ h=Fg.$
Since $F$ is fully faithful there exists a unique morphism $k\in \mathcal{%
A}$ such that $Fk=h$ and from $Ff\circ h=Fg$ and faithfulness of $F$ we
conclude that $f\circ k=g$ as desired.
\end{proof}

\begin{remark}
\label{rem:P&U}Proposition \ref{pro:aureo} states that a morphism $g\in
\left\langle F|G\right\rangle $ is cartesian over $Pg$ whenever $GPg$ is a
monomorphism. In other words any morphism $g\in \mathsf{M}\left( GP\right) $
is cartesian over $Pg$ where
\begin{equation*}
\mathsf{M}\left( F\right) =\left\{ f\in \mathcal{A}\mid Ff\text{ is a
monomorphism}\right\} .
\end{equation*}
\end{remark}

\begin{proposition}
\label{pro:Unfibr}For every $n\in \mathbb{N}$, every morphism $g\in \mathsf{M%
}\left( U_{\left[ n\right] }\right) $ is cartesian over $U_{\left[ n\right]
}g.$
\end{proposition}

\begin{proof}
\ We proceed by induction on $n\in \mathbb{N}$. The first step is trivially
true since $U_{\left[ 0\right] }=\mathrm{Id}_{\mathcal{B}}$.

Let $n\geq 1$ and assume the statement true for $n-1.$ Let $g$ be a morphism
in $\mathsf{M}\left( U_{\left[ n\right] }\right) .$ Since $U_{\left[ n\right]
}=U_{\left[ n-1\right] }U_{\left[ n-1,n\right] }$ we get that $U_{\left[
n-1,n\right] }g\in \mathsf{M}\left( U_{\left[ n-1\right] }\right) .$ By
inductive hypothesis we have that $U_{\left[ n-1,n\right] }g$ is cartesian
over $U_{\left[ n-1\right] }U_{\left[ n-1,n\right] }g=U_{\left[ n\right] }g.$
By Lemma \ref{lem:transitive}, it remains to prove that $g$ is cartesian
over $U_{\left[ n-1,n\right] }g$ in order to conclude. To this aim observe
that $U_{\left[ n-1,n\right] }=P\left( \mathbb{T}_{\left[ n-1\right]
}\right) .$ Thus, by Remark \ref{rem:P&U} applied to $P=P\left( \mathbb{T}_{%
\left[ n-1\right] }\right) ,F=RL_{\left[ n-1\right] },G=U_{\left[ n-1\right]
}$, we get that $g$ is cartesian over $U_{\left[ n-1,n\right] }g$ as $g\in
\mathsf{M}\left( U_{\left[ n-1\right] }U_{\left[ n-1,n\right] }\right) .$
\end{proof}

The following result provides a characterization in term of suitable pullbacks for a morphism to be cartesian with respect to an adjoint functor. It will not be used in the sequel but we think it might be of some intrinsic interest.

\begin{proposition}
Let $\left( L,R\right) ,$ with $R:\mathcal{A}\rightarrow \mathcal{B}$, be an
adjunction with unit $\eta $ and counit $\epsilon .$

\begin{enumerate}
\item A morphism $f$ is cartesian with respect to $L$ over $Lf$ if and only
if the following diagram is a pullback.%
\begin{equation}  \label{diag:carteta}
\xymatrixcolsep{1.5cm}\xymatrixrowsep{0.5cm}\xymatrix{X\pulb \ar[d]_{\eta
X}\ar[r]^{f}&Y\ar[d]^{\eta Y}\\ RLX\ar[r]^{RLf}&RLY }
\end{equation}

\begin{invisible}
\begin{equation}
\begin{array}{ccc}
X & \overset{\eta X}{\longrightarrow } & RLX \\
f\downarrow & \lrcorner & \downarrow RLf \\
Y & \overset{\eta Y}{\longrightarrow } & RLY%
\end{array}%
\end{equation}
\end{invisible}

\item A morphism $f$ is cartesian with respect to $R$ over $Rf$ if and only
if $\epsilon Z\bot f$ (that is $\epsilon Z$ and $f$ are orthogonal) for every $Z\in
\mathcal{A}$ i.e. any commutative diagram as follows admits a unique
diagonal filler $k$ making both triangles commute.
\begin{equation}  \label{diag:ortogeps}
\xymatrixcolsep{1.5cm}\xymatrixrowsep{0.5cm}\xymatrix{LRZ
\ar[d]_{u}\ar[r]^{\epsilon Z}&Z\ar[d]^{v}\ar@{.>}[dl]_k\\ X\ar[r]_{f}&Y }
\end{equation}

\begin{invisible}
\begin{equation}
\begin{array}{ccc}
LRZ & \overset{\epsilon Z}{\longrightarrow } & Z \\
u\downarrow & k\swarrow & \downarrow v \\
X & \overset{f}{\longrightarrow } & Y%
\end{array}%
\end{equation}
\end{invisible}
\end{enumerate}
\end{proposition}

\begin{proof}
$\left( 1\right) .$ Assume that $f:X\rightarrow Y$ is cartesian over $Lf$
and let us prove that (\ref{diag:carteta}) is a pullback. Let $%
u:Z\rightarrow Y$ and $v:Z\rightarrow RLX$ be such that $\eta Y\circ
u=RLf\circ v.$ We have
\begin{equation*}
Lf\circ \epsilon LX\circ Lv=\epsilon LY\circ LRLf\circ Lv=\epsilon LY\circ
L\left( RLf\circ v\right) =\epsilon LY\circ L\left( \eta Y\circ u\right)
=\epsilon LY\circ L\eta Y\circ Lu=Lu
\end{equation*}%
so that the following diagram commutes.
\begin{equation*}
\xymatrixcolsep{1.5cm}\xymatrixrowsep{0.5cm}\xymatrix{&LZ\ar[d]^{Lu}%
\ar@{.>}[dl]_{\epsilon LX\circ Lv}\\ LX\ar[r]_{Lf}&LRLY }
\end{equation*}

\begin{invisible}
\begin{equation*}
\begin{array}{ccc}
&  & LZ \\
& \epsilon LX\circ Lv\swarrow & \downarrow Lu \\
LX & \overset{Lf}{\longrightarrow } & LRLY%
\end{array}%
\end{equation*}
\end{invisible}

Since $f$ is cartesian over $Lf,$ there is a unique $k:Z\rightarrow X$ such
that $Lk=\epsilon LX\circ Lv$ and $f\circ k=u.$ The condition $Lk=\epsilon
LX\circ Lv,$ via the adjunction isomorphism, is equivalent to $RLk\circ \eta
Z=R\left( \epsilon LX\circ Lv\right) \circ \eta Z$ i.e. $\eta X\circ
k=R\epsilon LX\circ RLv\circ \eta Z$ i.e. $\eta X\circ k=R\epsilon LX\circ
\eta RLX\circ v$ i.e. $\eta X\circ k=v.$ Thus there is a unique $%
k:Z\rightarrow X$ such that $\eta X\circ k=v$ and $f\circ k=u.$ In other
words (\ref{diag:carteta}) is a pullback.

Conversely assume that (\ref{diag:carteta}) is a pullback and let us prove
that $f:X\rightarrow Y$ is cartesian over $Lf$. Let $u:Z\rightarrow Y$ and $%
h:LZ\rightarrow LX$ be such that $Lf\circ h=Lu.$ Set $v:=Rh\circ \eta Z.$
Then
\begin{equation*}
RLf\circ v=RLf\circ Rh\circ \eta Z=RLu\circ \eta Z=\eta Y\circ u.
\end{equation*}%
By the universal property of the pullback there is a unique morphism $%
k:Z\rightarrow X$ such that $\eta X\circ k=v$ and $f\circ k=u.$ The
condition $\eta X\circ k=v,$ via the adjunction isomorphism, is equivalent
to $\epsilon LX\circ L\left( \eta X\circ k\right) =\epsilon LX\circ Lv$ i.e.
$\epsilon LX\circ L\eta X\circ Lk=\epsilon LX\circ L\left( Rh\circ \eta
Z\right) $ i.e. $Lk=\epsilon LX\circ LRh\circ L\eta Z=h\circ \epsilon
LZ\circ L\eta Z=h.$ Thus there is a unique morphism $k:Z\rightarrow X$ such
that $Lk=h$ and $f\circ k=u.$ In other words $f$ is cartesian over $Lf.$

$\left( 2\right) .$ Assume that $f:X\rightarrow Y$ is cartesian over $Rf$
and let us prove that $\epsilon Z\bot f$. Consider a commutative square as
in (\ref{diag:ortogeps}). Set $h:=Ru\circ \eta RZ.$ Then $Rf\circ h=Rf\circ
Ru\circ \eta RZ=Rv\circ R\epsilon Z\circ \eta RZ=Rv.$ Since $f$ is cartesian
over $Rf,$ there is a unique $k:Z\rightarrow X$ such that $Rk=h$ and $f\circ
k=v.$ By the adjunction, the condition $Rk=h$ is equivalent to $\epsilon
X\circ LRk=\epsilon X\circ Lh$ i.e. $k\circ \epsilon Z=\epsilon X\circ
L\left( Ru\circ \eta RZ\right) $ i.e. $k\circ \epsilon Z=u\circ \epsilon
LRZ\circ L\eta RZ=u.$ Thus there is a unique $k:Z\rightarrow X$ such that $%
k\circ \epsilon Z=u$ and $f\circ k=v.$ Hence $\epsilon Z$ and $f$ are
orthogonal.

Conversely suppose $\epsilon Z\bot f$ and let us prove  $f$ is cartesian
over $Rf.$ Let $v:Z\rightarrow Y$ and $h:RZ\rightarrow RX$ be such that $%
Rf\circ h=Rv.$ By the adjunction the later equality is equivalent to $%
\epsilon Y\circ L\left( Rf\circ h\right) =\epsilon Y\circ LRv$ i.e. $f\circ
\epsilon X\circ Lh=v\circ \epsilon Z.$ Thus, if we set $u:=\epsilon X\circ
Lh $, by the orthogonality there is a unique morphism $k:Z\rightarrow X$
such that $k\circ \epsilon Z=u$ and $f\circ k=v.$ By the adjunction the
condition $k\circ \epsilon Z=u$ is equivalent to $R\left( k\circ \epsilon
Z\right) \circ \eta RZ=Ru\circ \eta RZ$ i.e. $Rk\circ R\epsilon Z\circ \eta
RZ=R\left( \epsilon X\circ Lh\right) \circ \eta RZ$ i.e. $Rk=R\epsilon
X\circ RLh\circ \eta RZ=R\epsilon X\circ \eta RX\circ h=h.$ Thus there is a
unique $k:Z\rightarrow X$ such that $k\circ \epsilon Z=u$ and $Rk=h$
i.e. $f$ is cartesian over $Rf.$

A somewhat faster proof could be obtained by applying Remark \ref{rem:cartpullb}.
\begin{invisible}
The fact that $f$ is cartesian over $Gf$ means that the following
diagram is a pullback%
\begin{equation*}
\xymatrixcolsep{1.5cm}\xymatrixrowsep{0.5cm}\xymatrix{\mathrm{Hom}\left(
Z,X\right)\ar[d]_{G_{Z,X}}\ar[r]^{\mathrm{Hom}\left(
Z,f\right)}&\mathrm{Hom}\left( Z,Y\right)\ar[d]^{G_{Z,Y}}\\
\mathrm{Hom}\left( GZ,GX\right)\ar[r]^{\mathrm{Hom}\left(
GZ,Gf\right)}&\mathrm{Hom}\left( GZ,GY\right) }
\end{equation*}
\begin{equation*}
\begin{array}{ccc}
\mathrm{Hom}\left( Z,X\right) & \overset{\mathrm{Hom}\left( Z,f\right) }{%
\longrightarrow } & \mathrm{Hom}\left( Z,Y\right) \\
G_{Z,X}\downarrow & \lrcorner & G_{Z,Y} \\
\mathrm{Hom}\left( GZ,GX\right) & \overset{\mathrm{Hom}\left( GZ,Gf\right) }{%
\longrightarrow } & \mathrm{Hom}\left( GZ,GY\right)%
\end{array}%
\end{equation*}
\end{invisible}

If $F=L\ $or $F=R$ we can rewrite the corresponding pullback by means of the
adjunction obtaining respectively the diagrams%
\begin{equation*}
\xymatrixcolsep{1.5cm}\xymatrixrowsep{0.5cm}\xymatrix{\mathrm{Hom}\left(
Z,X\right)\ar[d]_{\mathrm{Hom}\left( Z,\eta
X\right)}\ar[r]^{\mathrm{Hom}\left( Z,f\right)}&\mathrm{Hom}\left(
Z,Y\right)\ar[d]^{\mathrm{Hom}\left( Z,\eta Y\right)}\\ \mathrm{Hom}\left(
Z,RLX\right)\ar[r]^{\mathrm{Hom}\left( Z,RLf\right)}&\mathrm{Hom}\left(
Z,RLY\right) } \qquad \xymatrixcolsep{1.5cm}\xymatrixrowsep{0.5cm}%
\xymatrix{\mathrm{Hom}\left( Z,X\right)\ar[d]_{{\mathrm{Hom}\left( \epsilon
Z,X\right)}}\ar[r]^{\mathrm{Hom}\left( Z,f\right)}&\mathrm{Hom}\left(
Z,Y\right)\ar[d]^{{\mathrm{Hom}\left( \epsilon Z,Y\right)}}\\
\mathrm{Hom}\left( LRZ,X\right)\ar[r]^{\mathrm{Hom}\left(
LRZ,f\right)}&\mathrm{Hom}\left( LRZ,Y\right) }
\end{equation*}

\begin{invisible}
\begin{equation*}
\begin{array}{ccc}
\mathrm{Hom}\left( Z,X\right) & \overset{\mathrm{Hom}\left( Z,f\right) }{%
\longrightarrow } & \mathrm{Hom}\left( Z,Y\right) \\
\mathrm{Hom}\left( Z,\eta X\right) \downarrow & \lrcorner & \downarrow
\mathrm{Hom}\left( Z,\eta Y\right) \\
\mathrm{Hom}\left( Z,RLX\right) & \overset{\mathrm{Hom}\left( Z,RLf\right) }{%
\longrightarrow } & \mathrm{Hom}\left( Z,RLY\right)%
\end{array}%
\qquad
\begin{array}{ccc}
\mathrm{Hom}\left( Z,X\right) & \overset{\mathrm{Hom}\left( Z,f\right) }{%
\longrightarrow } & \mathrm{Hom}\left( Z,Y\right) \\
\mathrm{Hom}\left( \epsilon Z,X\right) \downarrow & \lrcorner & \downarrow
\mathrm{Hom}\left( \epsilon Z,Y\right) \\
\mathrm{Hom}\left( LRZ,X\right) & \overset{\mathrm{Hom}\left( LRZ,f\right) }{%
\longrightarrow } & \mathrm{Hom}\left( LRZ,Y\right)%
\end{array}%
\end{equation*}
\end{invisible}

The fact that the left-hand side diagram is a pullback means that (\ref%
{diag:carteta}) is a pullback, while the fact that the right-hand side
diagram is a pullback means that $\epsilon Z\bot f.$
\end{proof}

Next lemma will be used in order to prove Theorem \ref{teo:servaureo} that is our main tool to get Theorem \ref{teo:LambdaFibr} where the embedding $\Lambda_n$ is shown to be an $\mathsf{M}\left( U_{\left[ n\right] }\right) $-fibration.

\begin{lemma}
\label{lem:fibr1}The functor $\Lambda _{1}:\mathcal{B}_{1}\rightarrow
\mathcal{B}_{\left[ 1\right] }$ of Remark \ref{rem:adj} is an $\mathsf{M}%
\left( U_{\left[ 1\right] }\right) $-fibration.
\end{lemma}

\begin{proof}
Let $f_{\left[ 1\right] }:B_{\left[ 1\right] }\rightarrow \Lambda _{1}C_{1}$
be a morphism in $\mathsf{M}\left( U_{\left[ 1\right] }\right) $ i.e. such
that $U_{\left[ 1\right] }f_{\left[ 1\right] }$ is a monomorphism. Since $%
\Lambda _{1}$ is fully faithful, in order to conclude, by Lemma \ref%
{lem:fibrfulfaith}, it suffices to prove that there is $B_{1}\in \mathcal{B}%
_{1}$ such that $\Lambda _{1}B_{1}=B_{\left[ 1\right] }$.

Write $B_{\left[ 1\right] }=\left( B,b:RLB\rightarrow B\right) ,$ $%
C_{1}=\left( C,c:RLC\rightarrow C\right) $ and note that $C_{\left[ 1\right]
}:=\Lambda _{1}C_{1}=\left( C,c:RLC\rightarrow C\right) $ this time regarded
as an object in $\mathcal{B}_{\left[ 1\right] }.$ Set $f:=U_{\left[ 1\right]
}f_{\left[ 1\right] }:B\rightarrow C$ and consider the following diagrams.
\begin{equation*}
\xymatrixcolsep{1.5cm} \xymatrix{RLRLB\ar[d]_{RLRLf}\ar@<.5ex>[r]^{RLb}%
\ar@<-.5ex>[r]_{R\epsilon LB}&RLB\ar[d]^{RLf}\ar[r]^{b}& B\ar[d]^{f}\\
RLRLC\ar@<.5ex>[r]^{RLc}\ar@<-.5ex>[r]_{R\epsilon LC}&RLC\ar[r]^{c}& C}
\qquad
\xymatrixcolsep{1.5cm}\xymatrix{B\ar[d]_{f}\ar[r]^{%
\eta B}&RLB\ar[d]^{RLf}\ar[r]^b&B\ar[d]^f\\ C\ar[r]^{\eta C}&RLC\ar[r]^c&C }
\end{equation*}
\begin{invisible}
\begin{equation*}
\begin{array}{ccccc}
RLRLB & \overset{RLb}{\underset{R\epsilon LB}{\rightrightarrows }} & RLB &
\overset{b}{\longrightarrow } & B \\
RLRLf\downarrow &  & \downarrow RLf &  & \downarrow f \\
RLRLC & \overset{RLc}{\underset{R\epsilon LC}{\rightrightarrows }} & RLC &
\overset{c}{\longrightarrow } & C%
\end{array}%
\end{equation*}
\end{invisible}
The left-hand side one serially commutes since $f$ induces the morphism $f_{\left[ 1\right] }$
and by naturality of $\epsilon .$ Since $C_{1}\in \mathcal{B}_{1},$ we also
have that $c\circ RLc=c\circ R\epsilon LC$. Since $f$ is a monomorphism, we
deduce that $b\circ RLb=b\circ R\epsilon LB.$
\begin{invisible}
\begin{equation*}
\begin{array}{ccccc}
B & \overset{\eta B}{\longrightarrow } & RLB & \overset{b}{\longrightarrow }
& B \\
f\downarrow &  & \downarrow RLf &  & \downarrow f \\
C & \overset{\eta C}{\longrightarrow } & RLC & \overset{c}{\longrightarrow }
& C%
\end{array}%
\end{equation*}
\end{invisible}
A similar argument as above, but applied on the right-hand side diagram, shows that $b\circ \eta B=\mathrm{Id}_{B}.$ This
means that $B_{1}:=\left( B,b:RLB\rightarrow B\right) \in \mathcal{B}_{1}.$
By definition of $\Lambda _{1},$ we have that $\Lambda _{1}B_{1}=B_{\left[ 1%
\right] }.$
\end{proof}

The following lemma will be central in order to prove Propositions \ref{pro:serv0} and Proposition \ref{pro:serv1}.

\begin{lemma}
\label{lem:mono1}1)\ Let $F:\mathcal{A}\rightarrow \mathcal{B}$ be a functor
and let $f\in \mathsf{M}\left( F\right) $ be cartesian over $Ff.$ Then $f$
is a monomorphism.

2)\ Let $F:\mathcal{A}\rightarrow \mathcal{B}$ be an $\mathsf{M}\left(
G\right) $-fibration. Then any morphism $f\in \mathsf{M}\left( F\right) \cap
\mathsf{M}\left( GF\right) $ factors as $f=g\circ k,$ where $g$ is a
monomorphism and $Fk=\mathrm{Id}.$
\end{lemma}

\begin{proof}
1)\ Let $f:A\rightarrow A^{\prime }$ in $\mathsf{M}\left( F\right) $ be
cartesian over $Ff.$ By definition of $\mathsf{M}\left( F\right) $ we have that $Ff$ is a monomorphism. Let $a,b:A^{\prime \prime }\rightarrow A$ be
such that $f\circ a=f\circ b.$ Then $Ff\circ Fa=Ff\circ Fb.$ Since $Ff$ is a
monomorphism, we get $Fa=Fb.$ Call $h$ this morphism and $g:=f\circ a.$ Then
$Ff\circ h=Fg.$ Since $f$ is cartesian over $Ff$, there exists a unique
morphism $k\in \mathcal{A}$ such that $Fk=h$ and $f\circ k=g.$ Hence $a=k=b.$

2)\ Let $f:A\rightarrow A^{\prime }$ be in $\mathsf{M}\left( F\right) \cap
\mathsf{M}\left( GF\right) $. Then $Ff$ and $GFf$ are both monomorphisms.

In particular $f^{\prime }:=Ff:FA\rightarrow FA^{\prime }$ is in $\mathsf{M}%
\left( G\right) $ so that, by definition of $\mathsf{M}\left( G\right) $%
-fibration, there is $g:E\rightarrow A^{\prime }$ which is cartesian over $%
f^{\prime }.$ Since $Fg\circ \mathrm{Id}_{FA}=Ff,$ there is a unique $%
k:A\rightarrow E$ such that $Fk=\mathrm{Id}_{FA}$ and $g\circ k=f.$

Moreover $g\in \mathsf{M}\left( F\right) $ is cartesian over $Fg$ so that,
by 1), we get that $g$ is a monomorphism.
\end{proof}

We are now going to prove Propositions \ref{pro:serv0}, Proposition \ref{pro:serv1} and Theorem \ref{teo:serv1}. These results will be used to obtain Theorem \ref{teo:servaureo}.

\begin{proposition}
\label{pro:serv0}In the setting of Proposition \ref{pro:2}, assume that any
morphism $g\in \mathsf{M}\left( G\right) $ is cartesian over $Gg.$ Then for
every morphism $f^{\prime }:\left( B,\beta \right) \rightarrow S^{G}C_{\left[
1\right] }$ in $\mathsf{M}\left( GP\left( \mathbb{T}\right) \right) $ there
is a morphism $f:B_{\left[ 1\right] }\rightarrow C_{\left[ 1\right] }$ in $%
\mathsf{M}\left( GU_{\left[ 1\right] }\right) $ which is cartesian with
respect to $S^{G}:\mathcal{B}_{\left[ 1\right] }=\left\langle RL\mid \mathrm{%
Id}\right\rangle \rightarrow \left\langle GRL\mid G\right\rangle =\mathrm{I}%
\left( \mathbb{T}\right) $ over $f^{\prime }.$ In particular $S^{G}$ is an $%
\mathsf{M}\left( G\circ P\left( \mathbb{T}\right) \right) $-fibration.
\end{proposition}

\begin{proof}
Let $f^{\prime }:\left( B,\beta \right) \rightarrow S^{G}C_{\left[ 1\right]
} $ be a morphism in $\mathsf{M}\left( GP\left( \mathbb{T}\right) \right) $,
in particular it is a morphism in $\mathrm{I}\left( \mathbb{T}\right) .$ Write $%
C_{\left[ 1\right] }=\left( C,c:RLC\rightarrow C\right) $ so that $S^{G}C_{%
\left[ 1\right] }=\left( C,Gc\right) .$ The fact that $f^{\prime }\ $is a
morphism $\mathrm{I}\left( \mathbb{T}\right) $ means that the following
left-hand side diagram commutes
\begin{equation}  \label{diag:serv0}
\xymatrixcolsep{1.5cm}\xymatrixrowsep{0.5cm}\xymatrix{GRLB\ar[d]_{\beta}%
\ar[r]^{GRLP\left( \mathbb{T}\right) f^{\prime }}&GRLC\ar[d]^{Gc}\\
GB\ar[r]^{GP\left( \mathbb{T}\right) f^{\prime }}&GC } \qquad %
\xymatrixcolsep{1.5cm}\xymatrixrowsep{0.5cm}\xymatrix{RLB\ar[d]_{b}%
\ar[r]^{RLP\left( \mathbb{T}\right) f^{\prime }}&RLC\ar[d]^{c}\\
B\ar[r]^{P\left( \mathbb{T}\right) f^{\prime }}&C }
\end{equation}

\begin{invisible}
\begin{equation*}
\begin{array}{ccc}
GRLB & \overset{GRLP\left( \mathbb{T}\right) f^{\prime }}{\longrightarrow }
& GRLC \\
\beta \downarrow &  & \downarrow Gc \\
GB & \overset{GP\left( \mathbb{T}\right) f^{\prime }}{\longrightarrow } & GC%
\end{array}%
\end{equation*}
\end{invisible}

Since $f^{\prime }\in \mathsf{M}\left( GP\left( \mathbb{T}\right) \right) ,$
we get $P\left( \mathbb{T}\right) f^{\prime }\in \mathsf{M}\left( G\right) .$
By hypothesis $P\left( \mathbb{T}\right) f^{\prime }$ is cartesian with
respect to $G$ over $GP\left( \mathbb{T}\right) f^{\prime }.$ As a
consequence the diagram on the left above implies there is a unique morphism $%
b:RLB\rightarrow B$ such that $Gb=\beta $ and the right-hand side diagram in %
\eqref{diag:serv0} commutes.

\begin{invisible}
\begin{equation*}
\begin{array}{ccc}
RLB & \overset{RLP\left( \mathbb{T}\right) f^{\prime }}{\longrightarrow } &
RLC \\
b\downarrow &  & \downarrow c \\
B & \overset{P\left( \mathbb{T}\right) f^{\prime }}{\longrightarrow } & C%
\end{array}%
\end{equation*}
\end{invisible}

Set $B_{\left[ 1\right] }=\left( B,b\right) \in \mathcal{B}_{\left[ 1\right]
}.$ Then the last diagram means that there is a unique morphism $f:B_{\left[ 1%
\right] }\rightarrow C_{\left[ 1\right] }$ such that $U_{\left[ 1\right]
}f=P\left( \mathbb{T}\right) f^{\prime }.$ Hence $GU_{\left[ 1\right]
}f=GP\left( \mathbb{T}\right) f^{\prime }$ is a monomorphism so that $f\in
\mathsf{M}\left( GU_{\left[ 1\right] }\right) .$

Let us check that $f$ is cartesian with respect to $S^{G}$ over $f^{\prime
}. $

Note that $S^{G}B_{\left[ 1\right] }=\left( B,Gb\right) =\left( B,\beta
\right) $ so that $S^{G}f$ has the same domain and codomain of $f^{\prime
}:\left( B,\beta \right) \rightarrow S^{G}C_{\left[ 1\right] }.$ Thus we get
the equality $S^{G}f=f^{\prime }$ by the following computation%
\begin{equation*}
P\left( \mathbb{T}\right) S^{G}f=U_{\left[ 1\right] }f=P\left( \mathbb{T}%
\right) f^{\prime }
\end{equation*}%
and the fact that $P\left( \mathbb{T}\right) $ is faithful. Consider $g_{%
\left[ 1\right] }$ and $h$ as in the following left-hand side diagram.
\begin{equation*}
\xymatrixcolsep{1.5cm}\xymatrixrowsep{0.5cm}\xymatrix{&S^{G}D_{[1]}%
\ar[dl]_{h}\ar[d]^{S^{G}g_{[1]}}\\ S^{G}B_{[1]}\ar[r]^-{S^{G}f}&S^{G}C_{[1]}
} \qquad \xymatrixcolsep{1.5cm}\xymatrixrowsep{0.5cm}%
\xymatrix{&U_{[1]}D_{[1]}\ar[dl]_{P\left( \mathbb{T}\right)
h}\ar[d]^{U_{[1]}g_{[1]}}\\ U_{[1]}B_{[1]}\ar[r]^-{U_{[1]}f}&U_{[1]}C_{[1]} }
\end{equation*}

\begin{invisible}
\begin{equation*}
\begin{array}{ccc}
&  & S^{G}D_{\left[ 1\right] } \\
& h\swarrow & \downarrow S^{G}g_{\left[ 1\right] } \\
S^{G}B_{\left[ 1\right] } & \overset{S^{G}f}{\longrightarrow } & S^{G}C_{
\left[ 1\right] }%
\end{array}%
\qquad
\begin{array}{ccc}
&  & U_{\left[ 1\right] }D_{\left[ 1\right] } \\
& P\left( \mathbb{T}\right) h\swarrow & \downarrow U_{\left[ 1\right] }g_{%
\left[ 1\right] } \\
U_{\left[ 1\right] }B_{\left[ 1\right] } & \overset{U_{\left[ 1\right] }f}{%
\longrightarrow } & U_{\left[ 1\right] }C_{\left[ 1\right] }%
\end{array}%
\end{equation*}
\end{invisible}

By applying $P\left( \mathbb{T}\right) $ we get the right-hand side diagram
above. We know that $U_{\left[ 1\right] }f=P\left( \mathbb{T}\right)
f^{\prime }\in \mathsf{M}\left( G\right) .$ Then, by hypothesis $U_{\left[ 1\right]
}f$ is then cartesian with respect to $G$ over $GU_{\left[ 1\right] }f.$
Thus, by Lemma \ref{lem:mono1}, we get that $U_{\left[ 1\right] }f$ is a
monomorphism. Thus $f\in \mathsf{M}\left( U_{\left[ 1\right] }\right) .$ By
Proposition \ref{pro:Unfibr}, we get that $f$ is cartesian over $U_{\left[ 1%
\right] }f.$ As a consequence, the right-hand side diagram above implies
there is a unique morphism $d_{\left[ 1\right] }:D_{\left[ 1\right]
}\rightarrow B_{\left[ 1\right] }$ such that $U_{\left[ 1\right] }d_{\left[ 1%
\right] }=P\left( \mathbb{T}\right) h$ and $f\circ d_{\left[ 1\right] }=g_{%
\left[ 1\right] }.$ Note that%
\begin{equation*}
P\left( \mathbb{T}\right) S^{G}d_{\left[ 1\right] }=U_{\left[ 1\right] }d_{%
\left[ 1\right] }=P\left( \mathbb{T}\right) h
\end{equation*}
and hence $S^{G}d_{\left[ 1\right] }=h.$ It remains to prove the uniqueness
of $d_{\left[ 1\right] }.$ If there is another $k_{\left[ 1\right] }:D_{%
\left[ 1\right] }\rightarrow B_{\left[ 1\right] }$ such that $S^{G}k_{\left[
1\right] }=h$ and $f\circ k_{\left[ 1\right] }=g_{\left[ 1\right] }.$ Then
\begin{equation*}
U_{\left[ 1\right] }k_{\left[ 1\right] }=P\left( \mathbb{T}\right) S^{G}k_{%
\left[ 1\right] }=P\left( \mathbb{T}\right) h=U_{\left[ 1\right] }d_{\left[ 1%
\right] }
\end{equation*}%
so that $k_{\left[ 1\right] }=d_{\left[ 1\right] }$ as $U_{\left[ 1\right] }$
is faithful.
\end{proof}

Consider now the functor  $\Theta _{\left[ 1\right] }^{\prime }:%
\mathcal{B}_{\left[ 1\right] }^{\prime \prime }\rightarrow \mathcal{B}_{%
\left[ 1\right] }$ introduced in Notation \ref{not:Thetaprime}.

\begin{proposition}
\label{pro:serv1} Consider two adjoint triangles $\mathbb{T}$ and $\mathbb{T}'$ as in Remark \ref{rem:comptr}. Assume that

\begin{itemize}
\item any morphism $g\in \mathsf{M}\left( G\right) $ is cartesian over $Gg;$

\item $\Theta :\mathcal{B}^{\prime \prime }\rightarrow \mathcal{B}$ is an $%
\mathsf{M}\left( G\right) $-fibration,

\item $\theta $ is invertible.
\end{itemize}

Then for every morphism $f:B_{\left[ 1%
\right] }\rightarrow \Theta _{\left[ 1\right] }^{\prime }C_{\left[ 1\right]
}^{\prime \prime }$ in $\mathsf{M}\left( GU_{\left[ 1\right] }\right) $
there is $f_{\left[ 1\right] }^{\prime \prime }:B_{\left[ 1\right] }^{\prime
\prime }\rightarrow C_{\left[ 1\right] }^{\prime \prime }$ in $\mathsf{M}%
\left( U_{\left[ 1\right] }^{\prime \prime }\right) $ which is cartesian
with respect to $\Theta _{\left[ 1\right] }^{\prime }$ over $f.$ In
particular $\Theta _{\left[ 1\right] }^{\prime }:\mathcal{B}_{\left[ 1\right]
}^{\prime \prime }\rightarrow \mathcal{B}_{\left[ 1\right] }$ is an $\mathsf{%
M}\left( GU_{\left[ 1\right] }\right) $-fibration.
\end{proposition}

\begin{proof}
Let $f:B_{\left[ 1\right] }\rightarrow \Theta _{\left[ 1\right] }^{\prime
}C_{\left[ 1\right] }^{\prime \prime }$ be a morphism in $\mathsf{M}\left( GU_{\left[ 1%
\right] }\right) .$ Write $B_{\left[ 1\right] }=\left( B,b:RLB\rightarrow
B\right) $ and $C_{\left[ 1\right] }^{\prime \prime }=\left( C^{\prime
\prime },c^{\prime \prime }:R^{\prime \prime }L^{\prime \prime }C^{\prime
\prime }\rightarrow C^{\prime \prime }\right) .$ By definition of $\Theta _{\left[ 1\right] }^\prime$, we have%
\begin{equation*}
\Theta _{\left[ 1\right] }^{\prime }C_{\left[ 1\right] }^{\prime \prime
}=\Theta _{\left[ 1\right] }^{\prime }\left( C^{\prime \prime },c^{\prime \prime
}\right)  =\left( \Theta C^{\prime \prime },\Theta c^{\prime \prime }\circ
R\theta C^{\prime \prime }\right) .
\end{equation*}%
The fact that $f\in \mathcal{B}_{\left[ 1\right] }$ means that the first
diagram in \eqref{diag:serv1} commutes.
\begin{equation}  \label{diag:serv1}
\xymatrixcolsep{1.5cm}\xymatrixrowsep{0.5cm}\xymatrix{RLB\ar[d]_{b}%
\ar[r]^{RLU_{[1]}f}&RL\Theta C''\ar[d]_{\Theta \left(c''\circ R''\theta
C''\right)}\\ B\ar[r]_{U_{[1]}f}&\Theta C'' } \quad \xymatrixcolsep{1.5cm}%
\xymatrixrowsep{0.5cm}\xymatrix{\Theta R''L\Theta
B''\ar[d]_{b}\ar[r]^{\Theta R''L\Theta f''}&\Theta R''L\Theta
C''\ar[d]_{\Theta \left(c''\circ R''\theta C''\right)}\\ \Theta
B''\ar[r]_{\Theta f''}&\Theta C'' } \quad \xymatrixcolsep{1.5cm}%
\xymatrixrowsep{0.5cm}\xymatrix{ R''L\Theta B''\ar[d]_{\tau ''}\ar[r]^{
R''L\Theta f''}& R''L\Theta C''\ar[d]_{ c''\circ R''\theta C''}\\
B''\ar[r]_{f''}& C'' }
\end{equation}

\begin{invisible}
\begin{equation*}
\begin{array}{ccc}
RLB & \overset{RLU_{\left[ 1\right] }f}{\longrightarrow } & RL\Theta
C^{\prime \prime } \\
b\downarrow &  & \downarrow \Theta \left( c^{\prime \prime }\circ R^{\prime
\prime }\theta C^{\prime \prime }\right) \\
B & \overset{U_{\left[ 1\right] }f}{\longrightarrow } & \Theta C^{\prime
\prime }%
\end{array}%
\end{equation*}
\end{invisible}

Since $f\in \mathsf{M}\left( GU_{\left[ 1\right] }\right) $, we have that $%
U_{\left[ 1\right] }f\in \mathsf{M}\left( G\right) \mathsf{.}$ Since $\Theta
:\mathcal{B}^{\prime \prime }\rightarrow \mathcal{B}$ is an $\mathsf{M}%
\left( G\right) $-fibration, there is a morphism $f^{\prime \prime
}:B^{\prime \prime }\rightarrow C^{\prime \prime }$ which is cartesian (with
respect to $\Theta $) over $U_{\left[ 1\right] }f.$ In particular $\Theta
B^{\prime \prime }=B$ and $\Theta f^{\prime \prime }=U_{\left[ 1\right] }f.$
Note that $R=\Theta R^{\prime \prime }\ $so that the first diagram in %
\eqref{diag:serv1} rewrites as the second one therein.
\begin{invisible}
\begin{equation*}
\begin{array}{ccc}
\Theta R^{\prime \prime }L\Theta B^{\prime \prime } & \overset{\Theta
R^{\prime \prime }L\Theta f^{\prime \prime }}{\longrightarrow } & \Theta
R^{\prime \prime }L\Theta C^{\prime \prime } \\
b\downarrow &  & \downarrow \Theta \left( c^{\prime \prime }\circ R^{\prime
\prime }\theta C^{\prime \prime }\right) \\
\Theta B^{\prime \prime } & \overset{\Theta f^{\prime \prime }}{%
\longrightarrow } & \Theta C^{\prime \prime }%
\end{array}%
\end{equation*}
\end{invisible}
Since $f^{\prime \prime }$ is cartesian (with respect to $\Theta $) over $U_{%
\left[ 1\right] }f=\Theta f^{\prime \prime },$ this diagram implies
there is a unique morphism $\tau ^{\prime \prime }:R^{\prime \prime
}LB\rightarrow B^{\prime \prime }$ such that $\Theta\tau''=b$ and the third diagram in %
\eqref{diag:serv1} commutes.

\begin{invisible}
\begin{equation*}
\begin{array}{ccc}
R^{\prime \prime }L\Theta B^{\prime \prime } & \overset{R^{\prime \prime
}L\Theta f^{\prime \prime }}{\longrightarrow } & R^{\prime \prime }L\Theta
C^{\prime \prime } \\
\tau ^{\prime \prime }\downarrow &  & \downarrow c^{\prime \prime }\circ
R^{\prime \prime }\theta C^{\prime \prime } \\
B^{\prime \prime } & \overset{f^{\prime \prime }}{\longrightarrow } &
C^{\prime \prime }%
\end{array}%
\end{equation*}
\end{invisible}

Set $b^{\prime \prime }:=\tau ^{\prime \prime }\circ \left( R^{\prime \prime
}\theta B^{\prime \prime }\right) ^{-1}$ and $B_{\left[ 1\right] }^{\prime
\prime }:=\left( B^{\prime \prime },b^{\prime \prime }\right) $ Then%
\begin{equation*}
f^{\prime \prime }\circ b^{\prime \prime }\circ R^{\prime \prime }\theta
B^{\prime \prime }=f^{\prime \prime }\circ \tau ^{\prime \prime }=c^{\prime
\prime }\circ R^{\prime \prime }\theta C^{\prime \prime }\circ R^{\prime
\prime }L\Theta f^{\prime \prime }=c^{\prime \prime }\circ R^{\prime \prime
}L^{\prime \prime }f^{\prime \prime }\circ R^{\prime \prime }\theta
B^{\prime \prime }
\end{equation*}%
and hence $f^{\prime \prime }\circ b^{\prime \prime }=c^{\prime \prime
}\circ R^{\prime \prime }L^{\prime \prime }f^{\prime \prime }.$ As a
consequence $f^{\prime \prime }$ induces a morphism $f_{\left[ 1\right]
}^{\prime \prime }:\left( B^{\prime \prime },b^{\prime \prime }\right)
\rightarrow \left( C^{\prime \prime },c^{\prime \prime }\right) $ such that $%
U_{\left[ 1\right] }^{\prime \prime }f_{\left[ 1\right] }^{\prime \prime
}=f^{\prime \prime }.$ We compute
\begin{equation*}
U_{\left[ 1\right] }\Theta _{\left[ 1\right] }^{\prime }f_{\left[ 1\right]
}^{\prime \prime }=\Theta U_{\left[ 1\right] }^{\prime \prime }f_{\left[ 1%
\right] }^{\prime \prime }=\Theta f^{\prime \prime }=U_{\left[ 1\right] }f
\end{equation*}%
Since
\begin{equation*}
\Theta _{\left[ 1\right] }^{\prime }B_{\left[ 1\right] }^{\prime \prime
}=\left( \Theta B^{\prime \prime },\Theta b^{\prime \prime }\circ R\theta
B^{\prime \prime }\right) =\left( \Theta B^{\prime \prime },\Theta \left(
b^{\prime \prime }\circ R^{\prime \prime }\theta B^{\prime \prime }\right)
\right) =\left( \Theta B^{\prime \prime },\Theta \tau ^{\prime \prime
}\right) =\left( B,b\right) =B_{\left[ 1\right] }
\end{equation*}%
we have that $\Theta _{\left[ 1\right] }^{\prime }f_{\left[ 1\right]
}^{\prime \prime }$ and $f$ have the same domain (and codomain). Since $U_{%
\left[ 1\right] }$ is faithful, we get $\Theta _{\left[ 1\right] }^{\prime
}f_{\left[ 1\right] }^{\prime \prime }=f.$

Since $U_{\left[ 1\right] }\Theta _{\left[ 1\right] }^{\prime }f_{\left[ 1%
\right] }^{\prime \prime }=U_{\left[ 1\right] }f\in \mathsf{M}\left(
G\right) ,$ by hypothesis $U_{\left[ 1\right] }\Theta _{\left[ 1\right]
}^{\prime }f_{\left[ 1\right] }^{\prime \prime }$ is cartesian over $GU_{%
\left[ 1\right] }\Theta _{\left[ 1\right] }^{\prime }f_{\left[ 1\right]
}^{\prime \prime }.$ By Lemma \ref{lem:mono1}, we deduce that $U_{\left[ 1%
\right] }\Theta _{\left[ 1\right] }^{\prime }f_{\left[ 1\right] }^{\prime
\prime }$ is a monomorphism. Since $U_{\left[ 1\right] }\Theta _{\left[ 1%
\right] }^{\prime }=\Theta U_{\left[ 1\right] }^{\prime \prime },$ we get
that $U_{\left[ 1\right] }^{\prime \prime }f_{\left[ 1\right] }^{\prime
\prime }\in \mathsf{M}\left( \Theta \right) .$ The latter morphism is $%
f^{\prime \prime }$ which is cartesian (with respect to $\Theta $) over $%
\Theta f^{\prime \prime }.$ Again, by Lemma \ref{lem:mono1}, we deduce that $%
f^{\prime \prime }=U_{\left[ 1\right] }^{\prime \prime }f_{\left[ 1\right]
}^{\prime \prime }$ is a monomorphism i.e. $f_{\left[ 1\right] }^{\prime
\prime }:B_{\left[ 1\right] }^{\prime \prime }\rightarrow C_{\left[ 1\right]
}^{\prime \prime }$ in $\mathsf{M}\left( U_{\left[ 1\right] }^{\prime \prime
}\right) $ as desired.

Let us check that $f_{\left[ 1\right] }^{\prime \prime }:B_{\left[ 1\right]
}^{\prime \prime }\rightarrow C_{\left[ 1\right] }^{\prime \prime }$ is
cartesian with respect to $\Theta _{\left[ 1\right] }^{\prime }$ over $f.$

Let $g_{\left[ 1\right] }^{\prime \prime }:D_{\left[ 1\right] }^{\prime
\prime }\rightarrow C_{\left[ 1\right] }^{\prime \prime }$ in $\mathcal{B}_{%
\left[ 1\right] }^{\prime \prime }$ and $h:\Theta _{\left[ 1\right]
}^{\prime }D_{\left[ 1\right] }^{\prime \prime }\rightarrow B_{\left[ 1%
\right] }$ in $\mathcal{B}_{\left[ 1\right] }$ be such that $\Theta _{\left[
1\right] }^{\prime }f_{\left[ 1\right] }^{\prime \prime }\circ h=\Theta _{%
\left[ 1\right] }^{\prime }g_{\left[ 1\right] }^{\prime \prime }.$ By
applying on both sides $U_{\left[ 1\right] }$, we get $U_{\left[ 1\right]
}\Theta _{\left[ 1\right] }^{\prime }f_{\left[ 1\right] }^{\prime \prime
}\circ U_{\left[ 1\right] }h=U_{\left[ 1\right] }\Theta _{\left[ 1\right]
}^{\prime }g_{\left[ 1\right] }^{\prime \prime }$ i.e. $\Theta f^{\prime
\prime }\circ U_{\left[ 1\right] }h=\Theta U_{\left[ 1\right]
}^{\prime \prime }g_{\left[ 1\right] }^{\prime \prime }.$ Since $f^{\prime
\prime }:B^{\prime \prime }\rightarrow C^{\prime \prime }$ is cartesian over
$U_{\left[ 1\right] }f=\Theta f^{\prime \prime },$ we get that there is a
unique morphism $k^{\prime \prime }:D^{\prime \prime }\rightarrow B^{\prime
\prime }$ in $\mathcal{B}^{\prime \prime }$ such that $U_{\left[ 1\right]
}h=\Theta k^{\prime \prime }$ and $f^{\prime \prime }\circ k^{\prime \prime
}=U_{\left[ 1\right] }^{\prime \prime }g_{\left[ 1\right] }^{\prime \prime
}. $ Let us check that $k^{\prime \prime }$ induces a morphism $k_{\left[ 1%
\right] }^{\prime \prime }:D_{\left[ 1\right] }^{\prime \prime }\rightarrow
B_{\left[ 1\right] }^{\prime \prime }$ such that $\Theta _{\left[ 1\right]
}^{\prime }k_{\left[ 1\right] }=h.$ Write $D_{\left[ 1\right] }^{\prime
\prime }=\left( D^{\prime \prime },d^{\prime \prime }:R^{\prime \prime
}L^{\prime \prime }D^{\prime \prime }\rightarrow D^{\prime \prime }\right) .$
Knowing that $f^{\prime \prime }\circ b^{\prime \prime }=c^{\prime \prime
}\circ R^{\prime \prime }L^{\prime \prime }f^{\prime \prime }$ and that $g_{%
\left[ 1\right] }^{\prime \prime }:D_{\left[ 1\right] }^{\prime \prime
}\rightarrow C_{\left[ 1\right] }^{\prime \prime }$ belongs to $\mathcal{B}_{%
\left[ 1\right] }^{\prime \prime },$ we obtain%
\begin{equation*}
f^{\prime \prime }\circ b^{\prime \prime }\circ R^{\prime \prime }L^{\prime
\prime }k^{\prime \prime }=c^{\prime \prime }\circ R^{\prime \prime
}L^{\prime \prime }f^{\prime \prime }\circ R^{\prime \prime }L^{\prime
\prime }k^{\prime \prime }=c^{\prime \prime }\circ R^{\prime \prime
}L^{\prime \prime }U_{\left[ 1\right] }^{\prime \prime }g_{\left[ 1\right]
}^{\prime \prime }=U_{\left[ 1\right] }^{\prime \prime }g_{\left[ 1\right]
}^{\prime \prime }\circ d^{\prime \prime }=f^{\prime \prime }\circ k^{\prime
\prime }\circ d^{\prime \prime }.
\end{equation*}%
Since we proved that $f^{\prime \prime }\ is$ a monomorphism, we get that $%
b^{\prime \prime }\circ R^{\prime \prime }L^{\prime \prime }k^{\prime \prime
}=k^{\prime \prime }\circ d^{\prime \prime }$ i.e. that $k$ induces a
morphism $k_{\left[ 1\right] }^{\prime \prime }:D_{\left[ 1\right] }^{\prime
\prime }\rightarrow B_{\left[ 1\right] }^{\prime \prime }$ such that $U_{%
\left[ 1\right] }^{\prime \prime }k_{\left[ 1\right] }^{\prime \prime
}=k^{\prime \prime }.$ We have
\begin{equation*}
U_{\left[ 1\right] }\Theta _{\left[ 1\right] }^{\prime }k_{\left[ 1\right]
}^{\prime \prime }=\Theta U_{\left[ 1\right] }^{\prime \prime }k_{\left[ 1%
\right] }^{\prime \prime }=\Theta k^{\prime \prime }=U_{\left[ 1\right] }h.
\end{equation*}%
Since $\Theta _{\left[ 1\right] }^{\prime }B_{\left[ 1\right] }^{\prime
\prime }=B_{\left[ 1\right] }$ we have that $\Theta _{\left[ 1\right]
}^{\prime }k_{\left[ 1\right] }^{\prime \prime }$ and $h$ have the same
domain and codomain. Since $U_{\left[ 1\right] }$ is faithful, we obtain
that $\Theta _{\left[ 1\right] }k_{\left[ 1\right] }^{\prime \prime }=h.$
Moreover%
\begin{equation*}
U_{\left[ 1\right] }^{\prime \prime }\left( f_{\left[ 1\right] }^{\prime
\prime }\circ k_{\left[ 1\right] }^{\prime \prime }\right) =U_{\left[ 1%
\right] }^{\prime \prime }f_{\left[ 1\right] }^{\prime \prime }\circ U_{%
\left[ 1\right] }^{\prime \prime }k_{\left[ 1\right] }^{\prime \prime
}=f^{\prime \prime }\circ k^{\prime \prime }=U_{\left[ 1\right] }^{\prime
\prime }g_{\left[ 1\right] }^{\prime \prime }.
\end{equation*}%
Since $U_{\left[ 1\right] }^{\prime \prime }$ is faithful, we get $f_{\left[
1\right] }^{\prime \prime }\circ k_{\left[ 1\right] }^{\prime \prime }=g_{%
\left[ 1\right] }^{\prime \prime }$. Moreover $k_{\left[ 1\right] }^{\prime
\prime }$ is unique since $U_{\left[ 1\right] }^{\prime \prime }$ is
faithful and $f^{\prime \prime }\ $is a monomorphism. We have so proved that
$f_{\left[ 1\right] }^{\prime \prime }:B_{\left[ 1\right] }^{\prime \prime
}\rightarrow C_{\left[ 1\right] }^{\prime \prime }$ is cartesian over $f.$
\end{proof}

\begin{theorem}
\label{teo:serv1}In the setting of Proposition \ref{pro:5}, assume that

\begin{itemize}
\item any morphism $g\in \mathsf{M}\left( G\right) $ is cartesian over $Gg;$

\item $\Theta :\mathcal{B}^{\prime \prime }\rightarrow \mathcal{B}$ is an $%
\mathsf{M}\left( G\right) $-fibration,

\item $\theta $ is invertible.
\end{itemize}

Then for every morphism $f^{\prime }:\left( B,\beta \right) \rightarrow
\Theta _{\left[ 1\right] }C_{\left[ 1\right] }^{\prime \prime }$ in $\mathsf{%
M}\left( G\circ P\left( \mathbb{T}\right) \right) $ there is $f_{\left[ 1%
\right] }:B_{\left[ 1\right] }^{\prime \prime }\rightarrow C_{\left[ 1\right]
}^{\prime \prime }$ in $\mathsf{M}\left( U_{\left[ 1\right] }^{\prime \prime
}\right) $ which is cartesian with respect to $\Theta _{\left[ 1\right] }$
over $f^{\prime }.$ In particular $\Theta _{\left[ 1\right] }:\mathcal{B}_{%
\left[ 1\right] }^{\prime \prime }\rightarrow \mathrm{I}\left( \mathbb{T}%
\right) $ is an $\mathsf{M}\left( G\circ P\left( \mathbb{T}\right) \right) $%
-fibration.
\end{theorem}

\begin{proof}
First note that $\Theta _{\left[ 1\right] }=S^{G}\Theta _{\left[ 1\right]
}^{\prime }$ as they coincide on morphisms and for every $C_{\left[ 1\right]
}^{\prime \prime }=\left( C^{\prime \prime },c^{\prime \prime }:R^{\prime
\prime }L^{\prime \prime }C^{\prime \prime }\rightarrow C^{\prime \prime
}\right) \in \mathcal{B}_{\left[ 1\right] }^{\prime \prime },$ we have%
\begin{eqnarray*}
\Theta _{\left[ 1\right] }C_{\left[ 1\right] }^{\prime \prime } &=&\Theta _{%
\left[ 1\right] }\left( C^{\prime \prime },c^{\prime \prime }\right) =\left(
\Theta C^{\prime \prime },G^{\prime \prime }c^{\prime \prime }\circ
R^{\prime }\theta C^{\prime \prime }\right) =\left( \Theta C^{\prime \prime
},G\Theta c^{\prime \prime }\circ GR\theta C^{\prime \prime }\right) \\
&=&S^{G}\left( \Theta C^{\prime \prime },\Theta c^{\prime \prime }\circ
R\theta C^{\prime \prime }\right) =S^{G}\Theta _{\left[ 1\right] }^{\prime
}C_{\left[ 1\right] }^{\prime \prime }.
\end{eqnarray*}

Let $f^{\prime }:\left( B,\beta \right) \rightarrow \Theta _{\left[ 1\right]
}C_{\left[ 1\right] }^{\prime \prime }$ in $\mathsf{M}\left( G\circ P\left(
\mathbb{T}\right) \right) $. Since $\Theta _{\left[ 1\right] }C_{\left[ 1%
\right] }^{\prime \prime }=S^{G}\Theta _{\left[ 1\right] }^{\prime }C_{\left[
1\right] }^{\prime \prime }$, by Proposition \ref{pro:serv0},
there is a morphism $f:B_{\left[ 1\right] }\rightarrow \Theta _{\left[ 1%
\right] }^{\prime }C_{\left[ 1\right] }^{\prime \prime }$ in $\mathsf{M}%
\left( GU_{\left[ 1\right] }\right) $ which is cartesian with respect to $%
S^{G}$ over $f^{\prime }.$ By Proposition \ref{pro:serv1}, there
is $f_{\left[ 1\right] }^{\prime \prime }:B_{\left[ 1\right] }^{\prime
\prime }\rightarrow C_{\left[ 1\right] }^{\prime \prime }$ in $\mathsf{M}%
\left( U_{\left[ 1\right] }^{\prime \prime }\right) $ which is cartesian
with respect to $\Theta _{\left[ 1\right] }^{\prime }$ over $f.$ By Lemma %
\ref{lem:transitive}, the morphism $f_{\left[ 1\right] }^{\prime \prime }$
is cartesian with respect to $\Theta _{\left[ 1\right] }=S^{G}\Theta _{\left[
1\right] }^{\prime }$ over $f^{\prime }.$
\end{proof}

\begin{theorem}
\label{teo:servaureo}In the setting of Proposition \ref{pro:5}, assume that

\begin{itemize}
\item any morphism $g\in \mathsf{M}\left( G\right) $ is cartesian over $Gg;$

\item $\Theta :\mathcal{B}^{\prime \prime }\rightarrow \mathcal{B}$ is an $%
\mathsf{M}\left( G\right) $-fibration,

\item $\theta $ is invertible.
\end{itemize}

Then $\Theta _{1}:\mathcal{B}_{1}^{\prime \prime }\rightarrow \mathrm{I}%
\left( \mathbb{T}\right) $ is an $\mathsf{M}\left( G\circ P\left( \mathbb{T}%
\right) \right) $-fibration.
\end{theorem}

\begin{proof}
Let $f^{\prime }:\left( B,\beta :R^{\prime }LB\rightarrow GB\right)
\rightarrow \Theta _{1}C_{1}^{\prime \prime }$ be a morphism in $\mathsf{M}%
\left( G\circ P\left( \mathbb{T}\right) \right) $ i.e. a morphism in $%
\mathrm{I}\left( \mathbb{T}\right) $ such that $GP\left( \mathbb{T}\right)
f^{\prime }:GB\rightarrow GP\left( \mathbb{T}\right) \Theta
_{1}C_{1}^{\prime \prime }=G\Theta C^{\prime \prime }$ is a monomorphism,
where $C_{1}^{\prime \prime }=\left( C^{\prime \prime },c^{\prime \prime
}:R^{\prime \prime }L^{\prime \prime }C^{\prime \prime }\rightarrow
C^{\prime \prime }\right) .$
Since $\Theta _{1}=\Theta _{\left[ 1\right] }\Lambda _{1}$, by Theorem \ref%
{teo:serv1}, there is a morphism $f_{\left[ 1\right] }:\left( B^{\prime
\prime },b^{\prime \prime }\right) \rightarrow \left( C^{\prime \prime
},c^{\prime \prime }\right) $ in $\mathsf{M}\left( U_{\left[ 1\right]
}^{\prime \prime }\right) $ which is cartesian with respect to $\Theta _{%
\left[ 1\right] }$ over $f^{\prime }.$ In particular $\Theta _{\left[ 1%
\right] }\left( B^{\prime \prime },b^{\prime \prime }\right) =\left( B,\beta
\right) $ and $\Theta _{\left[ 1\right] }f_{\left[ 1\right] }=f^{\prime }$.
Since $f_{\left[ 1\right] }\in \mathsf{M}\left( U_{\left[ 1\right] }^{\prime
\prime }\right) $, by Lemma \ref{lem:fibr1}, there is $f_{1}:B_{1}^{\prime
\prime }\rightarrow C_{1}^{\prime \prime }$ which is cartesian with respect
to $\Lambda _{1}:\mathcal{B}_{1}\rightarrow \mathcal{B}_{\left[ 1\right] }$
over $f_{\left[ 1\right] }.$ We compute%
\begin{equation*}
\Theta _{1}f_{1}=\Theta _{\left[ 1\right] }\Lambda _{1}f_{1}=\Theta _{\left[
1\right] }f_{\left[ 1\right] }=f^{\prime }.
\end{equation*}
Since $f_{1}$ is cartesian with respect to $\Lambda _{1}$ over $f_{\left[ 1%
\right] }$ and $f_{\left[ 1\right] }$ is cartesian with respect to $\Theta _{%
\left[ 1\right] }$ over $f^{\prime },$ by Lemma \ref{lem:transitive}, $f_{1}$
is cartesian with respect to $\Theta _{1}=\Theta _{\left[ 1\right] }\Lambda
_{1}$ over $f^{\prime }.$
\end{proof}

\begin{theorem}
\label{teo:LambdaFibr}Let $n\in \mathbb{N}$. The functor $\Lambda _{n}:%
\mathcal{B}_{n}\rightarrow \mathcal{B}_{\left[ n\right] }$ of Remark \ref%
{rem:adj} is an $\mathsf{M}\left( U_{\left[ n\right] }\right) $-fibration.
Moreover it is a discrete isofibration and $\mathrm{\mathrm{Eim}}\left(
\Lambda _{n}\right) =\mathrm{Im}\left( \Lambda _{n}\right) =\mathrm{Im}%
^{\prime }\left( \Lambda _{n}\right) $ (see Definition \ref{def:images}).
\end{theorem}

\begin{proof}
We proceed by induction on $n\in \mathbb{N}$. The first step is trivially
true since $\Lambda _{0}=\mathrm{Id}_{\mathcal{B}}=U_{\left[ 0\right] }$.

Let $n\geq 1$ and assume the statement true for $n-1.$ Apply Theorem \ref%
{teo:servaureo} to $\Theta :=\Lambda _{n-1},G:=U_{\left[ n-1\right] },%
\mathbb{T=T}_{\left[ n-1\right] }$, $\mathbb{T}'=\boldsymbol{\Lambda}_{[n-1]}$ by noting that $\Lambda _{n}=\left(
\Lambda _{n-1}\right) _{1}$, that $\Theta $ fulfills the required conditions
by inductive hypothesis, $G$ also fulfills them  by Proposition \ref{pro:Unfibr} and $\theta =\lambda _{n-1}$ is
invertible by Remark \ref{rem:adj}. Thus $\Lambda _{n}\ $is an $\mathsf{M}%
\left( U_{\left[ n\right] }\right) $-fibration.
i
Any isomorphism in $\mathcal{B}_{\left[ n\right] }$ belongs trivially to $%
\mathsf{M}\left( U_{\left[ n\right] }\right).$ Moreover, by Remark \ref%
{rem:adj}, we know that $\Lambda _{n}$ is (fully) faithful and injective on
objects. Thus we can apply Corollary \ref{coro:isofibr} and Remark \ref%
{rem:isofibr} to obtain that $\Lambda _{n}$ is a discrete isofibration. From
Lemma \ref{lem:isofibfull} and the fact that $\Lambda _{n}$ is full, we get
that $\mathrm{\mathrm{Eim}}\left( \Lambda _{n}\right) \subseteq \mathrm{Im}%
^{\prime }\left( \Lambda _{n}\right) =\mathrm{Im}\left( \Lambda _{n}\right)
. $ We know that $\mathrm{Im}\left( \Lambda _{n}\right) \subseteq \mathrm{%
\mathrm{Eim}}\left( \Lambda _{n}\right) $ holds always.
\end{proof}

The following result gives conditions for an object in $\mathcal{B}_{\left[ n\right] }$ to be image via $\Lambda_n$ of an object in $\mathcal{B}_{n}$.

\begin{theorem}
\label{teo:main}Fix $n\in \mathbb{N}$ consider the functors $\Lambda _{n}:%
\mathcal{B}_{n}\rightarrow \mathcal{B}_{\left[ n\right] }$ and $U_{\left[ n%
\right] }:\mathcal{B}_{\left[ n\right] }\rightarrow \mathcal{B}.$

\begin{itemize}
\item[1)] For every morphism $B_{\left[ n\right] }\rightarrow \Lambda
_{n}C_{n}$ in $\mathsf{M}\left( U_{\left[ n\right] }\right) $ we have $B_{%
\left[ n\right] }\in \mathrm{Im}\left( \Lambda _{n}\right) .$

\item[2)] Let $B_{\left[ n\right] }\in \mathcal{B}_{\left[ n\right] }$ be
such that $\eta _{\left[ n\right] }B_{\left[ n\right] }$ is in $\mathsf{M}%
\left( U_{\left[ n\right] }\right) .$ Then $B_{\left[ n\right] }\in \mathrm{%
Im}\left( \Lambda _{n}\right) .$

\item[3)] Let $\left( B_{\left[ n\right] },b_{\left[ n\right] }\right) \in
\mathrm{Im}\left( \Lambda _{n+1}\right) .$ Then $\eta _{\left[ n\right] }B_{%
\left[ n\right] }$ is in $\mathsf{M}\left( U_{\left[ n\right] }\right) .$
\end{itemize}
\end{theorem}

\begin{proof}
1) By Theorem \ref{teo:LambdaFibr}, the functor $\Lambda _{n}:\mathcal{B}%
_{n}\rightarrow \mathcal{B}_{\left[ n\right] }$ is an $\mathsf{M}\left( U_{%
\left[ n\right] }\right) $-fibration. Thus, for every morphism $B_{\left[ n%
\right] }\rightarrow \Lambda _{n}C_{n}$ in $\mathsf{M}\left( U_{\left[ n%
\right] }\right) $ there is $B_{n}\in \mathcal{B}_{n}\ $such that $\Lambda
_{n}B_{n}=B_{\left[ n\right] }.$

2) Since $\eta _{\left[ n\right] }B_{\left[ n\right] }$ is a morphism $B_{%
\left[ n\right] }\rightarrow R_{\left[ n\right] }L_{\left[ n\right] }B_{%
\left[ n\right] }=\Lambda _{n}R_{n}L_{\left[ n\right] }B_{\left[ n\right] }$%
, we conclude by 1).

3)\ Since $\left( B_{\left[ n\right] },b_{\left[ n\right] }\right) \in
\mathrm{Im}\left( \Lambda _{n+1}\right) ,$ there is $B_{n+1}=\left(
B_{n},\mu _{n}:R_{n}L_{n}B_{n}\rightarrow B_{n}\right) \in \mathcal{B}_{n+1}$
such that $\left( B_{\left[ n\right] },b_{\left[ n\right] }\right) =\Lambda
_{n+1}B_{n+1}.$ Then $\mu _{n}\circ \eta _{n}B_{n}=\mathrm{Id}_{B_{n}}$ and $$B_{[n]}=U_{[n,n+1]}(B_{[n]},b_{[n]})=U_{[n,n+1]}\Lambda
_{n+1}B_{n+1} =\Lambda_{n} U_{n,n+1}B_{n+1}=\Lambda_{n} B_{n}.$$
\begin{invisible}
MI PARE UNUTILE QUANTO SEGUE:
We have  $\left( B_{\left[ n\right] },b_{\left[ n\right] }\right) =\left(
\Lambda _{n}B_{n},U_{n}\mu _{n}\circ R\lambda _{n}B_{n}\right) .$
Thus $B_{%
\left[ n\right] }=\Lambda _{n}B_{n}$ and
\begin{equation*}
b_{\left[ n\right] }=U_{n}\mu _{n}\circ R\lambda _{n}B_{n}=U_{n}\left( \mu
_{n}\circ R_{n}\lambda _{n}B_{n}\right) =U_{\left[ n\right] }\Lambda
_{n}\left( \mu _{n}\circ R_{n}\lambda _{n}B_{n}\right) =U_{\left[ n\right]
}\beta _{\left[ n\right] },
\end{equation*}%
where we set $\beta _{\left[ n\right] }:=\Lambda _{n}\left( \mu _{n}\circ
R_{n}\lambda _{n}B_{n}\right) :R_{\left[ n\right] }L_{\left[ n\right] }B_{%
\left[ n\right] }\rightarrow B_{\left[ n\right] }.$
\end{invisible}
Since $\lambda _{n}$ is the natural transformation inside the adjoint triangle $\boldsymbol{\Lambda}_{[n]}$, see Remark \ref{rem:adj}, we have $%
\lambda _{n}=\epsilon _{\left[ n\right] }L_{n}\circ L_{\left[ n\right]
}\Lambda _{n}\eta _{n}$ so that
\begin{eqnarray*}
R_{\left[ n\right] }\lambda _{n}\circ \eta _{\left[ n\right] }\Lambda _{n}
&=&R_{\left[ n\right] }\epsilon _{\left[ n\right] }L_{n}\circ R_{\left[ n%
\right] }L_{\left[ n\right] }\Lambda _{n}\eta _{n}\circ \eta _{\left[ n%
\right] }\Lambda _{n} \\
&=&R_{\left[ n\right] }\epsilon _{\left[ n\right] }L_{n}\circ \eta _{\left[ n%
\right] }\Lambda _{n}R_{n}L_{n}\circ \Lambda _{n}\eta _{n} =R_{\left[ n\right] }\epsilon _{\left[ n\right] }L_{n}\circ \eta _{\left[ n%
\right] }R_{\left[ n\right] }L_{n}\circ \Lambda _{n}\eta _{n}=\Lambda
_{n}\eta _{n}.
\end{eqnarray*}%
As a consequence we obtain
\begin{equation*}
\Lambda _{n}\mu _{n}\circ R_{\left[ n\right] }\lambda _{n}B_{n}\circ \eta _{%
\left[ n\right] }\Lambda _{n}B_{n}=\Lambda _{n}\mu _{n}\circ \Lambda
_{n}\eta _{n}B_{n}=\mathrm{Id}_{\Lambda _{n}B_{n}}.
\end{equation*}%
In particular $\eta _{\left[ n\right] }B_{\left[ n\right] }=\eta _{\left[ n%
\right] }\Lambda _{n}B_{n}$ is in $\mathsf{M}\left( U_{\left[ n\right]
}\right) .$
\end{proof}

%
%

\begin{corollary}
\label{coro:ff}Fix $n\in \mathbb{N}.$ If the functor $L_{\left[ n\right] }$
is fully faithful, then so is $L_{n}$ and $\Lambda _{n}:\mathcal{B}_{n}\rightarrow
\mathcal{B}_{\left[ n\right] }$ is a category isomorphism. In particular $R$
has a monadic decomposition of monadic length at most $n.$
\end{corollary}

\begin{proof}
We have the isomorphism $\lambda _{n}:L_{\left[ n\right] }\circ \Lambda
_{n}\rightarrow L_{n}.$ Thus if $L_{\left[ n\right] }$ is fully faithful we
get that $L_{n}$ is fully faithful being isomorphic to a composition of
fully faithful functors. By the dual version of \cite[Proposition 3.4.1 ]%
{Borceux1}, we have that $\eta _{\left[ n\right] }$ is invertible and hence $%
\eta _{\left[ n\right] }B_{\left[ n\right] }$ is in $\mathsf{M}\left( U_{%
\left[ n\right] }\right) $ for every $B_{\left[ n\right] }\in \mathcal{B}_{%
\left[ n\right] }.$ By Theorem \ref{teo:main}, $B_{\left[ n\right] }\in
\mathrm{Im}\left( \Lambda _{n}\right) .$ Thus $\Lambda _{n}$ is surjective
on objects. We already know that $\Lambda _{n}$ is injective on objects, see
Remark \ref{rem:adj}, thus it is bijective on objects. Since we know it is
also fully faithful, we deduce that it is an isomorphism.
\begin{invisible}
Since $L_{n}$ is fully faithful, we know that $U_{n,n+1}$ is an isomorphism
of categories. Since $U_{\left[ n,n+1\right] }\circ \Lambda _{n+1}=\Lambda
_{n}\circ U_{n,n+1},$ we deduce that $U_{\left[ n,n+1\right] }\circ \Lambda
_{n+1}$ is an isomorphism of categories.
\end{invisible}
\end{proof}

\begin{corollary}
\label{coro:monadic}Consider an adjunction $\left( L,R\right) $ such that $%
L_{\left[ 1\right] }$ and $L_{1}$ exist. If $R_{\left[ 1\right] }$ is an
equivalence of categories then $R$ is monadic. Moreover $\Lambda _{1}:\mathcal{B}%
_{1}\rightarrow \mathcal{B}_{\left[ 1\right] }$ is a category isomorphism.
\end{corollary}

\begin{proof}
If $R_{\left[ 1\right] }$ is an equivalence of categories then $L_{\left[ 1%
\right] }$ is an equivalence of categories and hence, by Corollary \ref%
{coro:ff}, $L_{1}$ is fully faithful and $\Lambda _{1}:\mathcal{B}%
_{1}\rightarrow \mathcal{B}_{\left[ 1\right] }$ is a category isomorphism. Since $R_{\left[ 1\right] }=\Lambda _{1}\circ R_{1}$ we get that
$R_{1}$ is an equivalence of categories. Equivalently $R$ is monadic.
\end{proof}

\begin{example}
\label{ex:tomega}Let us show that the converse of Corollary \ref{coro:monadic} is not true, in general. Let $\mathcal{B}=\mathrm{Vec}_{\Bbbk }.$ As a starting
adjunction consider $\left( T,\Omega \right) $ where $T:\mathcal{B}%
\rightarrow \mathrm{Alg}_{\Bbbk }$ is the tensor algebra functor and $\Omega
:\mathrm{Alg}_{\Bbbk }\rightarrow \mathcal{B}$ is the forgetful functor.  It is well-known that $\Omega $
is strictly monadic i.e. the comparison functor $\Omega _{1}:%
\mathrm{Alg}_{\Bbbk }\rightarrow \mathcal{B}_{1}$ is a
category isomorphism, see \cite[Theorem A.6]{AM-MM}. Given $B\in \mathcal{B},$ consider the zero map $b:\Omega TB\rightarrow B.$
Then $\left( B,b\right) \in \left\langle \Omega T|\mathrm{Id}\right\rangle =%
\mathcal{B}_{\left[ 1\right] }$ but $\left( B,b\right) \notin \mathrm{Im}%
\left( \Lambda _{1}\right) $ since $b\circ \eta B\neq \mathrm{Id}_{B}$, where
$\eta $ is the unit of
the adjunction $\left( T,\Omega \right) .$
Thus $\Lambda _{1}$ is not surjective whence not even a category isomorphism. By Corollary \ref{coro:monadic}, we conclude that
$\Omega _{\left[ 1\right] }$ is not an equivalence.
\begin{equation*}
\xymatrixcolsep{1.5cm}\xymatrixrowsep{0.7cm} \xymatrix{\Alg_{\Bbbk
}\ar[r]^\id\ar@{}[dr]|-{\lambda_1}\ar@<.5ex>[d]^{{\Omega_1}}&
\Alg_{\Bbbk }\ar@<.4ex>[d]^{\Omega_{[1]}}\\
\mathcal{B}_1\ar@<.4ex>@{.>}[u]^{{T_1}}\ar[r]^-{\Lambda_1}&%
\mathcal{B}_{[1]}=\langle\Omega
T,\id\rangle\ar@<.4ex>@{.>}[u]^{{T_{[1]}}}}
\end{equation*}
We have so proved that $R$ is monadic although $R_{\left[ 1\right] }$ is not an equivalence for $R=\Omega $.

\begin{invisible}
\begin{equation*}
\begin{array}{ccc}
\mathrm{Alg}_{\Bbbk } & \overset{\mathrm{Id}}{\longrightarrow } & \mathrm{Alg%
}_{\Bbbk } \\
T_{1}\uparrow \downarrow \Omega _{1} & \lambda _{1} & T_{\left[ 1\right]
}\uparrow \downarrow \Omega _{\left[ 1\right] } \\
\mathcal{B}_{1} & \overset{\Lambda _{1}}{\longrightarrow } & \mathcal{B}_{%
\left[ 1\right] }=\left\langle \Omega T|\mathrm{Id}\right\rangle%
\end{array}%
\end{equation*}
\end{invisible}
\end{example}

\section{Connection to augmented monads\label{sec:5}}

As an application of Theorem \ref{teo:main}, in this section we show
how to construct some functors $\Gamma_n:\mathcal{B}\to\mathcal{B}_{[n]}$
that factor through $\Lambda_n :\mathcal{B}_n\to\mathcal{B}_{[n]}$. The
existence of $\Gamma_n$ is related to the notion of augmented monad.

Recall that an \textbf{augmentation} for a monad $(M,m:MM\to M,u:\id\to M)$ is a natural transformation $\gamma:M\to \id$ such that $\gamma\circ u=\id$ and $\gamma\gamma=\gamma\circ m$. We will also say that the monad $M$ is \textbf{augmented} via the morphism $\gamma:M\to \id$.

We will mainly focus on the existence of an augmentation for the monad $\left(
RL,R\epsilon L,\eta \right) $ associated to a given adjunction $\left( L,R\right) $. We point out that such a monad  has an augmentation if and only if the left adjoint $L$ is h-separable, see \cite[Corollary 2.7]{AM-heavy}.

\begin{theorem}
\label{thm:main}Consider a diagram%
\begin{equation*}
\xymatrixcolsep{1.5cm}\xymatrixrowsep{0.7cm}\xymatrix{\mathcal{A}\ar[r]^{F}%
\ar@<.5ex>@{.>}[d]^{R}&\mathcal{A'}\ar@<.5ex>@{.>}[d]^{R'}\\
\mathcal{B}\ar[r]^{\id}\ar@<.5ex>[u]^{L}&\mathcal{B}\ar@<.5ex>[u]^{L'}}
\end{equation*}
\begin{invisible}
\begin{equation*}
\begin{array}{ccc}
\mathcal{A} & \overset{F}{\longrightarrow } & \mathcal{A}^{\prime } \\
L\uparrow \downarrow R &  & L^{\prime }\uparrow \downarrow R^{\prime } \\
\mathcal{B} & \overset{\mathrm{Id}}{\longrightarrow } & \mathcal{B}%
\end{array}%
\end{equation*}
\end{invisible}
where $\left( L,R,\eta ,\epsilon \right) $ and $\left( L^{\prime },R^{\prime
},\eta ^{\prime },\epsilon ^{\prime }\right) $ are adjunctions such that $%
F\circ L=L^{\prime }.$ Define $\xi
:R\rightarrow R^{\prime }F$ by%
\begin{equation}
R\overset{\eta ^{\prime }R}{\rightarrow }R^{\prime }L^{\prime }R=R^{\prime
}FLR\overset{R^{\prime }F\epsilon }{\rightarrow }R^{\prime }F.
\label{def:xi}
\end{equation}%
Then $\xi L:RL\rightarrow R^{\prime }L^{\prime }$ is a morphism of monads
such that
\begin{equation}
\epsilon ^{\prime }F\circ L^{\prime }\xi =F\epsilon .  \label{form:efeps}
\end{equation}

Assume that:

\begin{enumerate}
\item[1)] $\mathcal{A}$ has all coequalizers and that $F$ preserves them;

\item[2)] $R^{\prime }$ preserves coequalizers of pairs $\left( fe,f\right) $
where $f$ is composition of regular epimorphisms and $e$ is an idempotent
morphism;

\item[3)] $R^{\prime }$ preserves regular epimorphisms;

\item[4)] the monad $R^{\prime }L^{\prime }$ has an augmentation $\gamma
^{\prime }:R^{\prime }L^{\prime }\rightarrow \mathrm{Id}_{\mathcal{B}}.$
\end{enumerate}

Then the monad $RL$ is augmented via $\gamma :=\gamma ^{\prime }\circ \xi L:RL\rightarrow \mathrm{Id}_{%
\mathcal{B}}.$ For every $%
n\in \mathbb{N}$, there are a functor $\Gamma _{\left[ n\right] }:\mathcal{B}%
\rightarrow \mathcal{B}_{\left[ n\right] }$ and a natural transformation $%
\gamma _{\left[ n\right] }:RL_{\left[ n\right] }\Gamma _{\left[ n\right]
}\rightarrow \mathrm{Id}_{\mathcal{B}},$ such that $\Gamma _{\left[ 0\right] }:=\mathrm{Id}%
_{\mathcal{B}},\gamma _{\left[ 0\right] }:=\gamma $ and, for $n\geq 0,$%
\begin{equation*}
\Gamma _{\left[ n+1\right] }B=\left( \Gamma _{\left[ n\right] }B,\gamma _{%
\left[ n\right] }B\right) \in \mathcal{B}_{\left[ n+1\right] },\qquad \gamma
_{\left[ n\right] }\circ U_{\left[ n\right] }\eta _{\left[ n\right] }\Gamma
_{\left[ n\right] }=\mathrm{Id}_{\id_{\mathcal{B}}},\qquad \gamma _{\left[ n+1\right] }\circ
R\pi _{\left[ n,n+1\right] }\Gamma _{\left[ n+1\right] }=\gamma _{\left[ n%
\right] }.
\end{equation*}%
Moreover $U_{\left[ n,n+1\right] }\circ \Gamma _{\left[ n+1\right] }=\Gamma
_{\left[ n\right] }.$
\end{theorem}

\begin{proof}
First we have%
\begin{equation*}
\epsilon ^{\prime }F\circ L^{\prime }\xi =\epsilon ^{\prime }F\circ
L^{\prime }R^{\prime }F\epsilon \circ L^{\prime }\eta ^{\prime }R=F\epsilon
\circ \epsilon ^{\prime }L^{\prime }R\circ L^{\prime }\eta ^{\prime
}R=F\epsilon
\end{equation*}%
so that (\ref{form:efeps}) holds true.
It is easy to check that $\xi L:RL\rightarrow R^{\prime }FL=R^{\prime
}L^{\prime }$ is a morphism of monads.

\begin{invisible}
In fact%
\begin{eqnarray*}
\xi L\circ R\epsilon L &=&R^{\prime }F\epsilon L\circ \eta ^{\prime }RL\circ
R\epsilon L \\
&=&R^{\prime }F\epsilon L\circ R^{\prime }L^{\prime }R\epsilon L\circ \eta
^{\prime }RLRL \\
&=&R^{\prime }F\epsilon L\circ R^{\prime }FLR\epsilon L\circ \eta ^{\prime
}RLRL \\
&=&R^{\prime }F\epsilon L\circ R^{\prime }F\epsilon LRL\circ \eta ^{\prime
}RLRL \\
&=&R^{\prime }F\epsilon L\circ R^{\prime }\epsilon ^{\prime }L^{\prime
}RL\circ R^{\prime }L^{\prime }\eta ^{\prime }RL\circ R^{\prime }F\epsilon
LRL\circ \eta ^{\prime }RLRL \\
&=&R^{\prime }F\epsilon L\circ R^{\prime }\epsilon ^{\prime }FLRL\circ
R^{\prime }L^{\prime }\eta ^{\prime }RL\circ R^{\prime }F\epsilon LRL\circ
\eta ^{\prime }RLRL \\
&=&R^{\prime }\epsilon ^{\prime }L^{\prime }\circ R^{\prime }L^{\prime
}R^{\prime }F\epsilon L\circ R^{\prime }L^{\prime }\eta ^{\prime }RL\circ
R^{\prime }F\epsilon LRL\circ \eta ^{\prime }RLRL \\
&=&R^{\prime }\epsilon ^{\prime }L^{\prime }\circ R^{\prime }L^{\prime }\xi
L\circ R^{\prime }F\epsilon LRL\circ \eta ^{\prime }RLRL \\
&=&R^{\prime }\epsilon ^{\prime }L^{\prime }\circ R^{\prime }L^{\prime }\xi
L\circ \xi LRL=R^{\prime }\epsilon ^{\prime }L^{\prime }\circ \xi L\xi L
\end{eqnarray*}%
and
\begin{eqnarray*}
\xi L\circ \eta &=&R^{\prime }F\epsilon L\circ \eta ^{\prime }RL\circ \eta
=R^{\prime }F\epsilon L\circ R^{\prime }L^{\prime }\eta \circ \eta ^{\prime }
\\
&=&R^{\prime }F\epsilon L\circ R^{\prime }FL\eta \circ \eta ^{\prime }=\eta
^{\prime }
\end{eqnarray*}%
so that $\xi L:RL\rightarrow R^{\prime }FL=R^{\prime }L^{\prime }$ is a
morphism of monads.
\end{invisible}

Since $\xi L:RL\rightarrow R^{\prime }L^{\prime }$ is a morphism of monads
and $R^{\prime }L^{\prime }$ is augmented via $\gamma ^{\prime }:R^{\prime
}L^{\prime }\rightarrow \mathrm{Id}_{\mathcal{B}},$ we get that $\gamma
:=\gamma ^{\prime }\circ \xi L:RL\rightarrow \mathrm{Id}_{\mathcal{B}}$ is
an augmentation for $RL.$

\begin{invisible}
In fact%
\begin{eqnarray*}
\gamma \circ R\epsilon L &=&\gamma ^{\prime }\circ \xi L\circ R\epsilon
L=\gamma ^{\prime }\circ R^{\prime }\epsilon ^{\prime }L^{\prime }\circ \xi
L\xi L=\gamma ^{\prime }\gamma ^{\prime }\circ \xi L\xi L=\gamma \gamma , \\
\gamma \circ \eta &=&\gamma ^{\prime }\circ \xi L\circ \eta =\gamma ^{\prime
}\circ \eta ^{\prime }=\mathrm{Id}.
\end{eqnarray*}
\end{invisible}

We set $S_{\left[ n\right] }:=L_{\left[ n\right] }\Gamma _{\left[ n\right] }$
and we define iteratively $\Gamma _{\left[ n\right] },\gamma _{\left[ n%
\right] }^{\prime }:R^{\prime }FS_{\left[ n\right] }\rightarrow \mathrm{Id}_{\mathcal{B}}$
and
\begin{equation*}
\gamma _{\left[ n\right] }:=\gamma _{\left[ n\right] }^{\prime }\circ \xi S_{%
\left[ n\right] }:RS_{\left[ n\right] }\rightarrow \mathrm{Id}_{\mathcal{B}}
\end{equation*}%
such that
\begin{equation}\label{form:thmain1}
U_{[n]}\circ\Gamma_{[n]}=\id_{\mathcal{B}},\quad\gamma _{\left[ n\right] }\circ U_{\left[ n\right] }\eta _{\left[ n\right]
}\Gamma _{\left[ n\right] }=\mathrm{Id}_{\id_{\mathcal{B}}}\text{\qquad and\qquad }\gamma _{%
\left[ n\right] }^{\prime }\circ R^{\prime }F\pi _{\left[ n\right] }\Gamma _{%
\left[ n\right] }=\gamma ^{\prime }
\end{equation}%
as follows.

For $n=0,$ we set $\Gamma _{\left[ 0\right] }:=\mathrm{Id}_{\mathcal{B}%
},\gamma _{\left[ 0\right] }^{\prime }:=\gamma ^{\prime },\gamma _{\left[ 0%
\right] }:=\gamma $ as required.

Let $n\geq 0$. Suppose that $\Gamma _{\left[ n\right] },\gamma _{\left[ n%
\right] }^{\prime }$ such that \eqref{form:thmain1} hold are given and let us construct $%
\Gamma _{\left[ n+1\right] },\gamma _{\left[ n+1\right] }^{\prime }$ such
that $\gamma _{\left[ n+1\right] }\circ U_{\left[ n+1\right] }\eta _{\left[
n+1\right] }\Gamma _{\left[ n+1\right] }=\mathrm{Id}$ and $\gamma _{\left[
n+1\right] }^{\prime }\circ R^{\prime }F\pi _{\left[ n+1\right] }\Gamma _{%
\left[ n+1\right] }=\gamma ^{\prime }.$

Since $\mathcal{B}_{\left[ n+1\right] }=\left\langle RL_{\left[ n\right]
}|U_{\left[ n\right] }\right\rangle $ we can apply Lemma \ref{lem:lift},
taking $Q=\Gamma _{\left[ n\right] }$ and $q=\gamma _{\left[ n\right] }$, to
construct a unique functor $\Gamma _{\left[ n+1\right] }=\widetilde{\Gamma _{%
\left[ n\right] }}:\mathcal{B}\rightarrow \mathcal{B}_{\left[ n+1\right] }$
such that $U_{\left[ n,n+1\right] }\circ \Gamma _{\left[ n+1\right] }=\Gamma
_{\left[ n\right] }$ and $\psi \Gamma _{\left[ n+1\right] }=\gamma _{\left[ n%
\right] }.$ Explicitly $\Gamma _{\left[ n+1\right] }B=\left( \Gamma _{\left[
n\right] }B,\gamma _{\left[ n\right] }B\right) $ as desired.

For $B\in \mathcal{B}$ consider the coequalizer (\ref{coeq:Ln+1}) taking $B_{%
\left[ n+1\right] }:=\Gamma _{\left[ n+1\right] }B=\left( \Gamma _{\left[ n%
\right] }B,\gamma _{[n]}B\right) :$%
\begin{equation}  \label{coeq:main1}
\xymatrixcolsep{3cm} \xymatrix{LRS_{[n]}B\ar@<.5ex>[r]^{\pi_{[n]}%
\Gamma_{[n]}B\circ L\gamma_{[n]}B}\ar@<-.5ex>[r]_{\epsilon
S_{[n]}B}&S_{[n]}B\ar[r]^-{\pi _{[n,n+1]}\Gamma_{[n+1]}B}& S_{[n+1]}B}
\end{equation}

\begin{invisible}
\begin{equation*}
LRS_{\left[ n\right] }B\overset{\pi _{\left[ n\right] }\Gamma _{\left[ n%
\right] }B\circ L\gamma _{\left[ n\right] }B}{\underset{\epsilon S_{\left[ n%
\right] }B}{\rightrightarrows }}S_{\left[ n\right] }B\overset{\pi _{\left[
n,n+1\right] }\Gamma _{\left[ n+1\right] }B}{\longrightarrow }S_{\left[ n+1%
\right] }B.
\end{equation*}
\end{invisible}

Set $e_{\left[ n\right] }:=U_{\left[ n\right] }\eta _{\left[ n\right]
}\Gamma _{\left[ n\right] }\circ \gamma _{\left[ n\right] }.$ Then $e_{\left[
n\right] }$ is an idempotent natural transformation. Moreover, since $%
\mathbb{T}_{\left[ n\right] }$ is an adjoint triangle, we have
$\epsilon =\epsilon _{\left[ n\right] }\circ \pi _{\left[ n\right] }R_{\left[
n\right] }$
so that
\begin{eqnarray*}
\epsilon S_{\left[ n\right] }\circ Le_{\left[ n\right] } &=&\epsilon _{\left[
n\right] }S_{\left[ n\right] }\circ \pi _{\left[ n\right] }R_{\left[ n\right]
}S_{\left[ n\right] }\circ LU_{\left[ n\right] }\eta _{\left[ n\right]
}\Gamma _{\left[ n\right] }\circ L\gamma _{\left[ n\right] } \\
&=&\epsilon _{\left[ n\right] }S_{\left[ n\right] }\circ L_{\left[ n\right]
}\eta _{\left[ n\right] }\Gamma _{\left[ n\right] }\circ \pi _{\left[ n%
\right] }\Gamma _{\left[ n\right] }\circ L\gamma _{\left[ n\right] }=\pi _{%
\left[ n\right] }\Gamma _{\left[ n\right] }\circ L\gamma _{\left[ n\right] }
\end{eqnarray*}%
and hence $\left( \pi _{\left[ n\right] }\Gamma _{\left[ n\right] }B\circ
L\gamma _{\left[ n\right] }B,\epsilon S_{\left[ n\right] }B\right) =\left(
\epsilon S_{\left[ n\right] }B\circ Le_{\left[ n\right] }B,\epsilon S_{\left[
n\right] }B\right) .$ Moreover
\begin{equation*}
\epsilon S_{\left[ n\right] }B=\epsilon _{\left[ n\right] }S_{\left[ n\right]
}B\circ \pi _{\left[ n\right] }R_{\left[ n\right] }S_{\left[ n\right] }B
\end{equation*}%
is a composition of regular epimorphisms as $\epsilon _{\left[ n\right] }S_{%
\left[ n\right] }B$ is the coequalizer of the parallel pair of morphisms $%
\left( L_{\left[ n\right] }R_{\left[ n\right] }\epsilon _{\left[ n\right]
}S_{\left[ n\right] }B,\epsilon _{\left[ n\right] }L_{\left[ n\right] }R_{%
\left[ n\right] }S_{\left[ n\right] }B\right) $ (a split coequalizer, as $S_{[n]}=L_{[n]}\Gamma_{[n]}$) and
\begin{equation*}
\pi _{\left[ n\right] }R_{\left[ n\right] }S_{\left[ n\right] }B=\pi _{\left[
0,n\right] }R_{\left[ n\right] }S_{\left[ n\right] }B=\pi _{\left[ n-1,n%
\right] }R_{\left[ n\right] }S_{\left[ n\right] }B\circ \pi _{\left[ n-2,n-1%
\right] }R_{\left[ n-1\right] }S_{\left[ n\right] }B\circ \cdots \circ \pi _{%
\left[ 0,1\right] }R_{\left[ 1\right] }S_{\left[ n\right] }B
\end{equation*}

Since, by hypothesis, $F$ preserves coequalizers, we get that \thinspace $%
\left( FS_{\left[ n+1\right] }B,F\pi _{\left[ n,n+1\right] }\Gamma _{\left[
n+1\right] }B\right) $ is the coequalizer of $\left( F\left( \pi _{\left[ n%
\right] }\Gamma _{\left[ n\right] }B\circ L\gamma _{\left[ n\right]
}B\right) ,F\epsilon S_{\left[ n\right] }B\right) =\left( F\epsilon S_{\left[
n\right] }B\circ FLe_{\left[ n\right] }B,F\epsilon S_{\left[ n\right]
}B\right) $ where $FLe_{\left[ n\right] }B$ is still idempotent and $%
F\epsilon S_{\left[ n\right] }B$ is still a composition of regular epimorphisms.

By the hypothesis, the latter coequalizer is preserved by $R^{\prime }.$ Thus
we get the coequalizer%
\begin{equation*}
\xymatrixcolsep{3.5cm} \xymatrix{R'FLRS_{[n]}B\ar@<.5ex>[r]^{R'F\pi_{[n]}%
\Gamma_{[n]}B\circ R'FL\gamma_{[n]}B}\ar@<-.5ex>[r]_{R'F\epsilon
S_{[n]}B}&R'FS_{[n]}B\ar[r]^-{R'F\pi _{[n,n+1]}\Gamma_{[n+1]}B}&
R'FS_{[n+1]}B}
\end{equation*}

\begin{invisible}
\begin{equation*}
R^{\prime }FLRS_{\left[ n\right] }B\overset{R^{\prime }F\pi _{\left[ n\right]
}\Gamma _{\left[ n\right] }B\circ R^{\prime }FL\gamma _{\left[ n\right] }B}{%
\underset{R^{\prime }F\epsilon S_{\left[ n\right] }B}{\rightrightarrows }}%
R^{\prime }FS_{\left[ n\right] }B\overset{R^{\prime }F\pi _{\left[ n,n+1%
\right] }\Gamma _{\left[ n+1\right] }B}{\longrightarrow }R^{\prime }FS_{%
\left[ n+1\right] }B.
\end{equation*}
\end{invisible}

Let us check that $\gamma _{\left[ n\right] }^{\prime }B:R^{\prime }FS_{%
\left[ n\right] }B\rightarrow B$ together with the parallel pair above is a fork i.e.
\begin{equation} \label{form:thmain2}
\gamma _{\left[ n\right] }^{\prime }\circ R^{\prime }F\pi _{\left[ n\right]
}\Gamma _{\left[ n\right] }\circ R^{\prime }FL\gamma _{\left[ n\right]
}=\gamma _{\left[ n\right] }^{\prime }\circ R^{\prime }F\epsilon S_{\left[ n%
\right] }.
\end{equation}%
To this aim we first compute
\begin{eqnarray*}
&&\gamma _{\left[ n\right] }^{\prime }\circ R^{\prime }F\pi _{\left[ n\right]
}\Gamma _{\left[ n\right] }\circ R^{\prime }FL\gamma _{\left[ n\right]
}^{\prime }\circ R^{\prime }FLR^{\prime }F\pi _{\left[ n\right] }\Gamma _{%
\left[ n\right] } \\
&\overset{\eqref{form:thmain1}}{=}&\gamma ^{\prime }\circ R^{\prime }FL\gamma ^{\prime }=\gamma ^{\prime
}\gamma ^{\prime } \\
&=&\gamma ^{\prime }\circ R^{\prime }\epsilon ^{\prime }L^{\prime } \\
&\overset{\eqref{form:thmain1}}{=}&\gamma _{\left[ n\right] }^{\prime }\circ R^{\prime }F\pi _{\left[ n%
\right] }\Gamma _{\left[ n\right] }\circ R^{\prime }\epsilon ^{\prime
}L^{\prime } \\
&=&\gamma _{\left[ n\right] }^{\prime }\circ R^{\prime }F\pi _{\left[ n%
\right] }\Gamma _{\left[ n\right] }\circ R^{\prime }\epsilon ^{\prime }FL \\
&\overset{\text{nat. }\epsilon ^{\prime }}{=}&\gamma _{\left[ n\right]
}^{\prime }\circ R^{\prime }\epsilon ^{\prime }FS_{\left[ n\right] }\circ
R^{\prime }L^{\prime }R^{\prime }F\pi _{\left[ n\right] }\Gamma _{\left[ n%
\right] } \\
&=&\gamma _{\left[ n\right] }^{\prime }\circ R^{\prime }\epsilon ^{\prime
}FS_{\left[ n\right] }\circ R^{\prime }FLR^{\prime }F\pi _{\left[ n\right]
}\Gamma _{\left[ n\right] }
\end{eqnarray*}

Since $R^{\prime },F$ and $L$ preserve regular epimorphisms and $\pi _{\left[
n\right] }\Gamma _{\left[ n\right] }$ is a regular epimorphism, we get that $%
R^{\prime }FLR^{\prime }F\pi _{\left[ n\right] }\Gamma _{\left[ n\right] }$
is a regular epimorphism and hence%
\begin{equation*}
\gamma _{\left[ n\right] }^{\prime }\circ R^{\prime }F\pi _{\left[ n\right]
}\Gamma _{\left[ n\right] }\circ R^{\prime }FL\gamma _{\left[ n\right]
}^{\prime }=\gamma _{\left[ n\right] }^{\prime }\circ R^{\prime }\epsilon
^{\prime }FS_{\left[ n\right] }.
\end{equation*}%
Coming back to the equality \eqref{form:thmain2}, we compute%
\begin{eqnarray*}
\gamma _{\left[ n\right] }^{\prime }\circ R^{\prime }F\pi _{\left[ n\right]
}\Gamma _{\left[ n\right] }\circ R^{\prime }FL\gamma _{\left[ n\right] }
&=&\gamma _{\left[ n\right] }^{\prime }\circ R^{\prime }F\pi _{\left[ n%
\right] }\Gamma _{\left[ n\right] }\circ R^{\prime }FL\gamma _{\left[ n%
\right] }^{\prime }\circ R^{\prime }FL\xi S_{\left[ n\right] } \\
&=&\gamma _{\left[ n\right] }^{\prime }\circ R^{\prime }\epsilon ^{\prime
}FS_{\left[ n\right] }\circ R^{\prime }FL\xi S_{\left[ n\right] } \\
&=&\gamma _{\left[ n\right] }^{\prime }\circ R^{\prime }\epsilon ^{\prime
}FS_{\left[ n\right] }\circ R^{\prime }L^{\prime }\xi S_{\left[ n\right] }
\overset{(\ref{form:efeps})}{=}\gamma _{\left[ n\right] }^{\prime }\circ
R^{\prime }F\epsilon S_{\left[ n\right] }
\end{eqnarray*}%
Thus $\gamma _{\left[ n\right] }^{\prime }$ together with the parallel pair $(R^{\prime }F\pi _{\left[ n\right]
}\Gamma _{\left[ n\right] }\circ R^{\prime }FL\gamma _{\left[ n\right] },R^{\prime }F\epsilon S_{\left[ n\right] } )$ is a fork and hence, the
universal property of the above coequalizer yields a unique natural
transformation $\gamma _{\left[ n+1\right] }^{\prime }:R^{\prime }FS_{\left[
n+1\right] }\rightarrow \mathrm{Id}$ such that
\begin{equation*}
\gamma _{\left[ n+1\right] }^{\prime }\circ R^{\prime }F\pi _{\left[ n,n+1%
\right] }\Gamma _{\left[ n+1\right] }=\gamma _{\left[ n\right] }^{\prime }
\end{equation*}

We compute%
\begin{eqnarray*}
\gamma _{\left[ n+1\right] }^{\prime }\circ R^{\prime }F\pi _{\left[ n+1%
\right] }\Gamma _{\left[ n+1\right] } &=&\gamma _{\left[ n+1\right]
}^{\prime }\circ R^{\prime }F\left( \pi _{\left[ n,n+1\right] }\circ \pi _{%
\left[ n\right] }U_{\left[ n,n+1\right] }\right) \Gamma _{\left[ n+1\right] }
\\
&=&\gamma _{\left[ n+1\right] }^{\prime }\circ R^{\prime }F\pi _{\left[ n,n+1%
\right] }\Gamma _{\left[ n+1\right] }\circ R^{\prime }F\pi _{\left[ n\right]
}U_{\left[ n,n+1\right] }\Gamma _{\left[ n+1\right] } \\
&=&\gamma _{\left[ n\right] }^{\prime }\circ R^{\prime }F\pi _{\left[ n%
\right] }\Gamma _{\left[ n\right] }\overset{\eqref{form:thmain1}}{=}\gamma ^{\prime }.
\end{eqnarray*}%
We also have%
\begin{eqnarray*}
\gamma _{\left[ n+1\right] }\circ U_{\left[ n+1\right] }\eta _{\left[ n+1%
\right] }\Gamma _{\left[ n+1\right] } &=&\gamma _{\left[ n+1\right]
}^{\prime }\circ \xi S_{\left[ n+1\right] }\circ U_{\left[ n+1\right] }\eta
_{\left[ n+1\right] }\Gamma _{\left[ n+1\right] } \\
&\overset{(\ref{form:epseta[n+1]})}{=}&\gamma _{\left[ n+1\right] }^{\prime
}\circ \xi S_{\left[ n+1\right] }\circ U_{\left[ n\right] }\left( R_{\left[ n%
\right] }\pi _{\left[ n,n+1\right] }\circ \eta _{\left[ n\right] }U_{\left[
n,n+1\right] }\right) \Gamma _{\left[ n+1\right] } \\
&=&\gamma _{\left[ n+1\right] }^{\prime }\circ \xi S_{\left[ n+1\right]
}\circ U_{\left[ n\right] }R_{\left[ n\right] }\pi _{\left[ n,n+1\right]
}\Gamma _{\left[ n+1\right] }\circ U_{\left[ n\right] }\eta _{\left[ n\right]
}U_{\left[ n,n+1\right] }\Gamma _{\left[ n+1\right] } \\
&=&\gamma _{\left[ n+1\right] }^{\prime }\circ \xi S_{\left[ n+1\right]
}\circ R\pi _{\left[ n,n+1\right] }\Gamma _{\left[ n+1\right] }\circ U_{%
\left[ n\right] }\eta _{\left[ n\right] }\Gamma _{\left[ n\right] } \\
&=&\gamma _{\left[ n+1\right] }^{\prime }\circ R^{\prime }F\pi _{\left[ n,n+1%
\right] }\Gamma _{\left[ n+1\right] }\circ \xi S_{\left[ n\right] }\circ U_{%
\left[ n\right] }\eta _{\left[ n\right] }\Gamma _{\left[ n\right] } \\
&=&\gamma _{\left[ n\right] }^{\prime }\circ \xi S_{\left[ n\right] }\circ
U_{\left[ n\right] }\eta _{\left[ n\right] }\Gamma _{\left[ n\right]
}=\gamma _{\left[ n\right] }\circ U_{\left[ n\right] }\eta _{\left[ n\right]
}\Gamma _{\left[ n\right] }\overset{\eqref{form:thmain1}}{=}\mathrm{Id.}
\end{eqnarray*}%
Finally
\begin{eqnarray*}
\gamma _{\left[ n+1\right] }\circ R\pi _{\left[ n,n+1\right] }\Gamma _{\left[
n+1\right] } &=&\gamma _{\left[ n+1\right] }^{\prime }\circ \xi S_{\left[ n+1%
\right] }\circ R\pi _{\left[ n,n+1\right] }\Gamma _{\left[ n+1\right] } \\
&=&\gamma _{\left[ n+1\right] }^{\prime }\circ R^{\prime }F\pi _{\left[ n,n+1%
\right] }\Gamma _{\left[ n+1\right] }\circ \xi S_{\left[ n\right] }=\gamma _{%
\left[ n\right] }^{\prime }\circ \xi S_{\left[ n\right] }=\gamma _{\left[ n%
\right] }.
\end{eqnarray*}
\end{proof}

\begin{proposition}
\label{pro:Gamman}The functor $\Gamma _{\left[ n\right] }:\mathcal{B}%
\rightarrow \mathcal{B}_{\left[ n\right] }$ induces a functor $\Gamma _{n}:%
\mathcal{B}\rightarrow \mathcal{B}_{n}$ such that $\Lambda _{n}\circ \Gamma
_{n}=\Gamma _{\left[ n\right] }$ and $U_{n,n+1}\circ \Gamma _{n+1}=\Gamma
_{n}.$ Moreover there is $\gamma _{n}:R_{n}L_{n}\Gamma _{n}\rightarrow
\Gamma _{n}$ such that $\Gamma _{n+1}B=\left( \Gamma _{n}B,\gamma
_{n}B\right) ,$ for all $B\in \mathcal{B}$, and $U_{n}\gamma _{n}\circ
R\lambda _{n}\Gamma _{n}=\gamma _{\left[ n\right] }$. Note that $L_{\left[ n%
\right] }\Gamma _{\left[ n\right] }=L_{\left[ n\right] }\Lambda _{n}\Gamma
_{n}\overset{\lambda _{n}\Gamma _{n}}{\rightarrow }L_{n}\Gamma _{n}$ is
invertible.
\end{proposition}

\begin{proof}
The condition $\gamma _{\left[ n\right] }\circ U_{\left[ n\right] }\eta _{%
\left[ n\right] }\Gamma _{\left[ n\right] }=\mathrm{Id}$, given in Theorem \ref{thm:main}, implies that $\eta
_{\left[ n\right] }\Gamma _{\left[ n\right] }B\in \mathsf{M}\left( U_{\left[
n\right] }\right) $ for every $B\in\mathcal{B}.$ By Theorem \ref{teo:main}, we have that $\Gamma _{%
\left[ n\right] }B\in \mathrm{Im}\left( \Lambda _{n}\right) =\mathrm{Im}%
^{\prime }\left( \Lambda _{n}\right) .$ Thus $\mathrm{Im}^{\prime }\left(
\Gamma _{\left[ n\right] }\right) \subseteq \mathrm{Im}^{\prime }\left(
\Lambda _{n}\right) .$ Since $\Lambda _{n}$ is fully faithful and injective
on objects, by \cite[Lemma 1.12]{AM-MM}, there is a functor $\Gamma _{n}:%
\mathcal{B}\rightarrow \mathcal{B}_{n}$ such that $\Lambda _{n}\circ \Gamma
_{n}=\Gamma _{\left[ n\right] }.$
\begin{invisible}
From $\Gamma _{\left[ n\right] }B\in \mathrm{Im}\left( \Lambda _{n}\right) =%
\mathrm{\mathrm{Eim}}\left( \Lambda _{n}\right) $ we get that there is $%
B_{n}\in \mathcal{B}_{n}$ such that $\Gamma _{\left[ n\right] }B=\Lambda
_{n}B_{n}.$ Since $\Lambda _{n}$ i s injective on objects the object $B_{n}$
is unique and we can denote it by $\Gamma _{n}B.$ Similarly, since $\Lambda
_{n}$ is full and faithful, for every morphism $f$ there is a unique
morphism, denoted by $\Gamma _{n}f$ such that $\Gamma _{\left[ n\right]
}f=\Lambda _{n}\Gamma _{n}f.$ Moreover%
\begin{eqnarray*}
\Lambda _{n}\Gamma _{n}\left( f\circ g\right) &=&\Gamma _{\left[ n\right]
}\left( f\circ g\right) =\Gamma _{\left[ n\right] }\left( f\right) \circ
\Gamma _{\left[ n\right] }\left( g\right) =\Lambda _{n}\Gamma _{n}\left(
f\right) \circ \Lambda _{n}\Gamma _{n}\left( g\right) =\Lambda _{n}\left(
\Gamma _{n}\left( f\right) \circ \Gamma _{n}\left( g\right) \right) , \\
\Lambda _{n}\Gamma _{n}\left( \mathrm{Id}\right) &=&\Gamma _{\left[ n\right]
}\left( \mathrm{Id}_{\Gamma _{n}}\right) =\mathrm{Id}_{\Lambda _{n}\Gamma
_{n}}=\Lambda _{n}\left( \mathrm{Id}_{\Gamma _{n}}\right)
\end{eqnarray*}%
and since $\Lambda _{n}$ is faithful we obtain that $\Gamma _{n}$ is a
functor. By construction $\Lambda _{n}\circ \Gamma _{n}=\Gamma _{\left[ n%
\right] }.$
\end{invisible}
We compute%
\begin{equation*}
\Lambda _{n}\circ U_{n,n+1}\circ \Gamma _{n+1}=U_{\left[ n,n+1\right] }\circ
\Lambda _{n+1}\circ \Gamma _{n+1}=U_{\left[ n,n+1\right] }\circ \Gamma _{%
\left[ n+1\right] }=\Gamma _{\left[ n\right] }=\Lambda _{n}\circ \Gamma _{n}.
\end{equation*}%
Since $\Lambda _{n}$ is faithful and injective on objects, we get
$
U_{n,n+1}\circ \Gamma _{n+1}=\Gamma _{n}.
$
Moreover since $\Gamma _{n+1}B\in \mathcal{B}_{n+1}$ and $U_{n,n+1}\Gamma
_{n+1}B=\Gamma _{n}B$, there is $\gamma _{n}B:R_{n}L_{n}\Gamma
_{n}B\rightarrow \Gamma _{n}B$ such that $\Gamma _{n+1}B=\left( \Gamma
_{n}B,\gamma _{n}B\right) .$ From $\Lambda _{n+1}\circ \Gamma _{n+1}=\Gamma
_{\left[ n+1\right] },$ we get
$$(\Gamma_{[n]}B,\gamma_{[n]})=\Gamma_{[n+1]}B=\Lambda _{n+1} \Gamma _{n+1}B=\Lambda _{n+1}\left( \Gamma
_{n}B,\gamma _{n}B\right) =\left( \Lambda _{n}\Gamma
_{n}B,U_{n}\gamma _{n}B\circ R\lambda
_{n}\Gamma _{n}B\right)$$
and hence  $U_{n}\gamma _{n}\circ R\lambda
_{n}\Gamma _{n}=\gamma _{\left[ n\right] }$.
The last part follows by Remark \ref{rem:adj}.
\end{proof}

\begin{invisible}
\begin{equation*}
\begin{array}{ccccc}
RLRL_{\left[ n\right] }\Gamma _{\left[ n\right] } & \overset{R\pi _{\left[ n%
\right] }\Gamma _{\left[ n\right] }\circ RL\gamma _{\left[ n\right] }}{%
\underset{R\epsilon L_{\left[ n\right] }\Gamma _{\left[ n\right] }}{%
\rightrightarrows }} & RL_{\left[ n\right] }\Gamma _{\left[ n\right] } &
\overset{\gamma _{\left[ n\right] }}{\longrightarrow } & \mathrm{\mathrm{Id}}
\\
RLR\lambda _{n}\Gamma _{n}\downarrow &  & \downarrow R\lambda _{n}\Gamma _{n}
&  &  \\
RLRL_{n}\Gamma _{n} & \overset{R\pi _{n}\Gamma _{n}\circ RLU_{n}\gamma _{n}}{%
\underset{R\epsilon L_{n}\Gamma _{n}}{\rightrightarrows }} & RL_{n}\Gamma
_{n} & \overset{U_{n}\gamma _{n}}{\longrightarrow } & \mathrm{\mathrm{Id}}%
\end{array}%
\end{equation*}
\end{invisible}

\begin{lemma}
\label{lem:rank}In the setting of Theorem \ref{thm:main}, define $S_{\left[ n%
\right] }:=L_{\left[ n\right] }\Gamma _{\left[ n\right] }:\mathcal{B}%
\rightarrow \mathcal{A}.$ Given $B\in \mathcal{B},$ a morphism $f:S_{\left[ n%
\right] }B\rightarrow A$ in $\mathcal{A}$ together with the pair $\left( \pi _{%
\left[ n\right] }\Gamma _{\left[ n\right] }B\circ L\gamma _{\left[ n\right]
}B,\epsilon S_{\left[ n\right] }B\right) $ is a fork if and only if $Rf$ together with
the pair $\left( e_{\left[ n\right] }B,\mathrm{Id}_{RS_{\left[ n\right]
}B}\right) $ is a fork, where $e_{\left[ n\right] }:=U_{\left[ n\right] }\eta _{\left[
n\right] }\Gamma _{\left[ n\right] }\circ \gamma _{\left[ n\right] }.$

As a consequence, $\pi _{\left[ n,n+1\right] }\Gamma _{\left[
n+1\right] }B:S_{\left[ n\right] }B\rightarrow S_{\left[ n+1\right] }B$ is
invertible if and only if either $\gamma _{\left[ n\right] }B$ or $\eta _{%
\left[ n\right] }\Gamma _{\left[ n\right] }B$ is invertible. If $%
\pi _{\left[ n,n+1\right] }\Gamma _{\left[ n+1\right] }B$ is invertible so
is $\pi _{\left[ m,m+1\right] }\Gamma _{\left[ m+1\right] }B$
for all $m\geq n.$
\end{lemma}

\begin{proof}
In the proof of Theorem \ref{thm:main}, we have seen that $\pi _{\left[ n%
\right] }\Gamma _{\left[ n\right] }\circ L\gamma _{\left[ n\right]
}=\epsilon L_{\left[ n\right] }\Gamma _{\left[ n\right] }\circ Le_{\left[ n%
\right] }=\epsilon S_{\left[ n\right] }\circ Le_{\left[ n\right] }$ where $%
e_{\left[ n\right] }:=U_{\left[ n\right] }\eta _{\left[ n\right] }\Gamma _{%
\left[ n\right] }\circ \gamma _{\left[ n\right] }.$ As a consequence, $f:S_{%
\left[ n\right] }B\rightarrow A$ together with the pair $\left( \pi _{\left[ n%
\right] }\Gamma _{\left[ n\right] }B\circ L\gamma _{\left[ n\right]
}B,\epsilon S_{\left[ n\right] }B\right) $ is a fork if and only if $f\circ \epsilon
S_{\left[ n\right] }B\circ Le_{\left[ n\right] }B=f\circ \epsilon S_{\left[ n%
\right] }B$ if and only if $Rf\circ R\epsilon S_{\left[ n\right] }B\circ
RLe_{\left[ n\right] }B\circ \eta RS_{\left[ n\right] }B=Rf\circ R\epsilon
S_{\left[ n\right] }B\circ \eta RS_{\left[ n\right] }B$ if and only if $%
Rf\circ R\epsilon S_{\left[ n\right] }B\circ \eta RS_{\left[ n\right]
}B\circ e_{\left[ n\right] }B=Rf$ if and only if $Rf\circ e_{\left[ n\right]
}B=Rf$ if and only if $Rf$ together with the pair $\left( e_{\left[ n\right]
}B,\mathrm{Id}_{RS_{\left[ n\right] }B}\right)$ is a fork. Since $\pi _{\left[ n,n+1%
\right] }\Gamma _{\left[ n+1\right] }B$ is the coequalizer of the pair $%
\left( \pi _{\left[ n\right] }\Gamma _{\left[ n\right] }B\circ L\gamma _{%
\left[ n\right] }B,\epsilon S_{\left[ n\right] }B\right)$, we get that $\pi
_{\left[ n,n+1\right] }\Gamma _{\left[ n+1\right] }B$ is invertible if and
only if $\pi _{\left[ n\right] }\Gamma _{\left[ n\right] }B\circ L\gamma _{%
\left[ n\right] }B=\epsilon S_{\left[ n\right] }B$ if and only if $f=\mathrm{%
Id}$ together with the pair $\left( \pi _{\left[ n\right] }\Gamma _{\left[ n%
\right] }B\circ L\gamma _{\left[ n\right] }B,\epsilon S_{\left[ n\right]
}B\right) $ is a fork. By the foregoing this is equivalent to $Rf\circ e_{\left[ n%
\right] }B=Rf$ that is $e_{\left[ n\right] }B=\mathrm{Id}$ i.e. $U_{\left[
0,n\right] }\eta _{\left[ n\right] }\Gamma _{\left[ n\right] }B\circ \gamma
_{\left[ n\right] }B=\mathrm{Id.}$

Since $\gamma _{\left[ n\right] }B\circ U_{\left[ n\right] }\eta _{\left[ n%
\right] }\Gamma _{\left[ n\right] }B=\mathrm{Id},$ we conclude that $\pi _{%
\left[ n,n+1\right] }\Gamma _{\left[ n+1\right] }B$ is invertible if and
only if either $\gamma _{\left[ n\right] }B$ or $U_{\left[ n\right] }\eta _{%
\left[ n\right] }\Gamma _{\left[ n\right] }B$ is invertible. Since $U_{\left[
n\right] }$ reflects isomorphism, we have that $U_{\left[ n\right] }\eta _{%
\left[ n\right] }\Gamma _{\left[ n\right] }B$ is invertible if and only if $%
\eta _{\left[ n\right] }\Gamma _{\left[ n\right] }B$ is invertible.

If $\pi _{\left[ n,n+1\right] }\Gamma _{\left[ n+1\right] }B$ is invertible,
then $\gamma _{\left[ n\right] }B$ is invertible. Since $\gamma _{\left[ n+1%
\right] }B\circ R\pi _{\left[ n,n+1\right] }\Gamma _{\left[ n+1\right]
}B=\gamma _{\left[ n\right] }B,$ we obtain that $\gamma _{\left[ n+1\right]
}B$ is invertible. As a consequence $\pi _{\left[ n+1,n+2\right] }\Gamma _{%
\left[ n+2\right] }B$ is invertible. Going on this way, we obtain that $\pi
_{\left[ m,m+1\right] }\Gamma _{\left[ m+1\right] }B$ is invertible for all $%
m\geq n.$
\end{proof}

\section{Example on monoidal categories\label{sec:6}}

Given a category $\mathcal{A}$ and an object $X\in \mathcal{A}$ we denote by
$\mathcal{A}/ X$ the correspondent slice category consisting of pairs $%
\left( A,tA:A\rightarrow X\right) $ and where a morphism $f:\left(
A,tA\right) \rightarrow \left( B,tB\right) $ is a morphism $f:A\rightarrow B$
such that $tB\circ f=tA.$

Let $\mathcal{B}$ be a category with pullbacks. It is known that any
adjunction $\left( L,R\right) $ with unit $\eta $ and counit $\epsilon $ and
an object $\mathbf{1}\in \mathcal{B}$ induces an adjunction $\left( L/%
\mathbf{1},R/\mathbf{1}\right) $ as in the following left-hand side diagram
where $U_{\mathcal{A}}$ and $U_{\mathcal{B}}$ are the obvious forgetful
functors and $U_{\mathcal{A}}\circ L/\mathbf{1}=L\circ U_{\mathcal{B}}.$
\begin{equation*}
\xymatrixcolsep{1.5cm}\xymatrixrowsep{0.7cm}\xymatrix{\mathcal{A}/
L\mathbf{1}\ar[r]^-{U_{\mathcal{A}}}\ar@<.5ex>@{.>}[d]^{R/
\mathbf{1}}&\mathcal{A}\ar@<.5ex>@{.>}[d]^{R}\\ \mathcal{B}/ \mathbf{1}
\ar[r]^{U_{\mathcal{B}}}\ar@<.5ex>[u]^{L/
\mathbf{1}}&\mathcal{B}\ar@<.5ex>[u]^{L}}\qquad \xymatrixcolsep{1.5cm}%
\xymatrixrowsep{0.7cm}\xymatrix{KA\pulb\ar[r]^-{tKA}\ar@<.5ex>[d]_{kA}&%
\mathbf{1}\ar@<.5ex>[d]^{\eta\mathbf{1}}\\ R A\ar[r]^{RtA}&RL\mathbf{1}}
\end{equation*}

\begin{invisible}
\begin{equation*}
\begin{array}{ccc}
\mathcal{A}/ L\mathbf{1} & \overset{U_{\mathcal{A}}}{\longrightarrow } &
\mathcal{A} \\
L/ \mathbf{1}\uparrow \downarrow R/ \mathbf{1} &  & L\uparrow \downarrow R
\\
\mathcal{B}/ \mathbf{1} & \overset{U_{\mathcal{B}}}{\longrightarrow } &
\mathcal{B}%
\end{array}%
\end{equation*}
\end{invisible}

Explicitly $\left( L/\mathbf{1}\right) \left( B,tB:B\rightarrow \mathbf{1}%
\right) :=\left( LB,LtB\right) $ and $\left( L/\mathbf{1}\right) f=Lf.$ The
functor $R/\mathbf{1}$ associates to an object $\left( A,tA:A\rightarrow L%
\mathbf{1}\right) $ the pair $\left( KA,tKA\right) $ given by the
pullback in the right-hand side diagram above.

\begin{invisible}
\begin{equation*}
\begin{array}{ccc}
KA & \overset{tKA}{\longrightarrow } & \mathbf{1} \\
kA\downarrow & \lrcorner & \downarrow \eta \mathbf{1} \\
RA & \overset{RtA}{\longrightarrow } & RL\mathbf{1}%
\end{array}%
\end{equation*}
\end{invisible}

Given a morphism $f:\left( A,tA:A\rightarrow L\mathbf{1}\right) \rightarrow
\left( A^{\prime },tA^{\prime }:A^{\prime }\rightarrow L\mathbf{1}\right) $
then $\left( R/ \mathbf{1}\right) f:\left( KA,tKA\right) \rightarrow \left(
KA^{\prime },tKA^{\prime }\right) $ is defined by the universal property of
the pullback as the unique morphism such that
\begin{equation}
kA^{\prime }\circ U_{\mathcal{B}}\left( R/ \mathbf{1}\right) f=RU_{\mathcal{A%
}}f\circ kA.  \label{form:R/1morph}
\end{equation}

\begin{invisible}
We have $RtA^{\prime }\circ RU_{\mathcal{A}}f\circ kA=R\left( tA^{\prime
}\circ U_{\mathcal{A}}f\right) \circ kA=RtA\circ kA=\eta \mathbf{1}\circ
tKA. $ By pullback there is a unique morphism $g:KA\rightarrow KA^{\prime }$
such that $kA^{\prime }\circ g=RU_{\mathcal{A}}f\circ kA$ and $tKA^{\prime
}\circ g=tKA.$ The latter equality tells that there is a unique morphism $%
\left( R/ \mathbf{1}\right) f:\left( KA,tKA\right) \rightarrow \left(
KA^{\prime },tKA^{\prime }\right) $ such that $U_{\mathcal{B}}\left( R/
\mathbf{1}\right) f=g$ so that $kA^{\prime }\circ U_{\mathcal{B}}\left( R/
\mathbf{1}\right) f=RU_{\mathcal{A}}f\circ kA.$
\end{invisible}

The unit $\eta / \mathbf{1}$ and counit $\epsilon / \mathbf{1}$ of the adjunction
\begin{invisible}
\begin{equation*}
\eta / \mathbf{1}:\mathrm{Id}\rightarrow \left( R/ \mathbf{1}%
\right) \left( L/ \mathbf{1}\right) \qquad\epsilon / \mathbf{1}%
 :\left( L/ \mathbf{1}\right) \left( R/ \mathbf{1}\right) \rightarrow
\mathrm{Id}
\end{equation*}%
\end{invisible}
are uniquely defined by the following equalities%
\begin{equation*}
kLB\circ U_{\mathcal{B}}\left( \eta / \mathbf{1}\right) \left( B,tB\right)
=\eta B\qquad U_{\mathcal{A}}\left( \epsilon / \mathbf{1}\right) \left(
A,tA\right) =\epsilon A\circ LkA.
\end{equation*}

\begin{invisible}
First note that%
\begin{eqnarray*}
\left( R/ \mathbf{1}\right) \left( L/ \mathbf{1}\right) \left( B,tB\right)
&=&\left( R/ \mathbf{1}\right) \left( LB,LtB\right) =\left( KLB,tKLB\right) ,
\\
\left( L/ \mathbf{1}\right) \left( R/ \mathbf{1}\right) \left( A,tA\right)
&=&\left( L/ \mathbf{1}\right) \left( KA,tKA\right) =\left( LKA,LtKA\right) .
\end{eqnarray*}%
The object $\left( R/ \mathbf{1}\right) \left( LB,LtB\right) =\left(
KLB,tKLB\right) $ is given by the following pullback.%
\begin{equation*}
\begin{array}{ccc}
KLB & \overset{tKLB}{\longrightarrow } & \mathbf{1} \\
kLB\downarrow & \lrcorner & \downarrow \eta \mathbf{1} \\
RLB & \overset{RLtB}{\longrightarrow } & RL\mathbf{1}%
\end{array}%
\end{equation*}%
Note that $RLtB\circ \eta B=\eta \mathbf{1}\circ tB.$ By the universal
property of the pullback there is a unique morphism $f:B\rightarrow KLB$
such that $kLB\circ f=\eta B$ and $tKLB\circ f=tB.$ The last equality tells
that there is a unique morphism $\left( \eta / \mathbf{1}\right) \left(
B,tB\right) :\left( B,tB\right) \rightarrow \left( KLB,tKLB\right) $ such
that $U_{\mathcal{B}}\left( \eta / \mathbf{1}\right) \left( B,tB\right) =f.$
Thus $kLB\circ U_{\mathcal{B}}\left( \eta / \mathbf{1}\right) \left(
B,tB\right) =\eta B.$

Concerning $\epsilon / \mathbf{1},$ consider the diagram%
\begin{equation*}
\begin{array}{ccc}
KA & \overset{tKA}{\longrightarrow } & \mathbf{1} \\
kA\downarrow & \lrcorner & \downarrow \eta \mathbf{1} \\
RA & \overset{RtA}{\longrightarrow } & RL\mathbf{1}%
\end{array}%
\end{equation*}%
We have
\begin{eqnarray*}
tA\circ \left( \epsilon A\circ LkA\right) &=&tA\circ \epsilon A\circ
LkA=\epsilon L\mathbf{1}\circ LRtA\circ LkA \\
&=&\epsilon L\mathbf{1}\circ L\left( RtA\circ kA\right) =\epsilon L\mathbf{1}%
\circ L\left( \eta \mathbf{1}\circ tKA\right) \\
&=&\epsilon L\mathbf{1}\circ L\eta \mathbf{1}\circ LtKA=LtKA.
\end{eqnarray*}%
This means that there is a unique morphism $\left( \epsilon / \mathbf{1}%
\right) \left( A,tA\right) :\left( LKA,LtKA\right) \rightarrow \left(
A,tA\right) $ such that $U_{\mathcal{A}}\left( \epsilon / \mathbf{1}\right)
\left( A,tA\right) =\epsilon A\circ LkA.$

We compute%
\begin{eqnarray*}
&&U_{\mathcal{A}}\left[ \left( \epsilon / \mathbf{1}\right) \left( L/
\mathbf{1}\right) \left( B,tB\right) \circ \left( L/ \mathbf{1}\right)
\left( \eta / \mathbf{1}\right) \left( B,tB\right) \right] \\
&=&U_{\mathcal{A}}\left( \epsilon / \mathbf{1}\right) \left( LB,LtB\right)
\circ U_{\mathcal{A}}\left( L/ \mathbf{1}\right) \left( \eta / \mathbf{1}%
\right) \left( B,tB\right) \\
&=&\epsilon LB\circ LkLB\circ LU_{\mathcal{B}}\left( \eta / \mathbf{1}%
\right) \left( B,tB\right) \\
&=&\epsilon LB\circ L\left[ kLB\circ U_{\mathcal{B}}\left( \eta / \mathbf{1}%
\right) \left( B,tB\right) \right] =\epsilon LB\circ L\eta B=\mathrm{Id}=U_{%
\mathcal{A}}\left( \mathrm{Id}_{\left( B,tB\right) }\right)
\end{eqnarray*}%
so that $\left( \epsilon / \mathbf{1}\right) \left( L/ \mathbf{1}\right)
\left( B,tB\right) \circ \left( L/ \mathbf{1}\right) \left( \eta / \mathbf{1}%
\right) \left( B,tB\right) =\mathrm{Id}_{\left( B,tB\right) }.$ We also have%
\begin{eqnarray*}
&&kA\circ U_{\mathcal{B}}\left[ \left( R/ \mathbf{1}\right) \left( \epsilon
/ \mathbf{1}\right) \left( A,tA\right) \circ \left( \eta / \mathbf{1}\right)
\left( R/ \mathbf{1}\right) \left( A,tA\right) \right] \\
&=&kA\circ U_{\mathcal{B}}\left( R/ \mathbf{1}\right) \left( \epsilon /
\mathbf{1}\right) \left( A,tA\right) \circ U_{\mathcal{B}}\left( \eta /
\mathbf{1}\right) \left( R/ \mathbf{1}\right) \left( A,tA\right) \\
&&\overset{(\ref{form:R/1morph})}{=}RU_{\mathcal{A}}\left( \epsilon /
\mathbf{1}\right) \left( A,tA\right) \circ kLKA\circ U_{\mathcal{B}}\left(
\eta / \mathbf{1}\right) \left( KA,tKA\right) \\
&=&R\left( \epsilon A\circ LkA\right) \circ \eta KA=R\epsilon A\circ
RLkA\circ \eta KA=R\epsilon A\circ \eta RA\circ kA=kA
\end{eqnarray*}%
On the other hand $tKA\circ U_{\mathcal{B}}\left[ \left( R/ \mathbf{1}%
\right) \left( \epsilon / \mathbf{1}\right) \left( A,tA\right) \circ \left(
\eta / \mathbf{1}\right) \left( R/ \mathbf{1}\right) \left( A,tA\right) %
\right] =tKA$ as $\left( R/ \mathbf{1}\right) \left( \epsilon / \mathbf{1}%
\right) \left( A,tA\right) \circ \left( \eta / \mathbf{1}\right) \left( R/
\mathbf{1}\right) \left( A,tA\right) $ a morphism in $\mathcal{B}/ \mathbf{1}
$. The universal property of pullbacks tells that
\begin{equation*}
U_{\mathcal{B}}\left[ \left( R/ \mathbf{1}\right) \left( \epsilon / \mathbf{1%
}\right) \left( A,tA\right) \circ \left( \eta / \mathbf{1}\right) \left( R/
\mathbf{1}\right) \left( A,tA\right) \right] =\mathrm{Id}
\end{equation*}%
as we also have $kA\circ \mathrm{Id}=kA$ and $tKA\circ \mathrm{Id}=tKA.$
Since $\mathrm{Id}=U_{\mathcal{B}}\left( \mathrm{Id}_{\left( A,tA\right)
}\right) ,$ we obtain
\begin{equation*}
\left( R/ \mathbf{1}\right) \left( \epsilon / \mathbf{1}\right) \left(
A,tA\right) \circ \left( \eta / \mathbf{1}\right) \left( R/ \mathbf{1}%
\right) \left( A,tA\right) =\mathrm{Id}_{\left( A,tA\right) }.
\end{equation*}
\end{invisible}

\begin{remark}
As mentioned, the construction above is well-known. It can be recovered as follows. For
every morphism $f:X\rightarrow Y$ in a category $\mathcal{C}$ with pullbacks consider the
functor $\mathcal{C}/X\overset{f_{\ast }}{\mathbf{\rightarrow }}\mathcal{C}/Y$ defined on objects by $\left( C,g\right) :=\left( C,f\circ
g\right) $ and as the identity on morphisms. It is well-known that this
functor has a right adjoint $f^{\ast }$ given by pullbacks along $f$ in the
underlying category (see e.g. \cite[16.8.5]{Schubert}). Now note that
the functor $L/\mathbf{1}:\mathcal{B}/\mathbf{1\rightarrow }\mathcal{A}/L%
\mathbf{1}\ $can be written as the composition
\begin{equation*}
\mathcal{B}/\mathbf{1}\overset{\left( \eta \mathbf{1}\right) _{\ast }}{%
\mathbf{\rightarrow }}\mathcal{B}/RL\mathbf{1}\overset{L/RL\mathbf{1}}{%
\mathbf{\rightarrow }}\mathcal{A}/LRL\mathbf{1}\overset{\left( \epsilon L%
\mathbf{1}\right) ^{\ast }}{\mathbf{\rightarrow }}\mathcal{A}/L\mathbf{1}%
\text{.}
\end{equation*}
By \cite[16.8.7]{Schubert} the functor $R/L\mathbf{1}:\mathcal{A}/L%
\mathbf{1}\rightarrow \mathcal{B}/RL\mathbf{1:}\left( A,a\right) \mapsto
\left( RA,Ra\right) ,\alpha \mapsto R\alpha $ is a right adjoint for the
composition $\left( \epsilon L\mathbf{1} \right) ^{\ast }\circ L/RL\mathbf{1}
$. As a consequence we get that a right adjoint of $L/\mathbf{1}$ is given by
the composition $\left( \eta \mathbf{1}\right) ^{\ast }\circ R/L\mathbf{1}$
which is exactly the functor $R/\mathbf{1}$ defined above.
\end{remark}

\begin{invisible}
\lbrack REFERENCE\ NEEDED: https://ncatlab.org/nlab/show/over+category]
\end{invisible}

\begin{lemma}
\label{lem:createscomma}Let $\mathcal{C}$ be a category and $\mathbf{1}\in
\mathcal{C}$. The forgetful functor $U_{\mathcal{C}}:%
\mathcal{C}/ \mathbf{1}\longrightarrow \mathcal{C}\ $creates colimits.
\end{lemma}

\begin{proof}
Cf. \cite[Proposition 2.16.3 ]{Borceux1} or dual version of \cite[Exercice
1, page 108]{MacLane}.
\end{proof}

\begin{invisible}
Let $\mathcal{D}$ be a small category and let $F:\mathcal{D}\rightarrow
\mathcal{C}/ \mathbf{1}$ be a functor. Set $G:=U_{\mathcal{C}}\circ F:%
\mathcal{D}\rightarrow \mathcal{C}$ and assume $G$ has colimit $\left(
L,\left( sD:GD\rightarrow L\right) _{D\in \mathcal{D}}\right) .$ We can
write $FD=\left( GD,\gamma D:GD\rightarrow \mathbf{1}\right) .$ Given a
morphism $d:D\rightarrow D^{\prime }$ in $\mathcal{D}$, the fact that $Fd\in
\mathcal{C}/ \mathbf{1}$ means that $\gamma D^{\prime }\circ Gd=\gamma D.$
Thus $\left( \mathbf{1},\left( \gamma D:GD\rightarrow \mathbf{1}\right)
_{D\in \mathcal{D}}\right) $ is a cocone on $G.$ By the universal property
of colimits there is a unique morphism $\gamma L:L\rightarrow \mathbf{1}$
such that $\gamma L\circ sD=\gamma D.$ This means that $\left( L,\gamma
L\right) \in \mathcal{C}/ \mathbf{1}$ and that there is a unique morphism $%
\sigma D:\left( GD,\gamma D\right) \rightarrow \left( L,\gamma L\right) $
such that $U_{\mathcal{C}}\sigma D=sD.$

Clearly $U_{\mathcal{C}}\left( L,\gamma L\right) =L$ and $U_{\mathcal{C}%
}\left( \sigma D^{\prime }\circ Fd\right) =sD^{\prime }\circ Gd=sD=U_{%
\mathcal{C}}\left( \sigma D\right) $ so that $\sigma D^{\prime }\circ
Fd=\sigma D$ and $\left( \left( L,\gamma L\right) ,\left( \sigma D\right)
_{D\in \mathcal{D}}\right) $ is a cocone for $F.$ Given another cocone $%
\left( A,\left( fD:FD\rightarrow A\right) _{D\in \mathcal{D}}\right) $ such
that $U_{\mathcal{C}}A=L$ and $U_{\mathcal{C}}fD=sD,$ then $A=\left(
L,\lambda \right) $ for some $\lambda :L\rightarrow \mathbf{1.}$ Since $%
fD:FD\rightarrow A$ is a morphism in $\mathcal{C}/ \mathbf{1}$ we have that $%
\lambda \circ U_{\mathcal{C}}fD=\gamma D$ i.e. $\lambda \circ sD=\gamma D$
so that $\lambda =\gamma L$ by uniqueness of $\gamma L$. Thus $A=\left(
L,\lambda \right) =\left( L,\gamma L\right) .$ Moreover $U_{\mathcal{C}%
}fD=sD=U_{\mathcal{C}}\sigma D$ and since $fD:FD\rightarrow A=\left(
L,\gamma L\right) $ we get $fD=\sigma D.$ Hence $\left( A,\left( fD\right)
_{D\in \mathcal{D}}\right) =\left( \left( L,\gamma L\right) ,\left( \sigma
D\right) _{D\in \mathcal{D}}\right) $ as desired. This proves the uniqueness
of the cocone in the definition of creation.

Let us check that $\left( \left( L,\gamma L\right) ,\left( \sigma D\right)
_{D\in \mathcal{D}}\right) $ is the colimit of $F.$ Given another cocone $%
\left( B,\left( gD:FD\rightarrow B\right) _{D\in \mathcal{D}}\right) $ on $%
F, $ then $\left( U_{\mathcal{C}}B,\left( U_{\mathcal{C}}gD:GD\rightarrow U_{%
\mathcal{C}}B\right) _{D\in \mathcal{D}}\right) $ is a cocone on $G$ as $U_{%
\mathcal{C}}gD^{\prime }\circ Gd=U_{\mathcal{C}}\left( gD^{\prime }\circ
Fd\right) =U_{\mathcal{C}}gD.$ Thus there is a unique morphism $%
g:L\rightarrow U_{\mathcal{C}}B$ such that $g\circ sD=U_{\mathcal{C}}gD.$
Write $B=\left( U_{\mathcal{C}}B,\beta :U_{\mathcal{C}}B\rightarrow \mathbf{1%
}\right) .$ Then $\beta \circ g\circ sD=\beta \circ U_{\mathcal{C}}gD=\gamma
D=\gamma L\circ sD$ so that $\beta \circ g=\gamma L$ (by colimit). This
means that there is a morphism $g^{\prime }:\left( L,\gamma L\right)
\rightarrow B$ such that $U_{\mathcal{C}}g^{\prime }=g.$ Thus $U_{\mathcal{C}%
}g^{\prime }\circ U_{\mathcal{C}}\sigma D=g\circ sD=U_{\mathcal{C}}gD$ and
hence $g^{\prime }\circ \sigma D=gD.$ Assume there is another morphism $%
h:\left( L,\gamma L\right) \rightarrow B$ such that $h\circ \sigma D=gD.$ By
applying $U_{\mathcal{C}}$ we get $U_{\mathcal{C}}h\circ U_{\mathcal{C}%
}\sigma D=U_{\mathcal{C}}gD$ i.e. $U_{\mathcal{C}}h\circ sD=U_{\mathcal{C}%
}gD $ and hence $U_{\mathcal{C}}h=g$ by uniqueness of $g.$ Hence $U_{%
\mathcal{C}}h=U_{\mathcal{C}}g^{\prime }$ so that $h=g^{\prime }.$ Thus $%
\left( \left( L,\gamma L\right) ,\left( \sigma D\right) _{D\in \mathcal{D}%
}\right) $ is the colimit of $F.$
\end{invisible}

\begin{remark}
Since $\mathcal{B}$ is a category with pullbacks, if $\mathbf{1}\in \mathcal{%
B}$ is a terminal object then $\mathcal{B}$ would be finitely complete by
\cite[Proposition 2.8.2 ]{Borceux1}. As a consequence $\mathcal{B}/ \mathbf{1%
}$ is finitely complete whenever $\mathcal{B}$ is a category with pullbacks
(cf. \cite[Proposition 2.16.3 ]{Borceux1}).
\end{remark}

\begin{invisible}
First note that $\left( \mathbf{1},\mathrm{Id}_{\mathbf{1}}\right) $ is a
terminal object in $\mathcal{B}/ \mathbf{1.}$ In fact given $f:\left(
B,\beta \right) \rightarrow \left( \mathbf{1},\mathrm{Id}_{\mathbf{1}%
}\right) $ we get $\mathrm{Id}_{\mathbf{1}}\circ f=\beta $ i.e. $f=\beta .$
Hence there is a unique morphism $\left( B,\beta \right) \rightarrow \left(
\mathbf{1},\mathrm{Id}_{\mathbf{1}}\right) .$ Moreover $\mathcal{B}/ \mathbf{%
1}$ inherits the pullbacks of $\mathcal{B}.$ In fact, given a morphisms $%
f:\left( X,x\right) \rightarrow \left( Z,z\right) $ and $g:\left( Y,y\right)
\rightarrow \left( Z,z\right) $ let $\left( P,u,v\right) $ be the pullback
of $\left( f,g\right) $ in $\mathcal{B}.$
\begin{equation*}
\begin{array}{ccc}
\left( P,p\right) & \overset{v}{\longrightarrow } & \left( Y,y\right) \\
u\downarrow & \lrcorner & \downarrow g \\
\left( X,x\right) & \overset{f}{\longrightarrow } & \left( Z,z\right)%
\end{array}%
\end{equation*}%
Since $x\circ u=z\circ f\circ u=z\circ g\circ v=y\circ v$ we can set $%
p:=x\circ u=y\circ v.$ Let us check that $\left( \left( P,p\right)
,u,v\right) $ is the pullback of $\left( f,g\right) $ in $\mathcal{B}/
\mathbf{1.}$ Since $p:=x\circ u=y\circ v$ we have that $u$ and $v$ are
morphisms in $\mathcal{B}/ \mathbf{1}$. Given $\left( \left( P^{\prime
},p^{\prime }\right) ,u^{\prime },v^{\prime }\right) $ be such that $f\circ
u^{\prime }=g\circ v^{\prime }$, by the universal property of the pullback
in $\mathcal{B}$ there is a unique $\varphi :P^{\prime }\rightarrow P$ such
that $v\circ \varphi =v^{\prime }$ and $u\circ \varphi =u^{\prime }.$
Moreover $p\circ \varphi =x\circ u\circ \varphi =x\circ u^{\prime
}=p^{\prime }.$ Hence $\varphi :\left( P^{\prime },p^{\prime }\right)
\rightarrow \left( P,p\right) .$
\end{invisible}

In the rest of this section $\mathcal{M}$ is a non-empty preadditive braided
monoidal category such that

\begin{itemize}
\item $\mathcal{M}$ has equalizers, denumerable coproducts and coequalizers
of reflexive pairs of morphisms;

\item the tensor products are additive and preserve equalizers, denumerable
coproducts and coequalizers of reflexive pairs of morphisms.
\end{itemize}

\medskip

We include here a well-known result we need.

\begin{lemma}
\label{lem:zeroobj}A non-empty preadditive category $\mathcal{C}$ with
equalizers has a zero object.
\end{lemma}

\begin{proof}
For every $A,B$ in $\mathcal{C}$, the set $\mathrm{Hom}_{\mathcal{C}}(A,B)$ contains a zero morphism i.e. $\mathcal{C}$ is a pointed category. For any morphism $f:A\to B$ we can compute the equalizer of $f$ and the zero morphism $A\to B$, i.e. the kernel of $f$. By \cite{AHS}[7C(d), page 127], the category $\mathcal{C}$ has a zero object.
%
%
\end{proof}

\begin{remark}
Let us show that under the hypotheses above, the category $\mathcal{M}$ is a
pre-abelian, see \cite[pag 24]{Popescu}. First we see it is additive. By
Lemma \ref{lem:zeroobj}, the category $\mathcal{M},$ being non-empty and
preadditive, admits a zero object, say $\mathbf{0} $. Given two objects $%
X_{1},X_{2}$ in $\mathcal{M}$ we can set $X_{n}:=\mathbf{0}$ for all $n\in
\mathbb{N}$ with $n>2.$ Then the denumerable coproduct $\coprod_{n\in
\mathbb{N}}X_{n},$ which exists by assumption, is just the coproduct of $%
X_{1},X_{2}.$ By \cite[Proposition 1.2.4 and Definition 1.2.5]{Borceux2},
the category $\mathcal{M}$ has binary biproducts. Since $\mathcal{M}$ has a
zero object, then $\mathcal{M}$ is additive, see e.g. \cite[Definition 1.2.6]%
{Borceux2}.

By hypotheses $\mathcal{M}$ has all equalizers. Moreover, since $\mathcal{M}$
has binary coproducts and coequalizers of reflexive pairs, then $%
\mathcal{M}$ has all coequalizers: to check this one has to apply the procedure mentioned in \cite[page 20]{La} to replace a pair of morphisms by a reflexive pair with the same coequalizer.
Since $\mathcal{M}$ has a zero object, we get that $\mathcal{M}$ has all
kernels and cokernels. Thus $\mathcal{M}$ is a preabelian category. We point
out that, by \cite[Proposition 2.8.2]{Borceux1} and its dual form, the
category $\mathcal{M}$ is finitely complete and finitely cocomplete (this
makes sense since the dual of a preabelian category is preabelian, as
observed in \cite[page 24]{Popescu}).

We point out that, since, by hypothesis, denumerable coproducts and
coequalizers of reflexive pairs are preserved by tensor products,
all coequalizers are preserved too by \cite[Lemma 2.3]%
{AGM-Liftable}.
\end{remark}

By the assumptions above, we can apply \cite[Theorem 4.6]{AM-BraidedOb} to
give an explicit description of an adjunction $\widetilde{T}\dashv P:\mathrm{Bialg}\left( \mathcal{M}\right) \rightarrow \mathcal{M}.$
Note that $\mathbf{1}$ is a terminal object in $\mathrm{Bialg}\left(
\mathcal{M}\right) $ so that $\mathbf{0}:=P\mathbf{1}$ is a terminal object
in $\mathcal{M}$, as right adjoint functors preserve the terminal object. It is indeed a zero object in $\mathcal{M}$ by Lemma \ref{lem:zeroobj}.

As a particular case of the constructions above, consider the following
left-hand side diagram%
\begin{equation*}
\xymatrixcolsep{1.5cm}\xymatrixrowsep{0.6cm}\xymatrix{\Alg(\M)/
T\mathbf{0}\ar[r]^-{U_{\Alg(\M)}}\ar@<.5ex>@{.>}[d]^{\Omega/
\mathbf{0}}&\Alg(\M)\ar@<.5ex>@{.>}[d]^{\Omega}\\ \M/ \mathbf{0}
\ar[r]^{U_{\M}}\ar@<.5ex>[u]^{T/ \mathbf{0}}&\M\ar@<.5ex>[u]^{T}}
\qquad
\xymatrixcolsep{2cm}\xymatrixrowsep{0.6cm}\xymatrix{\Alg^+(\M)%
\ar[r]^-{U:=U_{\Alg(\M)}}\ar@<.5ex>@{.>}[d]^{\Omega^+}&\Alg(\M)%
\ar@<.5ex>@{.>}[d]^{\Omega}\\ \M
\ar[r]^{\id}\ar@<.5ex>[u]^{T^+}&\M\ar@<.5ex>[u]^{T}}
\end{equation*}
\begin{invisible}
\begin{equation*}
\begin{array}{ccc}
\mathrm{Alg}\left( \mathcal{M}\right) / T\mathbf{0} & \overset{U_{\mathrm{Alg%
}\left( \mathcal{M}\right) }}{\longrightarrow } & \mathrm{Alg}\left(
\mathcal{M}\right) \\
T/ \mathbf{0}\uparrow \downarrow \Omega / \mathbf{0} &  & T\uparrow
\downarrow \Omega \\
\mathcal{M}/ \mathbf{0} & \overset{U_{\mathcal{M}}}{\longrightarrow } &
\mathcal{M}%
\end{array}%
\end{equation*}
\end{invisible}
Since left adjoint functors preserve the initial object, we get that $T\mathbf{0}$ is initial. By uniqueness of initial object, we get $T\mathbf{0}\cong
\mathbf{1}$ as $\mathbf{1}$ is the initial object in $\mathrm{Alg}\left(
\mathcal{M}\right) $. Thus $\mathrm{Alg}\left( \mathcal{M}\right) / T\mathbf{%
0}$ is $\mathrm{Alg}\left( \mathcal{M}\right) / \mathbf{1}$ i.e. the
category of augmented algebras that will be denoted by $\mathrm{Alg}%
^{+}\left( \mathcal{M}\right) $. Note also that the functor $U_{\mathcal{M}}:%
\mathcal{M}/ \mathbf{0}\rightarrow \mathcal{M}$ is a category isomorphism
because $\mathbf{0}$ is terminal in $\mathcal{M}$. In light of these
observations we can rewrite the starting diagram as the right-hand side one above
\begin{invisible}
\begin{equation*}
\begin{array}{ccc}
\mathrm{Alg}^{+}\left( \mathcal{M}\right) & \overset{U:=U_{\mathrm{Alg}%
\left( \mathcal{M}\right) }}{\longrightarrow } & \mathrm{Alg}\left( \mathcal{%
M}\right) \\
T^{+}\uparrow \downarrow \Omega ^{+} &  & T\uparrow \downarrow \Omega \\
\mathcal{M} & \overset{\mathrm{Id}}{\longrightarrow } & \mathcal{M}%
\end{array}%
\end{equation*}
\end{invisible}
where $T^{+}:=\left( T/ \mathbf{0}\right) \circ \left( U_{\mathcal{M}%
}\right) ^{-1}$ and $\Omega ^{+}:=U_{\mathcal{M}}\circ \left( \Omega /
\mathbf{0}\right) .$ Explicitly%
\begin{eqnarray*}
T^{+}M &=&\left( T/ \mathbf{0}\right) \left( U_{\mathcal{M}}\right)
^{-1}M=\left( T/ \mathbf{0}\right) \left( M,tM\right) =\left( TM,TtM\right)
=\left( TM,\varepsilon _{TM}\right) , \\
\Omega ^{+}\left( A,\varepsilon \right) &=&U_{\mathcal{M}}\left( \Omega /
\mathbf{0}\right) \left( A,\varepsilon \right) =U_{\mathcal{M}}\left(
KA,tKA\right) =KA
\end{eqnarray*}
where $\Omega / \mathbf{0}$ associates $\left(
A,\varepsilon \right) \in \mathrm{Alg}^{+}\left( \mathcal{M}\right)$ the pair $\left( KA,tKA\right) $ defined by the pullback in $\mathcal{M}$
\begin{equation}\label{diag:defKaug}
\xymatrixcolsep{1.5cm}\xymatrixrowsep{0.6cm}\vcenter{\xymatrix{KA\pulb\ar[r]^-{tKA}%
\ar@<.5ex>[d]_{kA}&\mathbf{0}\ar@<.5ex>[d]^{\eta\mathbf{0}=i\Omega
\mathbf{1}}\\ \Omega A\ar[r]^{\Omega\varepsilon}&\Omega \mathbf{1\cong
}\Omega T\mathbf{0}} }
\end{equation}
\begin{invisible}
\begin{equation*}
\begin{array}{ccc}
KA & \overset{tKA}{\longrightarrow } & \mathbf{0} \\
kA\downarrow & \lrcorner & \downarrow \eta \mathbf{0}=i\Omega \mathbf{1} \\
\Omega A & \overset{\Omega \varepsilon }{\longrightarrow } & \Omega \mathbf{%
1\cong }\Omega T\mathbf{0}%
\end{array}%
\end{equation*}
\end{invisible}
where $iX:\mathbf{0}\rightarrow X$ is the unique morphism from the initial
object $\mathbf{0}$ and $tX:X\rightarrow \mathbf{0}$ the unique one into the
terminal object. This means that $\left( KA,kA\right) =\mathrm{Ker}\left(
\Omega \varepsilon \right) .$ Hence%
\begin{equation*}
T^{+}M=\left( TM,\varepsilon _{TM}\right) ,\qquad \left( \Omega ^{+}\left(
A,\varepsilon \right) ,kA\right) =\mathrm{Ker}\left( \Omega \varepsilon
\right) .
\end{equation*}%
Given a morphism $f:\left( A,\varepsilon \right) \rightarrow \left(
A^{\prime },\varepsilon ^{\prime }\right) $ then $\Omega ^{+}f:\Omega
^{+}\left( A,\varepsilon \right) \rightarrow \Omega ^{+}\left( A^{\prime
},\varepsilon ^{\prime }\right) $ is defined by
\begin{equation}
kA^{\prime }\circ \Omega ^{+}f=\Omega Uf\circ kA.  \label{form:Omega+morph}
\end{equation}
The unit $\eta ^{+}:=\left( U_{\mathcal{M}}\right) \left( \eta / \mathbf{0}%
\right) \left( U_{\mathcal{M}}\right) ^{-1}:\mathrm{Id}\rightarrow \Omega
^{+}T^{+}$ and the counit $\epsilon ^{+}:=\left( \epsilon / \mathbf{0}\right)
:T^{+}\Omega ^{+}\rightarrow \mathrm{Id}$ are uniquely determined by the
following equalities%
\begin{equation}
kTB\circ \eta ^{+}B=\eta B\qquad U\epsilon ^{+}\left( A,\varepsilon \right)
=\epsilon A\circ TkA.  \label{form:etaeps+}
\end{equation}
\begin{invisible}
In general a diagram as follows is a pullback if and only if $\left(
K,k\right) =\mathrm{Ker}\left( f\right) .$
\begin{equation*}
\begin{array}{ccc}
K & \overset{tK}{\longrightarrow } & \mathbf{0} \\
k\downarrow & \lrcorner & \downarrow iY \\
X & \overset{f}{\longrightarrow } & Y%
\end{array}%
\end{equation*}%
In fact, in order to check that $\left( K,k\right) =\mathrm{Ker}\left(
f\right) ,$ we have to check that the following fork is an equalizer in $%
\mathcal{M}$
\begin{equation*}
K\overset{k}{\rightarrow }X\overset{f}{\underset{iY\circ tX}{%
\rightrightarrows }}Y.
\end{equation*}

In fact $iY\circ tX\circ k=iY\circ tK=f\circ k$. Moreover given $%
g:W\rightarrow X$ such that $f\circ g=iY\circ tX\circ g,$ then $f\circ
g=iY\circ tW$ so that, by pullback, there is a unique morphism $g^{\prime
}:W\rightarrow K$ such that $k\circ g^{\prime }=g$ and $tK\circ g^{\prime
}=tW.$ Since the last equality is always true, we get that there is a unique
morphism $g^{\prime }:W\rightarrow K$ such that $k\circ g^{\prime }=g.$
\end{invisible}

Next aim is to show that the left-hand side diagram below fits into the
setting of Theorem \ref{thm:main}.
\begin{equation}  \label{square:T}
\vcenter{\xymatrixcolsep{2cm}\xymatrixrowsep{0.7cm}\xymatrix{\Bialg(\M)%
\ar@<.5ex>@{.>}[d]^P\ar[r]^{\mho
^{+}}&\Alg^+(\M)\ar[r]^-{U:=U_{\Alg(\M)}}\ar@<.5ex>@{.>}[d]^{\Omega^+}&\Alg(%
\M)\ar@<.5ex>@{.>}[d]^{\Omega}\\
\M\ar@<.5ex>[u]^{\widetilde{T}}\ar[r]^\id&\M
\ar[r]^{\id}\ar@<.5ex>[u]^{T^+}&\M\ar@<.5ex>[u]^{T}} }
\end{equation}

\begin{invisible}
\begin{equation}
\begin{array}{ccccc}
\mathrm{Bialg}\left( \mathcal{M}\right) & \overset{\mho ^{+}}{%
\longrightarrow } & \mathrm{Alg}^{+}\left( \mathcal{M}\right) & \overset{%
U:=U_{\mathrm{Alg}\left( \mathcal{M}\right) }}{\longrightarrow } & \mathrm{%
Alg}\left( \mathcal{M}\right) \\
\widetilde{T}\uparrow \downarrow P &  & T^{+}\uparrow \downarrow \Omega ^{+}
&  & T\uparrow \downarrow \Omega \\
\mathcal{M} & \overset{\mathrm{Id}}{\longrightarrow } & \mathcal{M} &
\overset{\mathrm{Id}}{\longrightarrow } & \mathcal{M}%
\end{array}%
\end{equation}
\end{invisible}

By construction of $\widetilde{T}$ we have $\mho \circ \widetilde{T}=T$. Here $\mho:\mathrm{Bialg}\left( \mathcal{M}\right)\to \mathrm{Alg}\left( \mathcal{M}\right)$ and %
$\mho ^{+}$ are the obvious forgetful functors such that $U\circ \mho ^{+}=\mho
\ $and $\mho ^{+}\circ \widetilde{T}=T^{+}.$

Moreover the assumptions guarantee that the category $\mathrm{Alg}\left(
\mathcal{M}\right) $ has coequalizers, see e.g. \cite[Proposition 2.5]%
{AGM-Liftable}, see also \cite[Theorem 2.3]{Porst}. Since $\mathrm{Bialg}\left( \mathcal{M}\right) =\mathrm{Coalg%
}\left( \mathrm{Alg}\left( \mathcal{M}\right) \right) $ we can apply \cite[%
Proposition 2.5]{Pareigis} to $\mathcal{C}:=\mathrm{Alg}\left( \mathcal{M}%
\right) ^{\mathrm{op}}$ to obtain that $\mho ^{\mathrm{op}}:\mathrm{Bialg}%
\left( \mathcal{M}\right) ^{\mathrm{op}}\rightarrow \mathrm{Alg}\left(
\mathcal{M}\right) ^{\mathrm{op}}$ creates limits, equivalently $\mho :%
\mathrm{Bialg}\left( \mathcal{M}\right) \rightarrow \mathrm{Alg}\left(
\mathcal{M}\right) $ creates colimits. Since $\mathrm{Alg}\left( \mathcal{M}%
\right) $ has coequalizers, we deduce that $\mathrm{Bialg}\left( \mathcal{M}%
\right) $ has coequalizers and $\mho $ preserves coequalizers. On the other
hand the functor $\Omega :\mathrm{Alg}\left( \mathcal{M}\right) \rightarrow
\mathcal{M}$ needs not to preserve coequalizers. Nevertheless $\Omega $
preserves the coequalizers of reflexive pairs of morphisms, see e.g. \cite[%
Corollary A.10]{AM-MM}. It is noteworthy that, since $\Omega $ has a left
adjoint $T$, then $\Omega $ is strictly monadic (the comparison functor is a
category isomorphism), see \cite[Theorem A.6]{AM-MM}.

Let $V\in \mathcal{M}.$ By construction $\Omega TV=\oplus _{n\in \mathbb{N}%
}V^{\otimes n},$ see \cite[Remark 1.2]{AM-BraidedOb}. Let$\alpha
_{n}V:V^{\otimes n}\rightarrow \Omega TV$ denote the canonical inclusion. The unit
of the adjunction $\left( T,\Omega \right) $ is $\eta :\mathrm{Id}_{\mathcal{%
M}}\rightarrow \Omega T$ defined by $\eta V:=\alpha _{1}V$ while the counit $%
\epsilon :T\Omega \rightarrow \mathrm{Id}$ is uniquely defined by the
equality
\begin{equation}
\Omega \epsilon \left( A,m,u\right) \circ \alpha _{n}A=m^{n-1}\text{ for
every }n\in \mathbb{N}  \label{form:epsTOmega}
\end{equation}%
where $m^{n-1}:A^{\otimes n}\rightarrow A$ is the iterated
multiplication of an algebra $\left( A,m,u\right) $ defined by $%
m^{-1}=u,m^{0}=\mathrm{Id}_{A}$ and, for $n\geq 2$, by $m^{n-1}=m\circ \left(
m^{n-2}\otimes A\right) .$

Denote by $\widetilde{\eta },\widetilde{\epsilon }$ the unit and counit of
the adjunction $( \widetilde{T},P) $.

\begin{lemma}
\label{lem:Omegareg}(\cite[Theorem 2.3]{Porst}) The functor $\Omega :\mathrm{Alg}%
\left( \mathcal{M}\right) \rightarrow \mathcal{M}$ preserves regular
epimorphisms.
\end{lemma}

Not that the previous result does not mean that $\Omega $ preserves
coequalizers.

\begin{lemma}
\label{lem:furbo}1) Let $e:A\rightarrow A$ and $f:A\rightarrow A^{\prime }$ be morphisms in $\mathcal{M}$ such
that $A\otimes f,f\otimes A$ are epimorphisms and $f\circ e\circ e=f\circ e.$
Then $\left( -\right) ^{\otimes 2}:\mathcal{M}\rightarrow
\mathcal{M}$ preserves the coequalizer of $\left( fe,f\right) .$

2) Let $a,b:A\rightarrow B$ be a reflexive pair of morphisms in $\mathcal{M}$
and let $g:B\rightarrow A^{\prime }$ be a morphism in $\mathcal{M}$ such
that $A\otimes g,g\otimes A$ are epimorphisms. Then the functor $\left(
-\right) ^{\otimes 2}:\mathcal{M}\rightarrow \mathcal{M}$ preserves the
coequalizer of $\left( ga,gb\right) .$

3) $\Omega :\mathrm{Alg}\left( \mathcal{M}\right) \rightarrow \mathcal{M}$
creates the coequalizers of those parallel pairs $\left( f,g\right) $ such
that the functor $\left( -\right) ^{\otimes 2}:\mathcal{M}\rightarrow
\mathcal{M}$ preserves the coequalizer of $\left( \Omega f,\Omega g\right) .$
\end{lemma}

\begin{proof}
Recall that $\mathcal{M}$ has all coequalizers and the tensor products
preserve coequalizers.

1) Consider the following left-hand side coequalizer
\begin{equation}  \label{coeq:furbo1}
\xymatrixcolsep{1cm} \xymatrix{A\ar@<.5ex>[r]^{fe}\ar@<-.5ex>[r]_{f}&A'%
\ar[r]^{p}& V} \qquad \xymatrixcolsep{1cm} \xymatrix{\duplica{A}%
\ar@<.5ex>[r]^{\duplica{fe}}\ar@<-.5ex>[r]_{\duplica{f}}&\duplica{A'}%
\ar[r]^{\duplica{p}}& \duplica{V}}
\end{equation}

\begin{invisible}
\begin{equation*}
A\overset{fe}{\underset{f}{\rightrightarrows }}A^{\prime }\overset{p}{%
\longrightarrow }V.
\end{equation*}
\end{invisible}

Let us check that the right-hand side one is a coequalizer too.

\begin{invisible}
\begin{equation*}
A\otimes A\overset{fe\otimes fe}{\underset{f\otimes f}{\rightrightarrows }}%
A^{\prime }\otimes A^{\prime }\overset{p\otimes p}{\longrightarrow }V\otimes
V
\end{equation*}%
is a coequalizer.
\end{invisible}

Let $\zeta :A^{\prime }\otimes A^{\prime }\rightarrow Z$ be such that $\zeta
\circ \left( fe\otimes fe\right) =\zeta \circ \left( f\otimes f\right) .$
Then%
\begin{eqnarray*}
\zeta \circ \left( fe\otimes A^{\prime }\right) \circ \left( A\otimes
f\right) &=&\zeta \circ \left( f\otimes f\right) \circ \left( e\otimes
A\right) =\zeta \circ \left( fe\otimes fe\right) \circ \left( e\otimes A\right) \\
&=&\zeta \circ \left( fe\otimes fe\right) =\zeta \circ \left( f\otimes f\right) =\zeta \circ \left( f\otimes
A^{\prime }\right) \circ \left( A\otimes f\right) .
\end{eqnarray*}%
Since $A\otimes f$ is an epimorphism, we deduce that $\zeta \circ \left(
fe\otimes A^{\prime }\right) =\zeta \circ \left( f\otimes A^{\prime }\right)
.$

Since the tensor products preserve coequalizers, the following are both
coequalizers.
\begin{equation}  \label{coeq:furbo2}
\xymatrixcolsep{1cm} \xymatrix{A\otimes A'\ar@<.5ex>[r]^{fe\otimes
A'}\ar@<-.5ex>[r]_{f\otimes A'}&A'\otimes A'\ar[r]^{p\otimes A'}& V\otimes
A'}\qquad \xymatrixcolsep{1cm} \xymatrix{V\otimes A\ar@<.5ex>[r]^{V\otimes
fe}\ar@<-.5ex>[r]_{V\otimes f}&V\otimes A'\ar[r]^{V\otimes p}& V\otimes V}
\end{equation}

\begin{invisible}
\begin{equation*}
A\otimes A^{\prime }\overset{fe\otimes A^{\prime }}{\underset{f\otimes
A^{\prime }}{\rightrightarrows }}A^{\prime }\otimes A^{\prime }\overset{%
p\otimes A^{\prime }}{\longrightarrow }V\otimes A^{\prime }
\end{equation*}
\end{invisible}

By using the left one there is a morphism $\zeta _{1}:V\otimes A^{\prime
}\rightarrow Z$ such that $\zeta _{1}\circ \left( p^{\prime }\otimes
A^{\prime }\right) =\zeta .$ We have%
\begin{eqnarray*}
\zeta _{1}\circ \left( V\otimes fe\right) \circ \left( pf\otimes A\right)
&=&\zeta _{1}\circ \left( p\otimes A^{\prime }\right) \circ \left( f\otimes
f\right) \circ \left( A\otimes e\right) =\zeta \circ \left( f\otimes f\right) \circ \left( A\otimes e\right) \\
&=&\zeta \circ \left( fe\otimes fe\right) \circ \left( A\otimes e\right) =\zeta \circ \left( fe\otimes fe\right) =\zeta \circ \left( f\otimes f\right) \\
&=&\zeta _{1}\circ \left( p\otimes A^{\prime }\right) \circ \left( f\otimes
f\right) \zeta _{1}\circ \left( V\otimes f\right) \circ \left( pf\otimes A\right) .
\end{eqnarray*}

Now, $f\otimes A$ is an epimorphism by assumption. Moreover $p\otimes A$ is
an epimorphism because the tensor products preserve coequalizers. Thus, from
the chain of equalities above, we deduce that $\zeta _{1}\circ \left(
V\otimes fe\right) =\zeta _{1}\circ \left( V\otimes f\right) $. Since the
right-hand side diagram in \eqref{coeq:furbo2} is a coequalizer,
\begin{invisible}
\begin{equation*}
V\otimes A\overset{V\otimes fe}{\underset{V\otimes f}{\rightrightarrows }}%
V\otimes A^{\prime }\overset{V\otimes p}{\longrightarrow }V\otimes V.
\end{equation*}
\end{invisible}
there is a morphism $\zeta _{2}:V\otimes V\rightarrow Z$ such that $\zeta
_{2}\circ \left( V\otimes p\right) =\zeta _{1}.$ Therefore%
\begin{equation*}
\zeta _{2}\circ \left( p\otimes p\right) =\zeta _{2}\circ \left( V\otimes
p\right) \circ \left( p\otimes A^{\prime }\right) =\zeta _{1}\circ \left(
p\otimes A^{\prime }\right) =\zeta .
\end{equation*}%
Note also that $p\otimes p=\left( V\otimes p\right) \circ \left( p\otimes
A^{\prime }\right) $ is an epimorphism. Thus the right-hand side diagram in %
\eqref{coeq:furbo1} is a coequalizer.
\begin{invisible}
\begin{equation*}
A\otimes A\overset{fe\otimes fe}{\underset{f\otimes f}{\rightrightarrows }}%
A^{\prime }\otimes A^{\prime }\overset{p\otimes p}{\longrightarrow }V\otimes
V
\end{equation*}%
is a coequalizer.
\end{invisible}

2) Let $s:A^{\prime }\rightarrow A$ be such that $a\circ s=\mathrm{Id}%
=b\circ s.$ Set $f:=g\circ b$ and $e:=s\circ a.$ Then $f\circ e=g\circ
b\circ s\circ a=g\circ a$ so that $\left( ga,gb\right) =\left( fe,f\right) $
and we can apply 1). We just point out that $f\otimes A=\left( g\otimes
A\right) \circ \left( b\otimes A\right) $ is an epimorphism as a composition
of the epimorphism $g\otimes A$ by the split-epimorphism $b\otimes A$.
Similarly $A\otimes f$ is an epimorphism.

3) It is straightforward.
\begin{invisible}
Let $f,g:\left( A,m_{A},u_{A}\right) \rightarrow \left( A^{\prime
},m_{A^{\prime }},u_{A^{\prime }}\right) $ be morphisms in $\mathrm{Alg}%
\left( \mathcal{M}\right) $ such that the functor $\left( -\right) ^{\otimes
2}:\mathcal{M}\rightarrow \mathcal{M}$ preserves the coequalizer of $\left(
f^{\prime }:=\Omega f,g^{\prime }:=\Omega g\right) .$ Thus, if we denote by
\begin{equation*}
A\overset{g^{\prime }}{\underset{f^{\prime }}{\rightrightarrows }}A^{\prime }%
\overset{p^{\prime }}{\longrightarrow }V
\end{equation*}%
the coequalizer of $\left( f^{\prime },g^{\prime }\right) ,$ we get the
coequalizer%
\begin{equation*}
A\otimes A\overset{g^{\prime }\otimes g^{\prime }}{\underset{f^{\prime
}\otimes f^{\prime }}{\rightrightarrows }}A^{\prime }\otimes A^{\prime }%
\overset{p^{\prime }\otimes p^{\prime }}{\longrightarrow }V\otimes V
\end{equation*}%
Consider the following diagram
\begin{equation*}
\begin{array}{ccccc}
A\otimes A & \overset{g^{\prime }\otimes g^{\prime }}{\underset{f^{\prime
}\otimes f^{\prime }}{\rightrightarrows }} & A^{\prime }\otimes A^{\prime }
& \overset{p^{\prime }\otimes p^{\prime }}{\longrightarrow } & V\otimes V \\
m_{A}\downarrow &  & m_{A^{\prime }}\downarrow &  & \downarrow \exists m_{V}
\\
A & \overset{g^{\prime }}{\underset{f^{\prime }}{\rightrightarrows }} &
A^{\prime } & \overset{p^{\prime }}{\longrightarrow } & V%
\end{array}%
\end{equation*}%
and set%
\begin{equation*}
u_{V}:=p^{\prime }\circ u_{A^{\prime }}.
\end{equation*}%
We get%
\begin{equation*}
m_{V}\circ \left( V\otimes m_{V}\right) \circ \left( p^{\prime }\otimes
p^{\prime }\otimes p^{\prime }\right) =m_{V}\circ \left( m_{V}\otimes
V\right) \circ \left( p^{\prime }\otimes p^{\prime }\otimes p^{\prime
}\right) ,
\end{equation*}%
\begin{equation*}
m_{V}\circ \left( V\otimes u_{V}\right) \circ r_{V}^{-1}\circ p^{\prime
}=p^{\prime }=m_{V}\circ \left( u_{V}\otimes V\right) \circ l_{V}^{-1}\circ
p^{\prime },
\end{equation*}

Since the tensor products preserve coequalizers we get that $p^{\prime
}\otimes p^{\prime }\otimes V$ and $A^{\prime }\otimes A^{\prime }\otimes p$
are epimorphisms and hence also $p^{\prime }\otimes p^{\prime }\otimes
p^{\prime }$ is. Hence we deduce that $\left( V,m_{V},u_{V}\right) \in
\mathrm{Alg}\left( \mathcal{M}\right) $ and $p^{\prime }$ induces $p:\left(
A^{\prime },m_{A^{\prime }},u_{A^{\prime }}\right) \rightarrow \left(
V,m_{V},u_{V}\right) $ in $\mathrm{Alg}\left( \mathcal{M}\right) $ such that
$p^{\prime }:=\Omega p.$ Let us check that
\begin{equation*}
\left( A,m_{A},u_{A}\right) \overset{g}{\underset{f}{\rightrightarrows }}%
\left( A^{\prime },m_{A^{\prime }},u_{A^{\prime }}\right) \overset{p}{%
\longrightarrow }\left( V,m_{V},u_{V}\right)
\end{equation*}%
is a coequalizer in $\mathrm{Alg}\left( \mathcal{M}\right) .$ Clearly $p$
coequalizes the parallel pair above since $\Omega $ is faithful. Let $%
h:\left( A^{\prime },m_{A^{\prime }},u_{A^{\prime }}\right) \rightarrow
\left( E^{\prime },m_{E^{\prime }},u_{E^{\prime }}\right) $ in $\mathrm{Alg}%
\left( \mathcal{M}\right) $ be such that $h\circ g=h\circ f.$ Set $h^{\prime
}:=\Omega h.$ Then $h^{\prime }\circ g^{\prime }=h^{\prime }\circ f^{\prime
} $ and hence there is $k^{\prime }:V\rightarrow E^{\prime }$ such that $%
k^{\prime }\circ p^{\prime }=h^{\prime }.$ Clearly%
\begin{eqnarray*}
k^{\prime }\circ m_{V}\circ \left( p^{\prime }\otimes p^{\prime }\right)
&=&k^{\prime }\circ p^{\prime }\circ m_{A^{\prime }}=h^{\prime }\circ
m_{A^{\prime }}=m_{E^{\prime }}\circ \left( h^{\prime }\otimes h^{\prime
}\right) =m_{E^{\prime }}\circ \left( k^{\prime }\otimes k^{\prime }\right)
\circ \left( p^{\prime }\otimes p^{\prime }\right) , \\
k^{\prime }\circ u_{V} &=&k^{\prime }\circ p^{\prime }\circ u_{A^{\prime
}}=h^{\prime }\circ u_{A^{\prime }}=u_{E^{\prime }},
\end{eqnarray*}%
so that $k^{\prime }$ induces $k:\left( V,m_{V},u_{V}\right) \rightarrow
\left( E^{\prime },m_{E^{\prime }},u_{E^{\prime }}\right) $ in $\mathrm{Alg}%
\left( \mathcal{M}\right) $ such that $k^{\prime }:=\Omega k.$ Clearly $%
k\circ p=h.$ Moreover $p$ is an epimorphism as $p^{\prime }$ is and $\Omega $
is faithful. Thus $k$ is unique.
\end{invisible}
\end{proof}

\begin{lemma}
\label{lem:Ucoeq}The forgetful functor $U:\mathrm{Alg}^{+}\left( \mathcal{M}%
\right) \longrightarrow \mathrm{Alg}\left( \mathcal{M}\right) $ creates
colimits. Moreover the category $\mathrm{Alg}^{+}\left( \mathcal{M}\right) $
has coequalizers and $U$ preserves all coequalizers.
\end{lemma}

\begin{proof}
By Lemma \ref{lem:createscomma}, the forgetful functor $U:\mathrm{Alg}%
^{+}\left( \mathcal{M}\right) \longrightarrow \mathrm{Alg}\left( \mathcal{M}%
\right) \ $creates colimits and since $\mathrm{Alg}\left( \mathcal{M}\right)
$ has coequalizers, we deduce that $\mathrm{Alg}^{+}\left( \mathcal{M}%
\right) $ has coequalizers and $U$ preserves them.
\end{proof}

\begin{lemma}
The forgetful functor $\mho ^{+}:\mathrm{Bialg}\left( \mathcal{M}\right)
\longrightarrow \mathrm{Alg}^{+}\left( \mathcal{M}\right) $ preserves
coequalizers.
\end{lemma}

\begin{proof}
Since $\mho :\mathrm{Bialg}\left(
\mathcal{M}\right) \rightarrow \mathrm{Alg}\left( \mathcal{M}\right) $
creates colimits, we have that $\mho$ preserves all colimits that exist in $\mathrm{Alg}\left(
\mathcal{M}\right) $, see \cite[11.5, page 106]{McLarty}. Since $\mathrm{Alg}\left(
\mathcal{M}\right) $ has all coequalizers, we get that $\mho$ preserves all coequalizers. Since $\mho=U\mho ^{+}$ and, by Lemma \ref{lem:Ucoeq}, $U$ creates, whence reflects
coequalizers \cite[Exercice 1, page 150]{MacLane}, we get that $\mho ^{+}$ preserves coequalizers as desired.
\begin{invisible}
Consider a coequalizer $\xymatrixcolsep{.5cm} \xymatrix{X\ar@<.5ex>[r]^{f}%
\ar@<-.5ex>[r]_{g}&Y\ar[r]^{c}& C}$
$X\overset{f}{\underset{g}{\rightrightarrows }}Y\overset{c}{\rightarrow }C$
in $\mathrm{Bialg}\left( \mathcal{M}\right) .$ Since $\mathrm{Alg}\left(
\mathcal{M}\right) $ has coequalizers, and $\mho :\mathrm{Bialg}\left(
\mathcal{M}\right) \rightarrow \mathrm{Alg}\left( \mathcal{M}\right) $
creates colimits, there is a coequalizer $\xymatrixcolsep{.5cm} %
\xymatrix{X\ar@<.5ex>[r]^{f}\ar@<-.5ex>[r]_{g}&Y\ar[r]^{c'}& C'}$
$X\overset{f}{\underset{g}{\rightrightarrows }}Y\overset{c^{\prime }}{%
\rightarrow }C^{\prime }$
which is preserved by $\mho .$ By uniqueness of coequalizers this
coequalizer is $\xymatrixcolsep{.5cm} \xymatrix{X\ar@<.5ex>[r]^{f}%
\ar@<-.5ex>[r]_{g}&Y\ar[r]^{c}& C}$.
$X\overset{f}{\underset{g}{\rightrightarrows }}Y\overset{c}{\rightarrow }C.$
Since $\xymatrixcolsep{.7cm}
\xymatrix{U\mho ^{+}X\ar@<.5ex>[r]^{U\mho
^{+}f}\ar@<-.5ex>[r]_{U\mho ^{+}g}&U\mho ^{+}Y\ar[r]^{U\mho ^{+}c}& U\mho
^{+}C}$
$U\mho ^{+}X\overset{U\mho ^{+}f}{\underset{U\mho ^{+}g}{\rightrightarrows }}%
U\mho ^{+}Y\overset{U\mho ^{+}c}{\rightarrow }U\mho ^{+}C$
is a coequalizer and $U$ creates (Lemma \ref{lem:Ucoeq}), whence reflects
coequalizers \cite[Exercice 1, page 150]{MacLane}, we get that $%
\xymatrixcolsep{.7cm}
\xymatrix{\mho ^{+}X\ar@<.5ex>[r]^{\mho
^{+}f}\ar@<-.5ex>[r]_{\mho ^{+}g}&\mho ^{+}Y\ar[r]^{\mho ^{+}c}& \mho ^{+}C}$
$\mho ^{+}X\overset{\mho ^{+}f}{\underset{\mho ^{+}g}{\rightrightarrows }}%
\mho ^{+}Y\overset{\mho ^{+}c}{\rightarrow }\mho ^{+}C$
is a coequalizer.\end{invisible}
\end{proof}

\begin{corollary}
\label{coro:OmUReg}$\Omega U:\mathrm{Alg}^{+}\left( \mathcal{M}\right)
\longrightarrow \mathcal{M}$ preserves coequalizers for pairs $\left(
fe,f\right) $ where $f:A\rightarrow A^{\prime }$ is composition of regular
epimorphisms in $\mathrm{Alg}^{+}\left( \mathcal{M}\right) $ and $%
e:A\rightarrow A$ is a morphism in $\mathrm{Alg}^{+}\left( \mathcal{M}%
\right) $ such that $f\circ e\circ e=f\circ e.$
\end{corollary}

\begin{proof}
Consider in $\mathrm{Alg}^{+}\left( \mathcal{M}\right) $ the following
left-hand side coequalizer.
\begin{equation*}
\xymatrixcolsep{.7cm} \xymatrix{A\ar@<.5ex>[r]^{f\circ
e}\ar@<-.5ex>[r]_{f}&A'\ar[r]^{p}& C} \qquad \xymatrixcolsep{.7cm} %
\xymatrix{UA\ar@<.5ex>[r]^{Uf\circ Ue}\ar@<-.5ex>[r]_{Uf}&UA'\ar[r]^{Up}& UC}
\end{equation*}

\begin{invisible}
\begin{equation*}
A\overset{f\circ e}{\underset{f}{\rightrightarrows }}A^{\prime }\overset{p}{%
\longrightarrow }C.
\end{equation*}
\end{invisible}

By Lemma \ref{lem:Ucoeq}, $U$ preserves coequalizers so that also the
right-hand side one is a coequalizer.

\begin{invisible}
we have the coequalizer%
\begin{equation*}
UA\overset{Uf\circ Ue}{\underset{Uf}{\rightrightarrows }}UA^{\prime }\overset%
{Up}{\longrightarrow }UC.
\end{equation*}
\end{invisible}

Since $\Omega U$ preserves the regular epimorphisms (as $U$ preserves
coequalizers and $\Omega $ preserves regular epimorphisms by Lemma \ref%
{lem:Omegareg}), we get that $\Omega Uf$ is composition of regular
epimorphisms. Since the tensor products preserves coequalizers, $%
\Omega UA\otimes \Omega Uf$ and $\Omega Uf\otimes \Omega UA$ are
epimorphisms.

Since $\Omega Uf\circ \Omega Ue\circ \Omega Ue=\Omega Uf\circ \Omega Ue$ and
$\Omega UA\otimes \Omega Uf,\Omega Uf\otimes \Omega UA$ are epimorphisms, we
can apply Lemma \ref{lem:furbo}-1) to $"f"=\Omega Uf$ and $"e"=\Omega Ue$ to
get that the coequalizer of $\left( \Omega Uf\circ \Omega Ue,\Omega
Uf\right) $ is preserved by $\left( -\right) ^{\otimes 2}$ and hence, by
Lemma \ref{lem:furbo}-3), $\Omega $ creates the coequalizer of $\left(
Uf\circ Ue,Uf\right) .$ As a consequence the above right-hand side displayed
coequalizer is preserved by $\Omega .$ Hence $\Omega U$ preserves the
starting coequalizer.
\end{proof}

\begin{proposition}
\label{pro:cosplitnat}Let $\nu :F\rightarrow G$ and $\tau :G\rightarrow F$
be natural transformations such that $\tau \circ \nu =\mathrm{Id}.$ Then $F$
preserves those colimits which are preserved by $G.$ Moreover $F$ preserves
regular epimorphisms which are preserved by $G.$
\end{proposition}

\begin{proof}
The first part follows from \cite[Lemma 1.7]{MW}. By the dual argument used therein, we have that $\tau$ is the coequalizer of the parallel pair $(\nu\circ\tau,\id)$. Therefore, given a morphism $p:A\rightarrow B$ such that $Gp$ is a regular epimorphism, the two rows in the following diagram are coequalizers.
\begin{equation*}
  \xymatrixcolsep{1.5cm} \xymatrixrowsep{.7cm}
  \xymatrix{
  GA\ar@{->>}[d]_{Gp}\ar@<+.5ex>[r]^{\nu A\circ\tau A}\ar@<-.5ex>[r]_{\id_{GA}}&GA\push\ar[d]_{Gp}\ar[r]^{\tau A}& FA\ar[d]^{Fp}\\
  GB\ar@<+.5ex>[r]^{\nu B\circ\tau B}\ar@<-.5ex>[r]_{\id_{GB}}&GB\ar[r]_-{\tau B }& FB}
\end{equation*}
Since $Gp$ is an epimorphism, by a well-known result (see e.g. \cite[Proposition 2.4]{Ho}), we have that the right square above is a pushout. Hence, by \cite[Proposition 4.3.8]{Borceux1}, we conclude that $Fp$ is a regular
epimorphism.
\begin{invisible}
\begin{equation*}
\xymatrixcolsep{.5cm} \xymatrixrowsep{.5cm}\xymatrix{GA\push\ar[d]_{Gp}\ar[r]^{\tau A}& FA\ar@/^1.5pc/[rdd]^-{h}\ar[d]^{Fp}\\ GB\ar@/_1pc/[rrd]_g\ar[r]_-{\tau B }& FB\ar@{.>}[dr]^f\\&&X}
\end{equation*}
Let $g,h$ be morphisms such that $g\circ Gp=h \circ \tau A$ and set $f:=g\circ\nu B$. We get $f\circ Fp=g\circ\nu B\circ Fp=g\circ Gp\circ\nu A=h\circ \tau A\circ\nu A=h$ and $f\circ\tau B\circ Gp=f\circ Fp\circ\tau A=h\circ\tau A=g\circ Gp.$ Since $Gp$ is an epimorphism, we get $f\circ\tau B=g$. Thus $f$ is unique as $\tau A$ is a split-epimorphism.
\end{invisible}
\end{proof}

Consider the natural transformation $\xi :P\rightarrow \Omega \mho $
defined, as in Theorem \ref{thm:main}, by%
\begin{equation*}
P\overset{\eta P}{\rightarrow }\Omega TP=\Omega \mho \widetilde{T}P\overset{%
\Omega \mho \widetilde{\epsilon }}{\rightarrow }\Omega \mho .
\end{equation*}%
As in the proof of the above theorem, we have (\ref{form:efeps}) that in
local notations becomes%
\begin{equation}
\epsilon \mho \circ T\xi =\mho \widetilde{\epsilon }.  \label{form:epstilde}
\end{equation}%
so that $\xi $ is exactly the natural transformation of \cite[Theorem 4.6]%
{AM-BraidedOb}, whose components are the canonical inclusions of the
subobject of primitives of a bialgebra $B$ in $\mathcal{M}$ into $\Omega
\mho B$ and hence they are regular monomorphisms.

\begin{invisible}
Note that the monad associated to $\left( T^{+},\Omega ^{+}\right) $ is
\begin{equation*}
\Omega ^{+}T^{+}M=\Omega ^{+}\left( TM,\varepsilon _{TM}\right) =\mathrm{Ker}%
\left( \Omega \varepsilon _{TM}\right) .
\end{equation*}
\end{invisible}

Since $UT^{+}=T,$ we can define
\begin{equation*}
\zeta :=\left( \Omega ^{+}\overset{\eta \Omega ^{+}}{\rightarrow }\Omega
T\Omega ^{+}=\Omega UT^{+}\Omega ^{+}\overset{\Omega U\epsilon ^{+}}{%
\rightarrow }\Omega U\right) .
\end{equation*}%
Given $\left( A,\varepsilon \right) \in \mathrm{Alg}^{+}\left( \mathcal{M}%
\right) ,$ we have $\zeta \left( A,\varepsilon \right) :\mathrm{Ker}\left(
\Omega \varepsilon \right) \rightarrow \Omega A.$

\begin{remark}
We compute%
\begin{eqnarray*}
\zeta \left( A,\varepsilon \right) &=&\left( \Omega U\epsilon ^{+}\circ \eta
\Omega ^{+}\right) \left( A,\varepsilon \right) =\Omega U\epsilon ^{+}\left( A,\varepsilon \right) \circ \eta \Omega
^{+}\left( A,\varepsilon \right) \\
&\overset{(\ref{form:etaeps+})}{=}&\Omega \left( \epsilon A\circ TkA\right)
\circ \eta \Omega ^{+}\left( A,\varepsilon \right) =\Omega \epsilon A\circ \Omega TkA\circ \eta \Omega ^{+}\left(
A,\varepsilon \right) =\Omega \epsilon A\circ \eta \Omega A\circ kA=kA.
\end{eqnarray*}%
where $kA$ is the morphism in diagram \eqref{diag:defKaug}. Thus
\begin{equation}
\zeta \left( A,\varepsilon \right) =kA.  \label{form:zetakappa}
\end{equation}
\end{remark}

\begin{invisible}
Questions: $\Omega ^{+}$ is monadic? AMMESSO\ CHE\ SERVA\ si\ potrebbe
provare a ricalcare la dim vista per $\Omega $. In general, if $R$ is
monadic what can we say about $L/ \mathbf{1}$ and $R/ \mathbf{1}? $
\end{invisible}

\begin{lemma}
\label{lem:cosplitzeta}There is a natural transformation $\tau :\Omega U\rightarrow
\Omega ^{+}$ such that $\tau \circ \zeta =\mathrm{Id}$. As a consequence $%
\Omega ^{+}:\mathrm{Alg}^{+}\left( \mathcal{M}\right) \rightarrow \mathcal{M}
$ preserves coequalizers for pairs $\left( fe,f\right) $ where $%
f:A\rightarrow A^{\prime }$ is composition of regular epimorphisms in $%
\mathrm{Alg}^{+}\left( \mathcal{M}\right) $ and $e:A\rightarrow A$ is a
morphism in $\mathrm{Alg}^{+}\left( \mathcal{M}\right) $ such that $f\circ
e\circ e=f\circ e.$
Moreover $\Omega ^{+}$ preserves regular epimorphisms.
\end{lemma}

\begin{proof}
Let $\left( A,\varepsilon \right) \in \mathrm{Alg}^{+}\left( \mathcal{M}%
\right) .$ As observed, the pullback \eqref{diag:defKaug} means that
 $\Omega ^{+}\left( A,\varepsilon \right) =KA=\mathrm{Ker%
}\left( \Omega \varepsilon \right) .$
\begin{invisible}
\begin{equation*}
\begin{array}{ccc}
KA & \overset{tKA}{\longrightarrow } & \mathbf{0} \\
kA\downarrow & \lrcorner & \downarrow \eta \mathbf{0}=i\Omega \mathbf{1} \\
\Omega A & \overset{\Omega \varepsilon }{\longrightarrow } & \Omega \mathbf{%
1\cong }\Omega T\mathbf{0}%
\end{array}%
\end{equation*}
\end{invisible}
The canonical inclusion is $kA$ which by (\ref{form:zetakappa}) equals $\zeta(A,\varepsilon)$. Thus we have the following kernel in $%
\mathcal{M}.$
\begin{equation*}
0\rightarrow \Omega ^{+}\left( A,\varepsilon \right) \overset{\zeta \left(
A,\varepsilon \right) }{\longrightarrow }\Omega A\overset{\Omega \varepsilon
}{\longrightarrow }\Omega \mathbf{1}=\mathbf{1}
\end{equation*}%
Since $\varepsilon $ is an algebra morphism, we have $\Omega \varepsilon
\circ u_{\Omega A}=\mathrm{Id}.$ Hence $\Omega \varepsilon \circ \left(
\mathrm{Id}_{\Omega A}-u_{\Omega A}\circ \Omega \varepsilon \right) =0$ so
that, by the universal property of the kernel we get a unique morphism $\tau
\left( A,\varepsilon \right) :\Omega A\rightarrow \Omega ^{+}\left(
A,\varepsilon \right) $ such that $\zeta \left( A,\varepsilon \right) \circ
\tau \left( A,\varepsilon \right) =\mathrm{Id}_{\Omega A}-u_{\Omega A}\circ
\Omega \varepsilon .$ Moreover $\tau \left( A,\varepsilon \right) \circ
\zeta \left( A,\varepsilon \right) =\mathrm{Id}_{\Omega ^{+}\left(
A,\varepsilon \right) }.$

\begin{invisible}
$\zeta \left( A,\varepsilon \right) \circ \tau \left( A,\varepsilon \right)
\circ \zeta \left( A,\varepsilon \right) =\left( \mathrm{Id}_{\Omega
A}-u_{\Omega A}\circ \Omega \varepsilon \right) \circ \zeta \left(
A,\varepsilon \right) =\mathrm{Id}_{\Omega A}=\zeta \left( A,\varepsilon
\right) \circ \mathrm{Id}_{\Omega ^{+}\left( A,\varepsilon \right) }.$ Since
$\zeta \left( A,\varepsilon \right) $ is a monomorphism, we get $\tau \left(
A,\varepsilon \right) \circ \zeta \left( A,\varepsilon \right) =\mathrm{Id}%
_{\Omega ^{+}\left( A,\varepsilon \right) }.$
\end{invisible}

It remains to check that $\tau \left( A,\varepsilon \right) $ is natural in $%
\left( A,\varepsilon \right) $. To this aim, first let $f:\left(
A,\varepsilon _{A}\right) \rightarrow \left( B,\varepsilon _{B}\right) $ be
a morphism in $\mathrm{Alg}^{+}\left( \mathcal{M}\right) $ and compute%
\begin{eqnarray*}
\zeta \left( B,\varepsilon _{B}\right) \circ \tau \left( B,\varepsilon
_{B}\right) \circ \Omega Uf &=&\left( \mathrm{Id}_{\Omega B}-u_{\Omega
B}\circ \Omega \varepsilon _{B}\right) \circ \Omega Uf =\Omega Uf-u_{\Omega B}\circ \Omega \varepsilon _{B}\circ \Omega Uf \\
&=&\Omega Uf-u_{\Omega B}\circ \Omega \left( \varepsilon _{B}\circ Uf\right)
=\Omega Uf-u_{\Omega B}\circ \Omega \varepsilon _{A} \\
&=&\Omega Uf-\Omega Uf\circ u_{\Omega A}\circ \Omega \varepsilon _{A} =\Omega Uf\circ \left( \mathrm{Id}_{\Omega A}-u_{\Omega A}\circ \Omega
\varepsilon _{A}\right) \\
&=&\Omega Uf\circ \zeta \left( A,\varepsilon _{A}\right) \circ \tau \left(
A,\varepsilon \right) =\Omega Uf\circ kA\circ \tau \left( A,\varepsilon \right) \\
&\overset{(\ref{form:Omega+morph})}{=}&kA^{\prime }\circ \Omega ^{+}f\circ
\tau \left( A,\varepsilon \right) =\zeta \left( B,\varepsilon _{B}\right) \circ \Omega ^{+}f\circ \tau
\left( A,\varepsilon \right) .
\end{eqnarray*}%
Since $\zeta \left( B,\varepsilon _{B}\right) $ is a monomorphism we deduce $%
\tau \left( B,\varepsilon _{B}\right) \circ \Omega Uf=\Omega ^{+}f\circ \tau
\left( A,\varepsilon \right) $ which means that $\tau $ is natural. Thus $%
\zeta :\Omega ^{+}\rightarrow \Omega U$ cosplits via $\tau :\Omega
U\rightarrow \Omega ^{+}$ i.e. $\tau \circ \zeta =\mathrm{Id}$.

Now, by Lemma \ref{lem:Omegareg}, the functor $\Omega :\mathrm{Alg}\left(
\mathcal{M}\right) \rightarrow \mathcal{M}$ preserves regular epimorphisms.

By Lemma \ref{lem:Ucoeq}, the forgetful functor $U:\mathrm{Alg}^{+}\left(
\mathcal{M}\right) \longrightarrow \mathrm{Alg}\left( \mathcal{M}\right) \ $%
creates colimits and preserves all coequalizers. As a consequence $\Omega U:%
\mathrm{Alg}^{+}\left( \mathcal{M}\right) \longrightarrow \mathcal{M}$
preserves regular epimorphisms. Hence by Proposition \ref{pro:cosplitnat}
also $\Omega ^{+}$ preserves regular epimorphisms. By Corollary \ref%
{coro:OmUReg}, the functor $\Omega U:\mathrm{Alg}^{+}\left( \mathcal{M}%
\right) \longrightarrow \mathcal{M}$ preserves coequalizers for pairs $%
\left( fe,f\right) $ where $f:A\rightarrow A^{\prime }$ is composition of
regular epimorphisms in $\mathrm{Alg}^{+}\left( \mathcal{M}\right) $ and $%
e:A\rightarrow A$ is a morphism in $\mathrm{Alg}^{+}\left( \mathcal{M}%
\right) $ such that $f\circ e\circ e=f\circ e.$ By Proposition \ref%
{pro:cosplitnat}, the functor $\Omega ^{+}$ preserves the same type of
coequalizers.
\end{proof}

Next aim is to show that the functor $T^{+}:\mathcal{M}\rightarrow \mathrm{%
Alg}^{+}\left( \mathcal{M}\right) $ is h-separable. First
note that there is a unique morphism $\omega V:\Omega TV\rightarrow V$ such
that%
\begin{equation}
\omega V\circ \alpha _{n}V=\delta _{n,1}\mathrm{Id}_{V}.
\label{form:gamalpha}
\end{equation}%
Given $f:V\rightarrow W$ a morphism in $\mathcal{M},$ we get for every $n\in
\mathbb{N},$%
\begin{equation*}
\omega W\circ \Omega Tf\circ \alpha _{n}V=\omega W\circ \alpha _{n}W\circ
f^{\otimes n}=\delta _{n,1}f^{\otimes n}=\delta _{n,1}f=f\circ \omega V\circ
\alpha _{n}V
\end{equation*}%
so that $\omega W\circ \Omega Tf=f\circ \omega V$ which means that $\omega
:=\left( \omega V\right) _{V\in \mathcal{M}}$ is a natural transformation $%
\omega :\Omega T\rightarrow \mathrm{Id}_{\mathcal{M}}.$

\begin{lemma}
\label{lem:omega} The natural transformation $\omega $ fulfills $\omega
\circ \eta =\mathrm{Id}$ and%
\begin{equation}
\omega \omega \circ \Omega T\zeta T^{+}=\omega \circ \Omega \epsilon T\circ
\Omega T\zeta T^{+}.  \label{form:omega}
\end{equation}
\end{lemma}

\begin{proof}
\begin{invisible}
Note that $U\epsilon ^{+}\left( A,\varepsilon \right) =\epsilon A\circ
TkA=\epsilon U\left( A,\varepsilon \right) \circ T\zeta \left( A,\varepsilon
\right) $ so that
\begin{equation*}
U\epsilon ^{+}=\epsilon U\circ T\zeta .
\end{equation*}%
Hence $\omega \circ \Omega \epsilon T\circ \Omega T\zeta T^{+}=\omega \circ
\Omega \epsilon UT^{+}\circ \Omega T\zeta T^{+}=\omega \circ \Omega \left(
\epsilon U\circ T\zeta \right) T^{+}=\omega \circ \Omega \left( \epsilon
U\circ T\zeta \right) T^{+}=\omega \circ \Omega U\epsilon ^{+}T^{+}.$
Moreover $T^{+}\left( \omega \circ \zeta T^{+}\right) =?=\epsilon
^{+}T^{+}\Leftrightarrow \Omega ^{+}T^{+}\left( \omega \circ \zeta
T^{+}\right) \circ \eta ^{+}\Omega ^{+}T^{+}=?=\Omega ^{+}\epsilon
^{+}T^{+}\circ \eta ^{+}\Omega ^{+}T^{+}$
\end{invisible}

\begin{invisible}
$\Leftrightarrow \eta ^{+}\circ \omega \circ \zeta T^{+}=?=\mathrm{Id.}$
Never true because $\eta ^{+}$ is an inclusion not surjective in general.
\end{invisible}

In \cite[Lemma 5.2]{AM-heavy} we prove that
$
\omega \omega \circ \Omega T\zeta ^{\prime }\widetilde{T}=\omega \circ
\Omega \epsilon T\circ \Omega T\zeta ^{\prime }\widetilde{T}
$
where $\zeta ^{\prime }:E\rightarrow \Omega \mho $ is a natural
transformation whose domain is the functor $E:\mathrm{Bialg}\left( \mathcal{M%
}\right) \rightarrow \mathcal{M}$ assigning to each bialgebra $A$ the kernel
$\left( EA,\zeta ^{\prime }A:EA\rightarrow \Omega \mho A\right) $ in $%
\mathcal{M}$ of its counit $\Omega \varepsilon _{\mho A},$ where here $%
\varepsilon _{\mho A}$ is regarded as an algebra map. Then, for every $M\in
\mathcal{M}$, we have%
\begin{eqnarray*}
\left( E\widetilde{T}M,\zeta ^{\prime }\widetilde{T}M\right) =\mathrm{Ker}%
\left( \Omega \varepsilon _{\mho \widetilde{T}M}\right) =\mathrm{Ker}\left(
\Omega \varepsilon _{TM}\right) =\left( \Omega ^{+}\left( TM,\varepsilon _{TM}\right) ,kTM\right) \overset{%
\left( \ref{form:zetakappa}\right) }{=}\left( \Omega ^{+}T^{+}M,\zeta
T^{+}M\right) .
\end{eqnarray*}%
Moreover, given $f:M\rightarrow N$, since, by (\ref{form:zetakappa}) we have
$kTM=\zeta T^{+}M,$ we obtain
\begin{eqnarray*}
\zeta T^{+}N\circ \Omega ^{+}T^{+}f &=&kTN\circ \Omega ^{+}T^{+}f=\Omega
UT^{+}f\circ kTM=\Omega Tf\circ \zeta T^{+}M \\
&=&\Omega \mho \widetilde{T}f\circ \zeta ^{\prime }\widetilde{T}M=\zeta
^{\prime }\widetilde{T}N\circ E\widetilde{T}f=\zeta T^{+}N\circ E\widetilde{T%
}f
\end{eqnarray*}
so that $\Omega ^{+}T^{+}f=E\widetilde{T}f.$ As a consequence $E\widetilde{T}%
=\Omega ^{+}T^{+}$ and $\zeta ^{\prime }\widetilde{T}=\zeta T^{+}.$ If we
substitute $\zeta ^{\prime }\widetilde{T}$ by $\zeta T^{+}$ in the starting
equality we obtain the desired one.
\end{proof}

The following result shows that $T^{+}$ is h-separable

\begin{lemma}
\label{lem:augmtilde}$\zeta T^{+}:\Omega ^{+}T^{+}\rightarrow
\Omega T$ is a monad morphism between the monads associated to $\left( T^{+},\Omega
^{+}\right) $ and $\left( T,\Omega \right) .$ Moreover
$\omega ^{+}:=\omega \circ \zeta T^{+}:\Omega
^{+}T^{+}\rightarrow \mathrm{Id}$ is an augmentation for the monad
associated to the adjunction $\left( T^{+},\Omega ^{+}\right) .$ Equivalently $T^{+}$ is h-separable.
\end{lemma}

\begin{proof}
The monads to consider are $\left( \Omega ^{+}T^{+},\Omega ^{+}\epsilon
^{+}T^{+},\eta ^{+}\right) $ and $\left( \Omega T,\Omega \epsilon T,\eta
\right) .$ We compute%
\begin{eqnarray*}
\zeta T^{+}\circ \Omega ^{+}\epsilon ^{+}T^{+}&\overset{\text{nat. }\zeta }{=}&%
\Omega U\epsilon ^{+}T^{+}\circ \zeta T^{+}\Omega ^{+}T^{+}\overset{(\ref%
{form:etaeps+}),(\ref{form:zetakappa})}{=}\Omega \epsilon UT^{+}\circ \Omega
T\zeta T^{+}\circ \zeta T^{+}\Omega ^{+}T^{+}=\Omega \epsilon T\circ \zeta
T^{+}\zeta T^{+}\\
\zeta T^{+}\circ \eta ^{+}&\overset{(\ref{form:zetakappa})}{=}&kUT^{+}\circ
\eta ^{+}=kT\circ \eta ^{+}\overset{(\ref{form:etaeps+})}{=}\eta .
\end{eqnarray*}
We have so proved that $\zeta T^{+}:\Omega ^{+}T^{+}\rightarrow \Omega T$ is
a morphism of monads. We compute%
\begin{eqnarray*}
\omega ^{+}\omega ^{+} &=&\omega \omega \circ \zeta T^{+}\zeta T^{+}=\omega
\omega \circ \Omega T\zeta T^{+}\circ \zeta T^{+}\Omega ^{+}T^{+} \overset{(%
\ref{form:omega})}{=}\omega \circ \Omega \epsilon T\circ \Omega T\zeta
T^{+}\circ \zeta T^{+}\Omega ^{+}T^{+}, \\
&=&\omega \circ \Omega \epsilon T\circ \zeta T^{+}\zeta T^{+}=\omega \circ
\zeta T^{+}\circ \Omega ^{+}\epsilon ^{+}T^{+}=\omega ^{+}\circ \Omega
^{+}\epsilon ^{+}T^{+}.
\end{eqnarray*}%
Moreover $\omega ^{+}\circ \eta ^{+}=\omega \circ \zeta T^{+}\circ \eta
^{+}=\omega \circ \eta =\mathrm{Id}.$ Thus $\omega ^{+}$ is an
augmentation for the monad $\left( \Omega ^{+}T^{+},\Omega ^{+}\epsilon
^{+}T^{+},\eta ^{+}\right) .$ By \cite[Corollary 2.7]{AM-heavy}, this means that
$T^{+}$ is h-separable.
\end{proof}

As a consequence of the results above, Theorem \ref{thm:main}
applies to the leftmost diagram in \eqref{square:T}.

\begin{invisible}
\begin{equation*}
\xymatrixcolsep{1.5cm}\xymatrixrowsep{0.7cm}\xymatrix{\Bialg(\M)\ar[r]^-{%
\mho ^{+}}\ar@<.5ex>@{.>}[d]^{P}&\Alg^+(\M)\ar@<.5ex>@{.>}[d]^{\Omega
^{+}}\\ \M\ar[r]^{\id}\ar@<.5ex>[u]^{\widetilde{T}}&\M\ar@<.5ex>[u]^{T^+}}
\end{equation*}\begin{equation*}
\begin{array}{ccc}
\mathrm{Bialg}\left( \mathcal{M}\right) & \overset{\mho ^{+}}{%
\longrightarrow } & \mathrm{Alg}^{+}\left( \mathcal{M}\right) \\
\widetilde{T}\uparrow \downarrow P &  & T^{+}\uparrow \downarrow \Omega ^{+}
\\
\mathcal{M} & \overset{\mathrm{Id}}{\longrightarrow } & \mathcal{M}%
\end{array}%
\end{equation*}
\end{invisible}

\section{Conclusions\label{sec:7}}

In this section we collect some fallouts of Theorem \ref{thm:main}. We
describe explicitly the functor $\Gamma _{\left[ n\right] }$ in case of
Yetter-Drinfeld modules and in particular of vector spaces. We infer an
analogue of the notion of combinatorial rank and we propose possible lines
of future investigation on the subject.

\begin{example}
\label{ex:YD}Let $H$ be a finite-dimensional Hopf algebra over a field $%
\Bbbk $. We want to apply the results of the previous sections in the case
when $\mathcal{M}$ is the category $_{H}^{H}\mathcal{YD}$ of (left-left)
Yetter-Drinfeld modules over $H.$ This category is braided as the antipode
of $H$ is invertible. Moreover $_{H}^{H}\mathcal{YD}$ satisfies all the
requirements of Section \ref{sec:6}. The related diagram rewrites as follows
and fulfills the assumptions of Theorem \ref{thm:main}.%
\begin{equation*}
\xymatrixcolsep{1.5cm}\xymatrixrowsep{0.7cm}\xymatrix{\Bialg(\yd)\ar[r]^-{%
\mho ^{+}}\ar@<.5ex>@{.>}[d]^{P}&\Alg^+(\yd)\ar@<.5ex>@{.>}[d]^{\Omega
^{+}}\\ \yd\ar[r]^{\id}\ar@<.5ex>[u]^{\widetilde{T}}&\yd\ar@<.5ex>[u]^{T^+}}
\end{equation*}

\begin{invisible}
\begin{equation*}
\begin{array}{ccc}
\mathrm{Bialg}\left( _{H}^{H}\mathcal{YD}\right) & \overset{\mho ^{+}}{%
\longrightarrow } & \mathrm{Alg}^{+}\left( _{H}^{H}\mathcal{YD}\right) \\
\widetilde{T}\uparrow \downarrow P &  & T^{+}\uparrow \downarrow \Omega ^{+}
\\
_{H}^{H}\mathcal{YD} & \overset{\mathrm{Id}}{\longrightarrow } & _{H}^{H}%
\mathcal{YD}%
\end{array}%
\end{equation*}
\end{invisible}

Let $V\in {}_{H}^{H}\mathcal{YD}$ and $\left( A,\varepsilon \right) \in
\mathrm{Alg}^{+}\left( _{H}^{H}\mathcal{YD}\right) .$ The object $\widetilde{%
T}V$ is the usual tensor algebra $TV$ that becomes the tensor algebra in $%
_{H}^{H}\mathcal{YD}$, as $V$ belongs to $_{H}^{H}\mathcal{YD}$, and that is
endowed with a braided bialgebra structure by means of its universal
property and the braiding of $V$. By definition
\begin{equation*}
T^{+}V=\left( TV,\varepsilon _{TV}\right) ,\qquad \left( \Omega ^{+}\left(
A,\varepsilon \right) ,kA\right) =\mathrm{Ker}\left( \Omega \varepsilon
\right) .
\end{equation*}%
Thus $\Omega ^{+}\left( A,\varepsilon \right) $ is nothing but the
augmentation ideal $A^{+}$ regarded as an object in $_{H}^{H}\mathcal{YD}$
being the kernel of $\varepsilon $ which is a morphism in this category. The
monad $\Omega ^{+}T^{+}$ is augmented via the morphism $\omega ^{+}:\Omega
^{+}T^{+}\rightarrow \mathrm{Id}$ of Lemma \ref{lem:augmtilde}. Note that,
for any $V\in {_{H}^{H}\mathcal{YD}},$ the map $\omega ^{+}V:\Omega
^{+}T^{+}V\rightarrow V$ is just the restriction to $\Omega
^{+}T^{+}V=\left( TV\right) ^{+}$ of the canonical projection $\omega
V:\Omega TV\rightarrow V$ onto $V$ (note that $\omega V$ is not in general
an augmentation for $\Omega T$ because $T$ is not h-separable \cite[%
Corollary 2.7 and Remark 5.3]{AM-heavy}).

By Theorem \ref{thm:main}, also the monad $P\widetilde{T}$ is augmented via $%
\gamma :=\omega ^{+}\circ \xi \widetilde{T}:P\widetilde{T}\rightarrow
\mathrm{Id}$. Explicitly $\gamma V$ is the restriction of $\omega V$ to the
Yetter-Drinfeld submodule of primitive elements of $\widetilde{T}V$. For
every $n\in \mathbb{N}$, there are a functor $\Gamma _{\left[ n\right]
}:{}_{H}^{H}\mathcal{YD}\rightarrow {}_{H}^{H}\mathcal{YD}_{\left[ n\right]
} $ and a natural transformation $\gamma _{\left[ n\right] }:P\widetilde{T}_{%
\left[ n\right] }\Gamma _{\left[ n\right] }\rightarrow \mathrm{Id},$ such
that $\Gamma _{\left[ 0\right] }:=\mathrm{Id},\gamma _{\left[ 0\right]
}:=\gamma $ and, for $n\geq 0,\Gamma _{\left[ n+1\right] }V=\left( \Gamma _{%
\left[ n\right] }V,\gamma _{\left[ n\right] }V\right) \in {}_{H}^{H}\mathcal{%
YD}_{\left[ n+1\right] }$ and $\gamma _{\left[ n\right] }\circ U_{\left[ n%
\right] }\widetilde{\eta} _{\left[ n\right] }\Gamma _{\left[ n\right] }=%
\mathrm{Id}$. Let us describe explicitly the functor
\begin{equation*}
S_{\left[ n\right] }:=\widetilde{T}_{\left[ n\right] }\Gamma _{\left[ n%
\right] }:{}_{H}^{H}\mathcal{YD}\rightarrow \mathrm{Bialg}\left( _{H}^{H}%
\mathcal{YD}\right).
\end{equation*}%
For $n=0$ we have $S_{\left[ 0\right] }:=\widetilde{T}_{\left[ 0\right]
}\Gamma _{\left[ 0\right] }=\widetilde{T}.$ Moreover $S_{\left[ n+1\right]
}V $ is given by the coequalizer, see \eqref{coeq:main1}:
\begin{equation*}
\xymatrixcolsep{2.3cm} \xymatrix{\widetilde{T}PS_{[n]}V\ar@<.5ex>[r]^-{%
\pi_{[n]}\Gamma_{[n]}V\circ
\widetilde{T}\gamma_{[n]}V}\ar@<-.5ex>[r]_{\widetilde{\epsilon}
S_{[n]}V}&S_{[n]}V\ar[r]^-{\pi _{[n,n+1]}\Gamma_{[n+1]}V}& S_{[n+1]}V}
\end{equation*}

\begin{invisible}
\begin{equation*}
\widetilde{T}PS_{\left[ n\right] }V\overset{\pi _{\left[ n\right] }\Gamma _{%
\left[ n\right] }V\circ \widetilde{T}\gamma _{\left[ n\right] }V}{\underset{%
\widetilde{\epsilon }S_{\left[ n\right] }V}{\rightrightarrows }}S_{\left[ n%
\right] }V\overset{\pi _{\left[ n,n+1\right] }\Gamma _{\left[ n+1\right] }V}{%
\longrightarrow }S_{\left[ n+1\right] }V.
\end{equation*}
\end{invisible}

By Lemma \ref{lem:rank}, a bialgebra map $f:S_{\left[ n\right] }V\rightarrow
B$ in $_{H}^{H}\mathcal{YD}$ together with the above parallel pair is a fork if and only if $%
Pf\circ e_{\left[ n\right] }V=Pf,$ where $e_{\left[ n\right] }:=U_{\left[ n%
\right] }\widetilde{\eta }_{\left[ n\right] }\Gamma _{\left[ n\right] }\circ
\gamma _{\left[ n\right] }.$ Since $Pf:PS_{\left[ n\right] }V\rightarrow PB$
is just the restriction of $f$ to the primitive elements, we get that $S_{%
\left[ n+1\right] }V$ is obtained by factoring out $S_{\left[ n\right] }V$
by its two-sided ideal generated by $\mathrm{Im}\left( \mathrm{Id}-e_{\left[
n\right] }V\right) .$ Since $e_{\left[ n\right] }V$ is idempotent, we have
that $\mathrm{Im}\left( \mathrm{Id}-e_{\left[ n\right] }V\right) =\mathrm{Ker%
}\left( e_{\left[ n\right] }V\right) $. By definition of $e_{\left[ n\right]
}V$ and since $U_{\left[ n\right] }\widetilde{\eta }_{\left[ n\right]
}\Gamma _{\left[ n\right] }V$ is split-injective, its retraction being $%
\gamma _{\left[ n\right] }V$, we get $\mathrm{Im}\left( \mathrm{Id}-e_{\left[
n\right] }V\right) =\mathrm{Ker}\left( \gamma _{\left[ n\right] }V\right) .$
Hence $S_{\left[ n+1\right] }V=\frac{S_{\left[ n\right] }V}{\left\langle
\mathrm{Ker}\left( \gamma _{\left[ n\right] }V\right) \right\rangle }.$

In order to give explicitly $\mathrm{Ker}\left( \gamma _{\left[ n\right]
}V\right) $ aind to get a complete description of the functors $\Gamma _{%
\left[ n\right] },$ let us take a closer look at $\gamma _{\left[ n\right] }$%
. By construction (see Theorem \ref{thm:main}), we have that $\gamma _{\left[
n\right] }:=\omega _{\left[ n\right] }^{+}\circ \xi S_{\left[ n\right] }:PS_{%
\left[ n\right] }\rightarrow \mathrm{Id}$ where $\omega _{\left[ n\right]
}^{+}:\Omega ^{+}\mho ^{+}S_{\left[ n\right] }\rightarrow \mathrm{Id}$
(denoted by $\gamma _{\left[ n\right] }^{\prime }$ in the quoted theorem as
it stems from $\gamma ^{\prime }=\omega ^{+}$) is defined iteratively by
the following equality $\omega _{\left[ n+1\right] }^{+}\circ \Omega
^{+}\mho ^{+}\pi _{\left[ n,n+1\right] }\Gamma _{\left[ n+1\right] }=\omega
_{\left[ n\right] }^{+}.$ Since $\Omega ^{+}\mho ^{+}\pi _{\left[ n,n+1%
\right] }\Gamma _{\left[ n+1\right] }$ is surjective, we get that $\omega _{%
\left[ n\right] }^{+}V:\Omega ^{+}\mho ^{+}S_{\left[ n\right] }V=\left( S_{%
\left[ n\right] }V\right) ^{+}\rightarrow V$ is just the projection onto $V$
passed to the quotient. Since $\gamma _{\left[ n\right] }:=\omega _{\left[ n%
\right] }^{+}\circ \xi S_{\left[ n\right] }$ we get that $\gamma _{\left[ n%
\right] }V:PS_{\left[ n\right] }V\rightarrow V$ is still the projection onto
$V.$ As a consequence $\mathrm{Ker}\left( \gamma _{\left[ n\right] }V\right)
$ is spanned by the homogeneous elements of $PS_{\left[ n\right] }V$ of
degree at least two.

Note that, if we forget the structure of Yetter-Drinfeld module and we just
keep the underlying braided bialgebra structure, the braided bialgebra $S_{%
\left[ n\right] }V$ is exactly what in \cite[Definition 3.10]{Ar-OntheComb}
was denoted by $S^{\left[ n\right] }\left( B\right) $ for $B:=\widetilde{T}%
V. $ As a consequence, the direct limit of the direct system
\begin{equation*}
\widetilde{T}V\rightarrow S_{\left[ 1\right] }V\rightarrow S_{\left[ 2\right]
}V\rightarrow \cdots
\end{equation*}%
is the Nichols algebra $\mathcal{B}\left( V,c\right) $ (\cite[Corollary 3.17
and Remark 5.4]{Ar-OntheComb})$,$ where $c:V\otimes V\rightarrow V\otimes V$
is the braiding of $V$ in $_{H}^{H}\mathcal{YD}$.
\end{example}

\begin{remark}
In the previous example $S_{\left[ n+1\right] }V$ is obtained by factoring
out $S_{\left[ n\right] }V$ by the two-sided ideal generated by the
primitive elements in $S_{\left[ n\right] }V$ of degree at least two.
Following \cite[Definition 4.1 and Section 5]{Ar-OntheComb}, we get that the
combinatorial rank of $V,$ regarded as braided vector space through the
braiding $c$ of $_{H}^{H}\mathcal{YD}$ as above, is the smallest $n$ such
that $\pi _{\left[ n,n+1\right] }\Gamma _{\left[ n+1\right] }V:S_{\left[ n%
\right] }V\rightarrow S_{\left[ n+1\right] }V$ is invertible, if such an $n$
exists. In this case obviously $S_{\left[ n\right] }V=\mathcal{B}\left(
V,c\right) $.

Since, in the setting of Theorem \ref{thm:main}, we can always define $S_{%
\left[ n\right] }:=L_{\left[ n\right] }\Gamma _{\left[ n\right] }:\mathcal{B}%
\rightarrow \mathcal{A}$, for every $B\in \mathcal{B}$ we are lead to the
following definition.
\end{remark}

\begin{definition}
\label{def:combrank}In the setting of Theorem \ref{thm:main}, consider the
functor $S_{\left[ n\right] }:=L_{\left[ n\right] }\Gamma _{\left[ n\right]
}:\mathcal{B}\rightarrow \mathcal{A}$. We define the \textbf{combinatorial
rank} of an object $B\in \mathcal{B}$ (with respect to the adjunction $%
\left( L,R\right) $) to be the smallest $n$ such that $\pi _{\left[ n,n+1%
\right] }\Gamma _{\left[ n+1\right] }B:S_{\left[ n\right] }B\rightarrow S_{%
\left[ n+1\right] }B$ is invertible, if such an $n$ exists.
\end{definition}

\begin{remark}
Thus a concept of combinatorial rank can be introduced and investigated in
this very general setting in which there are neither bialgebras nor braided
vector spaces but just an adjunction $\left( L,R\right) $ as in Theorem \ref%
{thm:main}. Note that, by Lemma \ref{lem:rank}, the morphism $\pi _{\left[
n,n+1\right] }\Gamma _{\left[ n+1\right] }B$ is invertible if and only if
either $\gamma _{\left[ n\right] }B$ or $\eta _{\left[ n\right] }\Gamma _{%
\left[ n\right] }B$ is invertible.

As we will see below, a case of interest is the one in which all objects in $%
\mathcal{B}$ have combinatorial rank at most one, equivalently $\eta _{\left[
1\right] }\Gamma _{\left[ 1\right] }:\Gamma _{\left[ 1\right] }\rightarrow
R_{\left[ 1\right] }S_{\left[ 1\right] }$ is invertible. Since, to this aim,
only the functor $\Gamma _{\left[ 1\right] }$ is needed, we can even more
relax our assumptions by taking just an adjunction $(L,R)$ with an
augmentation $\gamma :RL\rightarrow \mathrm{Id}$ for the associated monad,
avoiding the setting of Theorem \ref{thm:main} and define directly $\Gamma _{%
\left[ 1\right] }$ by $\Gamma _{\left[ 1\right] }B:=(B,\gamma B)$.
\end{remark}

\begin{theorem}
\label{thm:bound}In the setting of Theorem \ref{thm:main}, if the adjunction
$\left( L_{N},R_{N}\right) $ is idempotent for some $N\in \mathbb{N}$, then
every object in $\mathcal{B}$ has combinatorial rank at most $N$ with
respect to the adjunction $\left( L,R\right) .$ In particular the length of
the monadic decomposition of $R:\mathcal{A}\rightarrow \mathcal{B}$ is an
upper bound for the combinatorial rank of objects in $\mathcal{B}$ with
respect to the adjunction $\left( L,R\right) .$
\end{theorem}

\begin{proof}
The fact that the adjunction $\left( L_{N},R_{N}\right) $ is idempotent is
equivalent to require that $\eta _{N}U_{N,N+1}$ is an isomorphism. By
Proposition\ref{pro:Gamman}, we have that $\Gamma _{\left[ N\right]
}=\Lambda _{N}\Gamma _{N}$ and $U_{N,N+1}\circ \Gamma _{N+1}=\Gamma _{N}.$
Thus $\eta _{N}\Gamma _{N}=\eta _{N}U_{N,N+1}\Gamma _{N+1}$ is an
isomorphism. As in the proof of Theorem \ref{teo:main}, we get $R_{\left[
N\right] }\lambda _{N}\circ \eta _{\left[ N\right] }\Lambda _{N}=\Lambda
_{N}\eta _{N}.$ In particular we get $R_{\left[ N\right] }\lambda _{N}\Gamma
_{N}\circ \eta _{\left[ N\right] }\Lambda _{N}\Gamma _{N}=\Lambda _{N}\eta
_{N}\Gamma _{N}$ i.e. $R_{\left[ N\right] }\lambda _{N}\Gamma _{N}\circ \eta
_{\left[ N\right] }\Gamma _{\left[ N\right] }=\Lambda _{N}\eta _{N}\Gamma
_{N}.$ Since $\eta _{N}\Gamma _{N}$ and $\lambda _{N}$ are invertible, we
get that $\eta _{\left[ N\right] }\Gamma _{\left[ N\right] }$ is invertible.
By the foregoing, every object in $\mathcal{B}$ has combinatorial rank at
most $N.$

If $R$ has a monadic decomposition of monadic length $N$, then $L_{N}$ is
fully faithful i.e. $\eta _{N}$ is invertible. Thus, in particular, $\eta
_{N}U_{N,N+1}$ is an isomorphism and hence $\left( L_{N},R_{N}\right) $ is
idempotent. As a consequence every object in $\mathcal{B}$ has combinatorial
rank at most $N.$
\end{proof}

\begin{corollary}
\label{coro:MM}Let $\mathcal{M}$ be a symmetric MM-category in the sense of
\cite[Definition 7.4]{AM-MM}. Then every object in $\mathcal{M}$ has
combinatorial rank at most one with respect to the adjunction $(\widetilde{T}%
,P).$
\end{corollary}

\begin{proof}
By hypothesis all the requirements of Section \ref{sec:6} are satisfied so
that the adjunction $(\widetilde{T},P)$ is in the setting of Theorem \ref%
{thm:main}. By \cite[Theorem 7.2]{AM-MM} the adjunction $(\widetilde{T}%
_{1,}P_{1})$ is idempotent. We conclude by Theorem \ref{thm:bound}.
\end{proof}

As a consequence all the symmetric MM-categories given in \cite[Section 9]%
{AM-MM} have objects with combinatorial rank at most one.

\begin{example}
\label{ex:vec}Consider the particular case when $\mathcal{M}$ is the
category $\mathrm{Vec}$ of vector spaces over a field $\Bbbk $. Since $%
\mathrm{Vec}$ is just $_{H}^{H}\mathcal{YD}$ in case $H$ is the trivial Hopf
algebra $\Bbbk $, this is a particular case of Example \ref{ex:YD}. The
diagram above can be more easily written as follows
\begin{equation*}
\xymatrixcolsep{1.5cm}\xymatrixrowsep{0.7cm}\xymatrix{\Bialg\ar[r]^-{\mho
^{+}}\ar@<.5ex>@{.>}[d]^{P}&\Alg^+\ar@<.5ex>@{.>}[d]^{\Omega ^{+}}\\
\Vec\ar[r]^{\id}\ar@<.5ex>[u]^{\widetilde{T}}&\Vec\ar@<.5ex>[u]^{T^+}}
\end{equation*}

\begin{invisible}
\begin{equation*}
\begin{array}{ccc}
\mathrm{Bialg} & \overset{\mho ^{+}}{\longrightarrow } & \mathrm{Alg}^{+} \\
\widetilde{T}\uparrow \downarrow P &  & T^{+}\uparrow \downarrow \Omega ^{+}
\\
\mathrm{Vec} & \overset{\mathrm{Id}}{\longrightarrow } & \mathrm{Vec}%
\end{array}%
\end{equation*}%
Let $V\in \mathrm{Vec}$ and $\left( A,\varepsilon \right) \in \mathrm{Alg}%
^{+}.$ The object $\widetilde{T}V$ is the usual tensor algebra $TV$ endowed
with its canonical bialgebra structure that makes primitive the elements in $%
V$. Again, by definition, $T^{+}V=\left( TV,\varepsilon _{TV}\right) $ and $%
\left( \Omega ^{+}\left( A,\varepsilon \right) ,kA\right) =\mathrm{Ker}%
\left( \Omega \varepsilon \right) $ is nothing but the augmentation ideal $%
A^{+}$. The monad $\Omega ^{+}T^{+}$ is augmented via the morphism $\omega
^{+}:\Omega ^{+}T^{+}\rightarrow \mathrm{Id}$ of Lemma \ref{lem:augmtilde}
that on components is just the restriction to $\Omega ^{+}T^{+}V=\left(
TV\right) ^{+}$ of the canonical projection $\omega V:\Omega TV\rightarrow V$
onto $V$.By Theorem \ref{thm:main}, we get an augmentation $\gamma :=\omega
^{+}\circ \xi \widetilde{T}:P\widetilde{T}\rightarrow \mathrm{Id}$ for the
monad $P\widetilde{T}$ that on components is the restriction of $\omega V$
to the subspace of primitive elements of $\widetilde{T}V$. As above for
every $n\in \mathbb{N}$, there are a functor $\Gamma _{\left[ n\right] }:%
\mathrm{Vec}\rightarrow \mathrm{Vec}_{\left[ n\right] }$ and a natural
transformation $\gamma _{\left[ n\right] }:P\widetilde{T}_{\left[ n\right]
}\Gamma _{\left[ n\right] }\rightarrow \mathrm{Id}.$
\end{invisible}

As above we can define $S_{\left[ n\right] }:=\widetilde{T}_{\left[ n\right]
}\Gamma _{\left[ n\right] }:\mathrm{Vec}\rightarrow \mathrm{Bialg.}$ Thus $%
S_{\left[ 0\right] }:=\widetilde{T}$ and $S_{\left[ n+1\right] }V=\frac{S_{%
\left[ n\right] }V}{\left\langle \mathrm{Ker}\left( \gamma _{\left[ n\right]
}V\right) \right\rangle }$ is obtained by factoring out $S_{\left[ n\right]
}V$ by the two-sided ideal generated by the homogeneous primitive elements
in $S_{\left[ n\right] }V$ of degree at least two. Note that the procedure
we used to compute $S_{\left[ 1\right] }V=\frac{\widetilde{T}V}{\left\langle
\mathrm{Ker}\left( \gamma V\right) \right\rangle }$ is essentially the same
used to compute $L_{1}V_{1}$ in the proof of \cite[Theorem 3.4]%
{AGM-MonadicLie1}.

By \cite[Definition 6.8 and Theorem 6.13]{Ar-OntheComb}, if $char\left(
\Bbbk \right) =0,$ and \cite[Example 3.13]{Ardi-Univ}, if $char\left( \Bbbk
\right) =p,$ we get that $V,$ regarded as a braided vector space via the
braiding $c:V\otimes V\rightarrow V\otimes V:x\otimes y\mapsto y\otimes x$
of $\mathrm{Vec},$ has combinatorial rank at most one. Thus $\mathrm{Vec}$
is an example of braided monoidal category where every object has
combinatorial rank at most one with respect to the adjunction $(\widetilde{T}%
,P).$

By the foregoing $S_{\left[ 1\right] }V$ coincides with the Nichols algebra $%
\mathcal{B}\left( V,c\right) $ and all the maps $\pi _{\left[ 1,2\right]
}\Gamma _{\left[ 2\right] }V:S_{\left[ 1\right] }V\rightarrow S_{\left[ 2%
\right] }V$, $\gamma _{\left[ 1\right] }V:PS_{\left[ 1\right] }V\rightarrow V
$ and $U_{\left[ 1\right] }\eta _{\left[ 1\right] }\Gamma _{\left[ 1\right]
}V:V\rightarrow RS_{\left[ 1\right] }V$ are invertible. By Lemma \ref%
{lem:rank}, we have that $\pi _{\left[ n,n+1\right] }\Gamma _{\left[ n+1%
\right] }B$ is invertible for all $n\geq 1$ and hence $\gamma _{\left[ n%
\right] }V$ is invertible for all $n\geq 1.$

\label{ex:YDrank2}In Example \ref{ex:vec} we observed that $\gamma _{\left[ 1%
\right] }V:PS_{\left[ 1\right] }V\rightarrow V$ (equivalently $U_{\left[ 1%
\right] }\eta _{\left[ 1\right] }\Gamma _{\left[ 1\right] }V:V\rightarrow
RS_{\left[ 1\right] }V)$ is an isomorphism for $\mathcal{M}=\mathrm{Vec}$.
This fact may fail to be true if we change $\mathcal{M}$. For instance, let
us come back to the category $_{H}^{H}\mathcal{YD}$. By the foregoing we
have $S_{\left[ 1\right] }V=\frac{\widetilde{T}V}{\left\langle \mathrm{Ker}%
\left( \gamma V\right) \right\rangle }$ where $\gamma V:P\widetilde{T}%
V\rightarrow V$ is the projection on degree one and hence $\mathrm{Ker}%
\left( \gamma V\right) $ are the elements of $P\widetilde{T}V$ of degree at
least two. In order to see that the projection $\gamma _{\left[ 1\right]
}V:PS_{\left[ 1\right] }V\rightarrow V$ and the injection $U_{\left[ 1\right]
}\eta _{\left[ 1\right] }\Gamma _{\left[ 1\right] }V:V\rightarrow RS_{\left[
1\right] }V$ need not to be invertible we refer to \cite[Section 7]%
{Ar-OntheComb} where examples of braided vector spaces of combinatorial rank
greater than two, arising as object in $_{H}^{H}\mathcal{YD}$ and braided
via the braiding of $_{H}^{H}\mathcal{YD},$ are given.

\begin{invisible}
The fact that the combinatorial rank is greater than two means that the
projection $\pi _{\left[ 1,2\right] }\Gamma _{\left[ 2\right] }V:S_{\left[ 1%
\right] }V\rightarrow S_{\left[ 2\right] }V$ is not invertible and hence the
projection $\gamma _{\left[ 1\right] }V:PS_{\left[ 1\right] }V\rightarrow V$
is not injective and the injection $U_{\left[ 1\right] }\eta _{\left[ 1%
\right] }\Gamma _{\left[ 1\right] }V:V\rightarrow RS_{\left[ 1\right] }V$ is
not surjective.
\end{invisible}

It would be of interest to determine which conditions on $H$ guarantee that $%
U_{\left[ 1\right] }\eta _{\left[ 1\right] }\Gamma _{\left[ 1\right]
}V:V\rightarrow RS_{\left[ 1\right] }V$ is always invertible for every $V\in
{}_{H}^{H}\mathcal{YD}$, equivalently any object in $_{H}^{H}\mathcal{YD}$
has combinatorial rank at most one.
\end{example}

\begin{remark}
\label{rem:Milnor-Moore} In \cite[Theorem 3.4]{AGM-MonadicLie1} we showed
that the functor $P$ in case $\mathcal{M}=\mathrm{Vec}$ admits a monadic
decomposition of length at most two, represented in the following diagram.
\begin{equation*}
\xymatrix{\Bialg\ar@<.5ex>[d]^{P}&\Bialg\ar@<.5ex>[d]^{P_1}\ar[l]_{%
\mathrm{Id}}&\Bialg\ar@<.5ex>[d]^{P_2}\ar[l]_{\mathrm{Id}}\\
\Vec\ar@<.5ex>@{.>}[u]^{\widetilde{T}}&\Vec_1\ar@<.5ex>@{.>}[u]^{%
\widetilde{T}_1} \ar[l]_{U_{0,1}}&\Vec_2
\ar@<.5ex>@{.>}[u]^{\widetilde{T}_2} \ar[l]_{U_{1,2}}}
\end{equation*}%
This result was obtained by proving first that the adjunction $(\widetilde{T}%
_{1},P_{1})$ is idempotent or equivalently that $\widetilde{\eta }%
_{1}U_{1,2} $ is an isomorphism. Note that, by \cite[Proposition 2.3]%
{AGM-MonadicLie1}, we can take $\widetilde{T}_{2}:=\widetilde{T}_{1}U_{1,2},$
$U_{1,2}\widetilde{\eta }_{2}=\widetilde{\eta }_{1}U_{1,2}$ and $\widetilde{%
\epsilon }_{2}=\widetilde{\epsilon }_{1}.$ We have seen in \cite[Theorems
7.2 and 8.1]{AM-MM} and \cite[Theorem 3.3]{AGM-Restricted} that the category
$\mathrm{Vec}_{2}$ is equivalent to the category $\mathrm{Lie}$ of
(restricted) Lie algebras over $\Bbbk $ and that the adjunction $(\widetilde{%
T}_{2},P_{2})$ plays the role of the usual adjunction, between the
categories $\mathrm{Bialg}$ and $\mathrm{Lie}$, given by the (restricted)
universal enveloping algebra functor and the primitive functor. The fact
that the monadic decomposition has length at most two means that the unit $%
\widetilde{\eta }_{2}:\mathrm{Id}\rightarrow P_{2}\widetilde{T}_{2}$ is
invertible. In view of the identifications we mentioned, this is the
counterpart of half of the Milnor--Moore theorem \cite[Theorems 5.18(1) and
6.11(1)]{MM}. Now, given $V_{2}:=\left( V,\mu ,\mu _{1}\right) \in \mathrm{%
Vec}_{2}$, with $\mu :P\widetilde{T}V\rightarrow V,$ $V_{1}:=\left( V,\mu
\right) $ and $\mu _{1}:P_{1}\widetilde{T}_{1}V_{1}\rightarrow V_{1},$ one
has $\mu _{1}\circ \widetilde{\eta }_{1}=\mathrm{Id}$ and hence $\mu
_{1}=\left( \widetilde{\eta }_{1}V_{1}\right) ^{-1}$ (note that $\widetilde{%
\eta }_{1}V_{1}=\widetilde{\eta }_{1}U_{1,2}V_{2}$ is invertible). Moreover $%
\widetilde{T}_{2}V_{2}=\widetilde{T}_{1}U_{1,2}V_{2}=\widetilde{T}_{1}V_{1}$%
. Following the proof of \cite[Theorem 3.4]{AGM-MonadicLie1}, we can compute
explicitly $\widetilde{T}_{1}V_{1}\ $ as $\frac{\widetilde{T}V}{\left\langle
z-\mu \left( z\right) \mid z\in EV\right\rangle },$ where $EV$ denotes the
subspace of $P\widetilde{T}V$ spanned by element of homogeneous degree
greater than one, and hence we obtain that $\widetilde{T}_{1}V_{1}=U\left(
V,c,\mu \right) $ in the sense of \cite[Definition 3.5]{Ardi-PrimGen}, where
$c:V\otimes V\rightarrow V\otimes V:x\otimes y\mapsto y\otimes x$ is the
braiding of $\mathrm{Vec}$.

\begin{invisible}
$\widetilde{T}_{1}V_{1}$ is defined by the following coequalizer.%
\begin{equation*}
\widetilde{T}P\widetilde{T}V\overset{\widetilde{T}\mu }{\underset{\widetilde{%
\epsilon }\widetilde{T}V}{\rightrightarrows }}\widetilde{T}V\overset{\pi
_{0,1}V_{1}}{\longrightarrow }\widetilde{T}_{1}V_{1}.
\end{equation*}%
As in Example \ref{ex:vec}, we get that $\widetilde{T}_{1}V_{1}$ is the
quotient of $\widetilde{T}V$ by the two-sided ideal generated by $\mathrm{Id}%
_{P\widetilde{T}V}-\widetilde{\eta }V\circ \mu $ so that $\widetilde{T}%
_{1}V_{1}=\frac{\widetilde{T}V}{\left\langle z-\mu \left( z\right) \mid z\in
P\widetilde{T}V\right\rangle }.$
\end{invisible}

Note that, in the same quoted definition, it is set $S\left( V,c\right)
:=U\left( V,c,0\right) =\frac{\widetilde{T}V}{\left\langle z\mid z\in
EV\right\rangle }.$ Clearly $S\left( V,c\right) $ coincides with $S_{\left[ 1%
\right] }V$ of Example \ref{ex:vec}. In \cite[Corollary 5.5]{Ardi-PrimGen}
it is proved that $PU\left( V,c,\mu \right) \cong V$ using the fact that $%
PS\left( V,c\right) \cong V.$ In view of the above identifications, the
latter isomorphism means that $U_{\left[ 1\right] }\eta _{\left[ 1\right]
}\Gamma _{\left[ 1\right] }V:V\rightarrow PL_{\left[ 1\right] }\Gamma _{%
\left[ 1\right] }V=PS_{\left[ 1\right] }V$ is invertible and we already
observed that this is another way to say that $V$ has combinatorial rank at
most one (the primitive elements in $PS_{\left[ 1\right] }V$ are
concentrated in degree one). On the other hand, the first isomorphism
implies that $U_{1}\widetilde{\eta }_{1}V_{1}:V\rightarrow P\widetilde{T}%
_{1}V_{1}$ is invertible for any $V_{2}\in \mathrm{Vec}_{2}$. Equivalently $%
U_{1}\widetilde{\eta }_{1}U_{1,2}$ is invertible which is the same as
requiring that $\widetilde{\eta }_{1}U_{1,2}$ is invertible i.e. the
condition, recalled above, saying that the adjunction $(\widetilde{T}%
_{1},P_{1})$ is idempotent. Summing up, using that any object in $\mathrm{Vec%
}$ has combinatorial rank at most one, we can prove that $(\widetilde{T}%
_{1},P_{1})$ is idempotent and hence that $P$ has monadic decomposition of
length at most two.

As mentioned, we can consider an adjunction $\left( L,R\right) $ whose
associated monad is augmented. If every object in $\mathcal{B}$ has
combinatorial rank at most one, it is natural to wonder if, also in this
wider setting, it is true that $\left( L_{1},R_{1}\right) $ is idempotent
and hence $R$ has monadic decomposition of length at most two. In this way
the adjunction $\left( L_{2},R_{2}\right) $ would be involved in an analogue
of the Milnor--Moore theorem in the above sense. More generally one can ask
whether $\left( L_{N},R_{N}\right) $ is idempotent in case the combinatorial
rank of objects in $\mathcal{B}$ for an adjunction $\left( L,R\right) $ as
in Theorem \ref{thm:main} is at most $N\in \mathbb{N}$. This would provide
an inverse of Theorem \ref{thm:bound}.
\end{remark}

\end{document}